\documentclass[10pt, a4paper]{amsart}

\numberwithin{equation}{section}

\usepackage{dsfont}
\newcommand\Id{{\mathds{1}}}
\newcommand{\bphi}{{\Phi}}
\newcommand{\brho}{{\rho}}

\newcommand{\BB}{\mathcal{B}}
\newcommand{\CC}{\mathcal{C}}
\newcommand{\DD}{\mathcal{D}}
\newcommand{\EE}{\mathcal{E}} 
\newcommand{\FF}{\mathcal{F}} 
 
\newcommand{\FFF}{\mathbf{F}} 

\newcommand{\HHH}{\mathbf{H}}
\newcommand{\II}{\mathcal{I}}
\newcommand{\JJ}{\mathcal{J}}
\newcommand{\KK}{\mathcal{K}}
\newcommand{\LL}{\mathcal{L}}
\newcommand{\MM}{\mathcal{M}}
\newcommand{\NN}{\mathcal{N}}
\newcommand{\PP}{\mathcal{P}}

\newcommand{\RR}{\mathcal{R}}

\newcommand{\TT}{\mathcal{T}}
\newcommand{\AAc}{\mathcal{A}}

\newcommand{\ii}{\mathbf{i}}

\newcommand{\id}{{\rm id}}

\newcommand{\real}{\mathbb{R}}
\newcommand{\complex}{\mathbb{C}}
\newcommand{\integer}{\mathbb{Z}}
\newcommand{\AAA}{\mathbb{A}}

\newcommand{\MMM}{\mathbb{M}}
\newcommand{\PPP}{\mathbb{P}}

\newcommand{\Cs}{C_{\#}}
\newcommand{\epsilons}{\epsilon_{\#}}
\newcommand{\st}{\;|\;}
\DeclareMathOperator{\Ker}{Ker}

\newcommand{\TTT}{\mathbb{T}}
\newcommand{\UU}{\mathcal{U}}
\newcommand{\VV}{\mathcal{V}}

\newcommand{\ZZ}{\mathcal{Z}}

\renewcommand{\tilde}{\widetilde}

\newcommand{\norm}[2]{\left\| #1 \right\|_{#2}}

\DeclareMathOperator{\Card}{Card}
\DeclareMathOperator{\supp}{supp} 
\newcommand{\transposee}[1]{{\vphantom{#1}}^{\mathit t}{#1}}

\newtheorem{proposition}{Proposition}[section]
\newtheorem{lemma}[proposition]{Lemma}

\newtheorem{theorem}[proposition]{Theorem}

\newtheorem{corollary}[proposition]{Corollary}

\theoremstyle{remark}

\newtheorem{remark}[proposition]{Remark}

\theoremstyle{definition}
\newtheorem{definition}[proposition]{Definition}

\usepackage{amssymb}
\usepackage{enumerate, xspace}
\usepackage{pstricks}
\usepackage{tikz}   

\newcommand\cB{{\mathcal B}}
\newcommand\cC{{\mathcal C}}
\newcommand\cD{{\mathcal D}}

\newcommand\cL{{\mathcal L}}

\newcommand\cO{{\mathcal O}}
\newcommand\cP{{\mathcal P}}
\newcommand\cR{{\mathcal R}}

\newcommand\cW{{\mathcal W}}

\newcommand\bA{{\mathbb A}}

\newcommand\bC{{\mathbb C}}

\newcommand\bF{{\mathbb F}}
\newcommand\bG{{\mathbb G}}
\newcommand\bN{{\mathbb N}}

\newcommand\bR{{\mathbb R}}
\newcommand\bT{{\mathbb T}}
\newcommand\bW{{\mathbb W}}
\newcommand\bZ{{\mathbb Z}}

\newcommand\frp{{\mathfrak p}}

\newcommand\ve{\varepsilon}

\newcommand\vf{\varphi}
\newcommand{\blambda}{{\bar \lambda}}

\newcommand{\cu}{\upsilon}
\newcommand\up{\varkappa}
\newcommand\vu{{\tau_-}}
\newcommand\vuo{{\tau_0}}
\newcommand\vus{{\tau_+}}
\newcommand\ho{{\varpi}}
\newcommand\X{X_0}
\newcommand\cWo{{\overline\cW}}
\DeclareMathOperator*{\esssup}{ess-sup}

\begin{document}

\title[Decay of correlations for piecewise hyperbolic flows]{Exponential
decay of correlations for  piecewise cone  hyperbolic contact flows}
\author{Viviane Baladi and Carlangelo Liverani}
\address{D.M.A., UMR 8553, \'{E}cole Normale Sup\'{e}rieure,  75005 Paris, France}
\email{viviane.baladi@ens.fr}

\address{Dipartimento di Matematica
Universit\`a degli Studi di Roma Tor Vergata, Italy}
\email{liverani@mat.uniroma2.it}
\thanks{
This work was started in 2008 during a visit of CL to the DMA of Ecole Normale
Sup\'erieure Paris. It was continued in 2009 during a visit (funded by
Grefi-Mefi) of VB to the University Rome 2, in 2010 during stays of both authors at
the Institut Mittag-Leffler Institute (Djursholm, Sweden), and
a visit of VB in the Mathematics Center of Lund Technical University (Sweden),
and ended in 2011 during
visits of CL to DMA, Paris and to the Fields Institute, Toronto, Canada. Also CL acknowledges the partial support of the European Advanced Grant  Macroscopic Laws and Dynamical Systems (MALADY) (ERC AdG 246953). We are  grateful to S\'ebastien Gou\"ezel for useful
conversations and for  Lemma~\ref{tricksg}. We thank  Emmanuel Giroux,
Patrick Massot,   Mikael de la Salle, and Tomas Persson, for friendly advice on contact geometry, 
contact geometry again, complex interpolation,
and complexity growth, respectively. We are also grateful to the anonymous referees for constructive comments which led to an improved presentation}
\begin{abstract}
We prove exponential decay of correlations for a realistic model of
piecewise hyperbolic flows preserving a contact
form, in dimension three. This is the first time exponential  decay of correlations
is proved for continuous-time dynamics with singularities on a manifold. 
Our proof combines
the second author's version \cite {Li} of Dolgopyat's estimates for contact
flows and the first author's work with Gou\"ezel 
\cite{BG2} on piecewise hyperbolic
discrete-time dynamics.
\end{abstract}
\date{March 15, 2012, revised following the referees' suggestions}
\maketitle
\section{Introduction, definitions, and statement of the main theorem}
\label{intro}

\subsection{Introduction}

Many chaotic continuous-time dynamical systems
posess an ergodic physical (or SRB) measure
\cite{YY}, the simplest case
being when volume is preserved (and ergodic). When such systems are mixing, it is
natural to ask at which speed decay of correlations
takes place, for H\" older observables, say.
Controlling the rate of decay of correlations is notouriously more difficult
for continuous-time  than for discrete-time dynamics:
For mixing smooth Anosov (uniformly hyperbolic) flows,
exponential decay of correlations was obtained only in the late
nineties, in dimension three (or under a  bunching assumption)
in a groundbreaking work of 
Dolgopyat \cite{Do}, while
the analogous result for Anosov diffeomorphisms had been
known for almost twenty years, see e.g.  \cite{YY}. 
Dolgopyat's result implies 
that geodesic flows on surfaces of variable strictly negative curvature
are exponentially mixing, a generalisation of the
result  in  constant
negative curvature (\cite{Moo}, \cite{Ra}) obtained
more than a dozen years before.
Liverani \cite{Li} was then able to discard the bunching assumption when 
the flow preserved a contact form, generalising Dolgopyat's result to
geodesic flows on manifolds of variable strictly negative curvature
in any dimension.

Some natural chaotic systems are only piecewise smooth. The most
prominent example is given by dispersive (Sinai) billiards (see \cite{CM}
and references therein). Exponential decay of correlations was obtained
for a discrete-time version of the billiard 
(and other piecewise hyperbolic maps)
by Young \cite{Yo} (see also the work of Chernov
\cite{Ch0}--\cite{Ch} and earlier work of Liverani \cite{Li0} on piecewise hyperbolic maps).
For the actual billiard flow, only a stretched exponential upper bound
is known, recently proved by Chernov \cite{Ch2}.
(Since then, Melbourne \cite{Me} has proved superpolynomial decay
of correlations ---  a weaker result --- in a more general setting.)

Some results of exponential decay of correlations
for piecewise hyperbolic flows or semi-flows do exist, but
under assumptions which  avoid the main difficulties, 
making them unfit for generalisation to
realistic flows with singularities (such as the Sinai billiard):
Baladi--Vall\'ee  \cite{BV} extended Dolgopyat's results
to some systems with infinite Markov partitions,
but although this idea could later be applied to Teichm\"uller flows 
(\cite{AGY}) and Lorenz-type flows
(\cite{AV}), it does not seem applicable to billiards (in the infinite
Markov partition given by Young's tower \cite{Yo}, the relation
between the metric and the initial euclidean structure is lost, making
it impossible to exploit  uniform non joint integrability).
Stoyanov \cite{Sto} obtained exponential decay for 
{\em open} billiard flows,
where  the discontinuities in fact do not play a role,
and for Axiom A flows with $C^1$ laminations \cite{Sto2}.

We believe that a new approach is needed to attack 
exponential decay of correlations for chaotic flows with
singularities. Most proofs
\footnote{An important exception is given by the ``coupling" methods introduced in this context by L-S.Young in \cite{Yo1} and greatly generalised by D.Dolgopyat in \cite{Do1}  giving rise to the ``coupling of standard pairs" method. See Chernov and others \cite{CM} for a review on how to obtain exponential mixing for the discrete-time Sinai billiard via coupling of standard pairs. Implementing this strategy for flows, let alone the billiard flow, does not currently seem an approachable task.} of exponential decay of correlations
for a map $F$, or a flow $T_t$, boil down to a spectral gap for a transfer
operator (or a one parameter semigroup of operators). Classical
approaches \cite{Yo} first reduce the hyperbolic dynamics $F$ or $T_t$
to an expanding Markov system, and a great amount
of information is lost in this procedure. We think that studying
the original transfer operators
$$
\LL \psi = \frac{\psi \circ F^{-1}}{|\det DF| \circ F^{-1}} \, , \qquad
\LL_t \psi = \frac{\psi \circ T_{-t} }{|\det DT_{t}| \circ T_{-t}}
$$ 
on a suitable space of (anisotropic)
distributions on the manifold will be the
key to obtaining exponential decay of correlations for 
many  systems which have resisted
the traditional techniques.

Appropriate 
anisotropic Banach spaces were first
introduced by Blank, Keller, and Liverani \cite{BKL} 
to give a new proof of exponential decay of correlations
for smooth Anosov diffeomorphisms (as well as other
results). This approach was developed in the
next few years for smooth discrete-time
hyperbolic dynamics by Baladi \cite{Cinfty},
Gou\"ezel-Liverani \cite{GL1}--\cite{GL2}, and  Baladi--Tsujii
\cite{BT1}--\cite{BT2}, and more recently
for some smooth hyperbolic flows
(Butterley-Liverani \cite{BuL}, Tsujii \cite{Tsujii1} \cite{Tsujii2}).
Except for the Sobolev--Triebel spaces used in \cite{Cinfty}
(where a strong assumption of regularity of the dynamical
foliation was required), it
turns out that the Banach spaces appropriate for smooth
hyperbolic dynamics are not suitable for systems with discontinuities, because
multiplication by the characteristic function of a domain, however
nice, is not a bounded operator for the corresponding norms.
For this reason, we unfortunately cannot exploit directly
Tsujii's \cite{Tsujii2} remarkable  work on smooth contact hyperbolic
flows, which would give much more than exponential decay.
Very recently, new anisotropic spaces which satisfy this
bounded multiplier property were introduced by
Demers--Liverani \cite{DL}  and
Baladi--Gou\" ezel \cite{BG1}--\cite{BG2} to obtain
in particular exponential decay of correlations for various
piecewise hyperbolic maps.
The approach of \cite{BG1}--\cite{BG2} consists in adapting the
Sobolev--Triebel space results of \cite{Cinfty} to the piecewise hyperbolic
case, exploiting a key work of Strichartz \cite{Str}, and using families
of (noninvariant) foliations to replace the actual stable foliation, which
is only measurable in general for piecewise smooth systems.

In the present paper, building on the analysis from
\cite{BG2},
we introduce anisotropic spaces adapted to 
piecewise hyperbolic {\em continuous-time}
systems. Using these spaces, we then adapt Liverani's \cite{Li} version
of the Dolgopyat estimate
for Anosov contact flows  to obtain the first result of exponential
decay of correlations for hyperbolic systems with (true)
singularities. Our result applies to various natural examples
(Subsection ~\ref{ex2c}), and we explain in Remark~\ref{rkbil} below
how close we are to solving the actual Sinai billiard
flow problem.

After this paper was completed, we learned that
Demers and Zhang \cite{DZ} obtained a
spectral proof of exponential decay of correlations for the
discrete time Sinai billiard.

\tableofcontents

\subsection{Definitions and the main theorem}

In this subsection, we  state our main result and
outline the structure
of our argument (and of the paper), ending by the main formal
definitions.

Let $T_t:X_0 \to X_0$ be a piecewise $C^{2}$ cone hyperbolic flow
defined on a closed subset $X_0$ of
a compact $d$-dimensional ($d =2k+1\ge 3$)
$C^\infty$ manifold $M$
(see Definition ~ \ref{PiecHyp}, and note that $X_0$ is
the union of finitely many closed flow boxes).
Let $\alpha$ be a $C^2$ contact form on $M$. Recall that a flow generated by a vector field $V$ is the {\em Reeb flow} 
of $\alpha$ if $V\in\operatorname{Ker}(d\alpha)$ and $\alpha(V)=1$. This implies $\frac d{dt}\alpha(DT_t v)=0$ for each $v\in TM$, i.e., the flow is contact.
Assume also that $T_t$ is the Reeb flow of   $\alpha$.
In particular, $T_t$ preserves the volume $dx=\wedge^k d\alpha\wedge\alpha$, and  $|\det DT_t|\equiv 1$.
If $M=X_0$ and $T_t$ is a hyperbolic geodesic flow, or more generally
an Anosov flow preserving a contact form $\alpha$, then $T_t$ is  the
Reeb flow of $\alpha$, up to replacing $\alpha$ by $\alpha(V)^{-1}\alpha$, where
$V$ is the vector field generating the flow, see \cite[p. 1496 and \S2]{Tsujii1}.
More generally, in Appendix ~\ref{2c} we show that all ergodic
piecewise smooth  hyperbolic contact flows are Reeb. 
(For example, a billiard flow with speed one is the Reeb flow of the contact form $p\, dq$.)

Our main result is the following theorem
(its  proof can be found at the end of Section \ref{redux}):

\begin{theorem}[Exponential mixing for piecewise hyperbolic contact flows]
\label{main} Let $M$ be a compact $3$-dimensional  manifold.
Let $T_t$ be a piecewise $C^{2}$ cone hyperbolic flow
on  a closed subset $X_0$ of $M$,
which is the Reeb flow of  a $C^2$ contact form $\alpha$.
Assume in addition that  $T_t$ is ergodic for $dx$, that 
complexity grows subexponentially (Definition ~ \ref{domin}),
and that $T_t$ satisfies the transversality condition of Definition~ \ref{transcond}.
Then for each $\xi >0$  there exist $K_\xi>0$ and $\sigma_\xi>0$,
so that for any $C^\xi$ functions $\psi_1, \psi_2:M \to \complex$
$$
|\int \psi_1 (\psi_2 \circ T_t) \, dx - \int \psi_1 \, dx \int \psi_2\, dx|
\le K_\xi \|\psi_1\|_{C^\xi}\|\psi_2\|_{C^\xi} e^{-\sigma_\xi t} \, ,
\quad \forall t\ge 0 \, .
$$
\end{theorem}

Our proof is based on a spectral analysis of 
the linear operator $\LL_t \psi = \psi \circ T_{-t}$, defined 
initially on bounded functions, e.g. (By definition, $\LL_t^*$ preserves $dx$.)
The strategy, following \cite{Li}, is to study $\LL_t$ as an operator on a suitable Banach space $\widetilde \HHH$
of anisotropic distributions, and to prove Dolgopyat-like estimates. 
Just like in \cite{Li}, we do not claim that the transfer operator
$\LL_1$ associated to the time-one map $T_1$ has a spectral gap. However, the resolvent
method we adapt from \cite{Li} gives us a precise description of the spectrum
of the generator $X$ of the semigroup  of operators $\LL_t$ in a 
half-plane large enough to deduce exponential
decay of correlations (see Corollaries ~ \ref{rr} and ~ \ref{specX'}).
The spaces $\widetilde \HHH^{r,s,q}_p$ that we shall use are a modification 
(see Subsection~ \ref{spaces}) of the
spaces  $\HHH^{r,s}_p$ 
($s<0<r$ and $1<p<\infty$) 
of \cite{BG2} for piecewise hyperbolic maps (the spaces in \cite{BG2} generalise earlier
constructions in  \cite{BG1} and \cite{Cinfty}, more directly related
to standard Triebel spaces). In particular, the norm is defined by taking
a supremum over a class of admissible foliations (which are compatible
with the stable cones and satisfy some regularity property).
The main difference between  $\HHH^{r,s}_p$
and $\widetilde \HHH^{r,s,q}_p$, is due to the direction of the time that must be added to
the foliation class (leading to the additional regularity
parameter $q\ge 0$). As a consequence the proof of
the key  Lemma~3.3 from  \cite{BG2}
(invariance of the class of admissible foliations
under the action of the dynamics) had to be rewritten in full detail, because
a new phenomenon appears in continuous time
(see Lemma~\ref{lemcompose}): We get invariance only modulo 
precomposition
by a perturbation $\Delta$ limited to the flow direction.
This can be dealt with, up to a
worsening of the regularity exponents in the time direction
(Lemma~\ref{noglue}). It follows
that the ``bounded" term in our Lasota-Yorke bound (Lemma~\ref{LY0})
is not really bounded.  The ``compact" term in the Lasota--Yorke
bound is not compact either, due to a loss of regularity in the flow direction
of a more elementary origin (see Lemma~\ref{CompositionDure}), which
also played a part in Liverani's \cite{Li} proof.
Like in \cite{Li}, we may overcome these problems because
we work with the resolvent 
$\RR(z)= \int_0^{\infty} e^{-zt} \LL_t \, dt$ which involves
integration along the time direction. A price needs to be paid, in
the form of a power of the imaginary part of $z$ in the estimates,
see Lemma~\ref{bq}, and note that our Lasota-Yorke estimate
for the resolvent is Lemma~\ref{controlCeta}. 
Another difference with respect to \cite{BG2} is that we need to
decompose the time $t$, taking into account the Poincar\'e maps and
the return times, so that the proof of the Lasota-Yorke estimate
Lemma~\ref{LY0} needs to be rewritten in full (the use of the Strichartz bound ~Lemma~\ref{lem:multiplier} in
the argument is also a bit different).
With respect to Liverani's argument \cite{Li} for contact Anosov flows,
the key Dolgopyat estimate Lemma~\ref{dolgolemma}
(leading to Proposition~\ref{dolgo}) uses the same idea
of ``averaging in the (un)stable direction"
(see the definition of $\AAA_\delta$ in
Section ~\ref{molll}).
The main nontrivial difference is that, to prove 
Lemma~\ref{dolgolemma}, instead of the actual strong stable foliation $W^s$
used in \cite{Li}, but which is only measurable in the present setting,
we work with ``fake stable foliations"  
which lie in the stable cones
and {\em belong to the kernel of the contact form}.
This is possible because  the arguments in \cite[\S 6, App B]{Li}
(in particular Lemma B.7 there) do not require the fact that $W^s$ is the
actual invariant foliation of the flow. What matters is that the contact form
$\alpha$ vanishes along the leaves of the fake stable foliation.

This is why, although the contact assumption
is not needed for smooth Anosov flows in dimension three \cite{Do} since the foliation is $C^1$ (by this we mean that the tangent space to the leaves vary in a $C^1$ manner),
the contact assumption is  essential in the present  setting, where the foliation is only measurable. In fact, we show in Appendix \ref{sec:unstable} that, locally, one can effectively approximate the unstable foliation by a Lipschitz foliation, yet this alone does not suffice to apply Dolgopyat's argument. The contact form is our  leverage towards the lower bounds which
yield the  ``oscillatory integral"-type cancellations we need.

For smooth contact flows it is well-known \cite[Thm 3.6]{KB}
that ergodicity implies mixing, and a similar result holds for two-dimensional dispersing billiards (\cite[\S 6.9]{CM} and references therein). Yet, we are not aware of such a general theorem for contact
systems with discontinuities, even though it is probably true. In any case, we do not deduce mixing
directly from ergodicity and the contact property: In our uniformly hyperbolic setting it  follows
from the Dolgopyat estimate, Lemma~\ref{dolgolemma}, which   gives our stronger spectral/exponential
mixing result. (Our proof is thus organised a bit differently from \cite{Li}, where mixing was given by \cite[Thm 3.6]{KB}.)

We emphasize also that Lemma ~\ref{dolgolemma} does not involve
the anisotropic norms: It is formulated as an upper bound
on the supremum norm, with respect to the supremum
and $H^1_\infty$ norms. We are able to exploit this upper bound by using the
fact that our spaces $\widetilde \HHH^{r,0,0}_p$
(when $s=0$ and $q=0$)
are isomorphic to the ordinary Sobolev spaces
$H^{r}_p$, and by
using mollification operators (see Section ~\ref{molll}), and Sobolev
embeddings.
Finally, note that  we restrict to the three-dimensional setting in
this work to simplify as much as possible the intricate
estimate in Section~\ref{carlangelo}.  The other arguments hold in general
odd dimension $d\ge 3$,  and do not become shorter or simpler for
$d=3$.
We hope that the three-dimensional assumption
limitation can be removed (bunching, however, is necessary with
the present technology, as in \cite{BG2}, to prove invariance
of admissible charts, see 
Appendix~\ref{iteratechart}).

The paper is organised as follows: Subsection ~\ref{ex2c} discusses a simple class of examples to which our result applies.
After introducing the anisotropic Banach spaces in Section ~ \ref{defspace}, we show
in Section ~\ref{redux} how to reduce our
theorem to Lasota-Yorke estimates (Lemma~\ref{LY0}) and Dolgopyat estimates
(Proposition~ \ref{dolgo}, which hinges on
Lemma~\ref{dolgolemma}). Lemma~\ref {LY0} is proved in Section ~ \ref{LYBG}.
In Section~\ref{molll} we study mollification
operators $\MMM_\epsilon$, 
and stable-averaging operators $\AAA_\delta$. 
These operators are used to reduce to Lemma~\ref{dolgolemma},
the bound in Section ~\ref{carlangelo}, which
is the heart of the paper. In Section ~\ref{carlangelo}, we follow the lines of \cite[\S5, \S6]{Li}, but we must take into account the
fact that our Banach spaces are different. Section~ \ref{dodo}
contains the proof of Proposition ~ \ref{dolgo}.
Finally  Appendix ~\ref{2c} contains some useful facts about contact flows and changes of coordinates,  Appendices ~\ref{localspaces} and~ \ref{iteratechart} detail several basic results needed to  construct and study  our Banach spaces, and
Appendices ~\ref{sec:hoihoi} and~ \ref{sec:unstable} contain constructions fundamental for the arguments in Section ~\ref{carlangelo}.
\smallskip

We end  this subsection by defining piecewise $C^2$ cone hyperbolic flows
and the  assumptions needed for our theorem.

\begin{definition}
[Cones in $\real^d$]
A $k$-dimensional cone in $\real^d$, for an integer $1\le k \le d-1$, is
a closed subset $\CC$ of $\real^d$ so that there exists
a linear coordinate system $\real^{d-k}\times \real^k$ and a maximal rank linear
map $\AAc : \real^k  \to \real^d $ for which
\begin{equation}\label{conedeff}
\CC=\{ (x,y) \in \real^{d-k}\times \real^k  \mid |x| \le |\AAc y| \}\, .
\end{equation}
In particular
\footnote{See \cite[Def. 2.1]{BG2} and the remark thereafter for a more general notion
of cone and the corresponding notion of transversality.}, a cone $\CC$
has nonempty interior, it is invariant under scalar
multiplication, and its dimension is the maximal dimension of a vector subspace
included in $\CC$.
If $\CC$ is a  $k$-dimensional cone in $\real^d$
and $\CC'$ is a  $k'$-dimensional cone in $\real^d$
(not necessarily for the same coordinate systems), we
say that $\CC$ and $\CC'$ are transversal if $\CC \cap \CC'=\{0\}$.
We say that $\CC(w)$ depends continuously on $w$ if both the coordinate system
and the map $\AAc(w)$ depend continuously on $w$.
\end{definition}

Note that in our main application to three dimensional flows, 
only one-dimensional cones are
needed.

Let $M$ be a smooth $d$-dimensional compact manifold, with $d=d_u+d_s+1$,
for integers $d_u \ge 1$, $d_s\ge 1$.
Again, the reader only interested in the application to three-dimensional flows can 
focus
on visualising the case $d_u=d_s=1$. (It would not shorten the exposition
to restrict to that case.)

\begin{definition}
[Piecewise $C^2$ cone hyperbolic flows]\label{PiecHyp} 
A measurable flow $T_t: X_0 \to X_0$  is a  piecewise
$C^{2}$ hyperbolic flow  on a closed subset $X_0$ of
$M$ if there exist $\epsilon_0>0$ and finitely many  
codimension-one $C^{2}$ open hypersurfaces $\{O_i, i \in I\}$ of $M$, 
uniformly transversal to the flow direction, and for each $i\in I$,
there exists $J_i\subset I$ so that:

(0)  For each $j \in J_i$ there exists an open subset (in the sense of hypersurfaces)
$O_{i,j} \subset O_i$ so that  $O_i=\cup_{j\in J_i}O_{i,j}$
(modulo a zero Lebesgue measure set),  this union is disjoint,
and each boundary $\partial O_{i,j}$ is a finite union
of codimension-two $C^1$ hypersurfaces.
For each $j\in J_i$ there exists
a $C^2$ real-valued and strictly positive function $\tau_{i,j}$ defined on
a neighbourhood $\widetilde O_{i,j}$ of $\overline{O_{i,j}}$  (as hypersurfaces), so that
$$
T_{t}(w)\in X_0 \, , \, \forall w \in O_{i,j}\, ,
\forall t\in [0, \tau_{i,j}(z))\, ,
\quad T_{\tau_{i,j}(w)}(w)\in O_j \, , \, \forall w \in O_{i,j}\, ,
$$
and, setting,
$$B_{i,j}= \cup_{z \in O_{i,j}} \cup_{t \in [0, \tau_{i,j} (z))} T_t(z)\, ,
$$
the sets $B_{i,j}$, $i\in I$, $j\in J_i$ (called ``flow boxes") are two by two disjoint, and
$$
X_0=\cup_{i\in I}\cup_{j\in J_i} B_{i,j}
\mbox{ (modulo a zero Lebesgue measure set).}$$
For $w \in B_{i,j}$, we let  $z(w)\in O_{i,j}$ and $t(w)\in (0, \tau_{i,j}(z(w))$ be such that
\begin{equation}\label{deftwzw}w=T_{t(w)}(z(w))\, .
\end{equation}
Note that $\tau_{i,j}$ is the restriction to
$O_{i,j}$ of
the first return time
of $T_t$ to the section 
$$M_0:=\cup_k O_k\, .
$$

\smallskip

(1)   
For each $j \in J_i$  there exists 
a neighbourhood ${\widetilde B}_{i,j}$
of the closure of the flow box $B_{i,j}$ and a $C^2$ flow $T_{i,j,t}$ defined in ${\widetilde B}_{i,j}$
such that for each $w\in B_{i,j}$ and 
every $t$ such that $T_t(w)\in B_{i,j}$ we have $T_t(w)=T_{i,j,t}(w)$.
In addition,  there exists a neighbourhood $\widehat O_{i,j}$ in $M$
of the closure of $\widetilde O_{i,j}$  so that $T_{\tau_{i,j}(w)}(w):  O_{i,j}\to O_j$
extends to
a $C^{2}$ map $\PPP_{i,j}:\widehat O_{i,j}\to M$, which is a
diffeomorphism 
onto its image.
The $C^2$ map 
$$P_{i,j}:=\PPP_{i,j}|_{\widetilde O_{i,j}}$$ 
restricted to $O_{i,j}$ is the first return (Poincar\'e) map to the section
$M_0$.

\smallskip

(2) For each
$j \in J_i$,
there exist two continuous families $\CC^{(u)}_{i,j}(w)$ and $\CC^{(s)}_{i,j}(w)$ of
cones
on $\overline{B}_{i,j}$, where
$\CC^{(u)}_{i,j}(w)\subset T_w M$
is $d_u$-dimensional, $\CC^{(s)}_{i,j}(w)\subset T_w M$ is $d_s$-dimensional, 
and, denoting  the flow direction by $\mbox{flowdir}(w)\subset T_wM$,
$$\CC^{(u)}_{i,j}(w)\cap \CC^{(s)}_{i,j}(w)=\{0\}\, ,\quad
\CC^{(u)}_{i,j}(w)\cap \mbox{flowdir}(w)=\mbox{flowdir}(w)\cap \CC^{(s)}_{i,j}(w)
=\{0\}\, ;
$$
in addition, for any $t_{00}>0$, there exist
a smooth norm on $TM$ and continuous
functions $\lambda_{i,j,u}:\overline{B}_{i,j}\to (1, \infty)$
and $\lambda_{i,j,s}:\overline{B}_{i,j}\to (0,1)$  such that, for each $w \in \overline{B}_{ij}$
and each   $t\in (t_{00}, \tau_{i,j}(z(w))-t(w)]$,
letting $(k,\ell)$ be 
\footnote{Note that either $(k,\ell)=(i,j)$,
or $t=\tau_{i,j}(w)$ with $T_{i,j,t}(w) \in \overline {O_{k,\ell}}$ for $k=j$ and $\ell \in J_j$.}
such that $T_{i,j,t}(w)\in B_{k,\ell}$, we have 
$$
DT_{i,j,t}(w) \CC^{(u)}_{i,j}(w) \subset \CC^{(u)}_{k,\ell}(T_{i,j,t} (w) )\mbox{ and }
|DT_{i,j,t}(w) v|\ge \lambda^t_{i,j,u}(w) |v| \, , 
$$
for all $v\in \CC^{(u)}_{i,j}(w)$, 
and  
$$ 
DT_{i,j,t}^{-1}(T_{i,j,t}(w)) \CC^{(s)}_{k,\ell}(T_{i,j,t}(w)) \subset \CC^{(s)}_{i,j}(w)\mbox{ and }
|DT_{i,j,t}^{-1} (T_{i,j,t}(w))v|\ge \lambda_{i,j,s}^{-t}(w) |v|\, ,
$$
for all $v\in \CC^{(s)}_{k,\ell}(w)$.
\end{definition}

\bigskip

We must still formulate the transversality and complexity conditions.
We shall do this at the
level of the Poincar\'e maps $P_{i,j}$.

For $n\ge 1$, and $\ii \in I^{n+1}$, we let
$P_\ii^n=P_{i_{n-1} i_n}\circ \dots \circ P_{i_0 i_1}$, which is
defined on a neighbourhood of $O_\ii\subset M_0$, where  $O_{(i_0 i_1)}=O_{i_0, i_1}$,
and
\begin{equation}
O_{(i_0,\dots,i_{n})}=\{z\in O_{i_0, i_1} \st P_{i_0 i_1}(z)\in
O_{(i_1,\dots,i_{n})}\}\, .
\end{equation}
Conditions (0)-(1)-(2) imply that for each  iterate
of the Poincar\'e map $P^n_{\ii}$,  and every $z\in O_{\ii}$,
there exist  weakest contraction and
expansion constants
$\lambda_{\ii, s}^ {(n)}(z)<1$ and
$\lambda_{\ii,u}^{(n)}(z)>1$, and a strongest  expansion constant
$\Lambda_{\ii,u}^{(n)}(z)\ge \lambda_{\ii,u}^{(n)}(z)$. We put
$$\lambda_{s,n}(z)=\sup_{\ii} \lambda_{\ii, s}^ {(n)}(z)<1\, ,
\quad \lambda_{u,n}(z)=\inf_{\ii}\lambda_{\ii,u}^{(n)}(z)>1
\, .
$$
We can now formulate the bunching condition  on  a piecewise $C^{2}$ hyperbolic   flow: For some
$n\ge 1$
\begin{equation}
\label{bunch}
\sup_{\ii \in I^n,z\in M_0} \biggl ( \lambda_{\ii,s}^{(n)} (z)\lambda_{\ii,u}^{(n)}
(z)^{-1}\Lambda^{(n)}_{\ii,u}(z) \biggr ) <1\, .
\end{equation}
(The bunching condition \eqref{bunch}  is automatically satisfied
if $d_u=1$, which
implies that $\Lambda_{\ii,u}^{(n)}(z)/ \lambda_{\ii,u}^{(n)}(z)$
tends to $1$ as $n\to \infty$, uniformly.)
If \eqref{bunch} holds for $n$, there exists $\beta>0$ so that
\begin{equation}
\label{bunch2}
\sup_{\ii \in I^n,z\in M_0} \biggl ( (\lambda_{\ii,s}^{(n)} (z))^{1-\beta}\lambda_{\ii,u}^{(n)}
(z)^{-1}(\Lambda^{(n)}_{\ii,u}(z))^{1+\beta} \biggr ) <1\, .
\end{equation}
(It is in fact the above condition \eqref{bunch2} which appears in
our argument.)

As is usual in piecewise hyperbolic settings (see e.g. \cite{Yo}), we  assume transversality 
and subexponential  complexity.
In view of the transversality definition, it is convenient to assume
that the cone fields $\CC^{(s)}_{i,j}$ do not depend
on $i$ and $j$, i.e., they are continuous throughout (see \cite{BG2} 
for an alternative definition of transversality in the general case,
and Remark 2.4 there)
and we shall do so. (This allows us
to use a simplified definition of the
norm \eqref{defnormen}, and is useful also in the proof of Lemma~\ref{bq}
below: Otherwise a further argument is needed since we cannot apply
Strichartz' result \cite{Str}, \cite{BG1} for $q=1$, except if we have continuity at least in the time direction.)

\begin{definition}[Transversality]\label{transcond}
Let $T_t$ be a piecewise $C^{2}$ hyperbolic   flow.
We say that the flow $T_t$ satisfies the  transversality condition
if 
\begin{itemize}
\item
the cone fields $\CC^{(s)}_{i,j}$ do not depend
on $i$, $j$;
\item  each $\partial O_{i,j}$ is a finite union
of $C^1$ hypersurfaces $K_{i,j, k}$, the image of each of which by the Poincar\'e map
is transversal  to the stable cone  (i.e., for all $z \in K_{i,j,k}$, the tangent
space
$T_z (P_{i,j}(K_{i,j,k}))$ contains a $d_u$-dimensional subspace which intersects
$\CC_i^{(s)}$ only at $0$). 
\end{itemize}
\end{definition}

\begin{definition}[Subexponential complexity]
\label{domin}Let $T_t$ be a piecewise $C^{2}$ hyperbolic   flow.
For $n\ge 1$ and $\ii=(i_0,\dots,i_{n})\in I^{n+1}$, 
set
\begin{equation*}
D^b_n=\max_{z\in M_0} \Card \{ \ii=(i_0,\dots,i_{n}) \st z \in
\overline{O_{\ii}} \} \, ,
\end{equation*}
and
\begin{equation*}
D^e_n=\max_{z\in M_0} \Card \{ \ii=(i_0,\dots,i_{n}) \st z \in
\overline{P_\ii^n(O_{\ii})} \} \, .
\end{equation*}
We say that  complexity is subexponential  if 
\begin{equation}
\limsup_{n \to \infty} \frac{1}{n} \ln (D^e_n)=0 
\mbox{ and } \limsup_{n \to \infty} \frac{1}{n} \ln (D^b_n)=0\, .
\end{equation}
\end{definition}

\subsection{Examples}\label{ex2c}
We present a simple example to which our main theorem applies.

Given $M=\bT^2\times \bR$ and $\tau\in L^\infty(M,\bR)$, we define the set $X_0=\{(x,y,z)\in M\;:\; (x,y)\in \bT^{2}, z\in [0,\tau(x,y)]\}$.

To define the dynamics we consider a piecewise $C^2$ hyperbolic symplectic (with respect to the symplectic form $dx\wedge dy$) map $f:\bT^{2}\to\bT^{2}$. Let $\tau:[0,1]^{2}\to \bR_{+}$. Let $\{\hat O_i\}$ be the domains on $\bT^2$ in which $f$ is smooth and define $O_i=\{(x,y,z)\in M\;:\; (x,y)\in\hat O_i, z=0\}$, we assume that the $O_i$ are simply connected and that $f$ has a $C^2$ extension in a neighbourhood of each $\hat O_i$. 

We now define the flow on $X_0$ by
\[
T_{t}(x,y,z)=(x,y,z+t)
\]
for all $t\in[0,\tau(x,y)-z)$, while 
\[
T_{\tau(x,y)-z}(x,y,z)=(f(x,y),0)=(f_{1}(x,y), f_{2}(x,y),0) \, .
\]
The contact form is the standard one: $\alpha=dz-ydx$. In order to check that $\alpha$ is preserved by the flow, we must ensure that the form does not change while going through the roof. A direct computation shows that this is equivalent to requiring
\[
\begin{split}
&\partial_{x}\tau=y-f_{2}(x,y)\partial_{x}f_{1}(x,y)=:a\,,\\
&\partial_{y}\tau=-f_{2}(x,y)\partial_{y}f_{1}(x,y)=:b\, .
\end{split}
\] 
By the symplecticity of $f$ it follows that $a\, dx+b\, dy$ is a closed form, and hence $\tau$ is uniquely defined on each $\hat O_i$ apart from a constant. In particular, we can chose such constants as to ensure that there exists $\tau_->0$ such that $\inf \tau\geq \tau_-$.
We have thus a piecewise smooth contact flow. The sets $O_{i,j}$ are defined in the obvious way, and $P_{i,j}=T_{\tau}|_{O_{i,j}}$. If the map $f$ is uniformly hyperbolic, then one can define continuous cones $\hat \cC_{i,j}^{u}\subset \bR^2$ that are mapped strictly inside themselves by $df$ and such that each vector in them is expanded at least by some $\lambda>1$. We can then define the cones $\cC_{i,j}^u(x,y,z)=\{(\eta,\xi,\zeta)\in\bR^3\;:\; (\eta,\xi)\in\cC_{i,j}^u(x,y), \delta|\zeta|\leq \|(\eta,\xi)\|\}$. One can verify that this cone family is strictly invariant under the Poincar\'e maps and that the Poincar\'e map is hyperbolic, provided $\delta$ is chosen small enough. The transversality hypothesis is then satisfied by the flow if it is satisfied by the map $f$.

Next we provide an open set of examples in which this construction yields a flow that satisfies all our hypotheses,  many other similar examples can be constructed. Consider the map $f_0:\bT^2\to \bT^2$ defined by
\[
f_0(x,y)=\begin{cases} (x+y,\frac x2+\frac {3y}2)\mod 1\quad &\text{for }x\in[0,1), y\in[0,1-x]\\     
                                   (x+y,\frac x2+\frac {3y}2-\frac 12)\mod 1&\text{for }x\in[0,1), y\in(1-x,1)\, .
           \end{cases}
\]
Note that any cone of the type $\cC_a=\{(x,y)\;:\; |x-y|\leq a|x+2y|\}$ is strictly invariant. In particular $Df_0\cC_a\subset \cC_{\frac a4}$.\footnote{The eigenvalues of the matrix are $2,\frac 12$, for eigenvectors $(1,1)$ and $(1,-\frac 12)$. One must write the cone is such coordinates to have standard form used in Definition \ref{PiecHyp}.} To prove hyperbolicity one can first define the norm $\|(x,y,z)\|=\|(x,y)\|+\delta|\zeta|$ under which $T_t((x,y,z)$ is hyperbolic for $t\geq \tau(x,y)-z$, provided $a,\delta$ are small enough. Then one modifies the norm as to distribute the expansion and contraction in the return map evenly between $(x,y,0)$ and $(x,y,\tau(x,y))$. 
The discontinuity manifold is given by $\{y=1-x\}$, while the discontinuity line of the inverse map is $\{x=0\}$. 
Note that the discontinuity is not contained in the cone $\cC_1$, while its image is contained in $\cC_1$, thus, by defining the Poincar\'e map to be some higher power of $f_0$, the transversality holds. Since such properties are open, they hold also for all maps  $f=f_0\circ \phi$, where $\phi:\bT^2\to\bT^2$ is a smooth symplectic map sufficiently close to the identity.
For such maps, we have thus both transversality and also subexponential complexity growth since we can apply Theorem 2 of \cite{Per}.

The only property left to check is ergodicity. This follows from the ergodicity of $f$ that can be proved by applying the Main Theorem in \cite[section7]{LW}.

\begin{remark}\label{rkbil}
Our original motivation was to understand billiard flows (in dimension two, i.e., the flow
acts on a three-dimensional manifold), let us explain
now how close we are to this goal:
Sinai dispersive billiard flows  are of course contact flows. It follows from well-known results  (see e.g. \cite{CM}) 
that the ergodicity, transversality, and subexponential
complexity assumptions are satisfied for the flows of two-dimensional Sinai dispersive 
billiards with finite horizon.
Sinai billliards
are piecewise cone hyperbolic (see \cite{CM}, and also \cite[Section~ 3]{LW},
noting that our Poincar\'e map $P_{ij}$ includes the contribution
of what is called the ``collision map" $\Gamma$ there), 
except for the requirement that the flow is smooth all the way to the boundary of the domains
$B_{ij}$ (billiards flows are smooth on the open domains, but their derivatives
blow up along some of the boundaries).
This is a nontrivial difficulty, and we hope that the tools
being developed in \cite{BBG} will allow to solve it eventually.
It should be remarked that some other natural examples 
suffer from the same ``blowup of derivatives along
boundaries" problem that affects discrete and continuous-time
billiards, and hence  do not fit in our framework:
For example, consider a compact connected manifold partitioned in regions  $B_i$ with nice boundaries. Put on each region a different metric, all with strictly negative curvature, and consider the resulting geodesic flow. Generically, there will be geodesics tangent to the boundaries of the regions with non degenerate tangency. If we now consider a geodesic in $B_i$ tangent to the boundary between $B_i$ and $B_j$, then there exists an $\ve$-close geodesic that will spend a time $\sqrt \ve$ in $B_j$, and this means that the derivative of the flow at the boundary will be infinite, exactly as in the case of tangent collisions for billiards. Thus our result does not applies to examples obtained as patchwork of different geodesic flows, unless the cutting and pasting is done in the unitary tangent bundle (rather than on the manifold), where one can easily construct regions with boundaries uniformly transversal to the flow.
\end{remark}


\section{Definition of the Banach spaces $\HHH^{r,s,q}_p(R)$}
\label{defspace}
Throughout the paper $\Cs$ denotes a generic constant that may vary from line to line.

Let $\beta \in (0, 1)$ satisfy \eqref{bunch2}.
For $p\in (1, \infty)$  and real numbers $r$, $s$, and $q$,
we  shall introduce  in Subsection~ \ref{spaces}
scales of Banach spaces 
$\HHH_p^{r,s,q}(R)$ of (aniso\-tropic) distributions
on $M$, supported in $X_0$,
and parametrised by $\beta$, $p$,   $r$,  $s$,  $q$,
and a large zoom parameter $R>1$, and auxiliary real parameters $C_0>1$, $C_1>2C_0$.
When the meaning is clear, we write $\HHH_p^{r,s,q}(R)$, or just
$\HHH_p^{r,s,q}$. 
Just like in \cite{BG2}, the spaces will depend
on the stable cones and\footnote{See  \eqref{suffhyp} in Appendix~\ref{iteratechart}
for the use of $\beta$, which will
also play a role later in the compact embedding Lemma~\ref{embed}.} on $\beta$. 
In Lemma~\ref{embed}, we prove that $\HHH_p^{r,0,0}(R)$ is isomorphic
to the usual Sobolev space $H^r _p(X_0)$
on $X_0$ whenever $\max(-\beta, -1+1/p)< r < 1/p$.

\subsection{Anisotropic spaces $H^{r,s,q}_p$ in
$\real^d$ and the class $\FF(z_0,\CC^{s})$  of local foliations}
\label{sec:def_fol}

In this subsection, we  recall the anisotropic  spaces $H^{r,s,q}_p$  in $\real^d$ 
(generalising those  used in \cite{BG1} and \cite{BG2}),
and we define a class  of cone admissible
local foliations in $\real^d$  with uniformly bounded $C^1$ norms 
(in Lemma ~\ref{lemcompose} we  show that this class is invariant
under the action of the transfer operator). 
These are the two building blocks 
we shall use  in Subsection ~ \ref{spaces} to define
our spaces of anisotropic distributions.

Let $d=d_u+d_s+1$ with
$d_s \ge 1$ and $d_u\ge 1$.
(Once more, the reader is welcome to concentrate on the
case $d_s=d_u=1$.)
We write $x\in \real^d$ as $x=(x^u,x^s, x^0)$  with
$$
x^u=(x_1,\dots,x_{d_u})\, ,\, 
x^s=(x_{d_u+1},\dots, x_{d-1})\, , \,  x^0=x_{d}\, .
$$
The
subspaces $\{x^u\} \times \real^{d_s} \times \{ x^0\}$ of $\real^d$ will  be
referred to as the {\em stable leaves} in $\real^d$, 
and the lines $\{(x^u, x^s) \} \times \real^{1}$  are
the {\em flow directions} in $\real^d$. We say
that a diffeomorphism of $\real^ d$ {\it preserves stable leaves} or
{\em flow directions} if
its derivative has this property.
For $C>0$ and $x\in \real^d$, let us write 
\begin{align*}
&B(x,C)=\{y\in \real^d
\st |y^u-x^u|\le C, |y^s-x^s|\le C, |y^0-x^0|\le C\}\, ,\\
&B(x^u, x^s,C)=\{y\in \real^{d_u+d_s}
\st |y^u-x^u|\le C, |y^s-x^s|\le C\}\, ,\\
&B(x^u,C)=\{y^u\in \real^{d_u}
\st |y^u-x^u|\le C\}\, , \,\, 
B(x^s,C)=\{y^s\in \real^{d_s}
\st |y^s-x^s|\le C\}\, .
\end{align*}

We denote the Fourier transform in $\real^d$
by $\FFF$.  An element $\xi$ of the dual space of $\real^d$
will be written as $\xi=(\xi^u, \xi^s, \xi^0)$ with
$\xi^u \in \real^{d_u}$,
$\xi^s\in \real^{d_s}$ and $\xi^0 \in \real$.

The anisotropic Sobolev  spaces $H_p^{r,s,q}=H_p^{r,s,q}(\real^d)$   
belong to a class of spaces first studied by
Triebel \cite{Tr}:

\begin{definition}[Sobolev spaces $H^{r,s,q}_p$  and $H^r_p$ in $\real^d$]
\label{space} For $1<p<\infty$, $r$, $s$, and $q \in \real$,
let $H_p^{r,s,q}$ be the set of 
(tempered) distributions $v$ in $\real^d$ such that
\begin{equation}\label{norm}
\norm{v}{H_p^{r,s,q}}:=  \norm{\FFF^{-1}(a_{r,s,q}\FFF v)}{ L^p} < \infty\, ,
\end{equation}
where
\begin{equation}
a_{r,s,q}(\xi)=(1+|\xi^u|^2+|\xi^s|^2+|\xi^0|^2)^{r/2} (1+|\xi^s|^2)^{s/2}
(1+|\xi^0|^2)^{q/2} \, .
\end{equation}
We set $H^r_p=H^{r,0,0}_p$.
\end{definition}

Triebel proved that rapidly decaying $C^\infty$ functions are dense in each $H^{r,s,q}_p$ (see e.g.
\cite[Lemma 18]{BG1}). So we could equivalently define
$H_p^{r,s,q}$ to be the closure of rapidly decaying $C^\infty$ functions for the
norm (\ref{norm}).
Triebel also obtained complex interpolation results which apply
to the spaces $H_p^{r,s,q}$, see \cite[Lemma 18]{BG1}. Section 3.1 of \cite{BG1} 
contains  reminders about complex interpolation and references (such as \cite{BL} and \cite{TrB}).
In particular,  if $\BB_1$ and $\BB_2$ are two Banach spaces 
forming a compatible couple \cite[\S 2.3]{BL} and $0<\theta<1$ is real
then $[\BB_1,\BB_2]_\theta$ denotes their complex interpolation \cite[\S 4.1]{BL}.
In Appendix ~\ref{localspaces}, we adapt
the results in \cite[Section 4]{BG2} on the anisotropic spaces used there  to our current setting.

\smallskip

We next move to the definition of the cone-admissible foliations (also called admissible charts).
We shall work with local foliations indexed by points
$m$ in appropriate finite subsets of $\real^d$ (the sets will be introduced in Section ~ \ref{spaces}).  We
view  $\beta \in (0,1]$ as fixed, satisfying \eqref{bunch2}, while
the constants $C_0>1$ and $C_1>2C_0$ will be chosen later in Lemma~ \ref{lemcompose} (see also the quantifiers for the estimate
\eqref{2.10} in the proof of the Lasota-Yorke-type bound
Lemma~\ref{LY0}).
These constants play the following role: If $C_0$ is large,
then the admissible foliation covers a large domain; if $C_1$
is large, then the leaves of the foliation are almost parallel.
(We use the notation $\FF(m,\CC^{s},\beta, C_0, C_1)$ introduced in
\cite{BG2}, despite the fact that the spaces are slightly different in view
of the additional time direction.)

\begin{definition}[Sets $\FF(m,\CC^{s},\beta, C_0, C_1)$ of cone-admissible foliations]
\label{defcharts}
Let $\CC^{s}$ be a $d_s$-dimensional cone in $\real^d$, transversal to $\real^{d_u}\times
\{0\}\times \real$,  let  $\beta \in (0,1)$, let
$1<C_0 < C_1/2$, and let $m=(m^u,m^s, m^0)\in \real^d$.
Then  $\FF(m,\CC^s, \beta, C_0, C_1)$ is the set of maps
$$
\phi=\phi_F: B(m,C_0)\to \real^d
\, , 
\quad \phi_F(x^u,x^s,x^0)=(F(x^u, x^s), x^s, \tilde F(x^u, x^s,x^0))
\, ,
$$ 
where  $F:B(m^u, m^s,C_0) \to \real^{d_u}$  and
$\tilde F:B(m, C_0) \to \real$
are $C^1$, with $ F(x^u,m^s)= x^u$, $\forall |x^u|\le C_0$, 
$\tilde F(x^u, m^s,x^0)= x^0$, $\forall |x^0|\le C_0$, 
$\forall |x^u|\le C_0$,
$$
(\partial_{x^s} F(x) v,v, \partial_{x^s} \tilde F(x) v)\in \CC^{s} ,\,  
\forall v\in\real^{d_s}\, , 
\forall x \in B(m^u, m^s,m^0,C_0)\, ,
$$
and for all $(x^u, x^s, x^0),(y^u,y^s, y^0) \in B(m^u, m^s,C_0)$,
on the one hand
\begin{equation}\label{smooths}
| DF(x^u,x^s) -DF(x^u, y^s) | \le \frac {|x^s-y^s|}{C_1}   \, ,
\end{equation}
\begin{equation}\label{smoothu}
| DF(x^u,x^s) -DF(y^u, x^s) | \le \frac {|x^u-y^u|^\beta}{C_1}   \, ,
\end{equation}
and
\begin{equation}\label{smoothsecond}
| DF(x^u,x^s) -DF(x^u, y^s) - DF(y^u,x^s) +DF(y^u, y^s)| \le \frac {|x^s-y^s|^{1-\beta}|x^u-y^u|^\beta}{C_1}   \, ,
\end{equation}
and on the other hand,    
\begin{equation}\label{tildenew}
\partial_{x_0} \tilde F(x^u, x^s, x^0)\equiv 1\, ,
\end{equation}
where, writing $\tilde F(x^u, x^s, x^0)=x^0+\tilde f(x^u, x^s)$
\begin{equation}\label{smoothtilde}
|D\tilde f(x^u, x^s)-D \tilde f(x^u, y^s)|  \le \frac{|x^s-y^s|}{C_1}\, .
\end{equation}
\end{definition}

One easily proves that the set $\FF(m,\CC^{s},\beta, C_0, C_1)$ is large, adapting the
argument in \cite{BG2}, below Definition 2.8 there.
We refer to \cite[Remarks 2.10 and 2.11]{BG2} for comments on  the conditions
\eqref{smooths}, \eqref{smoothu} and \eqref{smoothsecond}, in particular, why they are natural
in view of the graph transform argument used in the proof of Lemma ~\ref{lemcompose}.

The fact that $F$ does not depend on $x^0$ is useful
e.g. in Lemma~\ref{bq}.
The smoothness condition on $\tilde f$ 
is new with respect to \cite{BG2}.  
Beware that we shall need to use  Lemma~\ref{noglue} below
because of Step 2 
in the proof of Lemma~\ref{lemcompose}, but that
Steps 4--5 of the same proof imply that we cannot force
$\tilde F(x^u, x^s, x^0)=f(x^0)$ in general (which is intuitively clear:
otherwise, all stable leaves would lie in planes $x^0=$constant, which
means that the ceiling times are cohomologous to a constant, a situation
we do not allow).

The following lemma will justify our ``foliation" terminology:
The  graphs 
$$\{(F(x^u,x^s),x^s ,\tilde F(x^u, x^s, x^0)), x^s\in B(m^s, C_0)
\}\, , \,\,
x^u \in B(m^u, C_0)\, , \, x^0\in B(m^0, C_0)\, , 
$$ 
form a partition of a neighbourhood of $m$
of size of the order $C_0$  \footnote{Through the $R$-zoomed charts to be introduced in Section~\ref{spaces},
this will correspond to a  neighbourhood of size 
of the order $C_0/R$ in the manifold.}, into sets whose ($d_s$-dimensional)
tangent space is everywhere contained in
$\CC^{s}$.  The maps $F$, $\tilde F$ thus define a local foliation, and
the map $\phi_F$  is a diffeomorphism straightening
this foliation, i.e., the leaves of the foliation are the images
of the stable leaves of $\real^d$ under the map $\phi_F$.
(The conditions in the definition 
imply that the local foliation defined by $F$, $\tilde F$ is $C^{1+Lip}$ along
the stable leaves.)
Note that the image of a flow direction by $\phi_F$
is again a flow direction, parametrised at constant unit speed.

\begin{lemma}[Admissible foliations are $C^{1+\beta}$ foliations]
\label{lempropphi}
Let $\CC^s$ be a  $d_s$-dimensio\-nal cone which is transversal to
$\real^{d_u}\times \{0\}\times \real$. Then 
there exists a constant $\Cs$ depending
only on $\CC^s$ such that, for any $1<C_0<C_1/2$, and any
$\phi_F \in \FF(m,\CC^s,\beta,  C_0, C_1)$, the map $\phi_F$ is a
diffeomorphism onto its image with
$$\norm{D\phi_F}{C^{\beta}}\le \Cs
\mbox{ and }
\norm{D\phi_F^{-1}}{C^{\beta}} \le \Cs\, .$$ 
Moreover, $\phi_F(
B(m,C_0))$ contains $B(m,\Cs^{-1} C_0)$.
\end{lemma}

The proof of  Lemma ~\ref{lempropphi}  does not require \eqref{smoothsecond}.

\begin{proof}
Lemma ~\ref{lempropphi} can be proved like \cite[Lemma 2.9]{BG2}, using
\cite[Appendix A.1]{BG2}. We just explain how to show that
the $C^1$ norm of $F$ is
uniformly bounded (the argument for
$\tilde F$ is similar, using that $\CC^s$ is transversal to $\{0\}\times \{0\}
\times \real$ and \eqref{smoothtilde}): First, $\partial_{x^s} F$ is bounded since the
cone $\CC^s$ is transversal to $\real^{d_u+1}\times\{0\}$. Next,
$F(x^u, m^s)=x^u$, so that $\partial_{x^u} F(x^u,m^s)=\id$, hence
\begin{equation}
\label{eqpartialxF}
|\partial_{x^u} F(x^u, x^s)-\id|=|\partial_{x^u} F(x^u,x^s) -\partial_{x^u} F(x^u,m^s)|
\le \frac{|x^s-m^s|}{C_1}
\le \frac{C_0}{C_1}<1 \,  .
\end{equation}
Finally, estimate (\ref{eqpartialxF}) implies that $D
F$ is everywhere invertible, and its inverse has uniformly
bounded norm.
\end{proof}


\subsection{Spaces of distributions}
\label{spaces}

In this subsection, we introduce appropriate $C^{2}$
coordinate
patches $\kappa_{i,j,\ell}=\kappa_\zeta$ on the manifold  and cones
$\CC^s_{i,j}$ in $\real^d$ (recall that the flow is piecewise $C^{2}$).
Combining them  with  admissible charts
in  suitable families $\FF(m, \CC^{s},\beta, C_0, C_1)$
we glue together the local spaces $H^{r,s,q}_p$  via a partition
of unity,  and, zooming by a large factor $R$, we define the space
$\HHH_{p}^{r,s,q}(R)$ of distributions.\footnote{This is a
modification of the space $\HHH^{r,s}_p$ in
\cite{BG2}.}   

\begin{definition}\label{extcone}
An {\em extended cone} $\CC$ is a set of four closed cones $(\CC^s,
\CC_0^s, \CC^u, \CC^u_0)$ in $\real^d$ such that:

\noindent $\CC^s\cap \CC^u=\{0\}$; $\CC^s\cap (\{0\} \times \{0\} \times \real) =\CC^u\cap (\{0\} \times \{0\} \times \real)
=\{0\}$.

\noindent $\CC^s_0$ is $d_s$-dimensional and
contains $\{0\}\times \real^{d_s}\times \{0\}$, 

\noindent $\CC^u_0$ is $d_u$-dimensional
and contains $\real^{d_u}\times \{0\}\times \{0\}$; 

\noindent $\CC^s_0\setminus \{0\}$ is contained in the
interior of $\CC^s$, 
$\CC^u_0 \setminus \{0\}$ is contained in the interior
of $\CC^u$. 

Given two extended cones $\CC$ and $\tilde\CC$, we
say that an invertible matrix $\AAc:\real^d\to \real^d$ {\em sends 
$\CC$ to $\tilde\CC$ compactly} if  $\AAc \CC^u$ is contained in
$\tilde{\CC}^u_0$, and  $\AAc^{-1}\tilde\CC^s$ is
contained in $\CC^s_0$. 
\end{definition}

For all $i \in I$ and $j \in J_i$, we fix once and for all a finite number of
open sets $U_{i,j,\ell, 0}$ of $M$, for $\ell\in N_{i,j}$, covering
$\overline{B_{i,j}}$, and included in 
the open neighbourhood of $\overline{B_{i,j}}$ where the stable and unstable
cones extend continuously
(recall Definition~\ref{PiecHyp}). Let also $\kappa_{i,j,\ell}:U_{i,j,\ell, 0} \to
\real^d$, for $i \in I$, $j \in J_i$, and $\ell\in N_{i,j}$, be a finite family
of $C^{2}$ charts mapping flow orbits to flow directions, i.e.,
for each $\tau$, each $(z,t)$ with $z\in \widetilde O_{i,j}$ and
each $t$ in a neighbourhood of $[0, \tau_{i,j}(z)-\tau]$
(given by Definition~\ref{PiecHyp})
so that $T_{i,j,\tau}(z,t)\in U_{i,j,\ell,0}$
\begin{align}\label{flowbox}
&\exists (x^u,x^s)(z)\in \real^{d-1} \, , \mbox{ s.t. }
\quad \kappa_{i,j,\ell}(T_{i,j,\tau}(z,t)) =(x^u(z),x^s(z),0)+(0,0,\tau+t)
\, ,
\end{align}
(where the function $(x^u,x^s)(\cdot)$ is
$C^{2}$, recalling (2) in Definition~\ref{PiecHyp}), and so that the images of the cones $\CC_{i,j}^{(s)}$
and $\CC_{i,j}^{(u)}$ satisfy
\begin{align}
\label{coneok}
\{0\}\times \real^{d_s}\times \{0\} \subset D\kappa_{i,j,\ell}(w) \CC_{i,j}^{(s)}(w)\, ,
\quad \real^{d_u} \times \{0\}\times \{0\} \subset D\kappa_{i,j,\ell}(w) \CC_{i,j}^{(u)}(w)\, , \,\,  \forall w \, .
\end{align}
We shall take as ``standard" contact form in $\real^d$  the following one form
\begin{equation}\label{contact}
\alpha_0= dx^0 - x^s dx^u \, .
\end{equation}
We  require in addition 
that each 
chart $\kappa_{i,j,\ell}$  is Darboux (see \cite[\S 2]{Tsujii1}), i.e., it satisfies
\begin{equation}\label{conntact}
\kappa_{i,j,\ell}^* (\alpha)=\alpha_0\, .
\end{equation}
(Condition \eqref{conntact} is used to study
the averaging operator $\AAA_\delta$ in Section ~\ref{molll}.)

Finally, for any fixed small $t_{00} >0$,
and each $i \in I$, $j \in J_i$,
let $\CC_{i,j,\ell}$ be extended cones in
$\real^d$ such that, for every
$t\ge t_{00}$ and each $x\in \real^d$ so that 
$\kappa_{i',j',\ell'}\circ T_t \circ
\kappa_{i,j,\ell}^{-1}(x)
$
is defined, 
\begin{equation}\label{concond}
D(\kappa_{i',j',\ell'}\circ T_t \circ
\kappa_{i,j,\ell}^{-1})_x \mbox{ sends
$\CC_{i,j,\ell}$ to $\CC_{i',j',\ell'}$  compactly. }
\end{equation}

Such charts and cones
exist, as we explain now. Since the flow is hyperbolic (recall (2) in Definition~\ref{PiecHyp})
and the
image of the unstable cone is included in the unstable cone,
small enlargements of the unstable cones are sent strictly into
themselves by the map $T_t$ for any $t\ge t_{00}$
(and similarly for the stable cones). 
Since
$T_t$ is the Reeb flow of $\alpha$, we can assume that
the same charts satisfy \eqref{conntact} (see \cite[p. 1496 and \S2]{Tsujii1}),
and Appendix ~\ref{2c}).
Therefore, if one considers charts with
small enough supports
satisfying \eqref{flowbox}, \eqref{coneok}, and \eqref{conntact},
locally constant cones
$\CC_{i,j,\ell}^s$, $\CC_{i,j,\ell}^u$ slightly larger than the cones
$D\kappa_{i,j,\ell}(w) \CC_{i,j}^{(s)}(w)$,
$D\kappa_{i,j,\ell}(w)\CC_{i,j}^{(u)}(w)$, and finally slightly smaller cones
$\CC^s_{i,j,\ell,0}, \CC^u_{i,j,\ell, 0}$, they satisfy the previous
requirements.  We also fix open sets
$U_{i,j,\ell, 1}$ covering $X_0$ such that
$\overline{U_{i,j,\ell, 1}}\subset U_{i,j,\ell, 0}$, and we let
$V_{i,j,\ell, k}=\kappa_{i,j,\ell}(U_{i,j,\ell,k})$, $k=0, 1$.

The spaces of distributions will depend on a large ``zoom''
parameter $R
\ge 1$: If $R\ge 1$ and
$W$ is a subset of $\real^d$, denote by $W^R$ the set $\{R\cdot z
\st z\in W\}$. Let also $\kappa_{i,j,\ell}^R(w)=R \kappa_{i,j,\ell}(w)$,
so that $\kappa_{i,j,\ell}^R(U_{i,j,\ell, k})=V_{i,j,\ell, k}^R$. Let
\begin{equation}
\label{fset}
\ZZ_{i,j,\ell}(R)
=\{ m\in V_{i,j,\ell, 0}^R \cap \integer^d\mid
B(m,C_0)\cap V_{i,j,\ell,1}^R\ne \emptyset\}\, ,
\end{equation}
and
\begin{equation}
\label{fset2}
\ZZ(R)=\{ (i,j,\ell, m) \st i\in I, j\in J_i, \ell\in N_{i,j}, m\in \ZZ_{i,j,\ell}(R)\}\, .
\end{equation}

To $\zeta=(i,j,\ell, m)\in \ZZ(R)$ is  associated the point
$w_\zeta:=(\kappa_{i,j,\ell}^R)^{-1}(m)$ of $M$. These are the
points around which we shall construct local foliations, as
follows. Let us first introduce useful notations: We write,
for  $\zeta =(i,j,\ell, m) \in \ZZ(R)$,
\begin{equation}\label{usenot}
O_\zeta=O_{i, j} \, , \quad
B_\zeta=B_{i, j} \, , \quad
U_{\zeta,k}=U_{i,j,\ell,k}\, ,\, k=0,1\, ,\quad
\kappa^R_\zeta=\kappa^R_{i,j,\ell}
\text{ and }
\CC_\zeta=\CC_{i,j,\ell}
\, .
\end{equation}
These are respectively the partition set, the chart and the
extended cone that we use around $w_\zeta$. Let us fix some
constants $C_0>1$ and $C_1>2C_0$. If $R$ is large enough, say
$R\ge R_0(C_0, C_1)$, then, for any $\zeta=(i,j,\ell, m)\in \ZZ(R)$
and any chart $\phi_\zeta\in \FF(m,\CC^s_\zeta, \beta ,C_0,C_1)$, we have
$\phi_\zeta(B(m,C_0)) \subset V_{i,j,\ell,0}^R$. For $\zeta=(i,j,\ell, m) \in
\ZZ(R)$, we can therefore consider the set of charts ($R$,
$C_0$ and $C_1$ do not appear in the notation for the sake of
brevity)
\begin{equation}\label{ourcharts}
\FF(\zeta):=\{\bphi_\zeta=(\kappa_\zeta^R)^{-1} \circ \phi_\zeta:B(m,C_0) \to M \, ,\,
\phi_\zeta\in \FF(m,\CC^s_\zeta, \beta, C_0,C_1)\}\, .
\end{equation}
The image under  a chart $\bphi_\zeta \in \FF(\zeta)$ of the stable
foliation in $\real^d$ is a local foliation around the point
$w_\zeta$, whose tangent space is everywhere contained in
$(D\kappa_\zeta^R)^{-1}(\CC^s_\zeta)$. This set is almost contained
in the stable cone $\CC^{(s)}_i(w_\zeta)$, by our choice of
charts $\kappa_{i,j,\ell}$ and extended cones $\CC_{i,j,\ell}$.

Let us fix once and for all a $C^\infty$ function\footnote{Such
a function exists since the balls of radius $d$ centered at
points in $\integer^d$ cover $\real^d$.} $\rho:\real^d\to [0,1]$  such that
\begin{equation*}
\rho(z)=0 \text{ if } |z| \ge d\qquad  \text{ and } \qquad \sum_{m\in \integer^d} \rho(z-m)=1\, .
\end{equation*}
For $\zeta=(i,j,\ell,m)\in \ZZ(R)$, let $\rho_m(z)=\rho(z-m)$, and
\[
\brho_{\zeta}:=\brho_{\zeta}(R)=\rho_m\circ \kappa_\zeta^R : M \to [0,1]\, .
\]
Since $\rho_m$ is compactly supported in
$\kappa_{i,j,\ell}^R(U_{i,j,\ell, 0})$ if $m\in \ZZ_{i,j,\ell}(R)$ (and $R$ is
large enough, depending on $d$), the above expression is
well-defined. This gives a partition of unity in the following
sense:
\[
\sum_{m\in \ZZ_{i,j,\ell}(R)} \brho_{i,j,\ell, m}(w)= 1 \,,  \forall w \in U_{i,j,\ell, 1}\, ,
\quad \brho_{i,j,\ell, m}(w)=0 \, , \forall w \notin  U_{i,j,\ell, 0}\, .
\]
Our choices ensure that the intersection multiplicity of this
partition of unity is bounded, uniformly in $R$, i.e., for any
point $w$, the number of functions such that
$\brho_{\zeta}(w)\not=0$ is bounded independently of $R$.

\medskip

The space we shall consider depends in an essential way on the
parameters $p$, $r$, $s$, and $q$. It will also depend, in an
inessential way, on  the choices we have made  (i.e.,  the
reference charts $\kappa_{i,j,\ell}$, the extended cones
$\CC_{i,j,\ell}$, the constants $C_0$ and $C_1$, the function
$\rho$, and $R\ge R_0(C_0, C_1)$): Different choices would
lead to different spaces, but all such spaces share the same
features. 

\begin{definition}[Spaces  $\HHH_{p}^{r,s,q}(R,C_0,C_1)$ of distributions on $M$]
\label{defnorm}
Let $1<p < \infty$, $r,s,q\in \real$, let $1<C_0<C_1/2$, and let
$R\ge R_0(C_0, C_1)$. For any system of charts $\Phi=\{
\bphi_{\zeta}\in \FF(\zeta) \st \zeta\in \ZZ(R)\}$, let for $\psi\in L^\infty(X_0)$
\begin{equation}
\label{defnormen}
\norm{\psi}{\Phi}=\left(\sum_{\zeta\in \ZZ(R)}
\norm
{(\brho_\zeta(R) \psi)\circ \bphi_\zeta}
{H_p^{r,s,q}}^p\right)^{1/p}\, ,
\end{equation}
and put $\norm{\psi}{\HHH_{p}^{r,s,q}(R,C_0,C_1)}
=\sup_{\Phi}\norm{\psi}{\Phi}$, the supremum ranging over all
such systems of charts $\Phi$.

The space $\HHH_{p}^{r,s,q}(R,C_0,C_1)$ is the closure, for the norm
$\norm{\psi}{\HHH_{p}^{r,s,q}(R,C_0,C_1)}$, of $\{\psi\in L^\infty(X_0)\;:\;\norm{\psi}{\HHH_{p}^{r,s,q}(R,C_0,C_1)}<\infty\}$.
\end{definition}

Recall that our assumption from Definition~\ref{transcond} implies that
the cones are continuous, this is why we replace
$(\brho_\zeta(R) \cdot \Id_{B_\zeta} \psi)\circ \bphi_\zeta$
in the definition of \cite{BG2} by
$
(\brho_\zeta(R) \cdot  \psi)\circ \bphi_\zeta
$.

\begin{remark}\label{natural2}
In general, $\HHH_{p}^{r,s,q}(R,C_0,C_1)$ is not isomorphic to a Triebel
space $H^{r,s,q}_p(X_0)$.   
However,  
Lemma~\ref{embed} 
implies  that the Sobolev space
$H_p^{\sigma}(X_0)$ is isomorphic with the Banach space
$\HHH^{\sigma,0,0}_p(R,C_0,C_1)$ if $\max(-\beta,-1+1/p) < \sigma < 1/p$,
and that
$\HHH_{p}^{\sigma,0,0}(R,C_0,C_1) \subset
\HHH_{p}^{r,s,q}(R,C_0,C_1)$ if $s \le 0$, $q \le 0$ 
and $r \le \sigma$.
\end{remark}

\section{Reduction of the theorem to Dolgopyat-like estimates}
\label{redux}

For each $t\in \real$ we define an operator on $L^\infty(X_0)$
by setting
$$\LL_t (\psi) = \psi \circ T_{-t}\, .$$
For $\beta \in (0, 1)$ satisfying \eqref{bunch2},
$p\in (1, \infty)$,  and real numbers $r$, $s$, and $q$,
we  introduced  in Subsection~ \ref{spaces}
a space
$\HHH_p^{r,s,q}(R)$ of (aniso\-tropic) distributions
on $M$, supported in $X_0$.
Since $L^\infty(X_0)$ is dense in 
$\HHH^{r,s,q}_p(R)$,  it makes sense to talk about the
extension of
$\LL_t$ to the Banach space $\HHH^{r,s,q}_p$.
In Section~ \ref{LYBG},  adapting the bounds in \cite{BG2},  we prove:

\begin{lemma}[Lasota-Yorke type estimate]\label{LY0}
Let  $T_t$ be a piecewise $C^2$ hyperbolic flow satisfying transversality
(Definition ~\ref{transcond})
and  Let $\beta \in (0,1)$ satisfy
\eqref{bunch2}. 
Fix $\epsilon >0$. 
Then, for all large enough $C_0$ and
$C_1$,   for all $1<p<\infty$, all
real numbers  $s$, $s'$, and $r$, $r'$   satisfying 
\begin{align}\label{condrs}
-1+1/p  <s'<s <0 \le  r'< r <1/p \, , \quad -\beta< r+s <0 \, ,
\end{align} 
all $q\ge 0$ satisfying
\begin{equation}\label{locall'}
(1+q/r)(r-s)<1
\end{equation}
and
\begin{align}
\label{starstar'}1/p-1<s (1+\frac{q}{r})\le &0\le r(1+\frac{q}{r}) <1/p   \, .
\end{align}
there exist $\Cs>1$ (independent of $C_1$ and $C_0$),
$t_0$,   $\tilde \tau>0$, 
and $A_0 \ge 0$,  so that for any $t\ge t_0$ there exists $R(t)$ so that
for all $R\ge R(t)$
\begin{equation}\label{reallybounded}
\| \LL_t (\psi)\|_{\HHH^{r,s,q}_p(R)} 
\le \Cs e^{A_0t}  \| \psi \|_{\HHH^{r,s,q}_p(R)} \, ,
\end{equation}
and
\begin{align}
\label{notLY} 
\| \LL_{t} (\psi)\|_{\HHH^{r,s,q}_p(R)}
&\le \Cs \lambda^{t} \| \psi \|_{\HHH_p^{r,s,q+r-s}(R)}\\ 
\nonumber
&\qquad\qquad+ \Cs  R^{2(r-r'+s-s')}e^{A_0 t} \| \psi\|_{\HHH^{r',s',q+2r-r'-s}_p(R)}\, ,
\end{align}
where
$$
\lambda := (1+\epsilon)^{1/\tilde \tau}
[ (D^e_{\lceil t/\tilde \tau\rceil})^{(p-1)/p} (D^b_{\lceil t/\tilde \tau\rceil })^{1/p} 
\|\max(\lambda_{u,\lceil t/\tilde \tau\rceil}^{-r}, \lambda_{s,\lceil t/\tilde \tau\rceil}^{-(r+s)})\|_{L^\infty}]^{1/\lceil t \rceil }\, .
$$
\end{lemma}

Throughout, $\Cs$ denotes a constant (which can vary
from line to line) depending 
on $r$, $s$, $q$, $p$, $r'$, $q'$,
and the dynamics, possibly on $C_0$, but
not on $C_1$ or $R$, not on the iterate $t$, and not on the
parameter $z=a+ib$ of the resolvent $\RR(z)$ to be introduced soon.

Let $t_0$ be as in Lemma~\ref{LY0}, and set
\begin{equation}
\|\psi\|_{\widetilde \HHH_p^{r,s,q} (R)}
=\sup_{t \in [0,t_0]}\|\LL_t(\psi)\|_{\HHH^{r,s,q}_p(R)} \, .
\end{equation}
We get as an easy corollary of \eqref{reallybounded} 
that there exists $A\ge A_0$ so that for all large enough $C_1$ and $R$
and all $t\ge 0$
\begin{equation}\label{reallybounded'}
\| \LL_t (\psi)\|_{\widetilde
\HHH^{r,s,q}_p(R)} \le \Cs e^{ At}  \| \psi \|_{\widetilde
\HHH^{r,s,q}_p(R)} \, .
\end{equation}
(It is not clear in general that $\LL_t$ is bounded
for small $t<t_0$,
for the norm  $\HHH_p^{r,s,q} (R)$. It seems that it is necessary
in particular to find charts so that the changes of charts 
preserve stable leaves.)
The bounds \eqref{reallybounded'} and \eqref{notLY} imply
that if $q=2r-r'-s$ satisfies  \eqref{locall'} then 
for every $t>0$ 
\begin{align}\label{needit}
&\| \LL_t (\psi)\|_{\widetilde \HHH^{r,s,q}_p(R)}\le
\\ \nonumber &
\qquad \Cs^{t_0} 
\lambda^t \| \psi \|_{\widetilde \HHH_p^{r,s,q+r-s}(R)} 
+ \Cs R^{2(r-r'+s-s')} e^{A t} \| \psi\|_{\widetilde \HHH^{r',s',q+2r-r'-s}_p(R)}\, .
\end{align}

For $r,s,q,p$ and $R$
as in Lemma~\ref{LY0}, we define $\widetilde \HHH^{r,s,q}_p(R)$ to be the closure of~\footnote{In \cite{BG2}, we took the closure of $L^\infty(X_0)$
in the analogous definition. The present definition
is adapted in particular for  \eqref{strongl}. }
\begin{equation}\label{closure'}
\{\LL_t(\psi) \mid
\psi \in C^1(X_0)\, , \, t \ge 0\, , \, 
\norm{ \psi}{\widetilde \HHH_{p}^{r,s,q}(R)}<\infty\}
\end{equation}
for the norm
$\norm{\psi}{\widetilde\HHH_{p}^{r,s,q}(R)}$, setting also
$\widetilde \HHH= \widetilde \HHH^{r,s,0}_p(R)$ and
$\|\cdot \|=\|\cdot \|_{\widetilde \HHH}$. 
Note that $\widetilde \HHH$ is not included in $\widetilde \HHH^{r',s',2r-r'-s}_p(R)$, and
$\widetilde \HHH$ is not included in $\widetilde \HHH^{r,s,r-s}_p(R)$ (a fortiori there is
no compact embedding).
Therefore, the inequality \eqref{notLY} does not give a ``true"
Lasota-Yorke inequality. Like in \cite{Li}, we shall overcome this problem
by working with the resolvent $\RR(z)$
(there is a  difference with \cite{Li}
here: even the ``bounded term" of our
Lasota-Yorke inequality is unbounded!).

For $1<p<\infty$
and $\sigma\in \real$,
denote by  $H_p^{\sigma}(M)$  the standard (generalised)
Sobolev  space on $M$ and set
\footnote{For  $\max(-\beta,-1+1/p)<\sigma<1/p$, Corollary~\ref{StrStr} implies that
$H_p^{\sigma}(X_0)$ coincides with
the ``Whitney" definition, i.e., the restriction
to $X_0$ of elements in  $H_p^{\sigma}(M)$; beware however that if
$\sigma \ge 1$ then, for example, $1_{X_0}
\notin H_p^{\sigma}(X_0)$.}
$$H_p^{\sigma}(X_0)
=\{ \psi \in H_p^\sigma(M)\mid
\mbox{support}\, (\psi)\subset X_0\}\, ,
$$
endowed with the $H_p^\sigma(M)$-norm. We put  $L^p(X_0)=H^0_p(X_0)$, also for $p=\infty$.
We have the following  embedding and compact embedding properties
\footnote{
The proof requires the Jacobian
of the charts in Definition~\ref{defcharts} to be $\beta$-H\"older.}:

\begin{lemma}[Bounded and compact embeddings]\label{embed}
Let $1<p<\infty$. Then for all
$\max(-\beta,-1+1/p) <\sigma <1/p$,
the Banach space $H^\sigma_p(X_0)$ is isomorphic
with $\HHH^{\sigma,0,0}_p(R)$
and with
$\widetilde \HHH^{\sigma,0,0}_p(R)$.

If $r' \le r$, $s' \le s$, and $q' \le q$ we have the following continuous
inclusions
\begin{equation}
\label{**}
\widetilde \HHH_p^{r,s,q}(R) \subset \widetilde \HHH^{r',s',q'}_p(R)\, ,
\quad
\widetilde \HHH_p^{r,s,q}(R) \subset \widetilde \HHH^{r-|s|-|q|,0,0}_p(R)\, ,
\end{equation}
(and similarly for $\HHH_p^{r,s,q}(R)$).
If 
$$\max(-\beta, r')<r-|s|\, ,
$$
the following inclusion is compact
\begin{equation}\label{cpctt}
\widetilde \HHH_p^{r,s,0} (R)\subset \widetilde \HHH^{r',0,0}_p(R)\, .
\end{equation}
\end{lemma}

\begin{proof}[Proof of Lemma ~\ref{embed}]
We fix $R$, $C_0$, $C_1$. 
(For fixed $R$, the sum in \eqref{defnormen} involves a
uniformly bounded number of terms.)

The continuous embedding claims
\eqref{**}
follow
from the definitions and properties of Triebel spaces, taking
the supremum over all admissible charts
(for example, since $H_p^{r,s,q}(\real^d)$ is included in
$H_p^{r-|s|-|q|}(\real^d)$, it follows by taking the supremum over the
admissible charts that $\HHH_p^{r,s,q}$ is included in
$\HHH_p^{r-|s|-|q|,0,0}$). We thus only need to prove the compact
embedding statement and the relation with  Sobolev spaces.

Fix $s\le 0\le r$, $q$, and $r_0<r$, with
$r_0-|s|-|q|>-\beta$. 
For any admissible charts
$\phi_1, \phi_2 \in \FF(\zeta)$ for some $\zeta$ (recall
\eqref{ourcharts}), the change of coordinates $\phi_2\circ
\phi_1^{-1}$ is $C^1$ and has a (uniformly) $C^\beta$ Jacobian.
Since $r-|s|-|q|>-\beta$,
it follows from the functional analytic preliminary  in \cite[Lemma 4.4]{BG2}
(which requires Lemma~\ref{Leib} for $\tilde \beta=\beta$)
that
changing the system $\Phi$ of charts in the definition of the
$\HHH_p^{\sigma,0,0}$-norm gives equivalent norms. Hence,
$\HHH_p^{\sigma,0,0}$ is isomorphic to the Sobolev space
$H_p^{\sigma}(X_0)$. 

To prove that $\|\psi\|_{\widetilde \HHH^{\sigma,0,0}_p(R)}\le \Cs
\|\psi\|_{\HHH^{\sigma,0,0}_p(R)}$, it suffices to
apply \eqref{prel} and the definition. 
For the converse bound,
$\|\psi\|_{\HHH^{\sigma,0,0}_p(R)}\le \|\psi\|_{\widetilde \HHH^{\sigma,0,0}_p(R)}$
just use the definition and $\LL_0 (\psi)=\psi$.

To prove the compact embedding statement
\eqref{cpctt}, we 
use   $\widetilde \HHH_p^{r,s,0} \subset \HHH_p^{r-|s|,0,0}$ 
and that
the inclusion $H_p^{r-|s|}(X_0) \subset H^{r'}_p(X_0)$
is compact  since $r'<r-|s|$ and $X_0$ is compact
(see e.g. \cite[Lemma 2.2]{Cinfty}).
\end{proof}

Clearly, $\LL_0$ is the identity and
$\LL_{t'}\circ \LL_t = \LL_{t'+t}$ for all
$t, t' \in \real_+$.  
We claim that for any fixed 
$\psi \in \widetilde \HHH^{r,s,0}_p$
\begin{equation}\label{strongl}
\lim_{t \downarrow 0} \LL_t (\psi) = \psi \, .
\end{equation}
Indeed, by the definition, in particular \eqref{closure'}, we can
approach any $\psi$ in $\widetilde \HHH^{r,s,0}_p$ by a sequence,
$\LL_{t_n}(\psi_n)$ with $\psi_n\in C^1$ and $t_n\ge 0$. Then we write
$$
\LL_t (\psi) -\psi =\LL_t(\psi-\LL_{t_n}\psi_n)- \LL_{t_n}(\LL_t(\psi_n)-\psi_n)
+(\LL_{t_n}\psi_n-\psi)\, .
$$
Recall \eqref{reallybounded'}.
The first and last term in the right-hand-side
above tend to zero in $\widetilde \HHH$ as $n\to \infty$,
uniformly in $0\le t \le 1$.
Then, using the bounded inclusion $H^1_p(X_0)\subset  \widetilde \HHH$
from Lemma~\ref{embed},
$$\|\LL_{t_n}(\LL_t(\psi_n)-\psi_n)\|\le C e^{At_n}\|\LL_t(\psi_n)-\psi_n\|_{H^1_p(X_0)}\, ,
$$
so that
it suffices to see that
$\lim_{t \downarrow 0}\|\LL_t \psi - \psi \|_{H^1_p(X_0)}=0$ for any $C^1$ function
$\psi$.
Now this is easy to check, because for any fixed small $t$, the flow $T_t$
is $C^2$ except on a set of Lebesgue measure zero.
By \cite[Prop 1.18]{Da} it follows  that the map
$(t, \psi)\mapsto \LL_t (\psi)$ is jointly continuous on
$\real_+ \times \widetilde \HHH_p^{r,s,0}$.

In particular, $\LL_t$ acting on $\widetilde \HHH$ is a one parameter
semigroup, and we can define its infinitesimal generator $X$ by
$$
X (\psi)= \lim_{t \downarrow 0} \frac{\LL_t (\psi) - \psi}{t} \, ,
$$
the domain of $X$ being the set of $\psi$ for which the limit exists.

Using  Lemma~\ref{embed}
to get  $H^\sigma_p(X_0)\subset \widetilde \HHH$  boundedly
for any $\sigma \ge r$ (we take $\sigma <1$), and  since
a piecewise $C^2$ flow is $C^2$ except on a set of Lebesgue
measure $0$, it is not difficult to see 
that $C^2(X_0)$ is included in the domain
of $X$. A priori, $X$ is not a bounded operator on $\widetilde \HHH$,
but it is closed (see e.g. \cite{Da}).
For $z$ not in the spectrum of $X$ (i.e., so that
$z-X$ is invertible from the domain $D(X)$
of $X$ to $\widetilde \HHH$, or, equivalently,  bijective from
$D(X)$ to $\widetilde \HHH$), we shall consider
the resolvent 
$$\RR(z)=(z - X) ^{-1}\, ,
$$ 
which is a bounded operator.
For $z \in \complex$ with $\Re z > A$, classical results
\cite{Da} imply
\begin{equation}\label{1'}
\RR(z) (\psi)= \int_0^\infty e^{-zt} \LL_t( \psi)\, dt \, .
\end{equation}

\begin{remark}
We will deduce below from the ergodicity and contact assumptions  that $\LL_t$ does not have any eigenvalue
of modulus strictly larger than $1$ on $\widetilde \HHH$, 
and that its only eigenvalue of modulus $1$
is the simple eigenvalue corresponding to the fixed point $\psi \equiv 1$.
However, since we do not know if
the essential spectral radius of $\LL_t$ is strictly smaller than $1$, we cannot
deduce from this eigenvalue control
that $A=0$. We shall overcome this problem by working with the resolvent
$\RR(z)$.
\end{remark}

The following lemma, together with the compact
embeddings from Lemma~\ref{embed}, will allow us to
deduce from Lemma~\ref{LY0}  a bound on the essential spectral radius
of $\RR(z)$ (Lemma~\ref{ress}). In Section~\ref{dodo}, Lemma ~\ref{bq} will also be used
with the Dolgopyat
bound to obtain the key estimate, Proposition~\ref{dolgo},
on the resolvent.

\begin{lemma}[$\RR(z)$ improves regularity in the flow direction]\label{bq}
Fix $1<p<\infty$,
$s<0<r$, and $q\ge 0$ as in Lemma \ref{LY0}. Then 
for any $q'\ge q$ there exists $\Cs>0$ so that
for each $z =a+ib\in \complex$ with $a > A$, $|b| \ge 1$
\begin{align}
\|\RR(z) (\psi)\|_{\widetilde \HHH_p^{r,s,q'}(R)}
\le \Cs  \bigl ( 1+\frac{|z|}{R}\bigr )^{q'-q} 
\bigl ( \frac{1}{a-A}+1 \bigr )  \| \psi\|_{\widetilde \HHH_p^{r,s,q}(R)}\, .
\end{align}
\end{lemma}

\begin{proof}
We first check that for any $q=q'\ge 0$ satisfying  \eqref{locall'},
\begin{equation}
\label{q=0}
\|\RR(z) (\psi)\|_{\widetilde \HHH_p^{r,s,q}}
\le  \frac{\Cs}{a-A}\| \psi\|_{ \widetilde \HHH_p^{r,s,q}}\, .
\end{equation}
This is a
simple computation, using  
(\ref{reallybounded'}):
\begin{equation*}
\|\RR(z)( \psi)\|_{\widetilde  \HHH^{r,s,q}_p}
\le \int_0^\infty e^{-at} \|\LL_{t}  \psi\|_{\widetilde \HHH^{r,s,q}_p}\, dt 
\le C \int_0^\infty e^{-(a-A)t}  \|\psi\|_{\widetilde \HHH^{r,s,q}_p} \, dt
\le C\frac{\|\psi\|}{a-A}\, ,
\end{equation*}
just note that $\int_0^\infty e^{-u}\, du=1$.

To show the claim for $q' >q$,
we shall use  the easily proved fact that 
\begin{equation}\label{partialR}
\partial_t (\RR(z) (\psi)) \circ \LL_t |_{t=t_1} = z\bigl  (\RR(z) (\LL_{t_1}(\psi))
\bigr ) - \LL_{t_1} (\psi)  \quad
\forall t_1\in \real_+\, .
\end{equation}
The rest of the proof uses the Triebel norms and the
admissible charts defined in Section~\ref{defspace}, as well as interpolation
tricks presented in Section~\ref{molll}, and is postponed to 
Subsection~\ref{endbq}.
\end{proof}

Using Lemma~ \ref{LY0}, following the arguments of \cite[\S 2]{Li},
and simplifying the argument in \cite[\S 4]{Li},  we next estimate the essential
spectral radius of $\RR(z)$:

\begin{lemma}[Essential spectral radius of $\RR(z)$]\label{ress}
In the setting of Lemma ~ \ref{LY0}, assume that $-\beta < r+s$
and
in addition that  complexity is subexponential
so that, up to choosing a  larger
$t_0$, we have $\lambda <1$.

Then, for each space $\widetilde \HHH_p^{r,s,0}(R)$ 
so that \eqref{reallybounded} and \eqref{needit} hold for
some $A_0 \le A$, and for each $z \in \complex$ with $a=\Re z > A$, 
the operator $\RR(z)$ on  $\widetilde \HHH_p^{r,s,0}(R)$
has  essential spectral radius bounded by
$(a+ \ln(1/\lambda))^{-1}$.
\end{lemma}

\begin{proof}
By induction, one gets from (\ref{1'}) that if $\Re z > A$ then
(see \cite[\S2]{Da})
\begin{equation}\label{Rn}
\RR(z)^n (\psi)= 
\frac{1}{(n-1)!} \int_0^\infty t^{n-1} e^{-zt} \LL_t( \psi)\, dt \, .
\end{equation}
From \eqref{Rn} and the consequence
\eqref{needit} of  Lemma~ \ref{LY0}, we  obtain 
for all large enough
$C_0$, $C_1$, and $R$
a constant  $\Cs$ so that, for
all $a > A$, $n\ge 0$ and $\psi \in \widetilde \HHH$, writing $z=a+ib$ 
\begin{align}
\label{LYR}
&\|\RR(z)^{n+1} (\psi)\|\le
\Cs \int_{0}^{\infty}\frac{t^{n-1}}{(n-1)!}
\bigl ( e^{-t(a+ \ln(\lambda^{-1}))} (\| \RR(z)(\psi) \|_{\widetilde \HHH_p^{r,s,r-s}} \\
\nonumber&\qquad\qquad\qquad\qquad\qquad\qquad\qquad
+ 
e^{-t(a-A)} \| \RR(z)(\psi)\|_{\widetilde \HHH_p^{r',s',2r-r'-s}}) \bigr )\, dt \\
\nonumber &\qquad\qquad\qquad\le \Cs\biggl (\frac{\|\RR(z)( \psi )\|_{\widetilde \HHH_p^{r,s,r-s}} }{(a+ \ln(1/\lambda))^{n}}+\frac{\|\RR(z)(\psi)\|_{\widetilde \HHH_p^{r',s',2r-r'-s}}}{ (a-A)^{n}}\biggr ) \, .
\end{align}

Lemma~\ref{bq} applied to $s'<s$, $-\beta <r'<s+s'$ and $q=2r-r'-s$
implies that 
$$
\|\RR(a+ib)(\psi)\|_{\widetilde \HHH_p^{r',s',2r-r'-s}}\le 
\Cs\biggl ( \frac{|z|}{R}+1\biggr )^{2r-r'-s-s'}
\biggl (\frac{1}{a-A}+1\biggr )\| \psi\|_{\widetilde \HHH_p^{r',s',0}}\, ,
$$
and applied to $r$, $s$ and $q=r-s$ gives
$$
\|\RR(a+ib)(\psi)\|_{\widetilde \HHH_p^{r,s,r-s}}\le 
\Cs
\biggl ( \frac{|z|}{R}+1\biggr )^{r-s}
\biggl (\frac{1}{a-A}+1\biggr ) \|\psi\|_{\widetilde \HHH_p^{r,s,0}}\, .
$$
Since $s'\le 0$,  Lemma~\ref{embed} gives
$$
\| \psi\|_{\widetilde \HHH_p^{r',s',0}(R)}\le \Cs
\| \psi\|_{\widetilde \HHH_p^{r',0,0}(R)}
\le \Cs \| \psi\|_{ H_p^{r'}(X_0)}\, .
$$
To prove the bound on the essential spectral
radius,  take a high enough iterate $n$, depending on 
$ |z|/R$ and on  $r$, $r'$, $s$, and 
use that  $\widetilde \HHH_p^{r,s,0} \subset \widetilde \HHH_p^{r-|s|,0,0}$ 
and that
the inclusion $\widetilde \HHH_p^{r-|s|,0,0} \subset H^{r'}_p(X_0)$
is compact by \eqref{cpctt} in
Lemma~\ref{embed} since $r'<s<r-|s|$ and $-\beta < r-|s|=r+s$.
(We use Hennion's theorem \cite{hen}.)
\end{proof}

It will next be easy to bound the spectral radius of $\RR(z)$:

\begin{lemma}[Spectral radius of $\RR(z)$]\label{rr}
Under the assumptions  of Lemma~\ref{ress}, 
for each $z \in \complex$ with $a=\Re z > A$, 
the operator $\RR(z)$ on $\widetilde \HHH_p^{r,s,0}$ has spectral radius
bounded by $a^{-1}$. In addition, if there exists $\psi\in \widetilde \HHH_p^{r,s,0}$ and $\rho\in\bC$, $|\rho|=a^{-1}$, such that $\RR(z)(\psi)=\rho\psi$, then $\psi\in L^\infty$. Conversely, if there exists $\psi\in L^1$ and $|\rho|=a^{-1}$, such that $\RR(z)(\psi)=\rho\psi$,\footnote{Remember that $\RR(z)$ is a well-defined operator both on $L^\infty$ and $L^1$, abusing notation we use the same name for the operator defined on different spaces.} then
$\psi\in \widetilde \HHH_p^{r,s,0}\cap L^\infty$. 
\end{lemma}

\begin{proof}
Lemma~ \ref{ress} implies that the spectrum of $\RR(z)$ outside of
the disk $\{|\rho|\leq (a+\ln(1/\lambda))^{-1}\}$, for $\Re z=a>A$, consists only of
isolated eigenvalues of finite multiplicity. Let us assume that there is
a unique maximal eigenvalue $\rho(z)$ (the case of finitely many maximal eigenvalues is treated in exactly the same way, apart for the necessity of a heavier notation).
If $|\rho(z)|\geq a^{-1}$ then by spectral decomposition, \cite{Ka}, we can write
\[
\RR(z)=\rho(z)\Pi(z)+N(z)+Q(z)\, ,
\]
where\footnote{To ease notation, we will suppress the $z$ dependence when irrelevant or no confusion can arise.} $\Pi, N, Q$ commute and $\Pi Q=NQ=0$; $\Pi, N$ are finite rank, $\Pi^2=\Pi$, $\Pi N=N$, and, if $N\neq 0$, there exists $d(z)\in\bN$ such that $N(z)^{d(z)+1}=0$ while $N(z)^{d(z)}\neq 0$. In addition there exist $C(z)>0$, $\rho_0(z)<\rho(z)$ such that $\|Q^n(z)\|\leq C(z)\rho_0(z)^n$. Note that\footnote{The convergence is in $\widetilde \HHH_p^{r,s,0}$.}
\[
\begin{split}
\lim_{n\to\infty}\frac1{n^{d+1}}\sum_{k=0}^{n-1}\rho^{-k}\RR(z)^k
&=\lim_{n\to\infty}\frac1{n^{d+1}}\sum_{k=0}^{n-1}\left[\sum_{j=0}^d\binom kj\rho^{-j}\Pi N^j+\cO(\rho^{-k}\rho_0^k)\right]\\
&=C_d N^d\, ,
\end{split}
\]
for some appropriate constant $C_d>0$. 

In the following, we will use two properties of $\widetilde \HHH_p^{r,s,0}$: There exists a set $\cD\subset L^\infty(X_0)$ that is dense in $\widetilde \HHH_p^{r,s,0}$ (see \eqref{closure'} and remember that $\LL_t$ is a contraction in $L^\infty$); if $\psi\in \widetilde \HHH_p^{r,s,0}$ and $\int_M\psi \vf\, dx=0$ for all $\vf\in C^2$, then $\psi=0$ (this follows from the embedding properties stated in Lemma \ref{embed} and the fact that $C^\infty$ is dense in the usual Sobolev spaces).

Let $\psi\in \cD$, then for $k\geq 1$, using that $\LL_t$ preserves volume,
\begin{equation}\label{crucialvol}
\left|\rho^{-k}\RR(z)^k(\psi)\right|_\infty
\leq |\rho^{-k}|\int_0^\infty \frac{t^{k-1}e^{-at}}{(k-1)!}  |\psi\circ T_{-t}|_\infty \, dt\leq |\rho^{-k}| a^{-k}|\psi|_\infty\, .
\end{equation}
But then, if $|\rho(z)|>a^{-1}$, for each $\vf\in C^2$, we have
\begin{equation}\label{eq:r-growth}
\begin{split}
C_{d(z)} \left|\int N^{d(z)}(z)\psi\cdot \varphi\, dx\right|
&\leq  \lim_{n\to\infty}\frac1{n^{d(z)+1}}\sum_{k=0}^{n-1}\left|\rho(z)^{-k}\int \RR(z)^k(\psi)\cdot \varphi\, dx\right|\\
&\leq  \lim_{n\to\infty}\frac{1}{n^{d(z)+1}}\sum_{k=0}^{n-1}|\rho(z)|^{-k}a^{-k}|\psi|_\infty |\varphi|_{L^1}=0\, .
\end{split}
\end{equation}
Hence $N^d\psi=0$, but then the density of $\cD$ implies $N^d=0$ contrary to the hypotheses. The only possibility left is that $N=0$, but then, arguing as before, one would obtain $\Pi=0$, also a contradiction.

Next, note that if $|\rho(z)|=a^{-1}$, then \eqref{eq:r-growth} implies again $N=0$ (no Jordan blocks). In other words we have the spectral representation
\begin{equation}\label{eq:spec-bound}
\RR(z)=\sum_{k}a^{-1}e^{i\theta_k(z)}\Pi_k(z)+Q(z)\, ,
\end{equation}
where $\theta_k\in\bR$, $\Pi_k\Pi_j=\delta_{jk}\Pi_k$, $\Pi_kQ=Q\Pi_k=0$ and $\|Q^n\|\leq C\rho_0^n$ for some constants $C>0$, $\rho_0\in (0,a^{-1})$.

Moreover, for $\theta\in\bR$, we have
\begin{equation}\label{eq:standard-spec}
\lim_{n\to\infty} \frac 1n\sum_{m=0}^{n-1} a^{m}e^{-im\theta}\RR(z)^m=\begin{cases}\Pi_k(z) \quad&\text{ if }\theta=\theta_k(z)\\
                   0&\text{otherwise.}
\end{cases}
\end{equation}
Hence, for each $\psi\in\cD$ and $\vf\in C^\infty$, arguing as in \eqref{crucialvol},
\begin{equation}\label{eq:Pi-bound}
\left|\int \Pi_k(z) \psi\cdot \vf\, dx\right| \leq |\psi|_{L^\infty} |\vf|_{L^1}\, .
\end{equation}
This implies that $\Pi_k(\cD)\subset L^\infty$, but since the range of $\Pi_k$ is finite-dimensional, it follows that the $\Pi_k(z)$ are bounded operators from $\widetilde \HHH_p^{r,s,0}$ to $L^\infty$.

On the other hand, suppose that  there exist $\psi(z)\in L^1\setminus\{0\}$ and $\rho(z)$, $|\rho(z)|=a^{-1}$, such that 
$\RR(z)(\psi(z))=\rho(z)\psi(z)$, then we can consider a sequence $\{\psi_\ve(z)\}\subset C^\infty$ that converges to $\psi(z)$ in $L^1$ and consider as before
\[
\lim_{n\to\infty}\frac 1n\sum_{k=0}^{n-1}\rho^{-k}\RR(z)^k(\psi_\ve)\, .
\]
Note that by \eqref{eq:standard-spec} such a limit always exists, and 
it is zero if $\rho(z)\not\in\sigma (\RR(z))$, while it equals $\Pi(z)(\psi_\ve(z))$ if $\rho(z)\in\sigma (\RR(z))$, where $\Pi(z)$ is the eigenprojector associated to $\rho(z)$.
In the first case, for each $\vf\in C^2$,
\[
\begin{split}
\left|\int\psi(z) \vf\, dx\right|
&\leq \lim_{\ve\to 0}\lim_{n\to\infty}\frac 1n\sum_{k=0}^{n-1}a^k\int_0^\infty
e^{-at}\frac{t^{k-1}}{(k-1)!}|\psi_\ve(z)-\psi(z)|_{L^1}\cdot|\vf\circ T(t)|_{L^\infty}\, dt\\
&\leq \lim_{\ve\to 0}|\psi_\ve(z)-\psi(z)|_{L^1}\cdot|\vf|_{L^\infty}=0\, .
\end{split}
\]
We would then have $\psi(z)=0$, which is a contradiction. Hence $\rho(z)$ must be an eigenvalue, and by the same computation as above
\begin{equation}\label{eq:close-psi}
\left|\int(\psi(z)-\Pi(z)(\psi_\ve(z))) \vf\right|\leq |\psi_\ve(z)-\psi(z)|_{L^1}\cdot|\vf|_{L^\infty}\, .
\end{equation}
Since $\Pi(z)$ is a projector we can write it as $\Pi(z)(\tilde \psi)=\sum_k\psi_k(z) [\ell_k(z)](\tilde \psi)$, where $\psi_k(z)\in \widetilde \HHH_p^{r,s,0}\cap L^\infty$, the $\ell_k(z)$ belong to the dual of $\widetilde \HHH_p^{r,s,0}$ and $\ell_k(\psi_j)=\delta_{kj}$. Then \eqref{eq:close-psi} shows that the sequences $\ell_k(\psi_\ve)$ are bounded. We can then extract a  subsequence $\{\psi_{\ve_j}(z)\}$ such that $\Pi(z)(\psi_{\ve_j}(z))$ is convergent. In turn, this shows that $\psi(z)$ is a linear combination of the $\psi_k(z)$, which concludes the proof.
\end{proof}

As a consequence of Lemmata ~\ref{ress} and ~\ref{rr},  we get:

\begin{corollary}[Spectrum of $X$]\label{spX}
Under the assumptions  of Lemma~\ref{ress},
the spectrum of $X$ on  $\widetilde \HHH=\widetilde \HHH^{r,s,0}_p(R)$
is contained in the left half-plane $\Re z \le 0$.
Also, the spectrum of $X$ on  $\widetilde \HHH$ in the half-plane
$
\{ z \in \complex \mid \Re z > \ln \lambda \}
$
consists of at most countably many isolated points, which are all eigenvalues of finite
multiplicity. The spectrum on the imaginary axis is a finite union of discrete additive subgroup of $\bR$.
If the flow is ergodic, then the eigenvalue zero has algebraic multiplicity one.
\end{corollary}

\begin{proof}
A nonzero
$\rho\in \complex$ lies in the spectrum of $\RR(z)$ on $\widetilde \HHH$ 
if and only if $\rho=(z-\rho_0)^{-1}$,
where $\rho_0$ lies in the
spectrum of $X$ as a closed operator on $\widetilde \HHH$ 
(see e.g. \cite[Lemma 2.11]{Da}).
The first claim then follows from Lemma ~ \ref{rr}. Indeed, if $\rho_0$ is in the spectrum
of $X$, then for all $a\ge A$ and all $b$, we have $|\rho_0-a-ib|\ge a$. The corresponding complement
of union of discs is contained in the left-half plane $\Re (\rho_0)\le 0$.
Similarly,
the second claim follows from Lemma ~ \ref{ress}.

To prove the third claim note that if $X(\psi)=ib\psi$, $b\neq 0$, then, for $z=a+ib$, $\RR(z)(\psi)=a^{-1}\psi$. 
Another simple computation, using
\cite[Theorem 1.7]{Da} 
(noting that $\tilde \psi_t = e^{
ib t} \psi$ satisfies $\partial_t \tilde \psi_t|_{t=s}= X (\tilde \psi_s)$
so that $\LL_t (\tilde \psi_0)=\tilde \psi_t$) gives 
\begin{equation}\label{eq:flow-inv}
\psi \circ T_t=e^{-ib t} \psi\, .
\end{equation}
Moreover Lemma \ref{rr} implies that $\psi\in L^\infty$. Then if $X(\psi_k)=ib_k\psi_k$, $k\in\{1,2\}$, we have $\psi_1,\psi_2\in L^\infty$ and
\[
\begin{split}
\RR(z)(\psi_1\psi_2)&=\int_0^\infty e^{-zt}(\psi_1\circ T_{-t})(\psi_2\circ T_{-t})\, dt
=\psi_1\psi_2 \int_0^\infty e^{-zt +i(b_1+b_2) t}dt\\
&=(z-ib_1-ib_2)^{-1}\psi_1\psi_2\, .
\end{split}
\]
By Lemma~ \ref{rr} again, it follows that either $\psi_1\psi_2=0$ or $ib_1+ib_2\in\sigma(X)$. On the other hand, a similar argument applied to $\bar\psi_k$ shows that $-ib_k\in\sigma(X)$. Thus $|\bar \psi_k|^2$ belongs to the finite dimensional eigenspace of the eigenvalue zero and $\{im b_k\}_{m\in\bZ}\subset \sigma(X)$. Finally, if $A$ is a positive measure invariant set, then $\Id_A$ is a eigenvector associate to zero and if $\psi$ is an eigenvector associated to zero then $\{\psi\geq \lambda\}$ are invariant sets,\footnote{We can assume $\psi$ real since, if not, then its real and imaginary part must also be invariant.} that is $\psi$ must be piecewise constant (otherwise zero would have infinite multiplicity). In other words the eigenspace of zero is spanned by the characteristic functions of the ergodic decomposition of Lebesgue.
\end{proof}

In the setting of  \cite{Li}, it was straightforward to bound
the norm of $\RR(z)^n$ by $C a^{-n}$.
Here, by Lemma~\ref{rr} we have for each $\eta>0$
a constant $C_\eta(z)$ so that $\| \RR(z)^n\|\le C_\eta(z) (a-\eta)^{-n}$
for all $n$. This abstract nonsense bound (with no control on the
$z$-dependence of $C_\eta$) will not suffice. In addition, we shall need
in Section~\ref{dodo}
a Lasota-Yorke type estimate for $\RR(z)$  improving
the one obtainable from \eqref{LYR} (which contains an unfortunate
$(a-A)^n$ factor). This is
the purpose of the following lemma:

\begin{lemma}[Lasota-Yorke estimate for $\RR(z)$]\label{controlCeta}
For $1<p<\infty$, $s<-r<0<r$, and  $R>1$ as in
setting of Lemma~\ref{LY0}, assume  in addition that  complexity is subexponential
so that  $\lambda <1$ (up to choosing a  larger
$t_0$), and assume that $|s|\in (0, 2r)$.
Then 
there exists $\Lambda>0$, depending on
$p$, but not on $r$, $s$, 
and there exists $\Cs$, so that for any  $N\ge 1$ such that
\begin{equation}\label{replacewidehatC}
(1+3N)r < \min\biggl \{\frac{1}{3},\frac{1}{p}, 1-\frac{1}{p} \biggr \}\, ,
\end{equation}
then, for  all $z=a+ib$, for $a>A$ (with $A$  given by \eqref{reallybounded'}),
and all $n \ge 0$, we have 
\begin{align}
\label{controlCeta'} 
\|\RR(z)^{n+1} \|_{\widetilde \HHH_p^{r,s,0}(R)}\le
\Cs^N \frac{  (1+\frac{|z|}{R})^{N(r-s)}}{(a-\Lambda |s|-\frac{A}{N})^n}  \, ,
\end{align}
and, for every $-1+1/p<s'\le s$,
\begin{align}
\label{controlCeta''} 
\|\RR(z)^{n+1} (\psi) \|_{\widetilde \HHH_p^{r,s,0}(R)}&\le
\Cs (\frac{1}{a-A}+1)\frac{  (1+\frac{|z|}{R})^{N(r-s)}}{(a+\ln (1/\lambda))^n}\|\psi\|_{\widetilde \HHH_p^{r,s,0}(R)}\\
\nonumber&
\qquad\qquad+
\Cs^N \frac{  (1+\frac{|z|}{R})^{(N-1)(r-s)+r-s'}}{(a-\Lambda |s|-\frac{A}{N})^n} 
\|\psi\|_{\widetilde \HHH_p^{s',0,0}(R)} \, .
\end{align}
\end{lemma}

This lemma is proved in Section ~\ref{LYBG}, after the proof of Lemma~\ref{LY0}.

The main bound in the paper is the following
Dolgopyat-like estimate (in the style of \cite[Prop 2.12]{Li}), which will
be proved in Section~ \ref{dodo}:

\begin{proposition}\label{dolgo}
Under the assumptions of Lemma~\ref{ress},
and if $d=3$, 
then, up to taking larger $p>1$ and smaller $|r|$ and $|s|$, there exist 
$C_A \ge 10$, $b_0 \ge 1$,
$0<\tilde c_1<\tilde c_2$, and $\nu \in (0,1)$,  so
that 
$$
\| \RR(a+ib)^{n}\|_{\widetilde \HHH^{r,s,0}_p(R)} \le 
\left ( \frac{1}{a+\nu}\right )^{ n}\, ,
$$ 
for all  $|b|>b_0$, 
$a \in[ C_A A, b]$ and
$n \in \lceil \tilde c_1 a \ln |b|, \tilde c_2 a \ln | b| \rceil$.
\end{proposition}

(The assumption that $d=3$ is only used to prove
Lemma~\ref{dolgolemma} which is a key ingredient of the proof of Proposition~\ref{dolgo}.)

Proposition~ \ref{dolgo} immediately implies the following strengthening
of Corollary ~ \ref{spX} (just like the proof
of \cite[Cor. 2.13]{Li}):

\begin{corollary}\label{specX'}
Under the assumptions of Lemma~\ref{ress}, if $d=3$
and $s<0<r$, $1<p<\infty$ are given by Proposition~\ref{dolgo}, then
there exists  $\delta_0 >0$ so that
the spectrum of $X$ on $\widetilde \HHH^{r,s,0}_p(R)$
in the half-plane
$$
\{ z \in \complex \mid \Re (z) > -\delta_0 \}
$$
consists only of the eigenvalue $0$. If the flow is ergodic zero is a simple eigenvalue.
\end{corollary}
\begin{proof}
By Proposition \ref{dolgo} and Corollary \ref{spX}  the set
$\{z\in \complex\;:\; \Re(z)>-\nu, \, |\Im(z)|> b_0\}$ and $\{z\in\complex\;:\; \Re(z)>0\}$ is included in
 the resolvent set of $X$. This can be deduced using, for $a,A',b\in\bR$,
\begin{equation}\label{2.36}
\RR(a+ib)=(1+(a-A') \RR(A'+ib))^{-1} \RR(A'+ib) 
\, .
\end{equation}
On the other hand, Corollary~  \ref{spX} implies that in the region $\{z\in \complex\;:\; \Re(z)>-\nu, \, |\Im(z)|\leq b_0\}$ there can be only finitely many eigenvalues. The 
first statement of the
lemma would then follow if we could prove that zero is the only eigenvalue on the imaginary axis. Suppose this is not the case and there exists $\psi$ such that $X(\psi)=ib\psi$, $0\neq |b|\leq b_0$. Then Corollary~  \ref{spX} implies that, for each $m\in \bZ$, $ibm\in\sigma(X)$, but this leads to a contradiction 
with Proposition~\ref{dolgo}
by choosing $m> b_0b^{-1}$. The last statement follows from Corollary~  \ref{spX} again.
\end{proof}

Our main theorem will then  follow:

\begin{proof}[Proof of Theorem~ \ref{main}]
Exponential decay for $\xi$-H\"older observables can be deduced from exponential decay
for $C^1$ observables by a standard approximation argument
(which may modify the decay rate). 
So it suffices to show that there exists $\sigma >0$ 
and $C>0$ so that for each $\psi$, $\varphi$ in $C^1$ with $\int \psi\, dx=0$
\begin{equation}\label{keybd}
|\int \varphi \LL_t (\psi) \, dx | 
\le C e^{-\sigma t} \| \psi\|_{C^1}  \| \varphi\|_{C^1}\, .
\end{equation}
Indeed, for any $\psi$, $\varphi$ in $C^1$, since
$\LL_t$ fixes  constant functions,
\begin{align*}
|\int \psi (\varphi \circ T_t) \, dx&-\int \psi \, dx \int \varphi\, dx |=\int \LL_t\biggl (\psi- 1 \cdot \int \psi\, dx \biggr ) \varphi \, dx \, .
\end{align*}

To show (\ref{keybd}),
like in \cite[Proof of Theorem 2.4]{Li}, we 
shall apply the following easily checked fact: Let $\BB$
be a Banach space on which $\LL_t$ is bounded. Then
for any $z$ in the resolvent set of $X$ (for $\BB$)
and any $\psi$ in the
domain of $X^2$ (for $\BB$) we have  (in $\BB$)
\begin{equation}\label{214}
\RR(z) (\psi) = z^{-1} {\psi} + z^{-2} X( \psi )
+ z^{-2}  \RR(z) (X^2 (\psi))\, .
\end{equation}

In view of applying \eqref{214} to $\BB=\widetilde \HHH=\widetilde \HHH^{r,s,0}_p$
(which will be necessary to exploit Proposition~\ref{dolgo} below), we 
fix a $C^\infty$ function $\phi:\real^+\to \real^+$, supported in $(0,1)$, with
$\int \phi(u)\, du=1$,
and, for $\psi \in C^1$,
we define (as in \cite[Proof of Theorem 2.4]{Li})
$$
\psi_\epsilon=\int_0^\infty \frac{ \phi(t/\epsilon)}{\epsilon}\LL_t(\psi)\, dt
\, .$$
(Note that $\int \psi_\epsilon\, dx=0$.)
For each $m\ge 1$ the function $\psi_\epsilon$ belongs to the domain of $X^m$
for $\widetilde \HHH$, and, letting $r_0=r+|s|$ (we assume
that $r_0 <1/p$), an easy computation shows
\begin{equation}\label{eq:XM}
\|X^m(\psi_\epsilon)\|_{H^{r_0}_p(X_0)}
\le C \epsilon^{-m} \|\phi^{(m)}\|_{L^1} \|\psi\|_{C^1}\, ,\, 
m=0, 1, 2 \, .
\end{equation}
In addition
\begin{equation}\label{app}
\|\psi_\epsilon -\psi\|_{L^\infty}
\le \int_0^\infty  \frac{\phi(u/\epsilon)}{\epsilon}
\|\psi\circ T_{-u} -\psi\|_{L^\infty} \, du
\le \epsilon \|\psi\|_{C^1} \, .
\end{equation}
Next, if $\psi, \varphi \in C^1$ then
\begin{equation}\label{eq:split-epsilon}
\begin{split}
|\int \varphi \LL_t (\psi) \, dx | 
&\le |\int \varphi \LL_t (\psi_\epsilon) \, dx | 
+|\int \varphi \LL_t (\psi_\epsilon -\psi) \, dx | \\
&\le |\int \varphi \LL_t (\psi_\epsilon) \, dx | 
+\|\varphi \|_{L^1} \cdot \|\psi-\psi_\epsilon\|_{L^\infty}
\end{split}
\end{equation}
To prove \eqref{keybd} it suffices then to prove that, for each zero average $\psi\in C^1$,
\begin{align}\label{initialpart}
&|\int \varphi \LL_t (\psi) | 
\le C e^{-\sigma_0 t} (\| X^2 (\psi)\|_{H^{r_0}_p(X_0)}+\|X (\psi)\|_{H^{r_0}_p(X_0)})  \| \varphi\|_{C^1}\, .
\end{align}
Indeed,  using \eqref{app}, \eqref{initialpart} and \eqref{eq:XM} to estimate \eqref{eq:split-epsilon} the result follows with $\sigma=\frac{\sigma_0}3$ after choosing $\varepsilon^3=e^{-\sigma_0t}$. 

Let us prove \eqref{initialpart}. We claim that \eqref{214} implies  that 
for each $a_0 >0$ 
\begin{equation}\label{2.8}
\LL_t(\psi)= \frac{1}{2 i \pi} \lim_{w \to \infty}
\int_{-w}^{w}  e^{a_0 t+ibt} \RR(a_0+ib) \psi\, db\, ,
\, \forall t >0  \, ,
\end{equation}
in the $L^\infty(X_0)$ norm.
Indeed, noting that the resolvent set
of $X$ for $L^\infty(X_0)$ contains the half-plane
$\Re (z) >0$, the identity \eqref{214} 
for $\BB=L^\infty$ implies
that for any $z$ with  $|z| > \zeta$
\begin{equation}\label{bis}
\| \RR(z) (\psi)\|_{L^\infty}
\le  \frac{ (\|\psi\|_{L^\infty}+\zeta^{-1} \|X(\psi)\|_{L^\infty}+
\zeta^{-1} \|\RR(z) (X^2(\psi))\|_{L^\infty})}{|z|}\, .
\end{equation}
Hence (adapting \cite[proof of (2.8), footnote 9]{Li}),
for almost all $x \in X_0$ and each fixed $a> \zeta$ the 
function $b \mapsto (\RR(a+ib)\psi) (x)$ is in $L^2(\real)$.
One can thus apply the inverse Laplace transform  for such
$x$ and get \eqref{2.8} pointwise (that is, the limit \eqref{2.8} takes place in
the $L^2([0,\infty], e^{-at}\, dt)$ sense, as a function of $t$).
On the other hand, $t\mapsto \LL_t (\psi)\in L^\infty(X_0)$
is continuous, and, using again \eqref{214},
$b \mapsto \RR(a+ib)\psi-(a+ib)^{-1}\psi$ is in $L^1(\real, L^\infty)$,
while, clearly, $b \mapsto \frac{e^{(a+ib)t}}{a+ib}$ is in $L^1(\real)$.
Hence, the limit in \eqref{2.8} converges in the $L^\infty(X_0)$ norm
for each $t\in \real ^+$.

The inverse Laplace transform expression  \eqref{2.8}
will be our starting point. 
We shall use a change of contour
to obtain an integral over a  vertical in the
left half-plane $\Re (z) < 0$. For this, we  first
study the map $z \mapsto \RR(z) (\psi)$.

Since the simple eigenvalue zero for $X$ corresponds to the eigenvector
$dx$ for $X^*$, and since $\int \psi\, dx=0$, Corollary ~ \ref{specX'} implies that
for suitable $s<0<r$, $1<p<\infty$, $R>1$, and $\delta_0>0$, 
the function $z \mapsto \RR(z)( \psi)$ is analytic from
the strip $\{ \Re (z) >- \delta_0\}$ to $\widetilde \HHH_p^{r,s,0}(R)$. 
Pick $-\sigma_0 \in (-\delta_0, 0)$ and
fix $b_0>1$.

We may assume that $a_0\le 1$, $\sigma_0\le 1$ and
$A\ge 1$ (recall \eqref{reallybounded'}).
We claim that  Proposition ~ \ref{dolgo}
implies that, up to taking smaller $r$ and $|s|$,
and larger
$b_0$ (possibly depending on $r$ and $s$), 
there exists  $K_1>0$ 
so that\footnote{We pick $1/2$ because any exponent 
$<1$ suffices, 
we could get arbitrarily small  exponent $>0$.} 
for $b_0\le |b| $,
\begin{equation}\label{claim1/2}
\sup_{a \in [-\sigma_0, a_0]} \| \RR(a+ib)\| \le 
K_1 \sqrt{ |b| }\, .
\end{equation}

Let us check \eqref{claim1/2}. Following \cite[Proof of Thm 2.4]{Li} we shall use
\eqref{2.36}.
Fix $A'\in [C_A A, b_0]$ and  $a \in [-\sigma_0, a_0]$ (in particular $|a|<A'$).
Setting $n= \lceil \tilde c_2 A' \ln |b| \rceil$, 
Proposition ~ \ref{dolgo} implies that  for some
$\nu>0$ 
$$
\| (a-A')^n \RR(A'+ ib)^n\|\le  (A'-a)^n (A'+\nu)^{-n}\le
(1+\nu/A')^{-n} \, .
$$
Clearly, if $|b|> b_0>1$   and $A'$ is large enough, then
$
(1+\nu/A')^{-\tilde c_2 A' \ln |b|} <  |b_0|^{-\tilde c_2 \nu}
$.

Next, if $N\ge 1$ is such that
$(1+3N)r<\min (1/3, 1/p, 1-1/p)$, 
then \eqref{controlCeta'} applied to $z=A'+ib$ and
$N$ gives,  
for all $0\le j < n$
\begin{align*}
\|(a-A')^j \RR(A'+ ib)^j\|
&\le \Cs^N \frac{ (1+|z|^{N(r-s)})(A'-a)^j}{(A'-\Lambda |s|-A/N)^{j}}
\\
&
\le \Cs^N  \frac{(1+|z|^{N(r-s)})}
{ (1-(\Lambda|s|+A/N)/A')^{j}} \, .
\end{align*}
Up to taking smaller  $r$,  larger $N$,
smaller $|s|$,  
we may assume that $A'$ is large enough so that
$$
\tilde c_2 \ln \biggl [\bigl (\frac{1}{1-((\Lambda|s|+A/N))/A'}  \bigr )^{A'} \biggr ]<2 \tilde c_2 \ln(e^{\Lambda|s|+A/N}) <1/8
\, .
$$
Finally, we can assume that $b_0>2 A'$ and
$N(r-s)<1/4$  so that
$$
(1+|z|)^{N(r-s)}\le
\Cs |b|^{N(r-s)}\le \Cs |b|^{1/4}\, .
$$
Therefore,  we can find a constant
$
K_1
$
so that for any $a \in [-\sigma_0, a_0]$,
$b_0\le |b|$, and large enough $A'$
\begin{align*}
&\|(1 + (a-A') \RR(A'+ ib))^{-1}\|\\
&\qquad\le \sum_{k=0}^\infty
\|(a-A')^{kn} \RR(A'+ ib)^{kn}\| \sum_{j=0}^{n-1}\|(a-A')^j \RR(A'+ ib)^j\|\\
&\qquad\le
\frac{\Cs}{1-(1+\frac{\nu}{A'})^{-\tilde c_2 A' \ln |b|}}
(1+|b|^{N(r-s)})c_2 A' \ln |b| \left (\frac{1}{1-\frac{\Lambda|s|+A/N}{A'}} \right )^{\tilde c_2 A' \ln |b|}
\\
&\qquad
\le K_1 |b|^{1/2} \, , \mbox{ proving (\ref{claim1/2}).}
\end{align*}

Now, since  $\|X^m(\psi )\|_{\widetilde \HHH}\le \Cs \|X^m(\psi)\|_{H^{r_0}_p(X_0)}$ 
for  $m=1, 2$, the identity (\ref{214}) gives the following upper bound for
$\| \RR(z) \psi\|_{\widetilde \HHH}$:
\begin{equation}  \frac{ \|\psi\|_{\widetilde \HHH}}{|z|}+
C(R)
\biggl (\frac{ \|X(\psi)\|_{H^{r_0}_p(X_0)}}{|z|^2}+
\frac{ \|\RR(z)\|_{\widetilde \HHH} \| X^2(\psi)\|_{H^{r_0}_p(X_0)}}{|z|^2}\biggr )\, .
\end{equation}
(The constant $C(R)$ may depend on $R$, but $R$ is fixed, so we
shall replace it by $\Cs$, slightly abusing notation.)

Therefore, 
the bound  (\ref{claim1/2}) implies that
for each fixed $b$ with $|b|\ge b_0$
the integral over the horizontal segments satisfies
\begin{align*}
&\|\int_{-\sigma_0}^{a_0}  e^{a t+ibt} \RR(a_0+ib) \psi\, da\|_{\widetilde \HHH}
\le \|\psi\|_{\widetilde \HHH}
\left |\int_{-\sigma_0}^{a_0} \frac{ e^{a t+ibt} }{a+ib}\, da
\right |\\
&\qquad \qquad\qquad +
\frac{2 \Cs |a+\sigma_0|e^{\max(\sigma_0,a)t}}{2\pi}
\left( \frac{\|X (\psi)\|_{H^{r_0}_p}}{|b|^{2}}+ K_1\frac{  \|X^2(\psi)\|_{H^{r_0}_p}}{|b|^{3/2}}\right)\, .
\end{align*}
Thus, since 
$$
\left|\int_{-\sigma_0}^{a_0}  \frac{e^{a t+ibt} }{a+ib}\, da\right|\le 
\frac{e^{\max(\sigma_0,a)t}}{|b|}
\, ,
$$ 
and  since 
the isomorphism
$H^{s}_p(X_0)\sim\HHH^{s,0,0}_p(R)$
given by Lemma~\ref{embed} imply that 
for any $\Psi \in \widetilde\HHH_p^{r,s,0}(R)$ and each $\varphi\in C^{|s|}$
\begin{equation}\label{trick}
|\int \Psi \varphi\, dx|
\le \Cs  \|\Psi\|_{H^{s}_p(X_0)} \|\varphi\|_{H^{|s|}_{p'}(X_0)}
\le\Cs  \|\Psi\|_{\widetilde \HHH} \|\varphi\|_{C^{|s|}} \, ,
\end{equation}
we get 
\begin{align*}
\|\int_{-\sigma_0}^{a_0}  e^{a t+ibt} \RR(a_0+ib) \psi\, da\|_{(C^{|s|})^*}
&\le \Cs  \|\int_{-\sigma_0}^{a_0}  e^{a t+ibt} \RR(a_0+ib) \psi\, da\|_{\widetilde \HHH}\\
&\le K_2 e^{\max(\sigma_0,a)t} \frac{\sum_{m=0}^2 \|X^m(\psi)\|_{H^{r_0}_p}}{|b|}\, ,
\end{align*}
which tends to zero for each fixed $t$ as
$|b|\to \infty$. Therefore, changing contours (the residue of the pole
at $z=0$ vanishes since $\int \psi\, dx=0$)
transforms (\ref{2.8}) into
$$
\LL_t(\psi)= \frac{1}{2 i \pi} \lim_{w \to \infty}
\int_{-w}^{w}  e^{-\sigma_0 t+ibt}
\RR(-\sigma_0+ib) (\psi)\, db \, ,
$$
where both sides above are viewed in $(C^1)^*$.
Since $|\int_{-w}^{w}  \frac{e^{-\sigma_0 t+ibt} }{-\sigma_0+ib}\, db|=O(|w|^{-1})$,
uniformly in $t$, we have
$$
\LL_t(\psi)= \frac{1}{2 i \pi} \lim_{w \to \infty}
\int_{-w}^{w}  e^{-\sigma_0 t+ibt}
(\RR(-\sigma_0+ib) (\psi)-\frac{\psi}{-\sigma_0+ib})\, db \, ,
$$
Thus, using again (\ref{214}),  (\ref{claim1/2}), 
and \eqref{trick}, we find a constant
$\Cs>0$ so that,
for  $\varphi\in C^1$, and arbitrary $t>0$,
\begin{equation}\label{keybddom}
\begin{split}
&|\int \varphi \LL_t (\psi) \, dx| \le \Cs  \frac{e^{-\sigma_0 t} \| \varphi\|_{C^1}}{2\pi}
\int_{\real} 
\| \RR(-\sigma_0+ib)   (\psi)-\frac{\psi}{-\sigma_0+ib}\|_{\widetilde \HHH} \, db  \\
&= \Cs   \frac{e^{-\sigma_0 t} \| \varphi\|_{C^1}}{2\pi}
\int_{\real} \frac{
\| \RR(-\sigma_0+ib)(X^2(\psi))+ X (\psi)\|_{\widetilde \HHH}}{ |-\sigma_0+ib|^{2}} \, db 
\\
&\le   \Cs \frac{  \| \varphi\|_{C^1}}{2\pi e^{\sigma_0 t} }
\bigl (  \|X^2(\psi)\|_{H^{r_0}_p(X_0)}\int \frac{K_1|b|^{1/2}}{\sigma_0+b^2}\, db +   \int \frac{ \|X(\psi)\|_{H^{r_0}_p(X_0)}}{\sigma_0+b^2}\, db \bigr )  \, .
\end{split}
\end{equation}
This proves equation \eqref{initialpart} and ends the proof of our main theorem.
\end{proof}

\section{The Lasota-Yorke estimates}\label{LYBG}

In this section, we prove the basic Lasota-Yorke estimate Lemma~ \ref{LY0} 
on $\LL_t$ and
the Lasota-Yorke estimate Lemma ~\ref{controlCeta} on $\RR(z)$. The section also includes
Lemma ~\ref{StrStr}, about multiplication by $1_{X_0}$.

The following easy lemma is the heart the proof of Lemma~ \ref{LY0}.
It is  the analogue of 
\cite[Lemma 4.6]{BG2}. Note however that a new phenomenon appears in
the present setting:
the loss of smoothness (of $r-r'$) in the time direction

\begin{lemma}
\label{CompositionDure}
For all $s<-r\le 0\le r$,  for all $p\in (1,\infty)$, $q\ge 0$,
and every $r'<r$, $s'\le s$, there exists a constant $\Cs$, depending only
on $r$, $s$, $p$, $s'$, $r'$, such that the following holds:

Let $D=\left(\begin{smallmatrix} A&0&0\\0&B&0\\0&0&1
\end{smallmatrix}\right)$
be a block diagonal matrix, with $A$ of dimension $d_u$,
$B$ of dimension $d_s$, and $1$ a scalar.
Assume that there exist $\lambda_u>1$,
$\lambda_s<1$ such that $|Av|\ge \lambda_u |v|$ and
$|Bv|\le \lambda_s |v|$. Then
\begin{align*}
&  \norm{w\circ D^{-1}}{H_p^{r,s,q}}\le \Cs
|\det D|^{-1/p}\bigl (\max(\lambda_u^{-r}, \lambda_s^{-(r+s)})\norm{w}{H_p^{r,s,q}}
 + \norm{w}{H_p^{r',s', q+r-r'}} \bigr )\, \\
  & \norm{w\circ D^{-1}}{H_p^{r,s,q}}\le \Cs|\det D|^{-1/p}\norm{w}{H_p^{r,s,q}}  \, .
\end{align*}
\end{lemma}

\begin{proof}[Proof of Lemma~\ref{CompositionDure}]
Write $\transposee{D^{-1}}=\left(\begin{array}{ccc} U&0&0\\0 & S&0\\0&0&1
\end{array}\right)$ with $|U \xi^u|\le
\lambda_u^{-1}|\xi^u|$ and $|S \xi^s|\ge \lambda_s^{-1} |\xi^s|$.
Let
\begin{align*}
b(\xi^u,\xi^s,\xi^0)&=a_{r,s,q}\circ \transposee{D^{-1}}(\xi^u,\xi^s,\xi^0)\\
&=(1+|U \xi^u|^2+|S\xi^s|^2+|\xi^0|^2)^{r/2}
(1+|S\xi^s|^2)^{s/2}  (1+|\xi^0|^2)^{q/2} \, .
\end{align*}
Let us prove that there exist $K_1$ and $K_1'$ depending only on $r$ and $s$
(but not on $D$), and $K_2$ depending only on $r$, $s$, $s'\le s$,  $r'<r$
(but not on $D$) so that we have
\begin{align}
\label{eq:opuispoif}
&b \le K_1 \max(\lambda_u^{-r}, \lambda_s^{-(r+s)})
a_{r,s,q}+ K_2 a_{r',s',q+ r-r'} \, , \qquad b \le K'_1 
a_{r,s,q}\, .
\end{align}
Assume that we can prove this bound, as well as the
corresponding estimates for the successive derivatives of $b$.
Then the Marcinkiewicz multiplier theorem applied to $b/(K_1
\max(\lambda_u^{-r}, \lambda_s^{-(r+s)}) a_{r,s,q}+ K_2
a_{r',s',q+r-r'})$ as in  \cite[Lemma 25]{BG1} gives
\begin{align*}
&\norm{\FFF^{-1}(b \FFF v)}{L^p} \le
C \norm{ \FFF^{-1}((K_1 \max(\lambda_u^{-r}, \lambda_s^{-(r+s)})
a_{r,s,q}+ K_2 a_{r',s',q+r-r'}) \FFF v)}{L^p}\, ,\\
&  \norm{\FFF^{-1}(b \FFF v)}{L^p} \le
C \norm{ \FFF^{-1}((K'_1 
a_{r,s,q}) \FFF v)}{L^p}\, ,
\end{align*}
(recall that $\FFF$ denotes the Fourier transform)
which yields the  two  claims of the lemma, using the arguments in
the first part of the proof of \cite[Lemma 25]{BG1}.

Let us now prove \eqref{eq:opuispoif} (the proof for the
derivatives of $b$ is similar). We shall freely use the
following trivial inequalities:
for $x\ge 1$ and $\lambda\ge
1$,
\begin{equation}
\frac{1}{\lambda}(1+\lambda x)\le 1+x \le
\frac{2}{\lambda}(1+\lambda x) \, .
\end{equation}
Assume first $  |S\xi^s|^2>|U\xi^u|^2$ and $|S\xi^s|^2 \ge 1$.
Then,  if $|S\xi^s|^2\ge |\xi^0|^2$, we get since $r\ge 0$ and $r+s<0$,
\begin{align*}
b(\xi^u,\xi^s,\xi^0) &\le (1+3 |S\xi^s|^2)^{r/2} (1+|S\xi^s|^2)^{s/2}
(1+|\xi^0|^2)^{q/2}\\
& \le 3^{r/2} (1+|S\xi^s|^2)^{r/2}(1+|S\xi^s|^2)^{s/2}(1+|\xi^0|^2)^{q/2}
\\&\le 3^{r/2} (1+\lambda_s^{-2} |\xi^s|^2)^{(r+s)/2}(1+|\xi^0|^2)^{q/2}\\
&\le 3^{r/2} (\lambda_s^{-2}/2)^{(r+s)/2} (1+|\xi^s|^2)^{(r+s)/2}(1+|\xi^0|^2)^{q/2}
\\& \le3^{r/2} 2^{-(r+s)/2} \lambda_s^{-(r+s)} a_{r,s,q}(\xi^u,\xi^s,\xi^0) \, ,
\end{align*}
and if $|S\xi^s|^2<|\xi^0|^2$, we get since $r\ge 0$ and $s<0$
\begin{align*}
b(\xi^u,\xi^s,\xi^0) &\le (1+3 |\xi^0|^2)^{r/2} (1+|S\xi^s|^2)^{s/2}
(1+|\xi^0|^2)^{q/2}\\
&\le 3^{r/2} (1+|\xi^0|^2)^{r/2}(1+|S\xi^s|^2)^{s/2}(1+|\xi^0|^2)^{q/2}
\\&\le 3^{r/2} (1+|\xi^u|^2+|\xi^s|^2+|\xi^0|^2)^{r/2} 
(1+\lambda_s^{-2} |\xi^s|^2)^{s/2}(1+|\xi^0|^2)^{q/2}\\
& \le 3^{r/2} (\lambda_s^{-2}/2)^{s/2} a_{r,s,q}(\xi^u,\xi^s,\xi^0)
\\& \le3^{r/2} 2^{-s/2} \lambda_s^{-(r+s)} a_{r,s,q}(\xi^u,\xi^s,\xi^0) \, .
\end{align*}

If $|U\xi^u|^2 >\max( |S\xi^s|^2, |\xi^0|^2)$ and $|U\xi^u|^2 \ge 1$, then
since $r\ge 0$ and $s<0$,
\begin{align*}
b(\xi^u,\xi^s,\xi^0)& \le (1+3 |U\xi^u|^2)^{r/2} (1+|S\xi^s|^2)^{s/2}
(1+|\xi^0|^2)^{q/2}\\
&\le 3^{r/2} (1+|U\xi^u|^2)^{r/2} (1+\lambda_s^{-2} |\xi^s|^2)^{s/2}(1+|\xi^0|^2)^{q/2}
\\&\le 3^{r/2} (1+ \lambda_u^{-2} |\xi^u|^2)^{r/2} (1+|\xi^s|^2)^{s/2}
(1+|\xi^0|^2)^{q/2}\\
&\le 3^{r/2} (2\lambda_u^{-2})^{r/2} (1+|\xi^u|^2)^{r/2} (1+|\xi^s|^2)^{s/2}(1+|\xi^0|^2)^{q/2}
\\&\le 3^{r} \lambda_u^{-r} a_{r,s,q}(\xi^u,\xi^s,\xi^0) \, .
\end{align*}

If $|\xi^0|^2\ge |U\xi^u|^2\ge |S\xi^s|^2$,
we get if $|S\xi ^s|^2\ge 1$, since $r\ge 0$ and $s<0$, 
\begin{align*}
b(\xi^u,\xi^s,\xi^0)& \le (1+3 |\xi^0|^2)^{r/2} (1+|S\xi^s|^2)^{s/2}(1+|\xi^0|^2)^{q/2}\\
&\le 3^{r/2} (1+|\xi^0|^2)^{r/2} (1+\lambda_s^{-2} |\xi^s|^2)^{s/2}(1+|\xi^0|^2)^{q/2}
\\&\le 3^{r/2} (\lambda_s^{-2}/2)^{s/2} (1+  |\xi^0|^2)^{r/2} (1+|\xi^s|^2)^{s/2}(1+|\xi^0|^2)^{q/2}\\
&\le 3^{r/2}   (\lambda_s^{-2}/2)^{s/2}(1+|\xi^u|^2+|\xi^s|^2+ |\xi^0|^2)^{r/2}    (1+|\xi^s|^2)^{s/2}(1+|\xi^0|^2)^{q/2}
\\&\le 3^{r/2}  (\lambda_s^{-2}/2)^{(r+s)/2}a_{r,s,q}(\xi^u,\xi^s,\xi^0) \, ,
\end{align*}
and, finally, if $|S\xi ^s|^2\le 1$, on the one hand,
\begin{align*}
b(\xi^u,\xi^s,\xi^0)& \le (1+3 |\xi^0|^2)^{r/2} (1+|S\xi^s|^2)^{s/2}
(1+|\xi^0|^2)^{q/2}\\
&\le  C_s 3^{r/2} (1+|\xi^0|^2)^{r/2} (1+ |\xi^s|^2)^{s/2}(1+|\xi^0|^2)^{q/2}\\
&\le C_s   3^{r/2}   (1+|\xi^u|^2+|\xi^s|^2+ |\xi^0|^2)^{r/2}     (1+|\xi^s|^2)^{s/2} 
(1+|\xi^0|^2)^{q/2}\, ,
\end{align*}
and on the other hand,
 \begin{align*}
&b(\xi^u,\xi^s,\xi^0) \le (1+3 |\xi^0|^2)^{r/2} (1+|S\xi^s|^2)^{s/2}
(1+|\xi^0|^2)^{q/2}\\
&\qquad
\le C_{s,s'} 3^{r/2} (1+|\xi^0|^2)^{r/2} (1+ |\xi^s|^2)^{s'/2}(1+|\xi^0|^2)^{q/2}
\\&\qquad
\le C_{s,s'} 3^{r/2}  (1+  |\xi^0|^2)^{r'/2} (1+  |\xi^0|^2)^{(q+r-r')/2}  (1+|\xi^s|^2)^{s'/2}\\
&\qquad
\le C_{s,s'} 3^{r/2}   (1+|\xi^u|^2+|\xi^s|^2+ |\xi^0|^2)^{r'/2}  
(1+  |\xi^0|^2)^{(q+r-r')/2}   (1+|\xi^s|^2)^{s'/2}
\\&\qquad\le C_{s,s'} 3^{r/2} a_{r',s',q+r-r'}(\xi^u,\xi^s,\xi^0) \, .
\end{align*}
Thus,
\eqref{eq:opuispoif} follows by choosing $K_2$ large enough depending
on $s$ and $s'$.
\end{proof}

Combining the lemma we just proved with  the results in
Appendix~\ref{localspaces} and Appendix~\ref{iteratechart}, we now prove
the Lasota-Yorke type estimate:

\begin{proof}[Proof of Lemma~\ref{LY0}]
We start with \eqref{notLY}.
We claim that  there exist $\Lambda>1$,
$\tilde \tau_0$, $\tilde \tau_1$,  
and $\Cs>0$ so that, for any
$N\ge 1$ there exists $C_1(N)$ so that, for any
$C_1\ge C_1(N)$  there exists
$t_0(C_1) \ge N$ so that, for any $t\ge t_0$
there exists   $R(t)$,   so that
for any  $R \ge R(t)$, setting $Y=2(r-r'+s-s')$
\begin{align}\label{2.10}
& \norm{\LL_t (\psi)}{\HHH^{r,s,q}_p(R,C_0,C_1)}\le \\
\nonumber& \Cs  
\biggl (\sum_{n=[t/\tilde \tau_0]}
^{[t/\tilde \tau_1 ]}(\Cs N^p)^{n/N}
(D^e_{n})^{\frac{(p-1)}{p}} (D^b_{n})^{\frac{1}{p}} \max(\lambda_{u,n}^{-r}, \lambda_{s,n}^{-(r+s)}) \biggr ) 
\norm{\psi}{\HHH^{r,s,q+r-s}_p(R,C_0,C_1)}\\
\nonumber&\,\,\, + \Cs 
R^{Y} \Lambda^{3 t}
\biggl (\sum_{n=[t/\tilde \tau_0]}
^{[t/\tilde \tau_1 ]}(\Cs N^p)^{n/N}
(D^e_{n})^{\frac{(p-1)}{p}} (D^b_{n})^{\frac{1}{p}}  \biggr )   \norm{\psi}{\HHH^{r',s',q+2r-r'-s}_p(R,C_0,C_1)} \, .
\end{align}
The bound \eqref{notLY} immediately follows from the above estimate. 

All the ingredients of the  proof of Lemma~5.1 in \cite{BG2}
are at our disposal to prove \eqref{2.10}:
The analogue of the iteration lemma  for charts \cite[Lemma 3.3]{BG2}
is  Lemma~\ref{lemcompose} in Appendix~\ref{iteratechart}.
Our Lemma~\ref{CompositionDure} plays the part of \cite[Lemma 4.6]{BG2} on
composition with
hyperbolic matrices.
The analogues of
Lemmata~ 4.1 (Leibniz formula), 4.2 (multiplication with a characteristic function),  4.3
(localisation),  4.5 (partition of unity), and  4.7 (composition with a $C^1$ diffeomorphism
which is $C^{1+Lip}$ along stable leaves)
from \cite{BG2} are Lemmata~\ref{Leib}, \ref{lem:multiplier}, \ref{lem:localization},
~\ref{lem:sum}, and ~\ref{lemcomposeD1alpha} in
Appendix ~\ref{localspaces} below. 
Note however that Lemma~\ref{lemcompose} does not have such a nice
form as \cite[Lemma 3.3]{BG2}, and this will force us in Step 2 of
the proof below to invoke a result specific to our continuous-time
setting, Lemma~\ref{noglue}. This point, together with the 
comparison of the ``continuous-time" and ``discrete-time" complexities
(via the set $\NN(t)$)
are the main differences between the present proof
and that of \cite[Lemma 5.1]{BG2}.

Before giving a detailed account of the proof of \eqref{2.10}, 
let us describe the order in which
we choose the constants. First, $N$ is fixed in the statement
(it will be used in the second step of the proof in order to
apply Lemma ~\ref{lem:multiplier}).
Then, we choose $C_1$ very large, depending on $N$,
also in the second step below, so
that the admissible charts $\phi_\zeta$ are close enough to
linear maps. Then, in Step 3 we fix $t$
very large  depending on $C_1$ (large enough so that
every branch of  $T_t$ is hyperbolic enough so that
Lemma~\ref{lemcompose} applies). Finally, we choose $R$ very large so
that, at scale $1/R$, all the iterates $T_s$ up to time $t$
look like linear maps, and all the boundaries of the sets we
are interested in look like hyperplanes. For the presentation
of the argument, we will start the proof with some values of
$C_1$, $t$, $R$, and increase them whenever necessary, checking each
time that $C_1$ does not depend on $t$, $R$, and that $t$ does not
depend on $R$, to avoid bootstrapping issues. We will denote by
$\Cs$ a constant which could
depend on $r$, $s$, $p$, $C_0$ but  does not depend on $N$, $C_1$, $t$, $R$, and may
vary from line to line.
We shall use  that 
for any $s'<s$ 
\begin{equation}\label{implicit}
\|\cdot\|_{H^{r'+s'-s,s,q}}\le \Cs \|\cdot\|_{H^{r',s',q}}\, .
\end{equation}

For every $\ii \in I^{n+1}$, we fix a small neighbourhood $\tilde
O_\ii$ of $\overline{O_\ii}$ (as a hypersurface) such that $P^n_\ii$ admits an
extension to $\tilde O_\ii$ with the same hyperbolicity
properties as the original $P^n_\ii$. Reducing these sets if
necessary, we can ensure that their intersection multiplicity
is bounded by $D_n^b$, and that the intersection multiplicity
of the sets $P^n_\ii (\tilde{O_\ii})$ is bounded by $D_n^e$.

Our assumptions (in particular the fact that the
return times are bounded from above and from
below) imply that there exist $\tilde \tau_0 >0$ and $\tilde \tau_1>0$, with the following properties:
For
each  flow box $B_{i j}$, 
for every $w\in B_{i j}$,  we let $z(w) \in {O}_{i,j}$
and $t(w)\in (0, \tau_{i,j} (z(w)))$   be 
as in \eqref{deftwzw}.
Then for every large enough $t >0$,
there exist uniquely defined  
\begin{align*}
&n=n(t,w)\ge 1\, ,\ 
\ii=\ii(t, w)=\ii(n, w) \in I^{n+1}\, , \, i_0=i\, ,\,  i_1=j
\, , \\
& 0 \le t_{n+1}(w) \le \tau_{i_{n-1} i_{n}} (P^{n}_{i_1\ldots i_{n}}(z(w))) \, , 
\end{align*}
so that, setting $t_0(w)=\tau_{i j} (z(w))-t(w)$
\begin{align}\label{deccomp}
&t=t_0(w)+t_{n+1}(w)+\sum_{\ell=1}^{n-1} \tau_{i_\ell i_{\ell+1}}(P^{\ell}_{i_1 \ldots i_\ell}(z(w))  )  \, ,
\end{align}
with $n\in [t/\tilde \tau_0 -\Lambda^{-1},t/\tilde \tau_1 -\Lambda^{-1}]$
for some constant $\Lambda$ depending only on the dynamics.

Fix $t>0$ large.
Let $\NN(t)$ be the finite set of 
possible values of $n(t, w)$,  when
$i$ and $j$ range in $I$, as $w$ ranges over $B_{i j}$. We define
for  $n \in \NN(t)$   the set 
\begin{align*}
\II(n,t)=  \{ \ii \in I^{n+1}\, \mid \ii=(i_0, \ldots, i_n)
\mbox{ appears in a decomposition \eqref{deccomp} for }
t \} \, ,
\end{align*}
and we put $\II(t)=\cup_{m \in \NN(t)} \II(m,t)$.
Introduce for $n \in \NN(t)$
and $\ii =(i_0, \ldots, i_{n+1}) \in \II(n,t)$ the refined flow boxes
$$
B_{\ii, t}=\{ w \in B_{i_0 i_1} \mid 
n=n(t,w) \, , \ii(t,w)= \ii\,\}\, .
$$
Note that if $w\in B_{\ii,t}$ then $z(w)\in O_{\ii}$, while
$t(w)$ lies between the graphs of two piecewise $C^2$ functions 
of $z(w)$, which  coincide either with
$0$ or the ceiling time, or
are the images by the flow  of the transversals
$O_{i_0 i_1}$ or $O_{i_1 i_2}$.
We let $T_{\ii, t}$ be the restriction of $T_t$ to $B_{\ii,t}$. By definition, $T_{\ii, t}$
admits a $C^2$ extension to  a neighbourhood 
$\tilde B_{\ii,t}$  of $B_{\ii,t}$.

For $\zeta=(i,j,\ell, m)\in \ZZ(R)$, let us write
$$A(\zeta)=A(\zeta,R)=(\kappa_\zeta^R)^{-1}(B(m,d)) \subset M\, .
$$
The set $A(\zeta)$ is a neighbourhood of $w_\zeta$, of diameter
bounded by $\Cs R^{-1}$, and containing the support of 
$\brho_{\zeta}$.

Let us fix some system of charts $\Phi$ as in the Definition ~
\ref{defnorm} of the ${\HHH_p^{r,s,q}(R)}$-norm. We
want to estimate $\norm{\LL_t \psi}{\Phi}$.

\emph{First step: Complexity at the end.} For any $n \in \NN(t)$
the  closures of the sets $\{T_{\ii,t}(B_{\ii,t})\st \ii \in \II(n,t) \}$,
or (up to taking a smaller neighbourhood) 
$\{T_{\ii,t}(\tilde B_{\ii,t})\st \ii \in \II(n,t) \}$
have intersection multiplicity at most $D_n^e$
(the partition cannot be refined along the time direction).
Therefore, writing\footnote{Elements of $L^\infty$ are defined almost everywhere,
and the transfer operator is defined initially on $L^\infty(X_0)$,
so the fact that 
$\bigcup_{ij} B_{ij}=X_0$ only modulo a zero
Lebesgue measure set is irrelevant.} 
$$
\LL_t (\psi)=\sum_{n \in \NN(t)} \sum_{\ii \in \II(n,t)}
\Id_{T_{\ii,t} B_{\ii,t}}( \psi)\circ T_{\ii,-t}\, ,
$$ 
we get by
Lemma \ref{lem:sum} that for each $\zeta \in \ZZ(R)$
\begin{align*}
\nonumber & \norm{(\brho_\zeta \LL_t (\psi))\circ
\bphi_\zeta}{H_p^{r,s,q}}^p
\\ 
\nonumber &\qquad\le \Cs\sum_{n \in \NN(t)} 
 (D_{n}^e)^{p-1}\sum_{\ii \in \II(n,t)}
\norm{(\brho_\zeta  \Id_{T_{\ii,t} B_{\ii,t}} \psi \circ T_{\ii,-t})\circ \bphi_\zeta}{H_p^{r,s,q}}^p \\
&\qquad\qquad+\Cs   R^{Y/2}
\sum_{\stackrel{n \in \NN(t)} {\ii \in \II(n,t)}} \lambda_{s,n}^{-Y/2}(D_{n}^e)^{p-1}
\norm{(\brho_\zeta   \Id_{T_{\ii,t} B_{\ii,t}} \psi \circ T_{\ii,-t})\circ \bphi_\zeta}{H_p^{r',s',q}}^p \, .
\end{align*}
(The factor $\Cs (\lambda_{s,n}^{-1} R)^{Y/2}=\Cs (\lambda_{s,n}^{-1} R)^{r-r'+s-s'}$ comes from the $C_{r-r'+s-s'}$ norm in \eqref{supcst} and \eqref{implicit}.)
Summing over $\zeta \in \ZZ(R)$, we obtain
\begin{align*}
&  \norm{\LL_t (\psi)}{\Phi}^p \le\\
&  \Cs 
 \sum_{n \in \NN(t)}  (D_{n}^e)^{p-1}
\sum_{\zeta \in \ZZ(R), \ii\in \II(n,t)}
\norm{(\brho_\zeta  \Id_{T_{\ii,t} B_{\ii,t}} \psi\circ
T_{\ii,-t})
\circ \bphi_\zeta}
{H_p^{r,s,q}}^p \\
&+ \Cs  R^{Y/2} 
\sum_{\stackrel{n \in \NN(t)} {\zeta \in \ZZ(R),\ii \in \II(n,t)}}  \lambda_{s,n}^{-(r-r'+s-s')}(D_{n}^e)^{p-1}
\norm{(\brho_\zeta  \Id_{T_{\ii,t} B_{\ii,t}} \psi\circ
T_{\ii,-t})
\circ \bphi_\zeta}
{H_p^{r',s',q}}^p
\, .
\end{align*}

For $i\in I$, $j\in J_i$, let  $U_{i,j,\ell, 2}$, $j\in N_{i,\ell}$, be arbitrary
open sets covering a fixed neighbourhood $\tilde B_{i,j}^0$ of
$\overline{B_{i,j}}$, such that $\overline{U_{i,j,\ell, 2}}\subset
U_{i,j,\ell, 1}$ (they do not depend on $t$ and $R$). For each $\zeta\in \ZZ(R)$, $n\in \NN(t)$, and
$\ii=(i_0,\dots,i_{n})\in \II(n,t)$ such that $T_{\ii, t} (B_{\ii,t})$
intersects $A(\zeta)$, the point $T_{\ii,-t}(w_\zeta)$ belongs to
$\tilde B_{i_0,i_1}^0$ if $R$ is large enough. We can therefore
consider $k$ such that $T_{\ii,-t}(w_\zeta)$ belongs to $U_{i_0,i_1, k,2}$. Then
$\sum_{\ell\in \ZZ_{i_0,k}(R)}\brho_{i_0,i_1, k,\ell}$ is equal to
$1$ on a neighbourhood of fixed size of $T_{\ii,-t}(w_\zeta)$, so
that $\sum_{\ell\in \ZZ_{i_0,i_1, k}(R)} \brho_{i_0,i_1, k,\ell} \circ
T_{\ii,-t}$ is equal to $1$ on $A(\zeta)$ if $R$ is large enough
(depending on $t$ but not on $\Phi$ or $\zeta$). 
Therefore, claim \eqref{converse} in Lemma~\ref{lem:localization}
(note that \eqref{locall} is \eqref{locall'})
gives, if $R$ is large enough (uniformly in
$\Phi$, $\zeta$, $k$, $\ii$)
\begin{align}
\nonumber &  \norm{
	(\brho_{\zeta} \Id_{T_{\ii,t} B_{\ii,t}} \psi
	\circ T_{\ii,-t}) \circ \bphi_\zeta
}{H_p^{r,s,q}}^p
\\
\label{Rbig} &\qquad \le
\Cs \sum_{\ell\in \ZZ_{i_0,i_1, k}(R)}
\norm{
	(\brho_\zeta  \Id_{T_{\ii,t} B_{\ii,t}}(\brho_{i_0,i_1,k,\ell} \cdot  \psi)
	\circ T_{\ii,-t}) \circ \bphi_\zeta
}{H_p^{r,s,q}}^p \, .
\end{align}

Taking $R$ large enough and summing over $\zeta\in \ZZ(R)$,
$n\in \NN(t)$,
$\ii\in \II(n,t)$ and $k$ in $N_{i_0 i_1}$ such that
$T_{\ii,-t}(w_\zeta) \in U_{i_0,i_1 k,2}$,  we get (writing
$\zeta'=(i_0,i_1, k,\ell)\in \ZZ(R)$)
\begin{align}
\label{klqjsdfml}
& \norm{\LL_t (\psi)}{\Phi}^p
\le \Cs \sum_{n,\zeta, \ii, \zeta'}
 (D_{n}^e)^{p-1}
\norm{(\brho_\zeta  \Id_{T_{\ii,t} B_{\ii,t}}( \brho_{\zeta'} \cdot  \psi)\circ T_{\ii,-t})
\circ \bphi_\zeta}
{H_p^{r,s,q}}^p \\
\nonumber
&\,+ \Cs \cdot R^{Y/2}
\sum_{n,\zeta, \ii, \zeta'} (D_{n}^e)^{p-1}\lambda_{s,n}^{-Y/2}
\norm{(\brho_\zeta  \Id_{T_{\ii,t} B_{\ii,t}}( \brho_{\zeta'} \cdot  \psi)\circ T_{\ii,-t})
\circ \bphi_\zeta}
{H_p^{r',s',q}}^p  \, ,
\end{align}
where the sum is restricted to those $(\zeta,\ii,\zeta')$ such
that the support of $\brho_{\zeta'}$ is included in $\tilde
B_{\ii,t}$, the support of $\brho_\zeta$ is included in 
$T_{\ii,t} (\tilde B_{\ii,t})$, and $B_{\zeta'}=B_{i_0 i_1}$ (this restriction will be
implicit in the rest of the proof).

\emph{Second step: Getting rid of the characteristic function.}
We claim that, if $R$ is large enough, then for any $\zeta$, $n$,
$\ii$, $\zeta'$ as in the right-hand-side of (\ref{klqjsdfml}), we have
\begin{multline}
\label{multipl}
\norm{(\brho_\zeta  \Id_{T_{\ii,t }B_{\ii,t}}
( \brho_{\zeta'} \cdot  \psi)\circ T_{\ii,-t}) \circ \bphi_\zeta}
{H_p^{r,s,q}}^p
\\ \le
\Cs (\Cs N^p)^{n/N} \norm{
(\brho_\zeta (  \brho_{\zeta'} \cdot  \psi)\circ T_{\ii,-t})
\circ \bphi_\zeta} {H_p^{r,s,q}}^p \, .
\end{multline}
To prove the above
inequality, it is sufficient to show that  multiplication
by  $\Id_{T_{\ii,t} (B_{\ii,t})}\circ
\bphi_\zeta$ acts boundedly on $H_p^{r,s,q}$, with norm bounded
by  $(\Cs N^p)^{n/N}$. 
First note that there exist $C^2$ functions 
$\tilde \tau_{\ii,k}(z,t)$, $k=0, 1$, so that the images of the
dynamically refined flow boxes
are of the form
$$
T_{\ii,t}( B_{\ii,t})=\{T_\tau(z)\mid z\in P_\ii^n(O_\ii)\, ,
\, \tau\in [\tilde \tau_{\ii,0}(z,t),\tilde \tau_{\ii,1}(z,t))\}\, .
$$
For $n\in \NN(t)$, let $\kappa = n/N$ and decompose
$\ii=(i_0,\dots,i_{n})$ into subsequences of length $N$, as
$(\ii_0,\dots,\ii_{\kappa-1})$. Then $\Id_{P_\ii^n O_\ii} \prod_{j=0}^{\kappa-1} \Id_{O_{\ii_j}} \circ
P_{\ii_j\ii_{j+1}\dots \ii_{\kappa-1}}^{-(\kappa-j)N}$. Define
a set 
$$\Omega_j=P_{\ii_j\ii_{j+1}\dots
\ii_{\kappa-1}}^{(\kappa-j)N}(O_{\ii_j})\, ,
$$ it is therefore
sufficient to show that multiplication by $\Id_{\{w\in B_{i_{n-jN}, i_{n-jN+1}}\mid
z(w)\in \Omega_j\}} \circ
\bphi_\zeta$ (this takes care of the lateral
boundaries) acts boundedly on $H_p^{r,s,q}$, with norm at most $\Cs
N^p$ and multiplication by  (this takes care of the top and bottom
boundaries)
\begin{equation}\label{lastass}
\Id_{\{w \mid z(w)\in O_{i_{n-1}, i_n}\, \, , \, \,
\tau(w)\in [\tilde \tau_{\ii,0}(w,t), \tilde \tau_{\ii,1}(w,t))\}}\circ
\bphi_\zeta
\end{equation}
acts boundedly on $H_p^{r,s,q}$, with norm at most $\Cs$.

Recall that \eqref{starstar'} holds.
Working with our flow box charts 
(see  \eqref{flowbox}), the assertion
on \eqref{lastass}
is an immediate application of Lemma~
\ref{lem:multiplier} (the number $M_{cc}$ of connected components being
then uniform
in $N$), since the functions $\tilde \tau_{\ii,k}$ are $C^2$, uniformly
in $t$, $k$, $\ii$ (they are obtained by composing the original roof
functions by a $C^2$ and hyperbolic
restriction of the composition of the Poincar\'e maps). 
Next, we assume for a moment that
each $\partial O_{i,j}$ is a finite union
of $C^1$ hypersurfaces $K_{i,j, k}$, each of which
is transversal  to the stable cone.  
Then the second step of the
proof of \cite[Lemma 5.1]{BG2} allows us to apply Lemma~
\ref{lem:multiplier} 
(using Fubini in  flow box coordinates) which
implies that, if $R$ and $C_1$ are large enough
(depending only on $N$), then 
the multiplication by the
characteristic function of each set
$\{w\in B_{i_{n-jN}, i_{n-jN+1}} \mid
z(w)\in \Omega_j\}$   has operator-norm on $H_p^{r,s,q}$
bounded by $\Cs NL$. 
Definition~\ref{transcond}
required  only transversality in the image,
but, using the fact that the cone fields $C_{i,j}^{(s)}$ 
do not depend on $i$ (see Definition~\ref{transcond}), we can apply
the arguments in \cite[Appendix C]{BG2}, in particular \cite[Theorem~C.1]{BG2}
there in the case where the set called $\Pi_1$ there coincides
with the entire manifold $M$.
This proves
~\eqref{multipl}.

Combining~\eqref{klqjsdfml} with~\eqref{multipl} and with
the analogous bound in 
$H^{r',s',q}_p$ for $-1+r<-1+1/p<r'<r$, we get
\begin{align}
\label{2step}
& \norm{\LL_t \psi}{\Phi}^p
\le \Cs  \sum_{\zeta, n, \ii, \zeta'}
(\Cs N^p)^{n/N}  (D_{n}^e)^{p-1}
\norm{(\brho_\zeta( \brho_{\zeta'} \cdot  \psi)\circ T_{\ii,-t})
\circ \bphi_\zeta}
{H_p^{r,s,q}}^p \\
\nonumber& +
\Cs   R^{Y/2}  \sum_{\zeta, n, \ii, \zeta'}\lambda_{s,n}^{-Y/2}(D_{n}^e)^{p-1}
(\Cs N^p)^{n/N}  
\norm{(\brho_\zeta( \brho_{\zeta'} \cdot  \psi)\circ T_{\ii,-t})
\circ \bphi_\zeta}
{H_p^{r',s',q}}^p 
\, .
\end{align}

\emph{Third step: Using the composition lemma.}  
In this step, we shall use
Lemma~\ref{lemcompose}, to pull the charts $\Phi_\zeta$ 
in the right hand
side of \eqref{2step} back at
time $-t$   
(glueing some  pulled-back charts together to
get rid of the summation over $\zeta$), and exploit
Lemma~\ref{CompositionDure} to obtain decay  from
hyperbolicity.

Let us partition $\ZZ(R)$ into finitely many subsets
$\ZZ^1,\dots,\ZZ^E$, such that $\ZZ^e$ is included in one of the
sets $\ZZ_{i,j,\ell}(R)$, and $|m-m'|\ge C(C_0)$ whenever
$(i,j,\ell,m)\not=(i,j,\ell, m')\in \ZZ^e$, where $C(C_0)$ is the constant
$C$ constructed in Lemma~\ref{lemcompose} (it only depends on
$C_0$). The number $E$ may be chosen independently of $t$ and
$n\in \NN(t)$.

We shall prove the following: For any ${\zeta'} \in \ZZ(R)$, any
$t>0$, any
$n\in \NN(t)$ and
$\ii\in \II(n,t)$ (such that the support of $\brho_{{\zeta'}}$ is
included in $\tilde B_{\ii,t}$ and $B_{\zeta'}=B_{i_0 i_1}$), and any
$1\le e\le E$, there exists an admissible chart
$\bphi'=\bphi'_{{\zeta'},\ii,e} \in \FF({\zeta'})$ such that
\begin{align}
\label{eqmainstep}
& \sum_{\zeta\in \ZZ^e} \norm{(\brho_\zeta( \brho_{\zeta'} \cdot  \psi)\circ T_{\ii,-t})
\circ \bphi_\zeta}
{H_p^{r,s,q}}^p
\le \Cs \chi_n \norm{( \brho_{\zeta'} \cdot\psi) \circ \bphi'_{{\zeta'},\ii,e}}{H_p^{r,s,q+r-s}}^p\\
\nonumber &\qquad\qquad+
\Cs   \lambda_{u,n}\cdot \lambda_{s,n}^{-1}
\norm{( \brho_{\zeta'} \cdot\psi) \circ \bphi'_{{\zeta'},\ii,e}}{H_p^{r', s', q+2r-r'-s}}^p
\, ,
\end{align}
where
\begin{equation}
\chi_n=\norm{
\max(\lambda_{u,n}^{-r}, \lambda_{s,n}^{-(s+r)})^p}{L^\infty}.
\end{equation}
As always, the sum on the left hand side of \eqref{eqmainstep}
is restricted to those values of $\zeta$ such that the support of
$\brho_\zeta$ is included in $T_{\ii,t} (\tilde B_{\ii,t})$

Let us fix ${\zeta'}$, $t$, $n\in \NN(t)$, $\ii\in \II(n,t)$ and $e$ as above, until the end of
the proof of \eqref{eqmainstep}. All the objects we shall now
introduce shall depend on these choices, although we shall not
make this dependence explicit to simplify the notations. Let
$i,j,\ell$ be such that $\ZZ^e\subset \ZZ_{i,j,\ell}(R)$, and let
$\JJ=\{m\st (i,j,\ell,m)\in \ZZ^e$\}. Since the points in $\JJ$ are
distant of at least $C(C_0)$, Lemma~\ref{lemcompose} will
apply.

Increasing
$R$, we can ensure that the map
\begin{equation*}
\TT:= \kappa_{\zeta}^R\circ T_{\ii,t} \circ
(\kappa_{\zeta'}^R)^{-1}
\end{equation*}
is arbitrarily close to its differential $\AAc=D\TT(\ell)$ at
$\ell:= \kappa_{\zeta'}(w_{\zeta'})$, i.e., the map
$(\TT^{-1} [\cdot + \TT(\ell)]-\ell) \circ \AAc$ is close to the
identity in the  $C^{2}$ topology, say on the ball $B(0,2d)$,
and that $n(\cdot,t)$ is constant on $ (\kappa_{\zeta'}^R)^{-1}B(0,2d)$ .
Moreover, the matrix $\AAc$ sends $\CC_{\zeta'}$ to $\CC_{\zeta}$
compactly (recall \eqref{concond}, and note
that $t_{00}$ can be fixed small while we may require $t\ge t_0$
with $t_0$ depending on $t_{00}$), and
\begin{equation}
\Cs \ge \lambda_u(\AAc,\CC_{\zeta'},\CC_{\zeta})/ \lambda^{(n)}_u( w_{\zeta'}) \ge \Cs^{-1} \, ,
\end{equation}
with similar inequalities for $\lambda_s$ and $\Lambda_u$.
In addition $\AAc(0,0,v^0)=(0,0,v^0)$ for any $v^0 \in \real$.
Since $T$ is uniformly hyperbolic and satisfies the bunching
condition
\eqref{bunch2}, we can ensure up
to taking larger $t$ (and thus $n$, and also $R$, in view of the requirements
in the beginning of the paragraph)  that $\AAc$ satisfies   the bunching
condition \eqref{suffhyp}
for the constant $\epsilon=\epsilon(C_0,C_1)$ constructed in
Lemma~\ref{lemcompose}. 

Let us make explicit the dependence of $t$ on $C_1>1$
and of $R$ on $t$ (and therefore on $C_1$).
Let $\lambda_\beta <1$ be the supremum in the
left-hand-side of the bunching condition \eqref{bunch2}
($\lambda_\beta$ only depends on the dynamics).
From \eqref{smalleps} in the proof of Lemma~\ref{lemcompose}, 
there is a constant $\Cs$ 
depending only on (the extended
cones and) $C_0$ so that we may take
$\epsilon(C_0,C_1)= \Cs C_1^{-1}$.
Therefore, there is a constant $\Cs$ depending 
only on $C_0$,  so that if 
$
n \ge   \ln (\epsilon^{-1})/\ln (\lambda_\beta^{-1})=\Cs  \ln (C_1)
$
then the bunching condition \eqref{suffhyp} holds
for $\epsilon(C_0,C_1)$.
Since $n \ge t/\tilde \tau_0-\Cs$, the condition on $n$ transforms to
$t \ge t_0(C_1)= \Cs \ln (C_1)$,
for a different constant $\Cs$ depending only on the dynamics.
Finally,  there are   $\Lambda >1$ depending only on the
dynamics,  and $\Cs$ depending only
on  the dynamics and on $C_0$, so that, for $t\ge t_0(C_1)$, if
$R\ge R(t)=\Cs \Lambda^t$ with 
$\Lambda^t \ge \Cs  C_1^{\Cs \ln \Lambda}$
then the requirements in the beginning of the previous paragraph
(and those in the first step regarding
\eqref{Rbig}) hold
for $t$.

Applying
\footnote{In order to apply Lemma~\ref{lemcompose},  we need to
extend $\TT$ to a diffeomorphism of $\real^d$, which can be done exactly
as in the third step of the proof of \cite[Lemma 5.1]{BG2}.}
Lemma \ref{lemcompose}, we obtain a block diagonal matrix $D=D_{\zeta'}$, a
chart $\phi'_\ell=\phi'_{\ell, \zeta'}$ around $\ell$, time-shifts
$\Delta_m=\Delta_{m,\zeta'}$
and diffeomorphisms $\Psi_m
=\Psi_{m,\zeta'}$, and
$\Psi_{\zeta'}$, such that, for any $m$ in the set $\JJ'$ of those
elements in $\JJ$ for which $\brho_\zeta\cdot\brho_{\zeta'}\circ
T_{\ii,-t}$ is nonzero,
\begin{equation}
\label{newchart}
\TT^{-1}\circ \phi_\zeta=\phi'_{\ell,\zeta'} \circ \Psi_{\zeta'} \circ D_{\zeta'}^{-1} \circ \Psi_{m,\zeta'} \circ \Delta_{m, \zeta'}
\end{equation}
on the set where $(\brho_\zeta\cdot\brho_{\zeta'}\circ T_{\ii,t})
\circ \Phi_\zeta$ is nonzero.

Writing $\psi_{\zeta'}=( \brho_{\zeta'} \cdot 
\psi)\circ (\kappa_{{\zeta'}}^R)^{-1}$, we have (recall that
$(i,j,\ell)$ is fixed so that $\ZZ^e\subset \ZZ_{i,j}(R)$)
\begin{align}\label{diff0}
\sum_{\zeta\in \ZZ^e}& \bigl\|(\brho_\zeta( \brho_{\zeta'} \cdot  \psi)\circ T_{\ii,-t})
\circ \bphi_\zeta\bigr\|_{H_p^{r,s,q}}^p\\
\nonumber &=
  \sum_{m\in \JJ'} \norm{(\rho_m\circ \phi_{i,j,\ell,m}) \cdot (\psi_{\zeta'}\circ \TT^{-1} \circ \phi_{i,j,\ell,m})} {H_p^{r,s,q}}^p
\\\nonumber &=
  \sum_{m\in \JJ'} \norm{ ((\rho_m\circ \phi_{i,j,\ell,m}\circ \Delta_m^{-1}) \cdot
(\psi_{\zeta'}\circ \phi'_\ell \circ \Psi \circ D^{-1} \circ \Psi_m))
\circ
\Delta_m}{H_p^{r,s,q}}^p\, .
\end{align}
Lemma~\ref{noglue} bounds the previous expression by
\begin{equation}\label{nog}
\Cs \sum_{m\in \JJ'} \norm{( (\rho_m\circ \phi_{i,j,\ell,m}\circ \Delta_m^{-1}
\circ \Psi_m^{-1}) \cdot
( \psi_{\zeta'}\circ \phi'_\ell \circ \Psi \circ D^{-1}))\circ \Psi_m)}{H_p^{r,s,q+r-s}}^p\, .
\end{equation}
Using the notations and results of Lemma ~\ref{lemcompose}, the
terms in this last equation are of the form $v\circ \Psi_m$,
where $v$ is a distribution supported in
$\Psi_m(\phi_{i,j,\ell,m}^{-1}(B(m,d)))\subset B(\Pi m,
C_0^{1/2}/2)$. 
Since the range of $\Psi_m$ contains $B(\Pi m^u,
C_0^{1/2})$, 
if $q\ge 0$ is small enough so that
$q+(r-s)<(1-r+s)/(r-s)$,
Lemma ~ \ref{lemcomposeD1alpha} gives 
$\norm{v\circ \Psi_m}{H_p^{r,s,q+r-s}}\le \Cs \norm{v}{H_p^{r,s,q+r-s}}$. Therefore
\eqref{nog} is bounded by
\begin{equation}\label{diff1}
\Cs \sum_{m\in \JJ'} \norm{ (\rho_m\circ \phi_{i,j,\ell,m}\circ \Delta_m^{-1}\circ \Psi_m^{-1} )  \cdot
( \psi_{\zeta'}\circ \phi'_\ell \circ \Psi \circ D^{-1})}{H_p^{r,s,q+r-s}}^p\, .
\end{equation}
The functions $\rho_m \circ \phi_{i,j,\ell,m}\circ \Delta_m^{-1}\circ \Psi_m^{-1}$ have
a bounded $C^1$ norm and are supported in the balls $B(\Pi
m,C_0^{1/2}/2)$, whose centers are distant by at least $C_0$,
by Lemma~\ref{lemcompose}~(a). Therefore, by the localisation
Lemma~\ref{lem:localization} (using that \eqref{locall'}
is \eqref{locall}), the sum in \eqref{diff1}
(and thus also the left-hand-side of \eqref{diff0}), is bounded by
\begin{equation}\label{toshrink}
\Cs \norm{\psi_{\zeta'}\circ \phi'_\ell \circ \Psi \circ D^{-1}}{H_p^{r,s,q+r-s}}^p\, .
\end{equation}
Similarly, but giving up the glueing in Step~2 of
Lemma~\ref{lemcompose} (and therefore the need to work
with $\Delta_m$ and the regularity loss in the flow direction),  up to the
cost of exponential growth, we get,
applying the  second estimate
in Lemma~\ref{CompositionDure} (recall that $r'\ge 0$) to the composition
with $D^{-1}$,
\begin{align}
\nonumber \sum_{\zeta\in \ZZ^e} \bigl\|(\brho_\zeta( \brho_{\zeta'} \cdot  \psi)\circ T_{\ii,-t})
\circ \bphi_\zeta\bigr\|_{H_p^{r',s',q}}^p
&\le
\Cs \lambda_{u,n} \norm{\psi_{\zeta'}\circ \phi'_\ell \circ \Psi \circ D^{-1}}{H_p^{r',s',q}}^p\\
\label{cpcterm}&\le  
\Cs \lambda_{u,n}  
\norm{\psi_{\zeta'}\circ \phi'_\ell }{H_p^{r',s',q}}^p
\, .
\end{align}
We next  apply the first estimate
in Lemma~\ref{CompositionDure} to the composition
with $D^{-1}$ in \eqref{toshrink}, in order to obtain a contraction in the $H_p^{r,s,q+r-s}$
norm, up to a bounded term in the $H_p^{r',s',q+2r-r'-s}$ norm
for $r'< r$. Since $\psi_{\zeta'}$ is supported in
$B(\ell, C_0^{1/2}/2)$ while the range of $\Psi$ contains
$B(\ell, C_0^{1/2})$ (by Lemma~\ref{lemcompose}),
Lemma~\ref{lemcomposeD1alpha} implies that the composition with
$\Psi$  is bounded. Summing up, we obtain
\begin{align}
\nonumber
\sum_{\zeta\in \ZZ^e} &
\norm{(\brho_\zeta( \brho_{\zeta'} \cdot  \psi)\circ T_{\ii,-t})
\circ \bphi_\zeta}{H_p^{r,s,q+r-s}}^p \\
\label{qklsjflmqsfd}
& \qquad \qquad
\le \Cs \bigl (\chi^{(0)}_n(w_{\zeta'})
\norm{( \brho_{\zeta'} \cdot  \psi)
	\circ (\kappa_{\zeta'}^R)^{-1}\circ \phi'_\ell }{H_p^{r,s,q+r-s}}^p
	\\  \nonumber&\qquad\qquad\qquad\qquad\qquad\qquad
	+
\norm{	( \brho_{\zeta'} \cdot  \psi)	\circ (\kappa_{\zeta'}^R)^{-1}\circ \phi'_\ell
}{H_p^{r',s',q+2r-r'-s}}^p) \, ,
\end{align}
where
\begin{equation*}
\chi^{(0)}_n(w_{\zeta'})=(
\max(\lambda_{u,n}^{-r}, \lambda_{s,n}^{-(s+r)})^p)(w_{\zeta'})
\le \chi_n \, ,
\end{equation*}
concluding the proof of \eqref{eqmainstep}. Summing over
all possible values of $\zeta'$, $n$, $\ii$, and $e$, we obtain
\begin{align}
& \nonumber \norm{\LL_t \psi}{\Phi}^p
\le \Cs  \cdot
\chi_{[t/\tilde \tau_0 ]}\sum_{n,\zeta', \ii} 
(\Cs N^p)^{n/N} (D_{n}^e)^{p-1}\sum_{e=1}^E
\norm{( \brho_{\zeta'} \psi)\circ \bphi'_{\zeta',\ii,e}}{H_p^{r,s,q+r-s}}^p \\
\label{concl} &\quad +
\Cs  \tilde \Lambda^{2[t/\tilde \tau_0 ]}\cdot R^{Y/2}
\sum_{n,\zeta', \ii} 
(\Cs N^p)^{n/N}(D_{n}^e)^{p-1}\sum_{e=1}^E \norm{( \brho_{\zeta'} \psi)\circ \bphi'_{\zeta',\ii,e}}{H_p^{r',s',q+2r-r'-s}}^p\, .
\end{align}
for some $\tilde \Lambda >1$ depending only on the dynamics.

\emph{Fourth step: Complexity at the end/ conclusion of the proof of
\eqref{notLY}.} The right hand side of
(\ref{concl}) is  of the form
$\norm{\psi}{\Phi'}^p$ for some family of admissible charts
$\Phi'$, except that 
several admissible charts may be assigned to a point $w_{\zeta'}$ for
$\zeta'\in \ZZ(R)$. Since $E$ is independent of $n$, the number of those charts
around $w_{\zeta'}$ is at most $\Cs \cdot \Card\{\ii \st \tilde
B_{\ii,t} \cap A({\zeta'})\not=\emptyset\}$. If $R$ is large enough,
we can ensure that this quantity is bounded by the intersection
multiplicity of the sets $\tilde B_{\ii,t}$, which is at most
$D_{[t/\tilde \tau_1 ]}^b$ by construction (using again
that there is no refinement in the flow direction). Therefore,  Lemma~\ref{lem:sum} gives
\begin{align*}
&\norm{\LL_t \psi}{\Phi}^p
\le \Cs 
\max_{n \in \NN(t)}\bigl \{
(\Cs N^p)^{n/N} (D_{n}^e)^{p-1} D_{n}^b \chi_n\bigr \}
 \norm{\psi}{\HHH^{r,s,q+r-s}_p(R)}^p\\
 &\, +
 \Cs   \tilde \Lambda^{2[t/\tilde \tau_0 ]}\cdot R^{Y}
\max_{n \in \NN(t)}\bigl \{
(\Cs N^p)^{n/N} (D_{n}^e)^{p-1}D_{n}^b \bigr \}
 \norm{\psi}{\HHH^{r',s',q+2r-r'-s}_p(R)}^p \, .
  \end{align*}
This concludes the proof of \eqref{2.10} and therefore of \eqref{notLY}.

\bigskip
\emph{Proof of \eqref{reallybounded}.}
To show \eqref{reallybounded}, we revisit the steps of the proof
of \eqref{2.10}, without attempting to get an exponential
contraction in the first term, or to obtain ``compactness" in the second term. 
In  the estimate \eqref{supcst} in the first step, 
we replace $H^{r-1, s, q}$ by
$H^{r, s, q}$.
We must explain how to avoid the loss from
the $H^{r,s,0}_p$ norm to the $H^{r,s,r-s}_p$-norm.
Since we may use the second bound of Lemma ~\ref{CompositionDure}
instead of the first,
this loss could occur only
through the introduction
of $\Delta_m$ in Step 1  of Lemma~\ref{lemcompose} (see Lemma~\ref{lemcomposeQ}),
invoked in  the third step of the present proof .
In Step ~1 of Lemma~\ref{lemcompose}, 
we  got rid both of the $x^s$ and $x^u$
dependence of $\tilde F^{(1)}_m$. The $x^s$ dependence was
a problem in  Step ~2  of Lemma~\ref{lemcompose} (glueing).
If we give up this glueing step  in
Lemma~\ref{lemcompose} then   $\tilde F^{(1)}_m$ 
can be allowed to depend on $x^s$,
to the cost of  exponential growth (in $\Lambda^{\tilde t}$,
for $\Lambda >1$ related to the dynamics). 
The $x^u$ dependence was
a problem in  Step 3  of Lemma~\ref{lemcompose}, where
$\tilde F^{(2)}(A x^u, Bx^s, x^0)$ could create 
exponential growth in the norm. Again, if we are willing
to deal with an exponential factor, we can allow 
$\tilde F^{(1)}_m$  to depend on $x^u$. In other words,
we can get rid of $\Delta_m$ in Lemma~\ref{lemcompose}.
This ends the proof of \eqref{reallybounded}, and of Lemma~ \ref{LY0}.
\end{proof}

Using our transversality assumption
and a simplification of the application of
Strichartz' result in Step 2 in the proof of Lemma~ \ref{LY0}, 
and recalling Lemma~\ref{embed}, we obtain the
following bound, which will be useful to exploit mollifying and
averaging operators to be introduced in the next section:

\begin{corollary}\label{StrStr}
For any $1<p<\infty$ and every $\max(-\beta,-1+1/p)<\sigma<1/p$, there is $\Cs>0$ so that
for any $\varphi$,
$$
\| \Id_{X_0}\varphi\|_{H^\sigma_p(M)}
\le \Cs \| \Id_{X_0}\varphi\|_{H^\sigma_p(X_0)}\le \Cs^2 \| \varphi\|_{H^\sigma_p(M)}\, .
$$
\end{corollary}

It remains to show the Lasota-Yorke estimates for $\RR(z)$:

\begin{proof}[Proof of Lemma~\ref{controlCeta}]
We start with a premilinary bound which is also 
useful elsewhere. Its proof will use interpolation, and we refer to
Section 3.1 of \cite{BG1}  for
reminders and references (such as \cite{BL} and
\cite{TrB}) about complex interpolation. 

Clearly, $\| \LL_t (\psi)\|_{H^{0}_p(X_0)}\le \|\psi\|_{H^{0}_p(X_0)}$
for all $1\le p \le \infty$.
Fix $1<p<\infty$ and $0<s_0<1/p'=1-1/p$.
A toy-model version of the proof
of Lemma~\ref{LY0} (working on isotropic spaces, and
using the original result of Strichartz \cite{Str}) easily
gives that
$$\| \LL_t (\psi)\|_{H^{s_0}_p(X_0)}\le 
\tilde \Lambda e^{\tilde \Lambda t} \|\psi\|_{H^{s_0}_{p'}(X_0)}
\, .$$ 
Here, $\tilde \Lambda$ may depend on $R$, and $s_0$, and $p'$
which are fixed. One can see that
$\tilde \Lambda$ is uniform in $p'$ as $p'\to 1$.  By duality
$\| \LL_t \psi\|_{H^{-s_0}_p(X_0)}\le 
\tilde \Lambda e^{\tilde \Lambda t} \|\psi\|_{H^{-s_0}_p(X_0)}$.
Therefore,
using complex interpolation for $s_0<s<0$ and Lemma~\ref{embed}
we get  $\Lambda$
(which may depend on $R$, $p$, and $s_0$, which are fixed) so that
for any $-s_0<s<0$
\begin{align}\label{prel}
\| \LL_t \psi\|_{\HHH^{s,0,0}_p}
&\le \Lambda   \| \LL_t \psi\|_{H^{s}_p(X_0)}
\le \Lambda  e^{\Lambda |s| t} \| \psi\|_{H^{s}_p(X_0)}
\le \Lambda  e^{\Lambda |s| t} \| \psi\|_{\HHH^{s,0,0}_p}\, .
\end{align}

Recall \eqref{Rn}. For large integer $N$, applying  \eqref{needit}  twice, and exploiting  Lemma~\ref{bq},
we  obtain  a constant $\Cs$ so that 
for all $a > A$, $n\ge 0$, and $\psi \in \widetilde \HHH$, writing $z=a+ib$ 
\begin{align*}
\nonumber&\|\RR(z)^{n+1} (\psi)\|_{\widetilde \HHH^{r,s,0}_p}\le
\Cs \int_{0}^{\infty}\frac{t^{n-1}}{(n-1)!}
\bigl ( e^{-at- t \ln(\lambda^{-1})/N} 
(\| \LL_{t-t/N} \RR(z)(\psi) \|_{\widetilde \HHH^{r,s,r-s}_p}\\
\nonumber &\qquad\qquad\qquad\qquad\qquad\quad
+ 
e^{-at+At/N} \| \LL_{t-t/N}\RR(z)(\psi)\|_{\widetilde \HHH^{s,0,2(r-s)}_p} \bigr )\, dt\\
&\le
\Cs \int_{0}^{\infty}\frac{t^{n-1}}{(n-1)!}
\biggl [ e^{-at- 2t\ln(\lambda^{-1})/N)} \|\LL_{t-2t/N} \RR(z)(\psi) \|_{\widetilde \HHH^{r,s,2(r-s)}_p}\\
&\qquad+ 
e^{-t \ln(\lambda^{-1})/N}
e^{-at+At/N} \| \LL_{t-2t/N}\RR(z)(\psi)\|_{\widetilde \HHH^{s,0,3(r-s)}_p}\\
\nonumber &\qquad
+ \Cs \bigl (1+\frac{|z|}{R}\bigr )^{2(r-s)}(\frac{1}{a-A}+1)
e^{-at+At/N} \| \LL_{t-t/N}(\psi)\|_{\widetilde \HHH^{s,0,0}_p} \biggr ]\, dt \, .
\end{align*}
Exploiting \eqref{prel},
and applying  
\eqref{needit}   $N-2$ more times, we get
\begin{align*}
\nonumber&\|\RR(z)^{n+1} (\psi)\|_{\widetilde \HHH^{r,s,0}_p}\le
\Cs \int_{0}^{\infty}\frac{t^{n-1}}{(n-1)!}
\biggl [ e^{-t(a- \ln(\lambda^{-1}))} \| \RR(z)(\psi) \|_{\widetilde \HHH^{r,s,N(r-s)}_p}\\
\nonumber &\qquad+ 
[  e^{-(N-2)t \ln(\lambda^{-1})/N}  e^{-(N-1)\Lambda t |s|/N}
+\cdots\\
\nonumber&\qquad\quad
+  e^{-t\ln(\lambda^{-1})/N}  \bigl (1+\frac{|z|}{R}\bigr )^{(3-N)(r-s)}\frac{e^{-2\Lambda t|s|/N}}{\Cs^{N-1}}
+ \bigl (1+\frac{|z|}{R}\bigr )^{(2-N)(r-s)}\frac{e^{-\frac{\Lambda t|s|}{N}}}{\Cs^{N-2}}] \\
\nonumber&\qquad\qquad\qquad\qquad\qquad \cdot \Cs^{N-1}
\bigl (1+\frac{|z|}{R}\bigr )^{N(r-s)} (\frac{1}{a-A}+1)e^{-t(a-\Lambda |s| -\frac{A}{N})} 
\| \psi\|_{\widetilde \HHH^{s,0,0}_p} \biggr ]\, dt\\
&\le   \Cs \bigl (\frac{1}{a-A}+1\bigr)\bigl (1+\frac{|z|}{R}\bigr )^{N(r-s)}
\biggl (\frac{1 }{(a+ \ln(1/\lambda))^{n}}
+\frac{\Cs^{N-1}}{ (a-\Lambda |s|-A/N)^{n}}\biggr )  \|\psi\|_{\widetilde \HHH^{r,s,0}_p}  .
\end{align*}
This is \eqref{controlCeta'}.
We have applied \eqref{needit} successively
to $q_j=j(r-s)$,
for $j=0, \ldots, N$. 
We need conditions \eqref{locall'}  and \eqref{starstar'}
to hold for each $q_j$. It suffices to 
check both inequalities
on $q_N$, and  they are satisfied if  $|s|\le 2r$ and
\eqref{replacewidehatC} hold.
A slight modification of the above argument gives
\eqref{controlCeta''} (the condition on $N$ does not depend on $s'$).
\end{proof}

\section{Mollifiers and stable-averaging operators}\label{molll}
It will be convenient in Section ~\ref{dodo} to mollify distributions, replacing them
by nearby distributions in $C^1$.  As is often the case, our mollification
operators $\MMM_\epsilon$ are obtained through
convolution.  
In this section, we also introduce a key tool in the Dolgopyat estimate
of Section~\ref{carlangelo},
the stable averaging operators $\AAA_\delta$.

\begin{remark}
[Minkowski-type integral inequalities]\label{minkoo}
We note for further use that
for any real $r$ and  any $1<p<\infty$, there exists
$C>0$ so that for any  integrable  $\tilde \eta : \real^{d}\to \real_+$ and
any  family $\omega_{y} \in H^{r}_p(\real^d)$, uniformly
bounded in $y$,
\begin{align}\label{minko}
\norm{ \int_{\real^{d}} \tilde \eta(y) \omega_{y}(\cdot )\, dy}{H^{r}_p(\real^d)}
&\le C \int_{\real^{d}} \tilde \eta(y)
\norm{ \omega_{y}}{H^{r}_p(\real^d)}\, dy \\
\nonumber &\le C \norm{\tilde \eta}{L^1(\real^{d})}
\sup_{y} \norm{ \omega_{y}}{H^{r}_p(\real^d)}\, .
\end{align}
Indeed, if $r=0$ the above estimate is the classical Minkowski integral inequality, see e.g. \cite[App.A]{Stein}.
The case  $|r|\le 1$ may be proved by considering first $r=1$,
recalling that
$\norm{\psi}{H_p^{1}}\sim \norm{\psi}{L^p} +
\norm{D\psi}{L^p}$ (see e.g. \cite[Prop 2.1.2 (iv)+(vii)]{RS}),
then $r=-1$, using duality 
($H^{r,0,0}_p=H^{-r,0,0}_{p'}$ with $p'=1/(1-1/p)$) and  $d$-dimensional Fubini 
(see e.g. \cite[Beginning of \S4]{BG1}), and finally using complex interpolation
(see \cite{Tr}, and \cite{BL} in particular
\S2.4, \S4.1, and Theorem 5.1.2, which says
that $[L_1(\BB_1,\nu),L_1(\BB_2,\nu)]_{\theta}=L_1([\BB_1,\BB_2]_\theta,\nu)$,
applied here to $\nu=\eta\, dx$, this is Theorem 1.18.4 in \cite{Tr}). 
The cases  $|r|>1$ are handled similarly by considering
higher order derivatives.
\end{remark}

\begin{remark}
[Variants of complex interpolation]\label{variants}
We shall use two easy variants
of complex interpolation. Let $\PP_1$, $\PP_2$, 
and $\PP_3$ be linear operators acting
on $\BB_1$, $\BB_2$, an interpolation pair of Banach spaces. Then,
interpolating at $\theta\in [0,1]$ between
\begin{equation}\label{version1}
\| \PP_1 \omega\|_{\BB_1} \le C_1 \| \PP_2 \omega\|_{\BB_1}\mbox{ and }
\| \PP_1 \omega\|_{\BB_2} \le C_2 \| \PP_2 \omega\|_{\BB_1}
\end{equation}
gives 
$$
\| \PP_1 (\omega)\|_{[\BB_1,\BB_2]_\theta} \le C_1^{1-\theta} 
C_2^{\theta} \| \PP_2( \omega)\|_{\BB_1}\, ,
$$
while interpolating at $\theta$ between
\begin{equation}\label{version2}
\| \PP_1( \omega)\|_{\BB_1} \le  \|  \PP_2 (\omega)\|_{\BB_1}\mbox{ and }
\| \PP_1 (\omega)\|_{\BB_2} \le C_1  \| \PP_2 (\omega)\|_{\BB_1} + 
C_2 \|  \PP_3 (\omega)\|_{\BB_1}\, , 
\end{equation}
gives 
$$
\| \PP_1( \omega)\|_{[\BB_1,\BB_2]_\theta} \le
\max( C_1^{\theta},  C_2^{\theta}) ( \| \PP_2 (\omega)\|_{\BB_1}+
\|  \PP_3 \omega\|_{\BB_1}) \, .
$$
For \eqref{version1}, one can easily adapt the argument in the last paragraph
of \cite[Theorem 4.1.2]{BL}. Using the notation there (in particular
$\FF_{\BB_1, \BB_2}$), let $\omega\in \BB_1$. Put $g(z)=C_1^{z-1}C_2^{-z} \PP_1(\omega)$.
Then $g\in \FF_{\BB_1,\BB_2}$
with 
$
\|g\|_{\FF}\le \|\PP_2(\omega)\|_{\BB_1}
$.
In addition, $g(\theta)=C_1^{\theta-1}C_2^{-\theta} \PP_1(\omega)$
so that 
$$
\|\PP_1(\omega)\|_{[\BB_1,\BB_2]_\theta}
\le C_1^{1-\theta}C_2^{\theta}\|g\|_\FF
\le C_1^{1-\theta}C_2^{\theta}  \|\PP_2(\omega)\|_{\BB_1}\, .
$$

For \eqref{version2}, put
$g(z)=(\max(C_1,C_2))^{-z} \PP_1(\omega)$. 
Then $g\in \FF_{\BB_1,\BB_2}$
with 
$
\|g\|_{\FF}\le \max(  \|  \PP_2 \omega\|_{\BB_1},
\|\PP_2(\omega)\|_{\BB_1}+\|\PP_3(\omega)\|_{\BB_1})\|\PP_2(\omega)\|_{\BB_1}+\|\PP_3(\omega)\|_{\BB_1}
$.
In addition, $g(\theta)=(\max(C_1,C_2))^{-\theta}\PP_1(\omega)$
so that 
$$
\|\PP_1(\omega)\|_{[\BB_1,\BB_2]_\theta}
\le (\max(C_1,C_2))^{\theta} \|g\|_\FF
\le \max(C_1^\theta,C_2^\theta)  (\|\PP_2(\omega)\|_{\BB_1}+\|\PP_3(\omega)\|_{\BB_1})\, .
$$
\end{remark}

\subsection{Mollification  $\MMM_\epsilon$ of  distributions on $M$}
\label{Meps}

Fix $\epsilon_0$ small.
Let $\eta: \real^{d} \to [0,\infty )$ be a bounded and compactly supported
$C^\infty$ function, supported in $|x|\le 1$
and bounded away from zero on  $|x|\le 1/3$,
with $\int \eta(x)\, dx =1$, and set, for $0<\epsilon<\epsilon_0$,
$$
\eta_\epsilon(y)=\frac{1}{\epsilon^{d}} \eta\bigl (\frac{y}{\epsilon} \bigr )
\, .
$$

Recall the family of $C^2$ charts $\kappa_{i,j,\ell}: U_{i,j,\ell, 0} \to \real^d$
from Section~\ref{spaces} and the standard contact form $\alpha_0$
from \eqref{contact}. We have $(0, x^s , 0) \in \Ker \alpha_0$ for
all $x^s$, and we shall use this as a local fake stable foliation
to define the averaging operator $\AAA^s_\delta$ in Section~ \ref{Aeps}
and in  the proof of the Dolgopyat estimate in Section ~\ref{carlangelo}.
Fix $C^\infty$ functions $\bar \theta_{i,j,\ell}:M \to [0,1]$,
supported in the set $U_{i,j,\ell,1}$ 
(recall that $\overline U_{i,j,\ell,1} \subset U_{i,j,\ell,0}$)
and so that 
$$\sum_{i,j,\ell} \bar \theta_{i,j,\ell} (q)=1\, \quad
\forall q\in Y_0:=\cup_{i,j,\ell}\,  \bar U_{i,j,\ell,2}\, .
$$
(Recall that $U_{i,j,\ell,2}$
was introduced in Step~1 of the proof of Lemma~\ref{LY0}, with
$\overline U_{i,j,\ell,2} \subset U_{i,j,\ell,1}$,
and that the definition ensures that
$X_0$ is contained in the interior of $Y_0$.)
As before, we write $\zeta=(i,j,\ell)$ to simplify notation.
(Note that the parameter $m \in \integer^d$ does not appear in this section.)

The mollifier operator $\MMM_\epsilon$ 
is defined for   $0<\epsilon<\epsilon_0$
\footnote{In the application, we shall take $\epsilon_0$ small enough as a
function of $n=\lceil c \ln |b| \rceil$, in particular much smaller than $1/R$ and than
the ``gap" between $Y_0$ and $X_0$.}   by first setting
for $\psi \in L^\infty(X_0)$ and $x \in \kappa_{i,j,\ell}(U_{i,j,\ell,1})$, 
\begin{equation*}
(\MMM_\epsilon (\psi))_\zeta(x)
= 
\int_{\real^{d}} \eta_\epsilon ( x-y) 
\psi ( \kappa_{\zeta}^{-1}(y))
\, dy = [\eta_\epsilon *  (\psi \circ \kappa_\zeta^{-1})](x)\, .
\end{equation*}
Then,
we set for $\psi \in L^\infty(X_0)$
\begin{equation}
\MMM_\epsilon (\psi)=\sum_\zeta \bar \theta_\zeta (( (\MMM_\epsilon (\psi))_\zeta \circ  \kappa_\zeta) \, .
\end{equation}
Since $\psi$ is supported in $X_0$, then 
$\MMM_\epsilon(\psi)$ is supported in $Y_0$ for small enough $\epsilon$.
We have the following  bounds for $\MMM_\epsilon$:

\begin{lemma}\label{mollbound1}
For each $p\in (1,\infty)$,
and all  $-1<r'\le 0$ and $r' \le r <2+r'$ in $\real$,
there exists $\Cs$ so that
for all small
enough $\epsilon >0$ and every $\psi\in H^{r'}_p(M)$,
supported in $X_0$,
\begin{equation}\label{claim1}
\|\MMM_\epsilon (\psi)\|_{H^{r}_p(M)} 
\le \Cs \epsilon^{r'-r} \|\psi\|_{H^{r'}_p(M)}  \, .
\end{equation}
\end{lemma}
(Smoothness of $\eta$ is important in the proof
of \eqref{claim1}.)

\begin{lemma}\label{mollbound2}
Let $p \in (1, \infty)$ and $-1< r'<s<0$.
There exists $\Cs >0$ so that for any  small enough $\epsilon >0$,  and all $\psi\in H^{s}_p(X_0)$
\begin{equation}\label{lemma1.1b}
\norm{\MMM_\epsilon (\psi)-\psi}{H^{r'}_p(M)}
\le \Cs \epsilon^{s-r'} \norm{\psi}{H^{s}_p(X_0)} \, .
\end{equation}
\end{lemma}

\begin{proof}[Proof of Lemma~ \ref{mollbound1}]
We assume that
$\epsilon$ is small enough so that 
$\MMM_\epsilon(\psi)$ is supported in the interior of $Y_0$. Hence,
we can implicitly replace $H^\sigma_p(M)$ by 
$H^\sigma_p(Y_0)$ in the argument. 

Since the charts and partition of unity
are $C^{2}$, the bound \eqref{claim1} on classical Sobolev
spaces is standard,  using
Remark~\ref{minkoo}  and interpolating (interpolation
is allowed for Triebel spaces on $\real^d$
or a manifold,
beware it cannot
be used directly for our spaces $\HHH$).
We provide details for the convenience of the reader:
We shall use that for any real $-1<r'<1$ and any $1<p<\infty$
\begin{equation}\label{important}
\norm{\psi}{H^{1+r'}_p(M)}\le \Cs (\norm{D\psi}{H^{r'}_p(M)}+ \norm{\psi}{H^{r'}_p(M)})\, .
\end{equation}
To prove \eqref{important}, combine $\norm{\omega}{H_p^{1+ r'}}\sim \norm{\omega}{H^{r'}_p} +
\norm{D\omega}{H^{r'}_p}$
(this can be proved by interpolation
and duality from the corresponding results for integer  $r'\ge 0$
available e.g. in \cite[Prop 2.1.2 (iv)+(vii)]{RS}),
with \eqref{version1},
using that $\bar \theta_\zeta$ and $\kappa_\zeta$ are $C^{2}$.

The Minkowski integral inequality \eqref{minko} implies that
for each  $r'\in \real$ and $1<p<\infty$, there exists $\Cs$ so that 
$\norm{\eta_\epsilon * \omega }{H^{r'}_p}\le \Cs 
\norm{\omega }{H^{r'}_p}$ for all
$\epsilon$.
Next,
$
|D^\ell( \eta_\epsilon) (x)| \le \Cs \epsilon^{-\ell} | (D^\ell\eta)_\epsilon|
$,
which implies
$
\int |D^\ell (\eta_\epsilon)(x)|\, dx \le \Cs \epsilon^{-\ell}$
for all integers $\ell \ge 1$.
Thus, by \eqref{minko}, for any integer $\ell\ge 1$,
\begin{equation*}
\norm{(D^\ell (\eta_\epsilon* \omega) }{H^{r'}_p}
=\norm{(D^\ell (\eta_\epsilon) )* \omega }{H^{r'}_p}\le \Cs \epsilon^{-\ell}
\norm{\omega }{H^{r'}_p} \, .
\end{equation*}
Therefore, applying \eqref{minko} again, this time with
\eqref{important},
and since
$$
D (\MMM_\epsilon \psi)= \sum_{\zeta} (D\bar \theta_\zeta)
(\MMM_\epsilon(\psi))_\zeta \circ  \kappa_\zeta
+\sum_{\zeta} \bar \theta_\zeta
D(\MMM_\epsilon(\psi))_\zeta \circ \kappa_\zeta \, D \kappa_\zeta
$$
with
$D (\MMM_\epsilon(\psi))_\zeta(x)=   D(\eta_\epsilon)*(\psi \circ  \kappa_\zeta^{-1})(x)$,
and similarly for the second derivative,
and also, putting $G_{\zeta' \zeta} =  \kappa_{\zeta'} \circ \kappa_\zeta^{-1}$, if the domain of the
map is nonempty,
\begin{align*}
\label{chain00}
& [((\bar \theta_\zeta \bar \theta_{\zeta'}) \circ \kappa_\zeta^{-1})
((\MMM_\epsilon(\psi))_{\zeta'} \circ G_{\zeta' \zeta}  ] (x)
\\ 
\nonumber &\quad
= \frac{(\bar \theta_\zeta\bar \theta_{\zeta'}) (\kappa_\zeta^{-1}(x))}
{\epsilon^{d}} \int_{\real^d} \eta\bigl (\frac{G_{\zeta' \zeta}(x)-y}{\epsilon}\bigr )
\psi (\kappa_\zeta^{-1} \circ G_{\zeta' \zeta}^{-1} (y))
\, dy\, ,
\end{align*}
we find for  $\ell=0$, $1$, $2$, and $-1<r'\le   \ell + r'<2$
\begin{equation}\label{forell}
\norm{\MMM_\epsilon (\psi)} {H^{\ell+r'}_p(M)}
\le \Cs \epsilon^{-\ell} \norm{\psi}{H^{r'}_p(M)} \, .
\end{equation}
Our assumptions imply that there exists $0\le \ell \le 1$ so that
$\ell+ r' \le  r < \ell +1+r'$. Interpolating between the corresponding inequalities
\eqref{forell} (at $(r-\ell-r')/r'$), we obtain
\eqref{claim1}.
\end{proof}

\begin{proof}[Proof of Lemma~ \ref{mollbound2}]
Lemma~\ref{mollbound1} implies that
\begin{equation}\label{interp0}
\|\MMM_\epsilon (\psi)-\psi\|_{H_p^{r'}(M)}
\le \Cs \|\psi\|_{H_p^{r'}(M)}\, .
\end{equation}
We shall next prove that 
there exists a constant $\Cs$  so that for all $-1\le r'\le 0$
\begin{equation}\label{interp1}
\|\MMM_\epsilon (\psi)-\psi\|_{H_p^{r'}(M)}
\le \Cs\epsilon \|\psi\|_{H_p^{r'+1}(M)}\, .
\end{equation}
Interpolating at $\theta=1-(s-r')\in (0, 1)$
between \eqref{interp0}  and \eqref{interp1} gives the result
since $\|\psi\|_{H_p^{\sigma}(M)}=\|\psi\|_{H_p^{\sigma}(X_0)}$
for $\psi \in {H_p^{\sigma}(X_0)}$, by our definition.

Putting $G_{\zeta' \zeta} =  \kappa_{\zeta'} \circ \kappa_\zeta^{-1}$, if the domain of the
map is nonempty,  we 
have 
\begin{align}
\label{chain0}
& [(\bar \theta_\zeta \bar \theta_{\zeta'} )\circ \kappa_\zeta^{-1}
((\MMM_\epsilon(\psi))_{\zeta'} \circ G_{\zeta' \zeta}  ] (x)
\\ 
\nonumber &\quad
= \frac{(\bar \theta_\zeta\bar \theta_{\zeta'}) (\kappa_\zeta^{-1}(x))}
{\epsilon^{d}} \int_{\real^d} \eta\bigl (\frac{y}{\epsilon}\bigr )
\psi (\kappa_\zeta^{-1} \circ G_{\zeta' \zeta}^{-1} (G_{\zeta' \zeta}(x)-y))
\, dy\, .
\end{align}

Therefore, to prove \eqref{interp1},  we must bound the $H^{r'}_p$ norm of
\begin{align}\label{toestimate}
&\frac{(\bar \theta_\zeta
\bar \theta_{\zeta'} )(\kappa_\zeta^{-1}(x))}{\epsilon^d}
\int_{\real^d} \eta(y/\epsilon)
\\ 
\nonumber &\qquad\qquad\qquad\cdot
[\psi (\kappa_\zeta^{-1} \circ G_{\zeta' \zeta}^{-1} (G_{\zeta' \zeta}(x)-y))
-\psi (\kappa_\zeta^{-1} \circ G_{\zeta' \zeta}^{-1} (G_{\zeta' \zeta}(x)) ]
\, dy .
\end{align}

Setting $u=G_{\zeta' \zeta}(x)$ and
$\omega=\psi \circ \kappa_\zeta^{-1} \circ G_{\zeta' \zeta}^{-1}$,
the integral remainder term in the order-zero
Taylor expansion gives
\begin{align}\label{TaylorM}
\omega(u-y)-\omega(u)
=-\sum_{\ell=1}^d
\int_0^1 y_\ell\cdot \partial_\ell \omega(u-t y)\, dt \, .
\end{align}

Therefore, using the Minkowski integral inequality \eqref{minko}, 
and setting $\bar \theta_{\zeta\zeta'}=(\bar \theta_\zeta
\bar \theta_{\zeta'} )\circ \kappa_\zeta^{-1}$, we get that
the $H^{r'}_p$ norm
of  \eqref{toestimate} is bounded by
$$
\Cs  \sup_{\zeta,\zeta', y \in 2 \epsilon \supp(\eta), \ell} 
\|\bar\theta_{\zeta \zeta'}(x)
[ \partial_\ell \psi (\kappa_\zeta^{-1} \circ G_{\zeta' \zeta}^{-1})] (G_{\zeta' \zeta}(x)-  y ))\|_{H^{r'}_p}\, .
$$

To conclude the proof of \eqref{interp1}, use that
for all $\ell=1, \ldots, d$, we have $|y_\ell|\le \Cs \epsilon$
for $y_\ell$ in \eqref{TaylorM} and
\begin{align*}
\nonumber &\bar\theta_{\zeta \zeta'}
\cdot
\partial_\ell [\psi \circ \kappa_\zeta^{-1}] 
=\partial_\ell [
\bar\theta_{\zeta \zeta'}
(\psi \circ \kappa_\zeta^{-1} )]
-\partial_\ell [\bar\theta_{\zeta \zeta'} ]
\psi \circ \kappa_\zeta^{-1}\, ,
\end{align*}
and that
$ \norm{ \partial_\ell \omega}{H_p^{r'}}
\le \Cs \norm{  \omega}{H_p^{r'+1}}$.
\end{proof}

\subsection
{The stable-averaging operator}
\label{Aeps}

We now discuss the main technical idea used to exploit
Dolgopyat's method \cite{Do}. It is borrowed from \cite{Li} and consists in replacing
$\psi$ by its average over a piece of (fake) stable
\footnote{In \cite{Li}, unstable manifolds were used because of the dual nature
of the argument there.} manifold.
\footnote{In particular, since we do not require
the equivalent of \cite[Sublemma 3.1]{Li} but only the analogue
of \cite[Sublemma 4.1]{Li}, we need 
less smoothness: Liverani \cite{Li} required $C^4$,
although he points out that $C^{2+ \epsilon}$ should suffice
in \cite[footnote 6]{Li}.}
Put 
$$
\eta_s: \real^{d_s} \to [0,1]\,,\quad
\eta_s(x)=\Id_{\{\|x^s\|\le 1\}}\, ,
\qquad 
\eta_{s,\delta}(x^s)= \frac{1}{S_{d_s}\delta^{d_s}} \eta_s(x^s/\delta)\, ,
$$
where $S_{d_s}$ is the volume of the $d_s$-dimensional unit ball.
For fixed small $\delta >0$
(so  that
\footnote{In Section~\ref{dodo}, we shall need to choose $\delta$ small enough as a
function of $n=\lceil c \ln |b| \rceil$.}
the 
$\delta$-neighbourhood of $ \kappa_\zeta(U_{\zeta, 1})$
is included in the domain of $\kappa_\zeta^{-1}$ for each
$\zeta=(i,j,\ell)$) we set
for each
$\zeta$, $\psi \in L^\infty(X_0)$, and $x \in  \kappa_\zeta(U_{\zeta, 1})$
\begin{equation*}
(\AAA^s_{\delta} (\psi))_\zeta (x^u, x^s, x^0)
=  \int_{\real^{d_s}} \eta_{s,\delta}( x^s-y^s) 
\psi ( \kappa_\zeta^{-1}(x^u, y^s,  x^0))
\, dy^s \, .
\end{equation*}
(In the above, we implicitly extend $\psi$ by zero outside of its support $X_0$.)

We set for $\psi \in L^\infty(X_0)$ 
\begin{equation}
\AAA^s_{\delta} (\psi)=\sum_{\zeta} \bar \theta_\zeta \cdot (\AAA^s_\delta (\psi))_\zeta \circ  \kappa_\zeta \, .
\end{equation}
We shall assume that $\delta$ is small enough so that $\AAA^s_{\delta} (\psi)$
is supported in the interior of $Y_0$.
The key estimate on $\AAA^s_{\delta}$ is contained in the next lemma.

\begin{lemma}\label{Abound}
Let  $1<p<\infty$ and $-1+1/p<r'<s\le 0$.
There exists $\Cs >0$  so that for every small enough $\delta >0$, 
and every bounded function $\psi$
supported in $X_0$
$$
\|\AAA^s_{\delta} (\psi)- \psi\|_{H^{r'}_p(M)} \le \Cs \delta^{s-r'} \|\psi\|_{H^{s}_p(M)} 
\,  ,
$$
and  also
$
\|\Id_{X_0} \AAA^s_{\delta} (\psi)- \psi\|_{H^{r'}_p(X_0)} 
\le \Cs \delta^{s-r'} \|\psi\|_{H^{s}_p(X_0)} 
\,  
$.
\end{lemma}

There is  no equivalent to Lemma~\ref{mollbound1}
($\AAA^s_{\delta}$ is not a mollifier!).

\begin{proof}
We proceed as in Lemma~\ref{mollbound2} to prove the
first bound. Note in particular that
the argument there only used that $\eta\in L^1$ was a probability
density (no smoothness of $\eta$ was required).

The analogue of \eqref{chain0} is
\begin{align}\label{chaina0}
& (\bar \theta_\zeta  \bar \theta_{\zeta'}) (\kappa_\zeta^{-1}(x))
\int_{\real^{d_s}} \frac{\eta_s\bigl (\frac{y^s}{\delta}\bigr )}{S_{d_s}\delta^{d_s}}
\psi (\kappa_\zeta^{-1}  (G_{\zeta'\zeta})^{-1} (G_{\zeta'\zeta}(x)-y^s))
\, dy^s
\, .
\end{align}

To prove the analogue of \eqref{interp1},  we put 
$\omega=\psi \circ \kappa_\zeta^{-1}  \circ G_{\zeta'\zeta}^{-1}$
and consider
\begin{align*}
&(\bar \theta_\zeta  \bar \theta_{\zeta'}) (\kappa_\zeta^{-1}(x))
\int_{\real^{d_s}}
\frac{\eta_s\bigl (\frac{y^s}{\delta}\bigr )}{S_{d_s}\delta^{d_s}}
[\omega(G_{\zeta'\zeta}(x)-  y^s ) - \omega (G_{\zeta'\zeta}(x)) ]
\, dy^s\, .
\end{align*}
Then,
\eqref{TaylorM} is replaced by the following zero-order Taylor expansion with
integral remainder term
\begin{align}\label{TaylorA}
&\omega (u-y^s)-\omega(u)
=-\sum_{\ell=d_u+1}^{d_u+d_s}
\int_0^1 
y^s_\ell\cdot \partial_\ell \omega (u-t y^s)\, dt\, .
\end{align}
Note that, even if $\ell$ corresponds
to a stable coordinate, 
$\partial_\ell( \psi \circ \kappa_\zeta^{-1}  \circ G_{\zeta'\zeta})^{-1})$
can involve all partial derivatives of $\psi$,
since $G_{\zeta'\zeta}$
does not preserve stable leaves in general.
We obtain the claimed  $\Cs \delta^{s-r'}$ factor
by interpolation, since $|y_\ell|\le \Cs \delta$.

The second claim follows from the first one, Corollary~\ref{StrStr} and the
identity 
$
(\Id_{X_0}\AAA^s_{\delta} (\psi))- \psi=\Id_{X_0}(\AAA^s_{\delta} (\psi)- \psi)
$ for any $\psi$ supported in $X_0$.
\end{proof}

\subsection{End of the proof of  Lemma~\ref{bq} on $\RR(z)$}
\label{endbq}

\begin{proof}[End of the proof of  Lemma~\ref{bq} on $\RR(z)$]
Let us deduce Lemma~\ref{bq} from (\ref{partialR}).
The proof is  by interpolation, but
this interpolation must be done at the level of Triebel spaces.
Also, the special form of the admissible charts must be used.

First note that $\LL_t \RR(z)=\RR(z)\LL_t$ so that
$$
\| \RR(z)(\psi)\|_{\widetilde \HHH^{r,s,q}_p(R)}=\sup_{\tau \in [0,t_0]} \| \RR(z) \LL_\tau(\psi)\|_{\HHH^{r,s,q}_p(R)}\ \, .
$$

Let us first take $q'=1$, and $q= 0$, setting
$\tilde \psi=\RR(z)(\LL_{\tau} (\psi))$ for
$0 \le \tau \le t_0$. By
Definition \ref{defnorm}, we must consider
$$
\left(\sum_{\zeta=(i,j,\ell,m)\in \ZZ(R)}
\norm
{[\rho_m\cdot (  \tilde \psi\circ 
(\kappa_\zeta^R)^{-1} )]\circ \phi_\zeta}
{H_p^{r,s,1}}^p\right)^{1/p}\, .
$$

Now,
\begin{align}
\nonumber \norm
{[\rho_m\cdot (  \tilde \psi\circ 
( \kappa_\zeta^R)^{-1}) ]\circ \phi_\zeta}
{H_p^{r,s,1}}
& \le \Cs
\norm
{[\rho_m\cdot ( \tilde \psi\circ 
(\kappa_\zeta^R)^{-1}) ]\circ \phi_\zeta}
{H_p^{r,s,0}}\\
\label{(a)} &\qquad\qquad +
\Cs 
\norm
{\partial_{x^0}[\rho_m\cdot ( \tilde \psi\circ 
(\kappa_\zeta^R)^{-1})\circ \phi_\zeta\ ]}
{H_p^{r,s,0}}\, .
\end{align}
We have
\begin{align}
\nonumber \norm
{  \partial_{x^0}[\rho_m\cdot 
( \tilde \psi\circ  (\kappa_\zeta^R)^{-1})\circ \phi_\zeta
 ]}
{H_p^{r,s,0}}
& \le 
\Cs\norm { [(\partial_{x_0} \rho_m)\cdot 
( \tilde \psi\circ  (\kappa_\zeta^R)^{-1})]\circ \phi_\zeta
 }
{H_p^{r,s,0}}\\
\label{(b)} &\qquad+
\Cs \norm{ [\rho_m\cdot 
 \partial_{x_0} (\tilde \psi
\circ  (\kappa_\zeta^R)^{-1})]\circ \phi_\zeta
 }
{H_p^{r,s,0}}\, .
\end{align}
Recall that  $
\phi_\zeta(x^u, x^s,x^0)=(F(x^u, x^s), x^s, x^0+f(x^u, x^s))
$.
If $t>0$ is small enough,
since the charts $\kappa_\zeta=\kappa_{i,j,\ell}$ preserve  the flow direction and
time units, 
we find
\begin{align}
\nonumber &\RR(z) \LL_\tau
(\psi (  (\kappa_\zeta^R)^{-1} (F(x^u, x^s), x^s, x^0+f(x^u, x^s)))\\
\nonumber &\qquad\qquad\qquad\quad-
\RR(z)( \LL_{\tau}\psi) ((\kappa_\zeta^R)^{-1} (F(x^u, x^s), x^s, x^0-t+f(x^u, x^s)) \\
\label{(c)}    &\qquad\qquad\qquad\qquad =( \RR(z)\LL_\tau(\psi) -\RR(z)\LL_{\tau+t/R} (\psi)) (  (\kappa_\zeta^R)^{-1}\circ \phi_\zeta )(x)\, .
\end{align}
Dividing by $t$ and letting $t\to 0$, it follows from \eqref{partialR} at $t_1=\tau$
and from (\ref{(a)}-\ref{(b)}-\ref{(c)})
that
\begin{align}
\label{interpbq1}& \norm
{[\rho_m\cdot ( \RR(z)(\LL_\tau (\psi))\circ 
(\kappa_\zeta^R)^{-1} )]\circ \phi_\zeta}
{H_p^{r,s,1}}\\
\nonumber & \qquad \qquad\le
\Cs \bigl (1+ \frac{|z|}{R}\bigr )
\norm
{[\rho_m\cdot ( \RR(z)( \LL_\tau(\psi))\circ 
(\kappa_\zeta^R)^{-1} )]\circ \phi_\zeta}
{H_p^{r,s,0}}\\
\nonumber  &\qquad \qquad\qquad \qquad+\Cs \norm
{[(\partial_{x_0} \rho_m)\cdot (\RR(z)(\LL_\tau( \psi)))\circ 
(\kappa_\zeta^R)^{-1} ]\circ \phi_\zeta}
{H_p^{r,s,0}}\\
\nonumber &\qquad \qquad\qquad \qquad
+\Cs \norm
{[\rho_m\cdot (\LL_\tau (\psi)\circ 
(\kappa_\zeta^R)^{-1} )]\circ \phi_\zeta}
{H_p^{r,s,0}}\, .
\end{align}
Since
$\sum_ {m'\in \ZZ_{i,j,\ell}(R)} \rho_{i,j,\ell,m'}\equiv 1$,
we may use
Lemma ~\ref{Leib} (for $\tilde \beta=\infty$) and \eqref{converse} in
Lemma ~\ref{lem:localization} (using \eqref{locall'}=\eqref{locall}) to replace $\partial_{x^0} \rho_m$ by a finite sum
of $\rho_{m'} \partial_{x^0} \rho_m$ for neighbours $m'$
of $m$ (the number of neighbours is uniformly bounded, which gives rise
to a bounded overcounting).
For $q>0$ and $q'\in (q, q+1)$,  setting
$q''=q(1-(q-q'))$ and applying complex interpolation
(see \eqref{version2}) at $\theta=q/q''$ between
\eqref{interpbq1} and
\begin{align*}
& \norm
{[\rho_m\cdot ( \RR(z)( \LL_\tau( \psi))\circ 
(\kappa_\zeta^R)^{-1} )]\circ \phi_\zeta}
{H_p^{r,s,q''}}\\
&\qquad\qquad\qquad\qquad\qquad\le
\norm
{[ \rho_m\cdot (\RR(z)(\LL_\tau (\psi)))\circ 
(\kappa_\zeta^R)^{-1} ]\circ \phi_\zeta}
{H_p^{r,s,q''}}\, ,
\end{align*}
we get,  since $1\le (\frac{|z|}{R}+1)^{q'-q}$,
summing all terms,
\begin{align*}
&\left(\sum_{\zeta=(i,j,\ell,m)\in \ZZ(R)}
\norm{ [\rho_m\cdot 
(  (\tilde \psi
\circ  (\kappa_\zeta^R)^{-1}]\circ \phi_\zeta)}
{H_p^{r,s,q'}}^p\right)^{1/p}\\
&\le
\Cs 
\bigl (\frac{|z|}{R}+1 \bigr )^{q'-q}
 \bigl [
\bigl(\sum_{\zeta=(i,j,\ell,m)\in \ZZ(R)}
\norm{ [\rho_m\cdot 
( \RR(z) (\LL_\tau  (\psi))
\circ  (\kappa_\zeta^R)^{-1})]\circ \phi_\zeta}
{H_p^{r,s,0}}^p\bigr)^{1/p}\\
&\qquad\qquad\qquad \quad+ \bigl(\sum_{\zeta=(i,j,\ell,m)\in \ZZ(R)}
\norm{ [\rho_m\cdot 
( \LL_\tau( \psi)
\circ  (\kappa_\zeta^R)^{-1})]\circ \phi_\zeta}
{H_p^{r,s,q}}^p\bigr)^{1/p}\bigr ] \, .
\end{align*}
In view of \eqref{q=0}, this ends the proof of Lemma~\ref{bq}.
\end{proof}

\section{The Dolgopyat estimate}\label{carlangelo}

The purpose of this section is to prove the following lemma.
\begin{lemma}\label{dolgolemma}
Assume $d=3$. There exist $\Cs>0$, $\blambda>1$ and $\gamma_{0}>0$ so 
that for all $a>1\,,b>1$, $\gamma\geq \gamma_{0}$, $m\geq \Cs a\gamma\ln b$, and $\tilde\psi\in H^1_\infty(M)$
\[
\|
\bA^s_\delta (\cR(a+ib)^{2m} (\tilde \psi))\|_{L^\infty(X_0)}\leq \Cs a^{-2m} b^{-\gamma_{0}}(\|\tilde\psi\|_{L^\infty(M)}+\nu_{a}^{-m}\|\tilde \psi\|_{H^1_\infty(M)})\, ,
\]
where $\delta=b^{-\gamma}$ and $\nu_{a}=(1+a^{-1}\ln\blambda)^{-1}$.
\end{lemma}

In the present section, $\Cs$, $\gamma_0$, and
$\bar \lambda$ denote constants which depend only on the dynamics
and not on $r$, $s$, or $p$. 

Note that, since $\overline{\cR(a+ib)}=\cR(a-ib)$ the obvious counterpart of the above lemma holds true for $b<0$.

The rest of the section consists in the proof of Lemma \ref{dolgolemma} and is a direct, but lengthy, computation.

In the present section, it is more convenient to work directly with the flow rather than with the Poincar\'e sections,
and  we will often look at the dynamics at different time steps $\vus>\vuo\gg\vu$. Let us be more precise: Remember from  \eqref{deccomp} the choice of $\tilde\tau_1\leq \inf_{i,j,z}\tau_{i,j}(z)$ 
and $\tilde\tau_0\geq \sup_{i,j,z}\tau_{i,j}(z)$. Fix once and for all $\vuo\leq \frac 14\tilde\tau_1$ and let $\mu_0=\lceil\frac{\tilde\tau_0}{\vuo}\rceil$, $\vus=\mu_0k_0\vuo\geq k_0\tilde\tau_0$.
The size of $k_0$ will be chosen large enough, but fixed, during the proof of Lemma \ref{lem:r-boundary}.  Also, we introduce the following  rough bounds for the minimal average expansion and contraction: Let  $\blambda_u=\inf_z\left[\lambda_{u,k_0}(z)\right]^{\frac 1{\vus}}>1$, $\blambda_s=\sup_z\left[\lambda_{s,k_0}(z)\right]^{\frac 1{\vus}}<1$ and set $\blambda=\min\{\blambda_u,\blambda_s^{-1}\}$. Thus the minimal expansion and contraction in a time $t\geq \vus$ will be bounded by $\blambda_{u,t}=\blambda^{t-\vus}$, $\blambda_{s,t}=\blambda^{-t+\vus}$.

This said, we start to compute. First of all it is convenient to localise in time: Consider a smooth function $\tilde p:\bR\to\bR$ such that $\operatorname{supp} \tilde p\subset (-1,1)$, $\tilde p(s)=\tilde p(-s)$, and $\sum_{\ell\in \bZ}\tilde p(t-\ell)=1$ for all $t\in\bR$. 
In the following it is convenient to proceed by very small time steps $\vu=\frac{\vus}{k_1}$, $k_1\in\bN$. Let $p(t)=\tilde p(\frac t\vu)$. 

For each $f\in L^\infty(M, \text{vol})$ we write
\begin{equation}\label{eq:step-1}
\begin{split}
\cR(z)^m(f)&=\sum_{\ell\in\bZ}\int_{0}^\infty p(t-\ell\vu) \frac{t^{m-1}}{(m-1)!}e^{-zt}\cL_t f\, dt\\
&=\sum_{\ell\in\bN^*}\int_{-\vu}^{\vu}  p(s) \frac{(s+\ell\vu)^{m-1}}{(m-1)!}e^{-z\ell\vu-zs}\cL_{\vu\ell}\cL_s f\, ds\\
&\quad+\int_{0}^{\vu} p(s) \frac{s^{m-1}}{(m-1)!}e^{-zs}\cL_s f\, ds\, .
\end{split}
\end{equation}
Next we recall the stable average introduced in Section \ref{Aeps}. Setting $p_{m,\ell,z}(s)=  p(s) \frac{(s+\ell\vu)^{m-1}}{(m-1)!}e^{-z\ell\vu-zs}$, for each $w\in\X$ we can write\footnote{\label{foo:pd} Here and in the following, when we write $\sum_{\ell\in\bN^{*}}$ we mean to include implicitly also the last term in \eqref{eq:step-1}. In any case, we will see in \eqref{eq:step3} that the total contributions of the first terms in the sum is negligible.}
\begin{equation}\label{eq:step0}
\bA^s_\delta (\cR(z)^m (f))(w)=\sum_{\ell\in\bN^{*},\zeta }\int_{-\vu}^{\vu}\hskip-.3cm p_{m,\ell,z}(s)\bar\theta_\zeta(w)\int_{W^s_{\delta, \zeta}(w)}\hskip-.5cm \tilde \frp_{\delta, \zeta}(w,\xi)\cdot\cL_{\ell\vu}\cL_s f(\xi)\, ,
\end{equation}
where (see Section \ref{Aeps}) $W^s_{\delta, \zeta}(w)=\{\kappa_\zeta^{-1}(\tilde \kappa_\zeta(w)^u, y^s,\tilde \kappa_\zeta(w)^0)\}_{y^s\in[-\delta,\delta]}\cap \X$. The integral is meant with respect to the volume form determined by the Riemannian metric restricted to $W^s_{\delta, \zeta}$ and $\tilde \frp_{\delta, \zeta}(w,\xi)=\eta_{s,\delta}(\kappa_\zeta(w)^s-\kappa_\zeta(\xi)^s)J_\zeta(w,\xi)$, where $J_\zeta$ is the Jacobian of the change of coordinates $\kappa_\zeta$ restricted to the manifold $W^s_{\delta, \zeta}$. Note that if $w$ is extremely close to a corner of $\X$, then $W^s_{\delta, \zeta}(w)$ could be extremely short, yet in such a case the integral will be trivially small. We can rewrite \eqref{eq:step0} as
\begin{equation}\label{eq:step02}
\begin{split}
\bA^s_\delta (\cR(z)^m (f))(w)=&\sum_{\ell\in\bN^*, \zeta}\int_{-\vu}^{\vu} \hskip-.3cm p_{m,\ell,z}(s)\bar \theta_\zeta(w)\int_{T_{-\ell\vu}W^s_{\delta, \zeta}}\hskip-.5cm\tilde \frp_{\delta, \zeta}\circ T_{\ell \vu}\cdot J^s_{\ell\vu}\cdot \cL_s f\\
=&\sum_{\ell\in\bN^*,\zeta}\!\!\bar\theta_\zeta(w)\hskip-.4cm\sum_{W\in \cW_\ell(w)}\int_{-\vu}^{\vu}\hskip-.4cm p_{m,\ell,z}(s)\!\!\int_{W}\hskip-.2cm\tilde \frp_{\delta, \zeta}\circ T_{\ell \vu} \cdot J^s_{\ell\vu} \cL_s f\, ,
\end{split}
\end{equation}
where  $J^s_{\ell\vu}$ is the Jacobian of the change of variable, $\int_{W^s_{\delta, \zeta}(w)}\tilde \frp_{\delta, \zeta}(w,z)\leq1$, and $\cW_\ell(w):=\{W_\alpha\}_{\alpha\in A_\ell(w)}$ is a decomposition of $T_{-\ell\vu}W^s_{\delta, \zeta}$ in {\em regular} connected pieces. By {\em regular} we mean that there exists $t\in[0,\tau_-]$ such that $T_{-t}W_\alpha$ is a $C^2$ manifold.~\footnote{ The issue here is that if $W_\alpha$ intersects one $O_i$,  then it may be discontinuous. Yet, such a lack of smoothness is only superficial since once the manifold is flowed past the section, it becomes smooth. More precisely, if $W_\alpha$ does not intersect any $O_i$, then it has uniformly bounded curvature (i.e., if $g$ is its parametrisation by arc-length, then $\|g''\|_{L^\infty}\leq \Cs$). This can be proved exactly as one proves the same bound on the curvature of the stable leaves of an Anosov map, see \cite{KH} for details.}

The decomposition $\cW_\ell$ is performed as follows. We choose $L_0>0$ such that a curve in the stable cone of length $L_0\|\lambda_{s,\mu_0k_0}\|_{L^\infty}$ intersecting an $O_i$ must lie entirely in its $\vuo/4$ neighbourhood and then we define $\cW_{k_1\kappa}$, $\kappa\in\bN$, recursively:~\footnote{ For simplicity we suppress the dependence on $w,\zeta$ when this does not create confusion.} First, let $\cW_0$ be the collection of the connected components of $W^s_{\delta, \zeta}\setminus \cup_{i,j} \partial B_{i,j}$. Given $\cW_{k_1\kappa}$ define first $\widetilde \cW_{k_1(\kappa+1)}$ to be the union of the connected regular pieces of the curves $T_{-\vus}W$ for $W\in\cW_{k_1\kappa}$. Next, if a curve $W\in \widetilde\cW_{k_1(\kappa+1)}$ is longer than $L_0$, we decompose it in curves of length $\frac 12 L_0$ apart from the last piece that will have length in the interval $[\frac 12 L_0,L_0)$. The set of curves so obtained is $\cW_{k_1(\kappa+1)}$. Finally, for the $\ell\in \{k_1\kappa+1,\dots,k_1(\kappa+1)-1\}$ we define $\cW_\ell$ as the collection of the connected regular pieces of  $T_{(k_1\kappa-\ell)\vu}W$, $W\in\cW_{k_1\kappa}$.
We will call the curves shorter than $L_0/2$ {\em short,} and the others {\em long.}
Note that, by construction, for each $\alpha\in A_\ell$ there exists $t_\alpha\in[0,\vuo]$ such that  $T_{-t_\alpha}W_\alpha$ is a $C^2$ curve with uniform $C^2$ norm and $T_{\ell\vus-t_\alpha}$, restricted to $T_{-t_\alpha}W_\alpha$, is a $C^2$ map.

The density $\tilde \frp_{\delta, \zeta}$ has the property $|\nabla\log\tilde \frp_{\delta, \zeta}|_{\infty}\leq \Cs$, for some fixed constant $\Cs$. Then, setting 
\begin{equation}\label{eq:ro-la}
\frp_{\ell, \alpha}=\frac{\tilde \frp_{\delta, \zeta}\circ T_{\ell \vu}\cdot J^s_{\ell\vu}}{Z_{\ell, \alpha}}\quad ;\quad Z_{\ell, \alpha}= \int_{W_\alpha}\tilde \frp_{\delta, \zeta}\circ T_{\ell \vu}\cdot J^s_{\ell\vu}\, ,
\end{equation}
we have $|\nabla\log\frp_{\ell,  \alpha}|_{\infty}\leq \Cs$, provided $\Cs$ is chosen large enough.\footnote{\label{foo:dist}Indeed, the flow induces 
one-dimensional maps between $T_{\kappa\vuo-t_\kappa}W_\alpha$ and $T_{(\kappa+1)\vuo-t_{\kappa+1}}W_\alpha$ ($t_k\in [0,\vuo]$ properly chosen), which, by parametrising the curves by arc-length, are uniformly $C^2$. So the claim follows by the usual distortion results on one-dimensional maps, see \cite{KH} for details.}
In addition, note that $\sum_{\alpha}Z_{\ell, \alpha}\leq1$.\footnote{In fact, the sum is exactly equal to one if no manifold is cut by the boundary of $\X$.}

Next, it is convenient to define the $r$-boundary $\partial_r(\cW_\ell)$ of the family $\cW_\ell$:
\[
\partial_r(\cW_\ell)=\cup_{\alpha\in A_\ell}\{x\in W_\alpha\;:\; d(x,\partial W_\alpha)\leq r\} \, ,
\]
where, given any two sets $A,B$, we define $d(A,B)=\inf_{x\in A, y\in B} d(x,y)$.
Not surprisingly, we will call $|\partial_r(\cW_\ell)|:=\sum_{\alpha\in A_\ell}Z_{\ell,\alpha}\int_{\{d(x,\partial W_\alpha)\leq r\}}\frp_{\ell,\alpha}$ the measure of the $r$-boundary of $\cW_\ell$.\footnote{\label{foo:meas}Note that $|\partial_r\cW_\ell |=\int_{T_{\ell\tau_-}\left[\partial_r(\cW_\ell)\right]}\tilde \frp_{\delta, \zeta}$ and hence the measure of $T_{\ell\tau_-}[\partial_r(\cW_\ell)]$ is bounded, above and below, by $\Cs\delta |\partial_r(W_\ell)|$.}

\begin{lemma}\label{lem:r-boundary}  There exists $\sigma\in (0,1)$ such that, for all $\delta$ small enough,  $|\partial_r(\cW_\ell)|\leq \Cs\max\{r, \sigma^{\frac\ell{k_1}} r\delta^{-1}\}$. In addition, for each $\upsilon\in (0,1)$ and $\sigma^{\ell}<\delta^{k_1}$, there exists $C_{\upsilon}>0$ such that the measure of  $\cup_{j\leq m} T_{j \vu}\partial_{\upsilon^{j\vu} r}(\cW_j)$ is bounded by $C_{\upsilon} r$ .
\end{lemma}
\begin{proof}
If $x\in\partial_r(\cW_\ell)$ then we have the following possibilities
\begin{enumerate}
\item $T_{\ell\vu}x$ belongs to a $r\blambda^{-\vu \ell}$-neighbourhood of $\partial \left(W^s_{\delta, \zeta}\setminus \cup_{i,j} O_{i,j}\right)$.
\item there exists $n\in \bN$, $n\leq \ell/k_1$, such that $T_{n\vus-t}x$, $t\in[0,\vus]$, intersects the $\blambda^{-\vu (\ell-nk_1)}r$ neighbourhood of the {\em lateral} boundaries of the flow boxes, i.e., $\cup_{i,j}\partial B_{i,j}\setminus(O_{i,j}\cup T_{\tau_{i,j}}(O_{i,j}))$.
\end{enumerate}
Let us call $\partial_r^1(\cW_\ell)$ and $\partial_r^2(\cW_\ell)$, respectively, the parts of $\partial_r(\cW_\ell)$ that satisfy the above two conditions.
Clearly, $|\partial_r^1(\cW_\ell)|\leq \Cs\blambda^{-\vu\ell} \delta^{-1}r$. 

To analyze the second case, we must follow the creation of the $W_\alpha$. By the complexity assumption \ref{domin}, for each $\nu\in (0,1/2)$ we can chose $k_0\in\bN$ and $L_0>0$ so that for each curve $W$ in the stable direction and of size smaller than $L_0$, the number of smooth connected pieces of $T_{-\vus}W$ is smaller than $\blambda^{\vus}\nu$.

Remark that, by construction, there exists $c_1>1$ such that, for each $j\in\{0,\dots, k_1\}$ and $r'>0$, $T_{j\vu} \partial_{r'}(\cW_{\kappa k_1+j})\subset  T_{k_1\vu}\partial_{c_1r'}(\cW_{(\kappa+1) k_1})$. Accordingly, it suffices to study $\partial_r (\cW_{k_1\ell'})$, $\ell'\in\bN$. Let $\cWo_j=\cW_{k_1 j}$.

For each $W'\in\cWo_{\kappa}$ we say that $W''\in \cWo_{j}$, $j\leq \kappa$, is its {\em ancestor} if $W''$ is long and, for each $l\in\{0,\dots, \kappa-j-1\}$, $T_{l\vus}W'$ never belongs to a long element of $\cWo_{\kappa-l}$.
Now for each ancestor $W''\in\cWo_\kappa$, we can consider the short pieces that are generated\footnote{That is the set of short pieces in $\cWo_{\ell'}$ whose ancestor is the given curve.}  in $\cWo_{\ell'}$, $\ell'=\frac\ell{k_1}$, by the complexity bound their number is less 
than $\blambda^{\vus({\ell'}-\kappa)}\nu^{{\ell'}-\kappa}$. Note that the image of a curve of length $r$ in $W\in\cWo_{{\ell'}}$ under $T_{\vus({\ell'}-\kappa)}$ will be of length smaller than $\blambda^{-\vus({\ell'}-\kappa)}r$. Thus, the union of the images, call it $P_{{\ell'},\kappa,r},$ of the $r$-boundary of the short pieces of $\cWo_{{\ell'}}$ in an ancestor belonging to $\cWo_\kappa$ will have total length bounded by $\nu^{{\ell'}-\kappa}r$. By the usual distortion estimate (see footnote \ref{foo:dist}) this implies that, calling $m$ the induced Riemannian measure on $W^{s}_{\delta, \zeta}$,
\[
\frac{m(T_{\kappa\vus}P_{{\ell'},\kappa,r})}{m(T_{\kappa\vus}(W''))}\leq \Cs L_0^{{-1}}\nu^{{\ell'}-\kappa}r \, .
\]
Thus the measure of the image, in $W^{s}_{\delta, \zeta}$, of the $r$-boundary belonging to short pieces with an ancestor in $\cWo_{\kappa}$ will have measure bounded by $\Cs \delta r\nu^{{\ell'}-\kappa}$. Hence, the total measure of such pieces will be bounded by $\Cs \delta r$, while the number of pieces that do not have any ancestor must be less 
than $\lambda^{-\vus{\ell'}}\nu^{{\ell'}}$ and thus their total measure will be less than $\nu^{\ell'} r$. The first statement follows then by footnote \ref{foo:meas} choosing\footnote{The square root is for later convenience, see the proof of Lemmata \ref{lem:disco} and \ref{lem:discard}.}
\begin{equation}\label{eq:sigmachoice}
\sigma=\max\{\blambda^{-\vus},\nu\}^{\frac12}\, .
\end{equation}

The last statement follows by applying the previous results together with footnote \ref{foo:meas} again.
\end{proof}

Next, it is convenient to localise in space as well. To this end we need to define a sequence of smooth partitions
of unity. 

Given a parameter $\theta\in (0,1)$ to be chosen later, there exists $\Cs>0$ such that, for each $r\in(0,1)$, there exists a $C^{\infty}$ partition of unity $\{\phi_{r,i}\}_{i=1}^{q(r)}$   enjoying the following properties~\footnote{For example, using the function $\tilde p$ introduced to partition in time, one can define, in the charts $\kappa_\zeta$, $\tilde p(\bar kr^{\theta}/2+2r^{-\theta}\vu\eta)\tilde p(\bar j r^{\theta}/2+2r^{-\theta}\vu\xi)\tilde p(\bar i r^{\theta}/2+2 r^{\theta}/2 \vu s)$.}
\begin{enumerate}[\bf (i)]
\item for each $i\in\{1,\dots,q(r)\}$, there exists $x_i\in U_{\zeta_i,1}\subset M$ such that $\phi_{r,i}(z)=0$ for all $z\not\in B_{r^\theta}(x_i)$ (the ball, in the sup norm of the chart $\kappa_{\zeta_i}$, of radius $r^\theta$ centered at $x_i$);
\item for each $r,i$ we have $\|\nabla\phi_{r,i}\|_{L^{\infty}}\leq\Cs r^{-\theta}$;
\item  $q(r)\leq \Cs r^{-3\theta}$. 
\end{enumerate}

\begin{figure}[ht]\ 
\begin{tikzpicture}
\draw (0,-.5)--(.5,-.25) node[right=10pt]{$W^s_{\delta,\zeta}$}--(2,.5);
\draw[->] (3.5,0)--(4,0) node[below=2pt]{$T_{-\ell\vu}$}--(4.5,0);
\draw[dash pattern=on 2 pt  off 1 pt](6,-2)--(9,-2);    
\draw[dash pattern=on 2 pt  off 1 pt](9,-2)--(11,-1); 
\draw[dash pattern=on .5 pt  off 4 pt](8,-1)--(11,-1);
\draw[dash pattern=on .5 pt  off 4 pt](8,-1)--(6,-2);
\draw[dash pattern=on 2 pt  off 1 pt](6,2)--(9,2);    
\draw[dash pattern=on 2 pt  off 1 pt](9,2)--(11,3); 
\draw[dash pattern=on 2 pt  off 1 pt](8,3)--(11,3);
\draw[dash pattern=on 2 pt  off 1 pt](8,3)--(6,2);
\draw[dash pattern=on 2 pt  off 1 pt](6,-2)--(6,1.2)node[left, xshift=-6pt]{$B_{cr^\theta}(x_i)$}-- (6,2);    
\draw[dash pattern=on 2 pt  off 1 pt](9,-2)--(9,2); 
\draw[dash pattern=on .5 pt  off 4 pt](8,-1)--(8,3);
\draw[dash pattern=on 2 pt  off 1 pt](11,-1)--(11,3);
\draw[scale around={.7: (8.5,.5)},  dash pattern=on 1 pt  off 1 pt](6,-2)--(9,-2);    
\draw[scale around={.7: (8.5,.5)},  dash pattern=on 1 pt  off 1 pt](9,-2)--(11,-1); 
\draw[scale around={.7: (8.5,.5)},  dash pattern=on 1 pt  off 4 pt](8,-1)--(11,-1);
\draw[scale around={.7: (8.5,.5)},  dash pattern=on 1 pt  off 4 pt](8,-1)--(6,-2);
\draw[scale around={.7: (8.5,.5)},  dash pattern=on 1 pt  off 1 pt](6,2)--(9,2);    
\draw[scale around={.7: (8.5,.5)},  dash pattern=on 1 pt  off 1 pt](9,2)--(11,3); 
\draw[scale around={.7: (8.5,.5)}, dash pattern=on 1 pt  off 1 pt](8,3)--(11,3);
\draw[scale around={.7: (8.5,.5)}, dash pattern=on 1 pt  off 1 pt](8,3)--(6,2);
\draw[scale around={.7: (8.5,.5)}, dash pattern=on 1 pt  off 1 pt](6,-2)--(6,1.4)node[right]{$\scriptstyle B_{r^\theta}(x_i)$}-- (6,2);    
\draw[scale around={.7: (8.5,.5)}, dash pattern=on 1 pt  off 1 pt](9,-2)--(9,2);
\draw[scale around={.7: (8.5,.5)}, dash pattern=on 1 pt  off 4 pt](8,-1)--(8,3);
\draw[scale around={.7: (8.5,.5)}, dash pattern=on 1 pt  off 1 pt](11,-1)--(11,0.5) node[left, xshift=-1.5cm]{$\cdot\scriptstyle x_{i}$}--(11,3);
\draw (6.5,-.4).. controls (7.6, 0).. (8.5,.4) .. controls (9.40, .8) .. (10.5, 1.2);       
\draw[xshift=-.3cm, yshift=.6cm] (6,-.4).. controls (7.6, 0).. (8.5,.4) .. controls (9.40, .8) .. (11, 1); 
\draw (7.4,-.6).. controls (8, 0).. (8.5,.2) ; 
\draw (8.5, -.2)..controls (9,0)..(10, .8);  
\draw[->] (5,-2.5) node[below=1pt]{{\tiny discarded manifolds in $D_{\ell,i}$}}--(7.3,-.65);
\draw[->] (5,-2.5)--(8.3,-.3);
\draw[->] (6.5, -2.2)--(7,-2.2) node[below]{{\tiny \hskip 1cm unstable}}--(8.5,-2.2);
\draw[->](9.2,-2)node[below,rotate=25]{{\tiny \hskip 2cm stable}}--(10.8,-1.2);
\draw[->](11.1, -.5)node[below,rotate=90]{{\tiny \hskip 3cm flow}}--(11.1,2.5);
\draw[<-](8,3.2)--(9,3.2)node[xshift=10pt, yshift=3pt]{$\scriptstyle c r^{\theta}$};
\draw[->](9.7,3.2)--(11,3.2);
\draw[<-](6,2.2)--(6.5, 2.45)node[right, yshift=4pt, rotate=30]{$\scriptstyle c r^{\theta}$};
\draw[->](7,2.7)--(8,3.2);
\draw[<-](5.85,-2)--(5.85,-.5) node[right, xshift=-2pt, rotate=90] {$\scriptstyle cr^\theta$};
\draw[->](5.85,0.1)--(5.85,2);
\end{tikzpicture}
\caption{The manifolds $\{W_\alpha\}_{\alpha\in A_{\ell,i}}$.}
\label{fig:mani}
\end{figure}
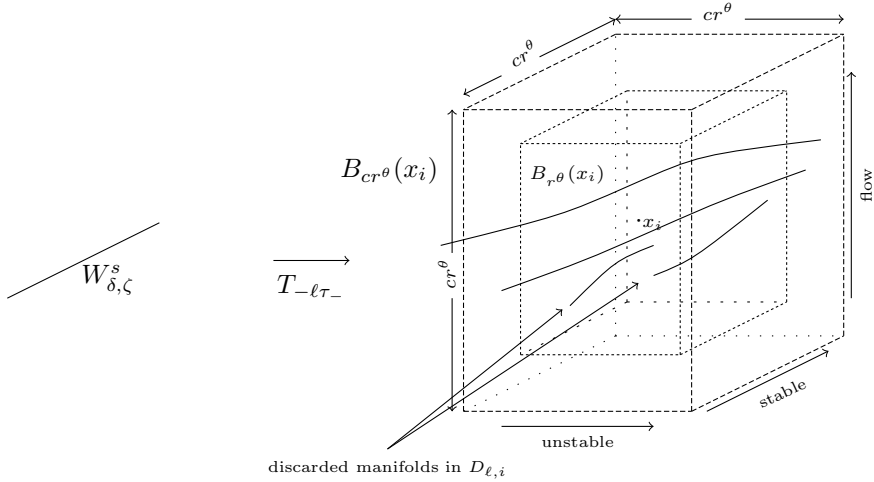

Fix $c>2$. For each $x_i$, let  $A_{\ell, i}=\{\alpha\in A_\ell\;:\; W_\alpha\cap B_{r^\theta}(x_i)\neq \emptyset\}$, $D_{\ell,i}=\{\alpha\in A_{\ell,i}\;:\;\partial(W_\alpha\cap B_{cr^\theta}(x_i))\not\subset\partial B_{cr^\theta}(x_i)\}$, $E_{\ell,i}=A_{\ell,i}\setminus D_{\ell,i}$.  Call the manifolds with index in $D_{\ell,i}$ the {\em discarded manifolds.} We choose $c$ large enough so that a manifold intersecting $B_{r^\theta}(x_{i})$ can intersect only the front and rear vertical part of the boundary of $B_{cr^\theta}(x_i)$, see figure \ref{fig:mani}.\footnote{This can be achieved thanks to the fact that the $W_{\alpha}$ belong to the stable cone.}

Set $W_{\alpha, i}=W_{\alpha}\cap B_{cr^\theta}(x_i)$ and
\begin{equation}\label{eq:ro-la1}
\frp_{\ell, \zeta, \alpha,i}=\frac{\tilde \frp_{\delta, \zeta}\circ T_{\ell \vu}\cdot J^s_{\ell\vu}}{Z_{\ell, \alpha,i}}\quad ;\quad Z_{\ell, \alpha,i}= \int_{W_{\alpha,i}}\tilde \frp_{\delta, \zeta}\circ T_{\ell \vu}\cdot J^s_{\ell\vu}.
\end{equation}
Moreover it is now natural to chose
\begin{equation}\label{eq:vu}
k_1=\lceil r^{-\theta} \tau_+\rceil,\, \text{ hence }\vu\cong r^\theta\, .
\end{equation}

Our next step is to estimate the contribution of the manifolds  $W_{\alpha,i}$, $\alpha\in D_{\ell,i}$.
Let $K_i\subset D_{\ell,i}$ be the collection of indices for which $W_{\alpha,i}\cap (\cup_j O_{j})\neq \emptyset$, then $\cup_{\alpha\in D_{\ell,i}\setminus K_i}W_{\alpha,i}\subset \partial_{cr^\theta}(\cW_\ell)$, hence, by Lemma \ref{lem:r-boundary} the total measure of the elements in $D_{\ell,i}\setminus K_i$, is bounded by $\Cs\max\{r^\theta, \sigma^{\frac\ell{k_1}} r^\theta\delta^{-1}\}$.\footnote{Remember that the $W_{\alpha,i}$ inherit from the $B_{cr^\theta}(x_i)$ the property of having a uniformly bounded number of overlaps.} It remains to estimate how many pieces are cut by a manifold $O_{i}$. This is
done in the following lemma whose proof can be found at the end of Appendix \ref{sec:hoihoi}.

\begin{lemma}\label{lem:disco} There exists $\ell_{0}\geq 0$ such that for each $\ell\geq k_1\ell_0$, any codimension-one disk~\footnote{Here we assume the curvature of the $\widetilde O$ to be bounded by some fixed constant.} $\widetilde O$ of measure $S>0$ and $\rho_{*}>0$, if we let
$D_{\ell}(\widetilde O,\rho_{*})=\{W_{\beta,i}\;:\; d(W_{\beta,i},\widetilde O)\leq \rho_{*}\}$ , then 
\[
\sum_{\{(\beta,i): W_{\beta,i}\in D_{\ell}(\widetilde O,\rho_{*})\}}Z_{\beta,i}\leq  \Cs\left[\sqrt{S(\rho_{*}+r^{\theta})}+S\sigma^{\ell}+\delta^{-1}\sigma^{\ell}\right] \, .
\]
\end{lemma}
Next, it is convenient to introduce the parameter $c_{*}\in(0,1)$ defined by
\begin{equation}\label{eq:h0}
c_{*}ea\vus=\sigma\, ,
\end{equation}
where $a=\Re(z)$, and assume that
$m$ is such that
\begin{equation}\label{eq:h1}
\sigma^{c_*m}\leq \delta r^{\frac\theta 2}.
\end{equation}
Applying Lemma \ref{lem:disco} with $\rho_{*}=r^{\theta}$, $\widetilde O= O_{l}$, we have that 
\[
\sum_i\sum_{\beta\in K_i}Z_{\beta,i}\leq \Cs r^{\frac\theta 2}\, .
\]
Thus,\footnote{ Note that if a box $B_{cr^\theta}(x_i)$ is cut by a manifold $O_i$ or by the boundary of $X_0$, then it is possible that $E_{\ell,i}=\emptyset$. In particular, the present estimate bounds the contributions of all the $W_{\alpha,i}$ for such a ``bad" box.}
\begin{equation}\label{eq:step03}
\left|\sum_{\ell\geq c_*k_1 m}\sum_{k,i}\sum_{\alpha\in D_{\ell,i}}\int_{-\vu}^{\vu}\!\!\!\!\! p_{m,\ell,z}(s)\int_{W_{\alpha,i}}\!\!\!\!\!\tilde \frp_{\delta, \zeta}\circ T_{\ell \vu}\cdot J^s_{\ell\vu}\cdot \cL_s f\right|\leq \Cs \frac{r^{\frac \theta 2}}{a^m}|f|_\infty\, .
\end{equation}

We are left with the small $\ell$ and the elements of $E_{\ell,i}$. To treat these cases, it is convenient to introduce  extra notation.  For each $\alpha\in A_{\ell}$, define $W^{c}_{\alpha}=\cup_{t\in[-\vu,\vu]}W_{\alpha}$. For each $W_{\alpha}\subset U_{\zeta', 0 }$, the manifold $\kappa_{\zeta'}(W_\alpha^{c})$ can be seen as the graph of $\bF_\alpha(\xi,s):=(F_\alpha(\xi),y_{\alpha}+\xi,N_\alpha(\xi)+s)$, $\xi,s\in\bR$, where $F_\alpha, N_\alpha$ are uniformly $C^{2}$ functions and $\bF_\alpha(\xi,0)$ is the graph of $W_\alpha$.

Given the above discussion,  to estimate the integrals in equation \eqref{eq:step0} it suffices to estimate the integrals
\begin{equation}\label{eq:step1}
\int_{-\vu}^{\vu}\!\!\!\!\! p_{m,\ell,z}(s)\int_{W_\alpha}\!\!\!\!\!\tilde \frp_{\delta, \zeta}\circ T_{\ell \vu}\cdot J^s_{\ell\vu}\cdot \cL_s f=\int_{W_\alpha^{c}}\!\!\!\!\! \bar p_{m,\ell,z,\alpha}\cdot\bar \frp_{\delta, \zeta}\circ T_{\ell \vu}\cdot \bar J^s_{\ell\vu}\cdot  f\, ,
\end{equation}
where $\bar p_{m,\ell,z,\alpha}\circ\kappa^{-1}_{\zeta'}\circ\bF_{\alpha}(\xi,s)=-p_{m,\ell,z}(-s)$,
$\bar \frp_{\delta, \zeta}\circ T_{\ell \vu}\circ \kappa^{-1}_{\zeta'}\circ\bF_{\alpha}(\xi,s)= \tilde \frp_{\delta, \zeta}\circ T_{\ell \vu}\circ\kappa^{-1}_{\zeta'}\circ\bF_{\alpha}(\xi,0)$, the same for $\bar J^{s}_{\ell,\nu}$ apart from a factor taking into account the speed of the flow, and the integrals are taken with respect to the volume form induced by the Riemannian metric.

Substituting \eqref{eq:step1} in \eqref{eq:step0}, setting $\bar \frp_{\ell, \zeta, \alpha,i}=\frac{\bar\frp_{\delta, \zeta}\circ T_{\ell \vu}\cdot \bar J^s_{\ell\vu}}{Z_{\ell, \alpha,i}}$ and remembering \eqref{eq:step03},\footnote{ The sum over the first $c_*m k_1$ elements is estimated by $|f|_\infty$ times the integral 
\[
\Cs\frac{a^{-m+1}}{(m-1)!}\int_0^{c_*m\vus}e^{-a x}(xa)^{m-1} dx\leq \Cs a^{-m}\frac{(c_*ma\vus)^{m-1}}{(m-1)!}\leq \Cs a^{-m}(c_*a e\vus)^m\, ,
\]
where we have used the Stirling formula.
}
\begin{equation}\label{eq:step3}
\begin{split}
\bA^s_\delta (\cR(z)^m (f))=&\sum_{\ell\in\bN^*}\sum_{\zeta,i}\sum_{\alpha\in A_{\ell,i}}Z_{\ell,\alpha,i}\bar\theta_{\zeta}\int_{W_{\alpha,i}^{c}} \bar p_{m,\ell,z,\alpha}(s)\phi_{r,i}\cdot\bar \frp_{\ell,\zeta,\alpha,i}\cdot f\\
=&\sum_{\ell\geq c_*mk_1}\sum_{\zeta,i}\sum_{\alpha\in E_{\ell,i}}Z_{\ell,\alpha,i}\bar\theta_{\zeta}\int_{W_{\alpha,i}^{c}} \bar p_{m,\ell,z, \alpha}(s) \phi_{r,i}\cdot\bar \frp_{\ell,\zeta, \alpha,i}\cdot  f\\
&\;+\cO(|f|_{\infty}([c_*ea\vus]^{m}+r^{\frac\theta 2})a^{-m}\, ,
\end{split}
\end{equation}
Note that \eqref{eq:h1} implies $[c_*ea\vus]^{m}=\sigma^{m}\leq \delta r^{\frac \theta 2}$.

To continue, following Dolgopyat, we must show that the sum over the manifolds in $E_{\ell,i}$ contains a lot of cancellations and this leads to the wanted estimate. In Dolgopyat scheme such cancellations take place when summing together manifolds that are at a distance larger than $r^{\vartheta}$, for some properly chosen $\vartheta>\theta$. To make the cancellations evident one must compare different leaves via the unstable holonomy (using the fact that it is $C^{1}$). In the present case, due to the discontinuities, the unstable holonomy is defined only on a Cantor set. 
To overcome this problem, we construct in Appendix \ref{sec:unstable} an approximate holonomy which is Lipschitz. Next (following \cite{Li}), we use the fact that the flow is contact to show that Lipschitz suffices to have the wanted cancellations, see Appendix \ref{sec:hoihoi}.  Unfortunately, the approximate holonomy is efficient only when its fibers are very short, in particular one cannot hope to use it effectively to compare leaves that are at a distance $r^{\theta}$. We need then to collect our weak leaves into groups that are at a distance smaller than $r$ and require that $\vartheta>1$.

To start with, we consider the line~\footnote{In the $\kappa_{\zeta_i}$ coordinates.} $x_{i}+(u,0,0)$, $u\in[-r^{\theta},r^{\theta}]$, and we partition it in intervals of length $r/3$. To each such interval $I$ we associate a point $x_{i,j}\in \cup_{\alpha\in E_{\ell,i}}W_{\alpha}^{c}\cap I$, if the intersection is not empty.  Next, we associate to each point $x_{i,j}$ Reeb coordinates $\tilde\kappa_{x_{i,j}}$. More precisely, let $x_{i,j}\in W_{\alpha}^{c}$, we ask that $x_{i,j}$ is at the origin in the $\tilde\kappa_{x_{i,j}}$ coordinates, that $\tilde\kappa_{x_{i,j}}(W_{\alpha}^{c})\subset \{(0,y,z)\}$, $y,z\in\bR$, and that the vector $D_{x_{i,j}}T_{-\ell \vu}(1,0,0)$ belongs to the unstable cone. Such changes of coordinates exist and are all uniformly smooth by Lemma \ref{lem:c2}. 

For each $x_{i,j}$, let us consider (in the coordinates $\tilde\kappa_{x_{i,j}}$)  the box $\cB_{r}=\{(\eta,\xi,s)\;:\; |\eta|\leq r,\,|\xi|\leq r^{\theta}, \, |s|\leq r^\theta\}$. We set $\cB_{r,i,j}=\tilde\kappa_{x_{i,j}}^{-1}(\cB_r)$. The next lemma ensures that Figure \ref{fig:smallbox} is an accurate representation of the manifolds intersecting $\cB_{r}$.
\begin{figure}[ht]\ 
\begin{tikzpicture}
\draw(6,-2)--(7.5,-2);    
\draw(7.5,-2)--(10,-1); 
\draw[dash pattern=on .5 pt  off 4 pt](8.5,-1)--(10,-1);
\draw[dash pattern=on .5 pt  off 4 pt](8.5,-1)--(6,-2);
\draw(6,2)--(7.5,2);    
\draw(7.5,2)--(10,3); 
\draw(8.5,3)--(10,3);
\draw(8.5,3)--(6,2);
\draw(6,-2)--(6,1.2)node[left, xshift=-2pt, yshift=12pt]{$\cB_{cr}$}-- (6,2);    
\draw(7.5,-2)--(7.5,2); 
\draw[dash pattern=on .5 pt  off 4 pt](8.5,-1)--(8.5,3);
\draw(10,-1)--(10,3);
\draw[<-](7.5,-2.1)--(9,-1.5)node[xshift=8pt, yshift=3pt, rotate=15]{$\scriptstyle c r^{\theta}$};
\draw[->](9.5,-1.3)--(10,-1.1);
\draw[<-](6,-2.1)--(6.5,-2.1)node[xshift=8pt, yshift=-1pt]{$\scriptstyle c r$};
\draw[->](7,-2.1)--(7.5,-2.1);
\pgfsetfillopacity{0.4}
\path[fill=gray](7,-1.5)..controls (8, -.9) ..(9.5, -.5)..controls (9.4, 1.5) ..(9.5,2.5)..controls(8,2).. (7,1.5)..controls (7,0).. (7,-1.5);
\pgfsetfillopacity{1}
\draw[dash pattern=on 2 pt  off 1 pt](7,-1.5)..controls (8, -.9)..(9.5, -.5); 
\draw[dash pattern=on 2 pt  off 1 pt] (9.5,2.5)..controls(8,2).. (7,1.5);
\draw(6.38,-1.95)..controls (6.8,-1.63)..(7,-1.5);
\draw(6.38,1.2)..controls (6.6,1.3).. (7,1.5);
\draw[dash pattern=on 2 pt  off 1 pt](9.5,2.5)--(10,2.67);
\draw(10,2.67)--(10.3,2.78);
\draw[dash pattern=on 2 pt  off 1 pt](9.5,-.5)--(10,-.4);
\draw(10,-.4)--(10.39,-.3);
\draw[scale around={1.6: (-1.5,.5)}, xshift=-3.575cm](9.5, 0)..controls (9.42, 1) ..(9.45,1.933)node[right,yshift=-.5cm]{$\tilde\kappa_{x_{i,j}}(W_{\alpha}^c)$};
\draw[scale around={1.6: (-1.5,.5)}, xshift=-3.575cm] (7,.94)..controls (7,0).. (7,-1.033);
\draw[->](5,-.5)node[left,yshift=-.2cm]{$\scriptstyle \tilde\kappa_{x_{i,j}}(W_{\alpha}^c)\cap \cB_{cr}$}--(7,.5);
\draw[<-] (5.9,-2)--(5.9,0.5)node[right, rotate=90, yshift=.1cm]{$\scriptstyle cr^\theta$};
\draw[->](5.9, 1)--(5.9,2);
\end{tikzpicture}
\caption{A manifold intersecting  $\cB_r$.}
\label{fig:smallbox}
\end{figure}

\begin{lemma}\label{lem:smallbox}
There exists $c>0$ such that, if a manifold $W_{\alpha}$, $\alpha\in E_{\ell,i}$, intersects $\cB_{r,i,j}$, then $\tilde\kappa_{x_{i,j}}(W_{\alpha}^c)\cap\partial \cB_{cr,i,j}$ is contained in the unstable boundary of $\cB_{cr,i,j}$, provided $\theta\in(\frac 12,1)$ and $\bar\lambda^{-c_{*}m}<r^{\frac{1-\theta}2}$.
\end{lemma}
\begin{proof}
By construction, for each $\cB_{r,i,j}$ there is a manifold $W_{\alpha}^c$ going through  its center and perpendicular to $(1,0,0)$ (in the $\tilde\kappa_{x_{i,j}}$ coordinates). Let $\tau_\alpha\in [-\vu,\vu]$ be such that $\tilde W_\alpha=T_{\tau_\alpha}W_\alpha\cap I\neq \emptyset$.\footnote{ When no confusion arises, to ease notation, we will identify $\tilde\kappa_{x_{i,j}}(W_{\alpha}^c)$ and $W_{\alpha}^c$.}

Next, consider another $W_{\beta}^{c}$, $\beta\in E_{\ell,i}$, intersecting  $\cB_{r,i,j}$ and let $u$ be the intersection point with $I$. Again, let $\tilde W_\beta=T_{\tau_\beta}W_\beta$. Consider the segment $J=[0,u]$ and its trajectory $T_{t} J$, $t\in[0,\ell\vu]$.  Let $t_{0}$ be the first time $t$ for which $T_{t-\vu}J\cap (\cup_{i}\partial O_{i})\neq \emptyset$. Since $T_{t_{0}}\tilde W_{\alpha}, T_{t_{0}}\tilde W_{\beta}$ are uniformly transversal to $\cup_{t\in[0,2\vu]}T_t\partial O_{i}=\partial^c O_i$ and do not intersect it, it follows that their length must be smaller than $\Cs|T_{t_{0}}J|$, see figure \ref{fig:cut}.

\begin{figure}[ht]\ 
\begin{tikzpicture}
\draw (0,-1)--(0,0) node[right]{$T_{t_0}(J)$}--(0,1);
\draw[line width=1pt] (-2,-2.2).. controls (0, .4).. (2,2)node[left, xshift=-.2cm, yshift=-1cm]{$\partial^c O_i$} ; 
\draw(-.8,1.1)node[yshift=.2cm]{$\scriptscriptstyle T_{t_0}(\tilde W_\beta)$}..controls (-.5,1.2)..(0,1)..controls (.3, .88)..(.58,.93);
\draw(-.8,-1.1)..controls (-.5,-1.2)..(0,-1)..controls (.45,- .88)..(.75,-1.05)node[xshift=.2cm, yshift=-.2cm]{$\scriptscriptstyle T_{t_0}(\tilde W_\alpha)$};
\end{tikzpicture}
\caption{Meeting $\partial O_i$.}
\label{fig:cut}
\end{figure}
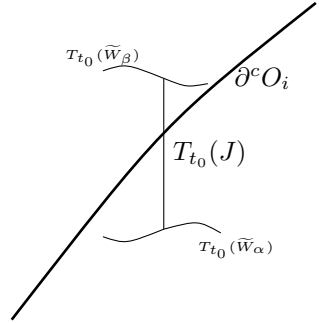

Hence, setting $\Lambda=\frac{|\tilde W_{\alpha,i}|}{|T_{t_0}\tilde W_{\alpha,i}|}$, by the invariance of the contact form it follows that
\[
\Cs^{-1}\Lambda\leq \frac {|T_{t_0}(J)|}{|J|}\leq \Cs\Lambda\, .
\]
Thus,
\begin{equation}\label{eq:cutcut}
cr^\theta\leq \Cs \Lambda |T_{t_0}\tilde W_{\alpha,i}|\leq\Cs \Lambda |T_{t_0}(J)|\leq \Cs \Lambda^2|J|\leq \Cs\Lambda^2 r\, .
\end{equation}
It follows that the forward dynamics in a neighbourhood of $J$ behaves like in the smooth case until it experiences a hyperbolicity $\Lambda$ of order at least $\Cs r^{-\frac{1-\theta}2}$. Moreover our conditions imply that this amount of hyperbolicity will be achieved in the time we are considering. In turn, this means that the tangent spaces of $\tilde W_{\alpha}$ and $\tilde W_{\beta}$, at zero and $u$ respectively, differ at most of $\Cs r^{1-\theta}$.\footnote{Since hyperbolicity $\Lambda$ implies that the image of the cone (which contains the tangent to the manifolds) has size, in the horizontal plane, $\Lambda^2$ while the axes of two cones at a distance $r$ can differ at most by $\Cs r$, which can be proved in the same manner in which is proven the $C^1$ regularity of the foliation in the  case of a smooth map, see \cite{KH}.} Thus the tangents to the manifolds at the above points can at most be at a distance $\Cs r$ inside the box. By the uniform $C^{2}$ bounds of the manifolds it follows that the distance between the two manifolds must be bounded by $\Cs (r+r^{2\theta})$. From this, Lemma~\ref{lem:smallbox} easily follows by choosing $c$ large enough.
\end{proof}
\begin{remark} Note that the estimates on the tangent plane to the manifolds contained in the previous lemma imply that the ``angle" between two nearby boxes $\cB_{r,i,j}$ is of the order  $r^{1-\theta}$, hence the maximal distance in the unstable direction is of order $r$. This implies that the covering $\{\cB_{r,i,j}\}$ has a uniformly bounded number of overlaps.
\end{remark}
Returning to the proof of Lemma~\ref{dolgolemma},
in view of the previous result it is convenient decompose $E_{\ell,i}$ as $\cup_{j}E_{\ell,i,j}$ where if $\alpha\in E_{\ell,i,j}$, then $\tilde\kappa_{x_{i,j}}(W_{\alpha,i})\cap \cB_{r}\neq\emptyset$. We can then rewrite \eqref{eq:step3} as
\begin{equation}\label{eq:step3.5}
\begin{split}
\bA^s_\delta (\cR(z)^m( f))
=&\sum_{\ell\geq c_*mk_1}\sum_{\zeta,i,j}\sum_{\alpha\in E_{\ell,i,j}}Z_{\ell,\alpha,i}\bar\theta_{\zeta}\int_{W_{\alpha,i}^{c}} \bar p_{m,\ell,z,\alpha}(s) \phi_{r,i}\cdot\bar \frp_{\ell,\zeta, \alpha,i}\cdot  f\\
&\;+\cO(|f|_{\infty}r^{\frac \theta 2})a^{-m}\, ,
\end{split}
\end{equation}

To elucidate the cancellation mechanism it is best to fix $\ell, i,j$ and use the above mentioned charts (in fact from now on, we will call the quantities in such a coordinate charts with the same names of the corresponding ones in the manifold). For each $\alpha\in E_{\ell,i,j}$, $W_\alpha\cap \cB_{cr}$ can be seen as the graph of $\bF_\alpha(\xi):=(F_\alpha(\xi),\xi,N_\alpha(\xi))$ for $\|\xi\|\leq cr$, where $F_\alpha, N_\alpha$ are uniformly $C^{2}$ functions, and, by Lemma \ref{lem:smallbox}, $|F_{\alpha}'|\leq \Cs r^{1-\theta}$. Note that, since both $\alpha$ and $d\alpha$ are invariant under the flow, and the manifolds are the image of manifolds with tangent space in the kernel of both forms, it follows that $N_\alpha'(\xi)=\xi F_\alpha'(\xi)$. 

To simplify notations, let us introduce the functions\footnote{To ease notation we suppress some indices.}
\begin{equation}\label{eq:defF}
\begin{split}
&\tilde \bF_\alpha(\xi,s)=(F_\alpha(\xi),\xi,N_\alpha(\xi)-s)\\
&\Xi_{\ell, r,i,\zeta,\alpha}=\phi_{r,i}\cdot\bar \frp_{\ell,\zeta,\alpha,i}\\
&{\bf F}_{\ell,m,i,\alpha}(\xi,s)= p(s)\frac{(\ell\vu-s)^{m-1}}{(m-1)!}e^{-z\ell\vu+as}\cdot \Xi_{\ell, r,i,\zeta,\alpha}\circ\bF_\alpha(\xi,s) \cdot\Omega_\alpha(\xi)\, ,
\end{split}
\end{equation}
where $\Omega_\alpha ds\wedge d\xi$ is the volume form on $W_\alpha^{c}$ in the coordinates determined by $\bF_\alpha$. Note that\footnote{By tracing the definition of $\bar\frp_{\delta, \zeta,i}$ just after \eqref{eq:step1}, of $Z_{\ell,\alpha,i}$ and  $\frp_{\ell,\zeta,\alpha,i}$ in \eqref{eq:ro-la1}, and of $\tilde\frp_{\delta,\zeta}$ after \eqref{eq:step0}, it follows that the only large contribution to the Lipschitz norm comes from $\phi_{r,i}$ and $p$. In particular \eqref{eq:ro-la1} implies that $|\frp_{\ell,\zeta,\alpha,i}|\leq\Cs r^{-\theta}$, the other estimate follows then by  property (ii) of the partition.}
\begin{equation}\label{eq:Xi}
\begin{split}
&|{\bf F}_{\ell,m,i,\alpha}|_{\infty}\leq \Cs r^{-\theta}\frac{(\ell\vu)^{m-1}}{(m-1)!}e^{-a\ell\vu}\\
&\|{\bf F}_{\ell,m,i,\alpha}\|_{\text{Lip}}\leq \Cs r^{-2\theta}\frac{(\ell\vu)^{m-1}}{(m-1)!}e^{-a\ell\vu}\, .
\end{split}
\end{equation}
We can then write,
\begin{equation}\label{eq:step4}
\begin{split}
&\int_{W_{\alpha,i}^{c}}  \bar p_{m,\ell,z,\alpha}\phi_{r,i}\cdot\bar \frp_{\ell,\zeta,\alpha,i}\cdot f\\
&=\int_{-\vu}^{\vu}ds\int_{\|\xi\|\leq cr^\theta} d\xi\; e^{ibs}{\bf F}_{\ell,m,i,\alpha}(\xi,s) f(F_\alpha(\xi),\xi,N_\alpha(\xi)+s)\, .
\end{split}
\end{equation}
At this point, we would like to compare different manifolds by sliding them along an approximate unstable direction.
To this end, we use the approximate unstable fibers $\Gamma_{i,j, r}^\up$ constructed in Appendix \ref{sec:unstable}. In short, for each coordinates $\tilde\kappa_{x_{i,j}}$, we can construct a Lipschitz foliation in $\cB_{cr}$ in a $\rho=r^\varsigma$ neighbourhood  of the ``stable" fiber (which is of length $\varrho=r^\theta$). In order to have the foliation defined in all $\cB_{cr}$ we need $\varsigma<1$, while for the foliation to have large part where it can be smoothly iterated backward as needed it is necessary that $\varsigma>\theta$, we thus impose
\begin{equation}\label{eq:varsigma}
\theta<\varsigma<1.
\end{equation}
The foliation can be locally described by a coordinate change $\bG_{i,j,\up}(\eta,\xi,s)=(\eta,G_{i,j,\up}(\eta,\xi), H_{i,j,\up}(\eta,\xi)+s)$ so that the fiber $\Gamma_{i,j,r}^{\up}(\xi,s)$ is the graph of $\bG_{i,j,\up}(\cdot, \xi,s)$. Note that, by construction, $\|G_{i,j,\up}'\|$ is small. We consider the holonomy $\Theta_{i,j,\alpha, \up}: W_\alpha\to W_*=\{x^u=0\}$ defined by $\{z\}=\Gamma_{i,j,r}^\up(\Theta_{i,j,\alpha,\up}(\{z\}))\cap W_\alpha$. Note that  
\[
\Theta_{i,j,\alpha,\up}(F_\alpha(\xi), \xi,N_\alpha(\xi))=(0,h_\alpha(\xi), \bar\omega_\alpha(\xi)).
\]
Accordingly, $\bG_{i,j,\up}(F_\alpha(\xi),h_\alpha(\xi),\bar\omega_\alpha(\xi))=\bF_\alpha(\xi)$, that is
\begin{equation}\label{eq:holo}
\begin{split}
&G_{i,j,\up}(F_\alpha(\xi),h_\alpha(\xi))=\xi\\
&H_{i,j,\up}(F_\alpha(\xi),h_\alpha(\xi))+\bar\omega_\alpha(\xi)=N_\alpha(\xi)\, .
\end{split}
\end{equation}

\begin{lemma}\label{lem:hpr}
There exists $\Cs,\ho_0>0$ such that for each $i,j, \ell$, $\ho\in [0,\ho_0]$ and $\alpha\in E_{\ell,i,j}$ the following holds true
\[
\begin{split}
&|h_\alpha(\xi)-\xi|+|h_\alpha^{-1}(\xi)-\xi|\leq \Cs r^{1-\varsigma}\;;\quad\left|1- h_\alpha'\right|\leq \Cs r^{1-\varsigma}\\
& |\bar\omega_{\alpha}|_{C^{1+\ho}}+ | h_{\alpha}|_{C^{1+\ho}}\leq \Cs\, , 
\end{split}
\]
provided 
\begin{equation}\label{eq:varvar}
\varsigma(1+\ho)\leq 1\, .
\end{equation}
\end{lemma}
\begin{proof}
By Lemma \ref{lem:unstable-fol} and Remark \ref{rem:lip}, both $h_\alpha, \bar\omega_\alpha$ are uniformly Lipschitz functions. Indeed, 
\begin{equation}\label{eq:h-a}
\begin{split}
\xi= &G_{i,j,\up}(0,h_\alpha(\xi))+\int_0^{F_\alpha(\xi)}\!\!\!\!\!\!\!\!dz\;\partial_zG_{i,j,\up}(z,h_\alpha(\xi))\\
=&
h_\alpha(\xi)+\int_0^{F_\alpha(\xi)}\!\!\!\!\!\!\!\! dz\int_0^{h_\alpha(\xi)}\!\!\!\!\!\!\!\!dw\;\partial_w\partial_zG_{i,j,\up}(z,w)
\end{split}
\end{equation}
that is $h_\alpha(\xi)=\xi(1+\cO(r^{1-\varsigma}))$.
Moreover, differentiating the first of \eqref{eq:holo},
\begin{equation}\label{eq:hprimo}
h'_\alpha(\xi)=\frac{1-\partial_\eta G_{i,j,\up}(F_\alpha(\xi),h_\alpha(\xi))F'_\alpha(\xi)}{\partial_\xi G_{i,j,\up}(F_\alpha(\xi),h_\alpha(\xi))},
\end{equation}
and 
\[
\begin{split}
\partial_\xi G_{i,j,\up}(F_\alpha(\xi),h_\alpha(\xi))&=\partial_\xi G_{i,j,\up}(0,h_\alpha(\xi))+\int_0^{F_\alpha(\xi)}\!\!\!\!\!\!\! \partial_z\partial_\xi G_{i,j,\up}(z,h_\alpha(\xi)) dz\\
&=1+\cO(r^{1-\varsigma})\\
\partial_\eta G_{i,j,\up}(F_\alpha(\xi),h_\alpha(\xi))&=\partial_\eta G_{i,j,\up}(F_\alpha(\xi),0)+\int_{0}^{h_\alpha(\xi)}\!\!\!\!\!\!\!\partial_z\partial_\eta G_{i,j,\up}(F_\alpha(\xi),z) dz\\
&=\partial_\eta G_{i,j,\up}(F_\alpha(\xi),0)+\cO(r^{\theta-\varsigma})\, ,
\end{split}
\]
where we have used property (1) of the foliation representation (see Appendix \ref{sec:unstable}). Since $\partial_\eta \bG_{i,j,\up}$ belongs to the unstable cone, then $\partial_\eta G_{i,j,\up}$ is uniformly bounded. Also remember the estimate $|F'|\leq \Cs r^{1-\theta}$ obtained in the proof of Lemma \ref{lem:smallbox}.
Hence, taking into account \eqref{eq:varsigma}, $|1- h_\alpha'|\leq \Cs r^{1-\varsigma}$, and $h_\alpha$ is invertible with uniform Lipschitz constant.

To estimate the H\"older norm of $h'$ we use the above equations together with property \eqref{it:6} of the foliation:\footnote{We use also the fact that $|f|_{C^{\ho}}\leq\Cs |f|_{C^{0}}^{1-\ho}|f'|_{C^{0}}^{\ho}$.}
\[
\begin{split}
|h_{\alpha}|_{C^{1+\ho}}&\leq \Cs\left\{1+|F_{\alpha}'|_{\infty}|\partial_{\eta}G_{i,j,\up}|_{C^{\ho}}+|1-\partial_{\xi}G_{i,j,\up}|_{C^{\ho}}\right\}\\
&\leq \Cs\left\{1+r^{1-\theta}|\partial_{\xi}\partial_{\eta} G_{i,j,\up}|_{\infty}^{\ho}+\int_{0}^{r}|\partial_{\xi}\partial_{\eta}G_{i,j,\up}|_{C^{\ho}}\right\}\\
&\leq \Cs\left\{1+r^{1-\varsigma(1+\ho)}\right\}.
\end{split}
\]
Analogously, 
\[
|\bar\omega_{\alpha}|_{C^{1+\ho}}\leq C|h_{\alpha}|_{C^{1+\ho}}+|F_{\alpha}'|_{\infty}|\partial_{\eta}H_{i,j,\up}|_{C^{\ho}}+|\partial_{\xi}H_{i,j,\up}|_{C^{\ho}}\leq \Cs\left\{1+r^{1-\varsigma(1+\ho)}\right\}\, .
\]
This concludes the proof of Lemma \ref{lem:hpr}.
\end{proof}

Next, remember that the fibers of $\Gamma_{i,j,r}^\up$ in the domain $\Delta_\up$ can be iterated backward a time $\up\vu$ and still remain in the unstable cone. In the following we will use the notation
\[
\partial_{\up,i}f=\sup_{(\xi,s)\in \Delta_\up}\esssup_{|\eta|\leq r}|\langle\partial_\eta\bG_{\up}(\eta,\xi,s), (\nabla f)\circ \bG_{\up}(\eta,\xi,s)\rangle|\, .
\]

With the above construction and notations, and using Lemma \ref{lem:unstable-fol}, we can continue our estimate left at \eqref{eq:step4}
\begin{equation}\label{eq:step5}
\begin{split}
&\int_{W_{\alpha,i}^{c}} \bar p_{m\ell,z,\alpha} \phi_{r,i}\cdot\bar \frp_{\ell,\zeta,\alpha,i}\cdot  f\\
&=\int_{-\vu}^{\vu}ds\int_{\|\xi\|\leq cr^\theta} d\xi\; e^{ibs}{\bf F}_{\ell,m,i,\alpha}(\xi,s) f\circ\Theta_{i,j,\alpha,\up}(F_\alpha(\xi), \xi,N_\alpha(\xi)+s)\\
&\quad+ \cO(r\partial_{\up,i}f +r^{\varsigma}|f|_\infty)
\frac{(\ell\vu)^{m-1}}{(m-1)!}e^{-a\vu\ell}\\
&=-\int_{-2\vu}^{2\vu}ds\int_{\|\xi\|\leq cr^\theta} d\xi\; e^{ib(\omega_\alpha(\xi)-s)}{\bf F}^*_{\ell,m,i,\alpha}(\xi,s) f(0, \xi,s)\\
&\quad+ \cO(r^{1-\theta}\partial_{\up,i}f +r^{\varsigma-\theta}|f|_\infty)
\frac{(\ell\vu)^{m-1}}{(m-1)!}e^{-a\vu\ell}\vu\, ,
\end{split}
\end{equation}
where we have set $\omega_\alpha(\xi)=\bar\omega_\alpha\circ h_\alpha^{-1}(\xi)$,
\begin{equation}\label{eq:defFs}
{\bf F}^*_{\ell,m,i,\alpha}(\xi,s)={\bf F}_{\ell,m,i,\alpha}(h_\alpha^{-1}(\xi),\omega_\alpha(\xi)-s) |h'_\alpha\circ h_\alpha^{-1}(\xi)|^{-1}\, .
\end{equation}

At last we can substitute \eqref{eq:step5} into \eqref{eq:step3.5} and use the Schwartz inequality (first with respect to the integrals and then with respect to the sum on $i$ and $j$) to obtain
\begin{equation}\label{eq:almost}
\begin{split}
|\bA^s_\delta &\cR(z)^mf|\leq \Cs\sum_{\ell\geq c_*mk_1}e^{-a\ell\vu}\sum_{\zeta, i,j} |f|_\infty r^{\theta}\\
&\times\bigg[\sum_{\alpha,\beta\in E_{\ell,i,j}}\!\!\!\!Z_{\alpha,i} Z_{\beta,i}\!\!\int_{-2\vu}^{2\vu}\!\!\!\!\!\!\!ds\int_{\|\xi\|\leq cr^\theta}\!\!\!\!\!\!\!\!\!\!\!\!\!\!\!\!\!d\xi \, e^{ib(\omega_\alpha(\xi)-\omega_\beta(\xi))}{\bf F}^*_{\ell,m,i,\alpha}\overline{{\bf F}^*_{\ell,m,i,\beta}}\bigg]^{\frac 12}\\
&+\Cs(r^{1-\theta}\sup_i \partial_{\up,i}f+(r^{\frac\theta 2}+r^{\varsigma-\theta})|f|_\infty)a^{-m}\\
&\quad\quad\quad\leq \Cs\sum_{\ell\geq c_*mk_1}e^{-a\ell\vu}\sum_{\zeta} |f|_\infty r^{-\frac 1 2}\\
&\times\bigg[\sum_{i,j}\sum_{\alpha,\beta\in E_{\ell,i,j}}\!\!\!\!Z_{\alpha,i} Z_{\beta,i}\!\!\int_{-2\vu}^{2\vu}\!\!\!\!\!\!\!ds\int_{\|\xi\|\leq cr^\theta}\!\!\!\!\!\!\!\!\!\!\!\!\!\!\!\!\!d\xi \, e^{ib(\omega_\alpha(\xi)-\omega_\beta(\xi))}{\bf F}^*_{\ell,m,i,\alpha}\overline{{\bf F}^*_{\ell,m,i,\beta}}\bigg]^{\frac 12}\\
&+C(r^{1-\theta}\sup_i \partial_{\up,i}f+(r^{\frac\theta 2}+r^{\varsigma-\theta})|f|_\infty)a^{-m}\, .
\end{split}
\end{equation}
To conclude the proof
of Lemma~\ref{dolgolemma}, we need two fundamental, but technical, results whose proofs can be found in Appendix \ref{sec:hoihoi}.
The first allows to estimate the contribution to the sum of manifolds that are enough far apart, the other shows that manifolds that are too close are few. For each two sets $A,B$ such that $A_{r,i,j}=A\cap \cB_{r,i,j}\neq \emptyset$ and $B_{r,i,j}=B\cap \cB_{r,i,j}\neq \emptyset$, let $d_{i,j}(A,B)=d(A_{r,i,j},B_{r,i,j})$.
\begin{lemma}\label{lem:cancel} We have
\begin{equation}\label{eq:below-o}
|\partial_\xi\left[\omega_\alpha-\omega_\beta\right]|\geq \Cs d_{i,j}(W_{\alpha},W_{\beta}) \, .
\end{equation}
In addition, if $\sigma^{c_*m}\leq \Cs r^{\frac{\vartheta-\theta}2}$, then
\begin{equation}\label{eq:up-o}
\begin{split}
\bigg|\int_{\|\xi\|\leq cr^\theta}\!\!\!\!\!\!\!\!\!\!\!\!d\xi\,&e^{ib(\omega_\alpha(\xi)-\omega_\beta(\xi))}{\bf F}^*_{\ell,m,i,\alpha}\overline{{\bf F}^*_{\ell,m,i,\beta}}\,\bigg|_{\infty}\!\!\!\leq \Cs\frac{(\ell\vu)^{2m-2}}{[(m-1)!]^2}\\
&\times r^{-\theta}\left[\frac1{d_{i,j}(W_\alpha,W_\beta)^{1+\ho} b^\ho}+\frac1{r^{\theta}d_{i,j}(W_\alpha,W_\beta) b}\right]\, .
\end{split}
\end{equation}
\end{lemma}

\begin{lemma}\label{lem:discard} There exists $\ell_{0}\in \bN$ such that for all $\ell\geq\ell_{0}k_1$,  $\vartheta>0$, $\alpha\in E_{\ell, i}$ and each $r>0$
\[
\sum_{\beta\in D^\vartheta_{\ell,i,j,\alpha}}Z_{\beta,i}\leq \Cs [r^{\frac{\vartheta+\theta}2}+\delta^{-1}\sigma^{\frac{\ell}{k_1}}]\, ,
\]
where $D^\vartheta_{\ell,i,j,\alpha}=\{\beta\in E_{\ell,i,j}\;:\; d(W_\alpha, W_\beta)\leq r^\vartheta\}$. 
\end{lemma}

To end the proof of Lemma~\ref{dolgolemma}, it is then convenient to assume that 
\begin{equation}\label{eq:h3}
\sigma^{c_{*}m}\leq r^{\frac{\vartheta+\theta}2}\delta\, .
\end{equation}
Applying Lemmata \ref{lem:cancel} and \ref{lem:discard} to \eqref{eq:almost} with $\up=m$ and $f=\cR(z)^m \tilde\psi$, we obtain\footnote{Recall that $\|f\|_{L^\infty}\leq a^{-m}\|\tilde\psi\|_{L^\infty}$ while $\partial_{m,i}f\leq \Cs (a+\ln\blambda)^{-m}\|\tilde\psi\|_{H^1_\infty}$ since $H^1_\infty$ functions are (or, better, have a representative) Lipschitz (see \cite[Section 4.2.3, Theorem 5]{EG}), with Lipschitz constant given by $\|\tilde\psi\|_{H^1_\infty}$. Lipschitz functions in $\bR^{3}$ are Lipschitz when restricted to a $C^{2}$ curve and Lipschitz functions are almost surely differentiable by Rademacher's Theorem.}
\[
\begin{split}
&|\bA^s_\delta \cR(z)^{2m}\tilde\psi|\\
&\qquad \leq \Cs\sum_{\ell\geq c_*m}|\tilde\psi|_{L^\infty}\frac{a^{-2m}e^{-a\ell\vu}(\ell\vu)^{m-1}}{(m-1)!\;r^{\frac 1 2}} \bigg[\frac{r^{-\vartheta(1+\ho)-\theta}}{b^{\ho}}+\frac{r^{-2\theta-\vartheta}}{b}+r^{\frac{\vartheta+\theta}2}\bigg]^{\frac 12}\\
&\qquad\qquad+\Cs(r^{1-\theta}\nu_{a}^{m}\|\tilde\psi\|_{H^1_\infty}+(r^{\frac\theta 2}+r^{\varsigma-\theta})|\tilde\psi|_{L^\infty})a^{-2m}\\
& \qquad\leq  \Cs a^{-2m}\left(\frac{r^{-\frac {1+(1+\ho)\vartheta+4\theta} 2}}{b^{\frac \ho 2}}+r^{\frac{\vartheta-3\theta-2}4}+r^{\varsigma-\theta}+r^{\frac\theta 2}\right)|\tilde\psi|_{L^\infty}\\
&\qquad\qquad+\Cs a^{-2m}r^{1-\theta}\nu_{a}^{m}\|\tilde\psi\|_{H^1_\infty}\, ,
\end{split}
\]
where $\nu_{a}=(1+a^{-1}\ln\blambda)^{-1}$.
For the reader's convenience, we collect all the conditions imposed along the computation
\[
c_{*}ea\vus=\sigma <1 \;;\;\; \sigma^{c_*m}\leq r^{\frac{\vartheta+\theta}2}\delta \;;\;\;\frac 12 <\theta<\varsigma<(1+\ho)^{-1} \, .
\]
We choose $\varsigma=\frac{3\theta}2$, $\vartheta=2+5\theta$, $b=r^{-\frac{4+8\vartheta+18\theta}{\ho}}$,
$\gamma_0=\frac{\ho\theta}{2(4+8\vartheta+18\theta)}$. Finally, if we choose $\theta=\frac 58$, $\ho<\frac 1{15}$ and $m\geq \Cs a\gamma \ln b$, for some appropriate fixed constant $\Cs$, we satisfy all the remaining conditions and  Lemma~\ref{dolgolemma} follows.
(Note that the best possible $\gamma_0$  given by this argument is
$\frac{\ho}{155}\leq \frac 1{2325}$.)
\qed

\section{Finishing the proof}
\label{dodo}

In this section, we prove Proposition~ \ref{dolgo}, using the
key bound from Lemma~\ref{dolgolemma}.

\begin{proof}[Proof of Proposition \ref{dolgo}]
Recall that $\|\cdot \|=\|\cdot\|_{\widetilde \HHH_p^{r,s,0}(R)}$
with $-1+1/p<s<-r<0<r<1/p$ for $p \in (0,1)$.
It will be convenient to assume
$$
p >3\, , |s|\le 2r\, .
$$
Let $A$  be as in \eqref{reallybounded'} and
\eqref{needit}, and let $\gamma_0>0$ be given
by Lemma~\ref{dolgolemma}. (They do not depend
on $r$, $s$, or $p$.)
Let us consider $|b|>b_0$, $a> 10 A$,
and decompose $n=Lm+m+2$, with $m$ even for even $n$, or
$n=Lm+m+3$, with $m$ even for odd $n$
(for simplicity, we consider even $n$), where $L\ge 1$  and $b_0$ will be chosen
later. We will collect at the end of the argument the upper  and lower bounds relating
$m$ and $a \ln |b|$ which will have appeared along the way, and check that they are consistent.

Take $r'$ with $\max(-1+1/p, r-\beta)<r' <s< 0$ and so that 
$r' > -\frac{\gamma_0}{2p}$.
We view $r'$  as fixed, but we
may  need to choose larger $p$ and smaller $r$, $|s|$,
this may affect some constants noted $\Cs$, but will not create
any problems, since we may take larger $b_0$  (without
affecting $\Cs$). 

Our starting point is 
the Lasota-Yorke estimate \eqref{controlCeta''} from  Lemma~\ref{controlCeta}:  For any integer $N\ge 1$  such that
$(1+3N)r<1/p$ (take the largest such integer), setting
$$ \eta :=\Lambda|s|+\frac{A}{N}\, , 
$$
for some constant $\Lambda$ independent of  $p$, $r$, $s$ and $r'$,
and also on $p$ for large enough $p$ (see the beginning of the proof of Lemma~\ref{controlCeta})
we have
\begin{align}
\label{start'}
\| \RR(z)^{Lm+1+m+1} (\psi)\|
&\le \Cs  (a+ \ln (1/\lambda))^{-Lm} |b|^{N(r-s)}  \|\RR(z)^{m+1}(\psi)\|\\
\nonumber &\,\,
+\Cs^N  (a-\eta)^{-Lm}|b|^{N(r-s)-r'} \|\RR(z)^{m+1}(\psi)\|_{H^{r'}_p(X_0)} \, ,
\end{align}
where we used that our assumptions ensure 
$(1+|z|)^q(1+1/(a-A)) <\Cs |b|^q$.
Recall that $\lambda\in(0,1)$ depends on $r$ and $s$.
(This is why the  consequence \eqref{LYR} of the Lasota-Yorke Lemma~\ref{LY0}  combined
with Lemma~\ref{bq} would not suffice here.) In fact, if $t_0$
is chosen large enough then there exists $\widehat c>0$ so that
$\ln (1/\lambda)\ge r \widehat c$, recalling our assumption $-s<2r$.

Using now \eqref{controlCeta'} from Lemma~\ref{controlCeta}, we get for the same $N$ and $\eta$,
$$
\|\RR(z)^{\ell+1}\| \le \Cs^N \frac{|b|^{N(r-s)} }{ (a-\eta)^{\ell}}\, ,
\quad \forall \ell \ge 0 \, .
$$
Thus, 
the first term of \eqref{start'} is bounded by
\begin{align}\label{eq2'}
\Cs^N\frac{  |b|^{N(r-s)}}{ (a+ \ln (1/\lambda))^{Lm}}  
\frac{|b|^{N(r-s)}}{(a-\eta)^{m}} 
\|\psi\|\, .
\end{align}
Now, if   $L$ is large enough so that  the strict inequality
below holds
\begin{equation}\label{lowL}
\eta=\Lambda|s|+\frac{A}{N}\le r(2\Lambda  + 3 p A) <L\frac{r \widehat c}{4}\le L \frac{\ln(1/\lambda)}{4}\, ,
\end{equation} 
then \eqref{eq2'}  is not larger than
\begin{align*}
&\Cs^N |b|^{2N(r-s)}  \left ( \frac {1} {a+\ln(1/\lambda)/2} \right
)^{Lm+m+2} \|\psi\|\, .
\end{align*}

Therefore,  if
\begin{equation}\label{bdlogb1}
Lm+m+2\ge \frac{2N(r-s) \ln |b|+2N \ln \Cs} {\ln (1+(\ln(1/\lambda)/(2a))-\ln(1+(\ln(1/\lambda) /(4a))}\, ,
\end{equation}
the first term of   \eqref{start'} is not larger than
$$ 
\left ( \frac {1} {a+\ln(1/\lambda)/8} \right )^{Lm+m+2} \|\psi\| \, .
$$

\smallskip
We may thus concentrate on the second term of   \eqref{start'}, that
is, the weak norm contribution. 
Taking
$\epsilon=b^{-\sigma}$ (for $\sigma >2$ to be determined later),  and
using \eqref{controlCeta'} from the proof
of Lemma~\ref{controlCeta} (which gives $\| \RR(z)^{m}\|_{H^{r'}_p(X_0)}\le \Cs \Lambda (a-\Lambda|r'|)^{-m}$), we see that
\begin{align}\label{eq4'}
& \Cs ^N
(a-\eta)^{-Lm} |b|^{N(r-s)-r'} \|\RR(z)^{m+1}(\psi)\|_{H^{r'}_p(X_0)}\\
\nonumber
&\qquad\qquad\le  \Cs^N   \frac{|b|^{N(r-s)-r'}}{(a-\eta)^{Lm}}
\bigl [
\|\RR(z)^{m+1}( \MMM_{\epsilon}(\psi))\|_{\widetilde \HHH^{r',0,0}_p}
\\
\nonumber& \qquad\qquad\qquad  \qquad\qquad\qquad\qquad
  +  \Cs (a-\Lambda |r'|)^{-m} \|\psi-\MMM_{\epsilon}(\psi)\|_{H^{r'}_p(X_0)}
  \bigr ]\, .
\end{align}
By Lemma~\ref{mollbound2} and Corollary~\ref{StrStr}
(using also $\widetilde \HHH^{r,s,0}_p\subset H_p^{s}(X_0)$),   the second term
in \eqref{eq4'} is bounded by
\begin{align}\label{eq5'}
\Cs^{N}    (a-\eta)^{-Lm} (a-\Lambda|r'|)^{-m}|b|^{N(r-s)-r'}
  \epsilon^{s-r'} \|\psi\|_{\widetilde \HHH^{r,s,0}_p}\, .
\end{align}
Next, 
we get
$$
\Cs^{N}     
\frac{|b|^{N(r -s)-r'-\sigma(s-r')}}{(a-\eta)^{Lm}(a-\Lambda|r'|)^{-m}}
\le \left ( \frac{1}{a+\eta}\right)^{Lm+m+2}
$$
if $\sigma$ and $|b|$ are  large enough and
\begin{equation}\label{bdlogb2}
Lm+m+2\le \frac{(\sigma(s-r')- N(r -s)+r')\ln |b|-N \ln \Cs} 
{\ln (1+\eta/a)-(1+1/L)^{-1}\ln(1-\eta/a)-(L+1)^{-1}\ln(1-\Lambda|r'|/a)}\, .
\end{equation}
(The right-hand-side above is positive for large enough $\sigma$.)  

\medskip
The first term in \eqref{eq4'} will be more tricky to handle.
Lemma~\ref{Abound}
implies that this  term  is bounded  by $\Cs^{N }  (a-\eta)^{-Lm} |b|^{N(r-s)-r'}$
multiplied by
\begin{align}
\nonumber    
\|\RR(z)( \Id_{X_0}\AAA_{\delta}& (\RR(z)^{m}(\MMM_{\epsilon}(\psi))))\|_{\widetilde \HHH^{r',0,0}_p}
\\ \nonumber &\qquad
+ 
\|\RR(z)((\id- \Id_{X_0}\AAA_{\delta} )(\RR(z)^{m}(\MMM_{\epsilon}(\psi))))\|_{\widetilde \HHH^{r',0,0}_p}\\
\label{eq6'}&\le
\Cs  
\| \Id_{X_0}\AAA_{\delta} (\RR(z)^{m}(\MMM_{\epsilon}(\psi)))\|_{L^p(X_0)}
\\ \label{forgotten} &\qquad
+ \Cs (a-A)^{-1} 
\delta^{s-r'}
\|\RR(z)^{m}(\MMM_{\epsilon}(\psi))\|_{H^{s}_p(X_0)}\, .
\end{align}
(We used   that $(\RR(z)^{m}(\MMM_{\epsilon}(\psi))$ is supported
in $X_0$, although $\MMM_{\epsilon}(\psi)$ is not necessarily supported
in $X_0$, and
the bounded inclusion
$L^p(X_0)\subset \widetilde \HHH^{r',0,0}_p(R)$
for   $r'\le 0$.)

Since $d=3$,
by the Dolgopyat bound (Lemma~\ref{dolgolemma}),
there exist $\Cs>0$, $\gamma_0 >0$, $\bar \lambda>1$, all
independent of $p$, $r$, and $s$,  so that for 
$\gamma> \gamma_0$ (we shall take
$\gamma >1$),
if 
\begin{equation}
\label{bdlogb0}
m \ge 2 \Cs a \gamma \ln| b|
\end{equation}
then \eqref{eq6'} times $\Cs^{N}   (a-\eta)^{-Lm} |b|^{N(r-s)-r'}$
is bounded by
\begin{align}\label{eq66'}
\frac{\Cs^N |b|^{N(r-s)-r'-\gamma_0}    }{(a-\eta)^{Lm}a^{m}}
\biggl (  \|\MMM_{\epsilon} (\psi)\|_{L^\infty(M)}
+\frac{ \|\MMM_{\epsilon}( \psi)\|_{H^1_\infty(M)}}{(1+(\ln \bar \lambda)/a)^{m/2}}
\biggr )\, .
\end{align}

Now,    Lemma~\ref{mollbound1}, Lemma~\ref{embed},
Corollary~\ref{StrStr}, and the Sobolev embeddings 
give
\begin{align}\label{eq7'}
\| \MMM_{\epsilon}( \psi)\|_{L^\infty(M)}
&\le  \Cs \|\MMM_\epsilon( \psi)\|_{H^{d/p}_p(M)}
 \le \Cs\epsilon^{s-d/p} \|\psi\|_{H^{s}_p(M)}\\
 \nonumber &=\Cs\epsilon^{s-d/p} \| \psi\|_{H^{s}_p(X_0)}\le \Cs\epsilon^{s-d/p} 
\|\psi\|_{\widetilde \HHH^{r,s,0}_p}\, ,
\end{align}
(using $\psi=\Id_{X_0}\psi$) and  
\begin{align}
\label{eq8'}
\|\MMM_\epsilon(\psi)\|_{H^{1}_\infty(M)}&
\le \Cs  \|\MMM_\epsilon(\psi)\|_{H^{1+d/p}_p(M)} 
\le C\epsilon^{s-1-d/p} \|\psi\|_{H^{s}_p(M)}\\
\nonumber &=C\epsilon^{s-1-d/p} \| \psi\|_{H^{s}_p(X_0)}
\le \Cs\epsilon^{s-1-d/p} \|\psi\|_{\widetilde \HHH^{r,s,0}_p}\, .
\end{align}

On the one hand, the estimate  \eqref{eq7'} then  gives
the following bound for the first term of \eqref{eq66'}
$$
\Cs^{N} \frac{ |b|^{N(r-s)+r'-\gamma_0+\sigma(d/p-s)}}
{(a-\eta)^{Lm}    a^{m}}
\le \left ( \frac{1}{a+\eta}\right)^{Lm+m+2}\, ,
$$
if  $\gamma_0-N(r-s)+r'-\sigma(d/p-s)>
\gamma_0/2$ (taking $r$ and $|s|$ smaller and
$p$ larger if necessary), $|b|$ is large enough, and
\begin{equation}\label{bdlogb3}
Lm+m+2\le \frac{(\gamma_0-N(r-s)+r'-\sigma(d/p-s))\ln |b|-N \ln \Cs}
{\ln (1+\eta/a)-\ln(1-\eta/a)}\, .
\end{equation}

On the other hand,  \eqref{eq8'} gives
(recall that $\gamma_0<1/2$) that the second term of
\eqref{eq66'} is bounded by
$$
\Cs^{N}  \frac{b^{N(r-s)-r'-\gamma_0+\sigma(d/p+1-s)}}
{(a-\eta)^{Lm}    a^{m}(1+(\ln \bar \lambda)/a)^{m/2}}
\le \left ( \frac{1}{a+(\ln \bar \lambda)/(4L)}\right)^{Lm+m+2}
$$
if 
\begin{equation}\label{bdlogb4}
m\ge 2 \frac{(\sigma(d/p+1-s) -\gamma_0+N(r-s)-r')\ln|b| +N \ln \Cs}
{\ln (1+(\ln \bar \lambda)/a)+2(L+1)[\ln(1-\eta/a) -\ln (1+(\ln \bar \lambda)/(4La))] } \, .
\end{equation}
We may assume that $\ln (1+(\ln \bar \lambda)/a)>
2(L+1)[ \ln (1+(\ln \bar \lambda)/(4La))- \ln(1+\eta/a)]$.

We must still estimate \eqref{forgotten}
times $\Cs^{N}  (a-\eta)^{-Lm} |b|^{N(r-s)-r'}$.
Take $\delta=b^{-\gamma}$.
By   \eqref{prel} in the proof
of Lemma~\ref{controlCeta}
(as for \eqref{eq4'}), and by Corollary ~\ref{StrStr}
followed by 
Lemma~\ref{mollbound1} applied to $r=r'=s$ (which does not give us any gain), 
we get,
using also the bounded inclusion
$H^{s}_p(X_0)\subset \widetilde  \HHH^{r,s,0}_p$,
\begin{align}\label{eq9'}
& \Cs^{N} (a-\eta)^{-Lm}|b|^{N(r-s)-r'}
\delta^{s-r'}
\|\RR(z)^{m}(\MMM_\epsilon(\psi))\|_{H^{s}_p(X_0)}\\
\nonumber &\qquad\le 
\Cs^N\Lambda (a-\eta)^{-Lm}   \frac{|b|^{N(r-s)-r'}}{(a-\Lambda|s|)^ {m}}
|b|^{-\gamma (s-r')}
\|\psi\|_{\widetilde \HHH^{r,s,0}_p}\, .
\end{align}
Then,
we have
$$
\Cs^N (a-\eta)^{-Lm}    (a-\Lambda|s|)^{-m} |b|^{N(r-s)-r'-\gamma (s-r')} 
\le \left ( \frac{1}{a+\eta}\right)^{m+Lm+2}
$$
if (recall \eqref{lowL} and note that $\Lambda|s| < \eta$)
\begin{equation}\label{bdlogb5}
m+Lm+2\le \frac{(-N(r-s)+r'+\gamma (s-r'))\ln |b|-N \ln \Cs }
{\ln (1+\eta/a)-\ln(1-\eta/a)} \, .
\end{equation}
(The right-hand side above is positive if, $r'<0$ and $\gamma>\gamma_0$ being fixed,
we take $|s|$ and $r$  close enough to $0$ 
and $p>3$ large enough, recalling $(1+3N)r<1/p$.)
This takes care of 
\eqref{eq9'}.

Along the way, we have collected the lower bounds \eqref{bdlogb1},
\eqref{bdlogb0},  and \eqref{bdlogb4}. 
Taking $b_0$ large enough (depending possibly on
$p$, $r$, $s$, in particular through
$L$ and $N$) and $|b|\ge b_0$, they are all implied by
\begin{equation}\label{bdlow}
m\ge \tilde c_1 a \ln |b| \, ,
\end{equation}
where $\tilde c_1$ is a possibly large constant, which is independent
of  $L$, $a$, $b$, $p$, $r$,  and $s$, but grows
linearly like $\gamma>1$.

Up to taking larger $\sigma$,  the upper bounds 
\eqref{bdlogb2}, \eqref{bdlogb3},
and \eqref{bdlogb5} are compatible with 
the lower bound \eqref{bdlow}
(this determines $\tilde c_2 > \tilde c_1$)
if  $p$ is large enough, $r$, $|s|$ are small enough, $\eta=\Lambda |s|+A/N$ is small enough,
and $b_0$ is large enough.
Finally, we take $\nu=\min(\frac{\ln(1/\lambda)}{8}, \eta, \frac{\ln(\bar \lambda)}{4L})$.
(Note that $\nu$ depends on $r$, $s$, $p$, through $\eta$, $L$, and $\lambda$.)
\end{proof}

\appendix

\section{Geometry of contact flows}\label{2c}

In this appendix, we recall some facts about the geometry of contact flows that may not be obvious to all  readers.

\begin{lemma} All ergodic cone hyperbolic contact flows on a compact manifold are Reeb flows.
\end{lemma}
\begin{proof}
Let $v^{u}$ be an unstable vector then, by the invariance of the contact form $\alpha$,
\[
\alpha(v^{u})=\lim_{t\to\infty}T_{-t}^{*}\alpha(v_{u})=\lim_{t\to \infty}\alpha((T_{-t})_*v^{u})=0\, .
\]
Analogously, $\alpha(v^{s})=0$.
Since $T_{t}^{*}d\alpha=d(T_{t}^{*}\alpha)$, we have that also $d\alpha$ is invariant.

Let $V$ be the vector field generating the flow, then for each tangent vector $\xi$ we can decompose it as $\xi=aV+v^{s}+v^{u}$. But $d\alpha(V,V)=0$, $d\alpha(V,v^{s})=\lim_{t\to\infty}d\alpha(V,(T_{t})_*v^{s})=0$ and analogously $d\alpha (V,v^{u})=0$, thus $V$ is the kernel of $d\alpha$.

Let us define $v(w)=\alpha_w(V(w))$, $w\in M$. By the invariance of $\alpha$ we have, for almost all $w$,
\[
v(w)=\alpha_{T_tw}(T_{*t}V(w))=\alpha_{T_t w}(V(T_t w))=v(T_t(w))\, .
\]
Since the flow is ergodic it follows that, almost surely, $v(w)=\bar v^{-1}\in\bR$.
Let us  define the new one form $\beta=\bar v\alpha$. Then $d\beta=\bar vd\alpha$, and $\beta\wedge d\beta=\bar v^2\alpha d\alpha$ is still a volume form, hence $\beta$ is still a contact form, and is invariant with respect to the flow. Moreover $\beta(V)=1$, hence the flow is Reeb.
\end{proof}

\begin{remark}
In the smooth case ergodicity is not necessary since contact implies automatically mixing, \cite{KB}. 
\end{remark}

\begin{lemma}\label{lem:contact-inv} All diffeomorphisms preserving the standard contact form  can be written as
\[
K(x,y,z)=(A(x,y), B(x,y), z+C(x,y)),
\]
where 
\[
\det\begin{pmatrix} \partial_{x}A&\partial_{y}A\\ \partial_{x}B&\partial_{y}B\end{pmatrix}=1\, ,
\] 
$\partial_{x} C=B\partial_{x}A-y$ and $\partial_{y}C=B\partial_{y}A$.
\end{lemma}
\begin{proof}
Let us consider the general change of coordinates defined by 
\[
K(x,y,z)=(A(x,y,z), B(x,y,z), C(x,y,z))\, .
\]
The condition that $\alpha$ is left invariant can be written as
\[
dC- BdA=dz-y dx
\]
that is $d(C-z)=BdA-y dx$. This is possible only if $BdA-y dx$ is a closed form, i.e.,
\begin{equation}\label{eq:change}
\begin{split}
&\partial_z  B\partial_y A-\partial_z A\partial_y B=0\\
&\partial_z  B\partial_x A-\partial_z A\partial_x B=0\\
&\partial_y B\partial_x A-\partial_y A\partial_x B=1\, .
\end{split}
\end{equation}

Multiplying the first by $\partial_{x}A$ and subtracting it to the second multiplied by $\partial_{y}A$, we obtain $\partial_{z}A\left\{\partial_{x}A\partial_{y}B-\partial_{y}A\partial_{x}B\right\}=0$. Which, by the last of the above, implies $\partial_{z}A=0$. This, in turn, implies $\partial_{z}B=0$. From which it follows that $C-z$ is independent of $z$. This implies the lemma.
\end{proof}

Using the above fact we can prove the following lemma:

\begin{lemma}\label{lem:c2}
Consider the leaf $(F(x^{s}), x^{s}, N(x^{s}))$ with tangent space in the stable cone, the point $(\bar x,\bar y,\bar z)=(F(\bar y), \bar y,N(\bar y))$,  and the vector $v$ in the unstable cone and in the kernel of $\alpha$. Then there exist Reeb coordinates in which the selected point is the origin of the coordinates, the leaf reads $(0,x^{s},0)$ and the vector $(1,0,0)$.
\end{lemma}
\begin{proof}
We can  consider the foliation 
\[
(\cW(x^{u},x^{s}), x^{s}, w(x^{u}, x^{s}, x^{0}))=(x^{u}+F(x^{s}), x^{s}, x^{0}+N(x^{s}))\, .
\]
Let us consider a change of coordinates that preserves $\alpha$. By Lemma \ref{lem:contact-inv} the change of coordinates reads  $(A(x^{u},x^{s}), B(x^{u},x^{s}), x^{0}+C(x^{u},x^{s}))$.
The fact that the foliation is sent into horizontal leaves means that $A(\cW(x^{u},x^{s}), x^{s})$ must be independent of $x^{s}$. Hence,  $A$ must be a solution of the first order PDE 
\[
\begin{split}
0&=(\partial_{x^{u}} A)(\cW(x^{u},x^{s}), x^{s})\partial_{x^{u}} \cW(x^{u},x^{s})+(\partial_{x^{s}} A)(\cW(x^{u},x^{s}),\xi)\\
&=\left[(\partial_{x^{u}} A)\cdot \Gamma+(\partial_{x^{s}} A)\right]\circ\bW(x^{u},x^{s})\, ,
\end{split}
\]
where $\bW(x^{u},x^{s})=(\cW(x^{u},x^{s}),x^{s})$ and $\Gamma(x^{u},x^{s})=(\partial_{x^{s}} \cW)\circ \bW^{-1}(x^{u},x^{s})$.\footnote{For example choose $A(x,y)=x-F(y)+F(\bar y)$.}
In other words, we have shown that there exist coordinates in which the foliation reads $(W(x^{u}),x^{s}, H(x^{u},x^{s}, x^{0}))$. At last, since the  leaves are in the kernel of $\alpha$, $\partial_{x^{s}}H=0$, the foliation is made of lines parallel to the $x^{s}$ coordinates, as required.

Let $(\tilde x, \tilde y,\tilde z)$ be the selected point in the new coordinates and perform the change of coordinates
\[
\begin{split}
&x=\xi+\tilde x\, , \qquad y=\eta+\tilde y\\
&z=\zeta+\tilde z+\tilde y \xi\, ,
\end{split}
\]
which sends the selected point to $(0, 0,0)$ and hence the manifold in $(0,x^s,0)$. 

Next, consider the changes of coordinates given by
\[
\begin{split}
&\eta=ax+b y\, , \qquad \xi=c x+dy\\
&\zeta=z+\frac{ac}2 x^{2}+ad xy+\frac{bd}2 y^{2}\, ,
\end{split}
\]
with $ad-cb=1$. Let the vector have coordinates $(u,s,t)$. Then $b=0$ ensures that the horizontal is sent to the horizontal, $a=u^{-1}, d=u, c=-s$, imply that in the new coordinates the vector reads $(1,0, \tilde t)$. But since the vector belongs to the kernel of the contact form and is at the point zero, we must have $\tilde t=0$.
\end{proof}


\section{Basic facts on the local spaces $H^{r,s,q}_p$}\label{localspaces}

We adapt  the bounds of
\cite[\S 4.1]{BG2}, 
state a result about interpolation between Lasota-Yorke
inequalities due to S. Gou\" ezel (Lemma~\ref{tricksg}),
and prove  Lemma~\ref{noglue}, necessary
in view of the ``glueing" procedure in Step 2 of Lemma~\ref{lemcompose}
used in Step ~3 of the proof of the Lasota-Yorke claim Lemma~\ref{LY0}.
We start with a Leibniz bound:

\begin{lemma}
\label{Leib} Fix $\tilde \beta \in (0,1)$.
Let $s\le 0\le r$, $q\ge 0$ be real numbers
with  
\begin{equation}\label{condLeib}
(1+q/r)(r-s)<\tilde \beta\, .
\end{equation} For any $p \in (1,\infty)$, there exists a
constant $\Cs$ such that for any $C^{\tilde \beta}$ function $g :
\real^d \to \complex$,
\begin{equation*}
\norm{ g \cdot \omega}{H_p^{r,s,q}}\le \Cs
\|g\|_{C^{\tilde \beta}} \norm{\omega}{H_p^{r,s,q}}\, .
\end{equation*}
\end{lemma}

\begin{proof}
Note that $(r+q)(1-s/r)>\max(r+q, r-s)$.

If $q=0$ then the proof of \cite[Lemma 22]{BG1} gives the statement
if $r-s <\tilde \beta$. If $s=r=0$ and $0<q<\tilde \beta$, we can use Fubini.
If $q > 0$ and $r-s>0$, it suffices to observe that $H^{r,s,q}_p$ can be obtained
by interpolating between $H^{r_0, s_0, 0}_p$ and $H^{0,0,q_0}_p$
where $r_0=r+q>0$, $s_0 = r_0 s/r<0$, $q_0= r+q$, at $r/r_0\in (0, 1)$, and that our conditions
ensure $r_0-s_0 <\tilde \beta$ and $q_0<\tilde \beta$.
\end{proof}

The following extension of a classical result of Strichartz
\cite{Str}  is an adaptation of  \cite[Lemma 4.2]{BG2}.

\begin{lemma}[Strichartz bound]
\label{lem:multiplier} Let $1<p<\infty$, $s \le 0 \le r$ and $q\ge 0$ be real
numbers so that
\begin{align}
\label{starstar}1/p-1<s (1+\frac{q}{r})\le &0\le r(1+\frac{q}{r}) <1/p   \, .
\end{align}
Let $e_1,\dots,e_d$ be a basis of $\real^d$, such
that $e_{d_u+1},\dots,e_{d-1}$ form a basis of $\{0\}\times
\real^{d_s}\times \{0\}$ and
$e_{d}$ forms a basis of $\{0\}\times
\real$. There exists a constant $\Cs$ (depending only on
$p,s,r, q$ and the norm of the matrix change of coordinate between
$e_1,\dots,e_d$ and the canonical basis of $\real^d$) so that, for
any subset $U$ of $\real^d$ whose intersection with almost
every line directed by a vector $e_i$ has at most $M_{cc}$ connected
components,
\begin{equation*}
\norm{\Id_{U} \omega}{H_p^{r,s,q}} \le \Cs M_{cc}
\norm{\omega}{H_p^{r,s,q}}\, .
\end{equation*}
\end{lemma}

\begin{proof}
Note that $(1+|q|/r) s \le s\le 0$.
The conditions on $r,s,q$ are obtained by interpolating as in the
proof of Lemma~\ref{Leib}, and using the condition
$-1+1/p <t < 1/p$ for $H^t_p(\real^d)$ coming from \cite{Str}
(see the proof of \cite[Lemma 23]{BG1}).
One finishes by a linear change of coordinates preserving the
$\real^{d_s}$ and $\real^{d_0}=\real$ directions.
\end{proof}

The following is essentially 
\cite[Lemma 4.3]{BG2}, itself based on \cite[Lemma 28]{BG1}.

\begin{lemma}[Localisation principle]
\label{lem:localization}
Let $\mathbb K$ be a compact subset of $\real^d$. For each $m\in
\integer^d$, consider a function $\eta_m$ supported in $m+\mathbb K$,
with uniformly bounded $C^1$ norm. For any $p\in (1,\infty)$
and $s\le 0 \le r$, $q\ge 0$ with 
\begin{equation}\label{locall}
(1+q/r)(r-s)<1\, ,
\end{equation} 
there exists $\Cs >0$
so that 
\begin{equation*}
\left(\sum_{m\in \integer^d} \norm{\eta_m \omega }{H_p^{r,s,q}}^p\right)^{1/p}
\le \Cs  \norm{\omega}{H_p^{r,s,q}} \, .
\end{equation*}
If, in addition, we assume that $\sum_{m\in \integer^d} \eta_m(x)=1$
for all $x$, then
\begin{equation}\label{converse}
\norm{\omega}{H_p^{r,s,q}}   \le \Cs \left(\sum_{m\in \integer^d} \norm{\eta_m \omega }{H_p^{r,s,q}}^p\right)^{1/p}  \, .
\end{equation}
\end{lemma}

\begin{proof}
For the first bound,
our condition on $r,s,q$ ensures we can apply Lemma~\ref{Leib} to $\tilde \beta=1$
so that $\Cs$ only depends on the $C^1$ norm of $\eta_m$ (see the proof of
\cite[Lemma 4.3]{BG2}).
For the second bound, we refer to \cite[Remark 29]{BG1} and 
\cite[Theorem 2.4.7(i)]{Trie}.
\end{proof}

The following lemma on partitions of unity is a modification of 
\cite[Lemma 32]{BG1}:
\begin{lemma}[Partition of unity]
\label{lem:sum} Let $r$, $s$, $q$ be
arbitrary real numbers. There exists a constant $\Cs$ such
that, for any distributions $v_1,\dots, v_l$ with compact
support in $\real^d$, belonging to $H_p^{r,s,q}$, there exists a
constant $C$  with
\begin{equation}\label{B.3}
\norm{ \sum_{i=1}^l v_i}{H_p^{r,s,q}}^p \le \Cs m^{p-1} \sum_{i=1}^l
\norm{v_i}{H_p^{r,s,q}}^p 
+ C \sum_{i=1}^l \norm{v_i}{H_p^{r-1,s,q}}^p \, ,
\end{equation}
where $m$ is the intersection multiplicity of the supports $K_i$ of
the $v_i$'s, i.e., $m=\sup_{x\in \real^d} \Card\{i \st x\in K_i \}$, and
letting $K'_i$ be neighbourhoods of the
$K_i$ having the same intersection multiplicity as the
$K_i$, and choosing $C^\infty$ functions $\Psi_i$ so that
$\Psi_i$ is supported in $K'_i$ and $\equiv 1$ on $K_i$, we have
\begin{equation}\label{supcst}
C \le \Cs m^{p-1}\sup_i \| \Psi_i\|_{C^1}\, .
\end{equation}
Replacing $H_p^{r-1,s,q}$ by $H_p^{r',s,q}$ with $r-1<r'\le r$
in the right-hand-side of \eqref{B.3}, 
we can replace the $C^1$ norm by the $C^{r-r'}$
norm \eqref{supcst}.
\end{lemma}

\begin{proof}
Apply the proof of \cite[Lemma 32]{BG1}, noting
that \cite[Lemma 2.7]{Cinfty} also holds for our symbol
$a_{r,s,q}$.
The bound \eqref{supcst},  including the claim
about $r-1<r'\le r$, comes from the term in the integration by parts
in the proof of \cite[Lemma 2.7]{Cinfty}, and interpolation
if $r-r'\notin \integer$.
\end{proof}

The following class of local diffeomorphisms (adapted from 
\cite{BG2}) will be useful:

\begin{definition}\label{3.1} For $C>0$ let $D^1_{2}(C)$ denote  the set of $C^1$
diffeomorphisms $\Psi$ defined on a subset of $\real^d$, sending
stable leaves to stable leaves, flow directions to flow directions, and such that
\begin{align*}
&\max\bigl ( \sup |D\Psi(x^u,x^s,x^0)|, \sup |D\Psi^{-1}(x^u,x^s, x^0)|,\\
&\qquad\qquad \qquad\sup_{x^u,x^s,\tilde x^0, x^0} \frac{|D\Psi(x^u,x^s, x^0)-D\Psi(x^s, x^s, \tilde x^0)|}{|x^0-\tilde x^0|}\\
&\qquad\qquad\qquad
\sup_{x^u,x^s,\tilde x^s, x^0} \frac{|D\Psi(x^u,x^s, x^0)-D\Psi(x^s,\tilde x^s, x^0)|}{|x^s-\tilde x^s|}\bigr )\le C \, .
\end{align*}
\end{definition}

Adapting the  proof of \cite[Lemma 4.7]{BG2} gives:

\begin{lemma}
\label{lemcomposeD1alpha}
Let $C>0$, and let $s \le 0\le r$ and $q\ge 0$ be so that
\eqref{locall} holds.
There exists a constant $C'>0$  so that for any
$\Psi\in D^1_{2}(C)$  whose range contains a ball
$B(z,C_0^{1/2})$, and for any distribution $\omega\in
H_p^{r,s,q}$ supported in $B(z,C_0^{1/2}/2)$, the composition
$\omega\circ \Psi$ is well defined, and
\begin{equation}
\norm{\omega \circ \Psi}{H_p^{r,s,q}}\le C' \norm{\omega}{H_p^{r,s,q}}\, .
\end{equation}
\end{lemma}

\begin{proof}
Without loss of generality, we may assume $z=\Psi^{-1}(z)=0$.
Let $\gamma$ be a $C^\infty$ function equal to $1$ on
$B(0,C_0^{1/2}/2)$ and vanishing outside of $B(0,C_0^{1/2})$.
We want to show that the operator $\MM : \omega\mapsto (\gamma
\omega) \circ \Psi$ is bounded by $C'$ as an operator from
$H_p^{r,s,q}$ to itself. By interpolation (see e.g. the
proof of Lemma~\ref{Leib}), it is sufficient to
prove this statement for for $L^p$, for $H_p^{1,0,0}$,   for
$H_p^{0,-1,0}$ and for $H_p^{0,0,1}$, $H_p^{0,0,-1}$. This can be done 
by adapting the second step of the proof
of Lemma 25 in \cite{BG2}. The result
there is formulated for $C^{2}$ diffeomorphisms, but a
glance at the proof there indicates that the $C^2$
regularity  is only used in the sense of Lipschitz regularity
of the Jacobian along the stable
leaves, in the argument for $H_p^{0,-1}$. In our setting,
we shall also need  Lipschitz regularity
of the Jacobian along the flow
directions, in the argument for $H_p^{0,0,-1}$. The
definition of $D^1_{2}(C)$ ensures that the Jacobian is
indeed Lipschitz  along stable leaves and along flow directions.
\end{proof}

The idea for the following lemma was explained to us by S\'ebastien Gou\"ezel:

\begin{lemma}[Interpolation of Lasota-Yorke-type inequalities]\label{tricksg}
Let $1< p < \infty$ and $s\le 0\le r$, $q\ge 0$,
$q', q''\ge 0$, and $s' \le s$, $r' \le r$.
Let $L$ be an operator for which there exist
constants $c_u$, $c_s$ and $C_u\ge 0$, $C_s\ge 0$ so that
\begin{align}\label{a}
\norm{L w}{H_p^{r,0,0}} < c_u \norm{w}{H_p^{r,0,0}}  + 
C_u \norm{w}{H_p^{r',0,q'}}\, ,
\forall w \in H_p^{r,0,0}\\
\label{b} \norm{L w}{H_p^{0,s,0}} < c_s \norm{w}{H_p^{0,s,0}}  + C_s \norm{w}{H_p^{0,s',q''}}\, ,
\forall w \in H_p^{0,s,0}\, ,
\end{align}
then for each $\theta \in [0,1]$   and every $\omega \in H_p^{\theta r,(1-\theta)s,0}$
\begin{align*}
\norm{L \omega}{H_p^{\theta r, (1- \theta)s,0}} 
&< c_u^\theta c_s^{1-\theta} 
\norm{\omega}{H_p^{\theta r, (1-\theta)s,0}}+ C_u^\theta C_s^{1-\theta}\norm{\omega}{H_p^{\theta r',(1-\theta)s',\theta q'+(1-\theta)q''}}  \\
&\qquad\qquad
+ c_u^\theta C_s^{1-\theta} \norm{\omega}{H_p^{\theta r,(1-\theta)s',(1-\theta) q''}}
+ C_u^\theta c_s^{1-\theta}\norm{\omega}{H_p^{\theta r',(1-\theta)s,\theta q'}} \, .
\end{align*}
\end{lemma}

\begin{proof}[Proof of Lemma ~\ref{tricksg}]
For a Triebel-type symbol $a(\xi^u,\xi ^s, \xi ^0)$ we write $|\omega|_a$ 
for the Triebel norm $\|\FF^{-1} (a \FF \omega) \|_{L^p}$.
By classical multiplier theorems (see \cite{Stein}),
the norm $c_u \norm{\omega}{H_p^{r,0,0}} + C_r \norm{\omega}{H_p^{r',0,q}}$ 
is equivalent to the Triebel norm with symbol
$a_u(\xi^u,\xi^s,\xi^0)$ equal to
$$
c_u (1+|\xi^u|^2+
|\xi^s|^2+|\xi^0|^2)^{r/2} + C_r(1+|\xi^u|^2+
|\xi^s|^2+|\xi^0|^2)^{r'/2} (1+|\xi^0|^2)^{q/2}\, .
$$ 
Therefore, (\ref{a}) reads
$$\norm{L (\omega)}{H_p^{r,0,0} }< |\omega|_{a_u}\, .$$
Similarly, defining
$$a_s(\xi^u,\xi^s, \xi^0)=c_s (1+|\xi^s|^2)^{s/2}+ C_s (1+|\xi ^s|^2)^{s'/2}
(1+|\xi^0|^2)^{q/2}\, ,
$$
(\ref{b}) becomes $\norm{L (\omega)}{H_p^{0,s,0}} < |\omega|_{a_s}$.
One may thus apply standard interpolation. 
A result of Triebel gives that the interpolation of $|.|_{a_u}$ and $|.|_{a_s}$ is
equivalent to $|.|_{a_u^\theta a_s^{1-\theta}}$,
with uniform equivalence constants over the norms
we are considering. 
\end{proof}

Finally, we shall need a new result, specific to our flow situation
(and which is proved with the help of Lemma~\ref{tricksg}):

\begin{lemma} [Composing with a  perturbation in the
flow direction]\label{noglue}
Let $\Delta$ be defined by
$
\Delta(x^u,x^s, x^0)=(x^u, x^s, x^0+  \delta(x^s,x^u))\, ,
x^s\in \real^{d_s}\, ,
x^u\in \real^{d_u}\, ,\,  \, x^0\in \real\, , 
$
where $\delta:\real^{d_s+d_u}\to \real$ is a $C^{1}$
map whose range contains a ball
$B(z,C_0^{1/2})$.
Then, for any $s<0<r$ and $q\ge 0$
so that \eqref{locall} holds, 
there is $C$ so that 
for any distribution $\omega$ supported in $B(z,C_0^{1/2}/2)$, 
$$
\| \omega \circ \Delta\|_{H^{r,s,q}_p}\le C \|\omega\|_{H^{r, s, q+r-s}_p}\, .
$$
\end{lemma}

\begin{proof}
As usual, we proceed by interpolation. 
The Jacobian of $\Delta$ and of $\Delta^{-1}$ is equal to $1$.
Therefore
\begin{equation}\label{ng0}
\| \omega \circ \Delta\|_{H^{0,0,0}_p}\le \|\omega\|_{H^{0, 0,0}_p}\, .
\end{equation}
Since $\partial_{x^0}(\omega \circ \Delta)=(\partial_{x^0}\omega) \circ \Delta$
and $\partial_{x^s}(v \circ \Delta^{-1})=(\partial_{x^s}v) \circ \Delta^{-1}+
(\partial_{x^0}v) \circ \Delta^{-1} \partial_{x^s} \delta$, we have
\begin{equation}\label{ng2}
\| \omega \circ \Delta\|_{H^{0,-1,0}_p}\le \|\omega\|_{H^{0,-1,0}_p}
+C \|D\delta\|_{L^\infty}  \|\omega\|_{H^{0,-1,1}_p}\, .
\end{equation}
Indeed, letting $\Omega_{s}\in H^{0,0,0}_p$ be such that $\partial_{x^s}\Omega_{s}=\omega$,
and writing $D\delta$ for $\partial_{x^s}\delta$,
\begin{align*}
&\sup_{\{\psi:|\psi |_{H^{0,1,0}_{p'}}\le 1\}}\int (\omega\circ \Delta)\cdot  \psi\, dx=\sup_\psi\int \omega\cdot  (\psi \circ \Delta^{-1})\, dx\\
&\qquad=\sup_\psi\biggl (-\int \Omega_{s} \cdot ((\partial_{x^s} \psi) \circ \Delta^{-1})
-\int \Omega_{s}\cdot  ((\partial_{x^0} \psi) \circ \Delta^{-1}) \cdot D \delta\biggr ) \\
&\qquad=\sup_\psi \biggl (-\int (\Omega_{s}\circ \Delta) \partial_{x^s} \cdot \psi
-\int (\Omega_{s}\circ \Delta )\cdot (\partial_{x^0} \psi) \cdot (D \delta\circ \Delta)\biggr ) \\
&\qquad\le \int |\Omega_{s}\circ \Delta |^p
+\sup_\psi  \int ( (\partial_{x^0} \Omega_{x^s})\circ \Delta) \cdot \psi 
\cdot (D \delta\circ \Delta)
\\
&\qquad\le (\int |\Omega_{x^s}|^p)^{1/p}
+ \sup_\psi \int( [(\partial_{x^0} \Omega_{x^s})D \delta]\circ \Delta ) \cdot \psi  \\
&\qquad\le |\omega|_{H^{0,-1,0}_p}
+ \sup |D \delta| (\int |(\partial_{x^0} \Omega_{x^s})\circ \Delta  |^p )^{1/p} \le |\omega|_{H^{0,-1,0}_p}
+ \sup |D \delta|  |\omega|_{H^{0,-1,1}_p}\, .
\end{align*}
Using also $\partial_{x^u}(\omega \circ \Delta)=(\partial_{x^u}\omega) \circ \Delta
+(\partial_{x^0}\omega \circ \Delta) \partial_{x^u} \delta$, 
$\partial_{x^s}(\omega \circ \Delta)=\partial_{x^s}(\omega) \circ \Delta+
(\partial_{x^0}\omega \circ \Delta) \partial_{x^s} \delta$, we get
\begin{equation}\label{ng3}
\| \omega \circ \Delta\|_{H^{1,0,0}_p}\le \|\omega\|_{H^{1,0,0}_p}
+C \|D \delta\|_{L^\infty}  \|\omega\|_{H^{0,0,1}_p}\, . 
\end{equation}
Beware that interpolating between Lasota-Yorke inequalities is not licit in general.
However, for the simple symbols in presence, we may use Lemma~\ref{tricksg} (for $C_u=r=r'=q'=0$ and
$s=-1$, $s'=-2$),
and we deduce from \eqref{ng0}--\eqref{ng2}  that
$$
\| \omega \circ \Delta\|_{H^{0,-(1-\theta''), 0}_p}\le \|\omega\|_{H^{0,-(1-\theta''),0}_p}
+C \|D \delta\|_{L^\infty} ^{1-\theta''} \|\omega\|_{H^{0,-2(1-\theta''),1-\theta''}_p}\, ,\, \,
\forall \theta'' \in [0,1]\, .
$$
and  from 
Lemma~\ref{tricksg} (for $C_s=s=s'=q''=0$,
$r=1$, $r'=0$)  and \eqref{ng0}--\eqref{ng3}  that
$$
\| \omega \circ \Delta\|_{H^{\theta',0,0}_p}\le \|\omega\|_{H^{\theta',0,0}_p}
+C \|D \delta\|_{L^\infty} ^{\theta'} \|\omega\|_{H^{0,0,\theta'}_p}\, , \,
\forall \theta' \in [0,1]\, .
$$
Since $\norm{w}{H_p^{\theta \theta',-2(1-\theta)(1-\theta''),(1-\theta) (1-\theta'')}}$ and
$\norm{w}{H_p^{0, -(1-\theta)(1-\theta''),\theta \theta'}}$,
are dominated by
$\norm{w}{H_p^{r, s,r-s}}$, 
using again  Lemma~\ref{tricksg} to interpolate
at $(r,s,0)=\theta (\theta',0,0)+(1-\theta)(0,-(1-\theta''),0)$, we get
\begin{align}\label{ng4}
\| \omega \circ \Delta\|_{H^{r,s,0}_p}&\le C\bigl ( \|\omega\|_{H^{r, s, 0}_p}+ 
\max(\|D \delta\|_{L^\infty}^{r-s},\|D \delta\|_{L^\infty}^r,
\|D \delta\|_{L^\infty}^{|s|}) \|\omega\|_{H^{r, s, r-s}_p}\bigr )\\
\nonumber &\le C  \|\omega\|_{H^{r, s, r-s}_p}\, .
\end{align}
Finally, since $\partial_{x^0}(\omega \circ \Delta)=(\partial_{x^0}\omega) \circ \Delta$, we have
\begin{equation}\label{ng5}
\| \omega \circ \Delta\|_{H^{0,0,1}_p}\le \|\omega\|_{H^{0,0,1}_p}\, .
\end{equation}
Therefore, interpolating with \eqref{ng4}, we get
the lemma.
\end{proof}

\section{Admissible charts are invariant under
composition}\label{iteratechart}

Lemma ~\ref{lemcompose} below is the analogue of \cite[Lemma 3.3]{BG2}.
(There will be a nontrivial difference, the presence of the map $\Delta_m$.)
We need one more notation:

\begin{definition}
Let $\CC$ and $\tilde \CC$ be extended cones (Definition ~
\ref{extcone}). If an invertible matrix $\AAc:\real^d\to \real^d$ sends
$\CC$ to $\tilde\CC$ compactly, let
$\lambda_u(\AAc)=\lambda_u(\AAc,\CC,\tilde\CC)$ be the least
expansion under $\AAc$ of vectors in $\CC^u$, and
$\lambda_s(\AAc)=\lambda_s(\AAc,\CC,\tilde\CC)$ be the inverse of the
least expansion under $\AAc^{-1}$ of vectors in $\tilde\CC^s$.
Denote by $\Lambda_u(\AAc)=\Lambda_u(\AAc,\CC,\tilde\CC)$ and $
\Lambda_s(\AAc)=\Lambda_s(\AAc,\CC,\tilde\CC)$ the strongest
expansion and contraction coefficients of $M$ on the same
cones.
\end{definition}

\begin{lemma}
\label{lemcompose}
Let $\CC$ and $\tilde\CC$ be extended cones
and let $\beta \in (0,1)$. For any large enough
$C_0$ (depending on $\CC$ and $\tilde\CC$) and any $C_1 >
2C_0$, there exist constants $C$ (depending on $\CC$,
$\tilde\CC$ and $C_0$) and $\epsilon$ (depending on $\CC$,
$\tilde \CC$, $C_0$ and $C_1$) satisfying the following
properties:

Let $\TT$ be a $C^{2}$ diffeomorphism of $\real^d=\real^{d_u}\times\real^{d_s}
\times \real$ with
$\TT(0)=0$,  so that,
$D\TT(z)(0,0,v^0)=(0,0,v^0)$ for all $z\in \real ^d$ and $v^0\in \real$, and,
setting $\AAc= D\TT(0)$,
\begin{align}
\nonumber &
\norm{ \TT^{-1}  \circ \AAc -\id }{C^{2}} \le \epsilon\, , \quad
\AAc \text{  sends } \CC \text{ to } \tilde\CC \text{ compactly, } \\
 \label{suffhyp}
&
\lambda_{s}(\AAc)^{1-\beta} \Lambda_{u}(\AAc)^{1+\beta} \lambda_{u}(\AAc) ^{-1}  <\epsilon \, ,
\qquad  \lambda_u(\AAc)>\epsilon^{-1}\, ,\,\,  \lambda_s(\AAc)^{-1}>\epsilon^{-1} \, .
\end{align}
Let $\JJ \subset \real^d$ be a finite set such that $|m-m'|\ge C$
for all $m\neq m'\in \JJ$, and consider a family of charts
$\{\phi_m\in \FF(m,\tilde\CC^s, \beta, C_0, C_1)\mid m\in \JJ\}$.
Then, defining
\begin{equation*}
\JJ'= \{ m \in \JJ \mid  B(m,d) \cap \TT(B(0,d))\ne \emptyset \} \, ,
\end{equation*}
and setting  $\Pi(x^u,x^s, x^0)=(x^u,0,0)$, we have:

(a) $|\Pi m -\Pi m'|\ge C_0$ for all $m\not=m'$ in $\JJ'$.

(b) There exist $\phi'\in \FF(0,\CC^s, \beta, C_0, C_1)$,  and
diffeomorphisms $\TTT_m$, for $m\in \JJ'$, such that
\begin{equation}\label{newphi}
\TT^{-1}\circ \phi_m=\phi'\circ \TTT_m
\quad \text{ on }
\phi_m^{-1}( B(m,d) \cap \TT(B(0,d))) \, ,\, \, \forall m \in \JJ'\, .
\end{equation}

(c) For each $m \in \JJ'$, we can write $\TTT_m=\Psi \circ
D^{-1} \circ  \Psi_m\circ \Delta_m$, where
\begin{itemize}
\item The diffeomorphism $\Psi_m(x^u, x^s, x^0)
=(\tilde \psi_m(x^u, x^s), x^0)$ is in
$D^1_{2}(C)$, its  range contains  $B(\Pi
(m^u,0),C_0^{1/2})$, and 
$$\Psi_m( \Delta_m
\phi_m^{-1}(B(m,d))) \subset
B(\Pi m, C_0^{1/2}/2)\, .
$$
\item
$\Delta_m(x^u, x^s, x^0)=(x^u, x^s, x^0+\delta_m(x^u, x^s))$, where
$\delta_m$ is a $C^1$ function 
defined on $B(m^u, 0), C_0^{2/3})$ with $\|D\delta_m\|_{L^\infty}\le C$.
\item The matrix $D$ is block diagonal, of the form $D=\left(\begin{smallmatrix} A&0&0\\0&B&0\\0&0&1
\end{smallmatrix}\right)$ with
\begin{equation*}
|Av|\ge C^{-1}\lambda_u(A) |v| \text{ and } |Bv|\le C\lambda_s(A) |v|\, .
\end{equation*}
\item The diffeomorphism $\Psi$ is in $D^1_{2}(C)$,
its range contains $B(0, C_0^{1/2})$.
\end{itemize}
\end{lemma}

Note that  (c) implies in particular that each $\TTT_m$ sends
stable leaves to stable leaves. Note also that if $C_0$ is
large enough, then $\phi' \in \FF(0, \CC^s, \beta, C_0, C_1)$ implies
$(\phi')^{-1}(B(0,d))\subset B(0, C_0^{1/2}/2)$ (because
$\|(\phi')^{-1}\|_{C^1} \le \Cs$ by Lemma ~ \ref{lempropphi}).

\begin{remark}\label{shorter}
Composing with translations, we deduce a more general result
from Lemma ~\ref{lemcompose}, replacing $0$ by $\ell \in
\real^d$, and allowing $\TT(\ell)\ne \ell$: Just replace $\AAc$ by
$D\TT(\ell)$, the projection $\Pi$ by
$\Pi(x^u,x^s,x^0)=(x^u,x^s_{\TT(\ell)},x^0_{\TT(\ell)})$, where 
$\TT(\ell)=(x^s_{\TT(\ell)}, x^u_{\TT(\ell)}, x^0_{\TT(\ell)})$, and assume that
\begin{equation*}
\norm{ (\TT^{-1} [\cdot + \TT(\ell)]-\ell) \circ D\TT(\ell)
-\id }{C^{2}} \le \epsilon
\end{equation*}
and that $D\TT(\ell)$ sends  $\CC $ to $\tilde\CC$ compactly.
One then uses the condition $B(m,d) \cap \TT(B(\ell,d))\ne
\emptyset$ to define $\JJ'$. Of course, $\phi'$ is then in
$\FF(\ell ,\CC^s,\beta,  C_0, C_1)$, equality \eqref{newphi} holds on
$\phi_m^{-1}( B(m,d) \cap \TT(B(\ell,d)))$, and the range of
$\Psi$  contains $B(\ell, C_0^{1/2})$. Finally, we have
$(\phi')^{-1}(B(\ell,d))\subset B(\ell, C_0^{1/2}/2)$.
\end{remark}

The proof below is an adaptation of the proof of
\cite[Lemma 3.3]{BG2}.

\begin{proof}
[Proof of Lemma~\ref{lemcompose}]
We shall write $\pi_1$, $\pi_2$, and
$\pi_3$ for, respectively, the first, second and third projection in $\real^d=\real^{d_u}\times \real^{d_s}\times\real$.

\emph{Step zero: Preparations.} We shall write $\Cs$ and
$\epsilons$ for a large, respectively small, constant,
depending only on $\CC,\tilde\CC$, that may vary from line to
line. For the other parameters, we will always specify if they
depend on $C_0$ or $C_1$.

The set $\AAc(\real^{d_u}\times \{0\}\times \{0\})$ is contained in
$\tilde\CC^u$, hence uniformly transversal to $\{0\}\times
\real^{d_s+1}$. Therefore, it can be written as a graph $\{(x^u,P(x^u))\}$
for some matrix $P$ with norm depending only on $\tilde\CC$.
Let $Q(x^u,x^s, x^0)=(x^u,(x^s, x^0)-P(x^u))$, 
so that $Q\AAc$ sends $\real^{d_u}\times
\{0\}\times \{0\}$ 
and $\{0\}\times \{0\} \times \real$
to itself. In the same way, $\AAc^{-1}(\{0\}\times
\real^{d_s}\times \{0\})$ is contained in $\CC^s$, hence it is a graph
$\{(P_u'(x^s),x^s, P'_0(x^s))\}$. Letting $Q'(x^u, x^s, x^0)(x^u-P_u'(x^s),x^s, x^0-P'_0(x^s))$, the matrix
$D=Q\AAc(Q')^{-1}$ leaves
$\{0\}\times \{0\}\times \real$, $\real^{d_u}\times \{0\}\times \{0\}$, 
and $\{0\}\times
\real^{d_s}\times \real$ invariant, i.e., it is block-diagonal, of the form $
\left(\begin{smallmatrix} A&0&0\\0&B&0\\0&0&\sigma
\end{smallmatrix}\right)$, and moreover
$|Av|\ge \Cs^{-1}\lambda_u |v|$ and $|Bv|\le \Cs\lambda_s
|v|$ (since the matrices $Q$ and $Q'$, as well as their
inverses,  are uniformly bounded in terms of $\CC$ and
$\tilde\CC$) and the scalar matrix $\sigma=1$.

We can readily prove  assertion (a) of the lemma. Let $m\in
\JJ'$, there exists $z \in B(m,d) \cap \TT(B(0,d))$. The set
$Q\TT(B(0,d))=DQ'(\TT^{-1}\AAc)^{-1}(B(0,d))$ is included in
$\{(x^u,x^s, x^0)\st |(x^s, x^0)|\le \Cs\}$ for some constant $\Cs$ (the role of
$Q$ is important here). Since $Qz\in Q\TT(B(0,d))$, we obtain
$|\pi_2(Q z)|\le \Cs$ and
$|\pi_3(Qz)|\le \Cs$. Since $|z-m|\le d$, we also have
$|Qz-Qm|\le \Cs$, hence $|\pi_2(Qm)|\le \Cs$,
$|\pi_3(Qm)|\le \Cs$ (for a different
constant $\Cs$). Since 
$$
Qm-\Pi m(m^u, \pi_2(Qm),\pi_3(Qm)) - (m^u,0,0)(0,\pi_2(Qm),\pi_3(Qm) )
\, ,$$ 
we obtain 
\begin{equation}
\label{piQm}
|Qm-\Pi m|\le \Cs\, .
\end{equation}

Since the points $m\in \JJ'$ are far apart by assumption, the
points $Qm$ for $m\in \JJ'$ are also far apart, and it follows
that the points $\Pi m$ are also far apart. Increasing the
distance between points in $\JJ'$, we can in particular ensure
that $|\Pi m-\Pi m'|\ge C_0$ for any $m\not=m'\in \JJ'$,
proving (a).

The strategy of the rest of the proof  is the
following: We write
\begin{equation}
\TT^{-1}=\TT^{-1}\AAc \cdot (Q')^{-1}\cdot D^{-1} \cdot Q \, .
\end{equation}
We shall start from the partial foliation given by the maps
$\phi_m$ for $m\in \JJ$, apply $Q$ (Step ~1) to obtain a new
partial foliation at $Q m$, modify it via glueing (Step ~2) to
obtain a global foliation, and then push this foliation
successively with $D^{-1}$ (Step~ 3), $(Q')^{-1}$ (Step ~4), and
$\TT^{-1}\AAc$ (last step).

We next define spaces of functions which will be used in the proof.
For $\Cs>0$, let us denote by $\DD(\Cs)$ the class of $C^1$
maps $\Psi$ defined on an open subset of $\real^j$
(the value of $j$ will be clear from the context), taking
their values in $\real^j$ and satisfying
\begin{equation}
\label{defD}
\Cs^{-1} |z-z'|\le |\Psi(z)-\Psi(z') | \le \Cs |z-z'|,
\end{equation}
for any $z,z'$ in the domain of definition of $\Psi$. It follows
that any such $\Psi$ is a local diffeomorphism, and that $\|D\Psi\|\le
\Cs$, $\|(D\Psi)^{-1}\|\le \Cs$.

For $\beta\in (0,1)$, we
denote by $\KK=\KK^{\beta}(C)$ the class of matrix-valued
functions $K$ on $\real^{d-1}$ such that,
for all $x^u,y^u\in \real^{d_u}$, for all $x^s,y^s\in \real^{d_s}$,
\begin{align*}
& |K(x^u,x^s)|\le C\, ,\\
& |K(x^u,x^s)-K(y^u,x^s)|\le C |x^u-y^u|^\beta \, ,
|K(x^u,x^s)-K(x^u,y^s)|\le C |x^s-y^s|\, ,\\
& |K(x^u,x^s)-K(y^u,x^s)-K(x^u,y^s)+K(y^u,y^s)| \le C |x^u-y^u|^\beta |x^s-y^s|^{1-\beta} \, .
\end{align*}
The spaces of local diffeomorphisms
$\DD(\Cs)$ and of matrix-valued functions
$\KK(\Cs)=\KK^{1,\beta}(\Cs)$ were introduced in 
\cite{BG2}. As in Remark~A.6 of 
Appendix
A of \cite{BG2}, we will
write $\KK(\Cs,A)$ for the functions defined on a set $A$ and
satisfying the inequalities defining $\KK(\Cs)$ ($A$ will
sometimes be omitted when the domain of definition is obvious).
The map $\phi_m=(F_m(x^u, x^s),x^s, x^0+\tilde
f_m(x^s, x^u))$ belongs to $\DD(\Cs)$ (see the proof of Lemma
~ \ref{lempropphi}), the matrix-valued function $DF_m$
belongs to $\KK(\Cs, B((m^u, m^s), C_0))$ (boundedness of $DF_m$ is
proved in Lemma ~ \ref{lempropphi}, while the H\"{o}lder-like
properties are given by \eqref{smooths}--\eqref{smoothsecond}).

\medskip

\emph{First step: Pushing the foliations with $Q$.} We
formulate in detail  the construction in this first step (a
version of Lemma~ \ref{lemcomposeQ} will be used also in the
last step, replacing $Q$ by $\TT^{-1} \AAc$, while steps 2-3-4 are
much simpler). The  statement below is the analogue in our setting
of \cite[Lemma 3.5]{BG2}:

\begin{lemma}
\label{lemcomposeQ} (Notation as in Lemma ~\ref{lemcompose} and Step 0 of its proof.)
There exists a constant $\Cs$ such that, if $C_0$ is large
enough and $C_1>2C_0$, for any $m=(m^u, m^s, m^0)\in \JJ'$ there
exist maps 
$\Psi_m:B((m^u, m^s),C_0^{2/3})\times \real\to \real^d$,
$\phi_m^{(1)}: B(\Pi m, C_0^{1/2}) \to \real^d$, and a map $\Delta_m$
from $B((m^u, m^s),C_0^{2/3})\times \real$ to
itself such that
\begin{equation*}
\phi_m^{(1)}\circ \Psi_m\circ \Delta_m
=Q\circ \phi_m \text{ on } \phi_m^{-1}(B(m,d)) \, .
\end{equation*}
Moreover, $\Psi_m$ is a diffeomorphism in $D^1_{2}(\Cs)$
whose  range contains $B(\Pi m, C_0^{1/2})$, and
$\Psi_m(\Delta_m \phi_m^{-1}(B(m,d))) \subset B(\Pi m, C_0^{1/2}/2)$,
while $\Delta_m(x^u, x^s, x^0)(x^u, x^s, x^0+\delta_m(x^u, x^s))$  with
$\|\delta_m\|_{C^1}\le C$.
Finally, 
$$\phi_m^{(1)}(x^u,x^s, x^0)=(F^{(1)}_{m}(x^u,x^s),x^s, 
x^0)
$$ 
on $B(\Pi m,
C_0^{1/2})$, with $F^{(1)}_m$ a $C^1$ map
with $DF^{(1)}_m$ in $\KK(\Cs,B(
(m^u,0), C_0^{1/2}) )$.
\end{lemma}

If $\EE$ is the foliation given by
$\phi_m(x^u,x^s, x^0)=(F_m(x^u,x^s),x^s, \tilde F_m(x^u,x^s,x^0))$, 
then by definition $\phi_m^{(1)}$
sends the stable leaves of $\real^d$ to the foliation $Q(\EE)$,
i.e., $\phi_m^{(1)}$ is the standard parametrisation of the
foliation $Q(\EE)$.

\begin{proof}[Proof of Lemma ~\ref{lemcomposeQ}]
Fix $m=(m^u,m^s, m^0) \in \JJ'$. The map $Q\circ \phi_m$ does not
qualify as $\phi_m^{(1)}$ for three reasons. First, $\pi_2 \circ
Q \circ \phi_m(x^u,x^s, x^0)$ is  not equal to $x^s$ in general. Second,
$ \pi_3 \circ Q \circ \phi_m(x^u,x^s, x^0)$ 
is  not equal to $x^0$ in general.
Thirdly,
$\pi_1 \circ Q\circ \phi_m(x^u,0,x^0)$ is not equal to $x^u$ in general.
We shall use two maps $\Gamma^{(0)}$ and $\Gamma^{(1)}$ (sending
stable leaves to stable leaves, and flow directions to flow directions) 
and a perturbation of the flow direction
$\Delta_m$ to solve these three
problems. The map $\Gamma^{(0)}$ will have the form
$\Gamma^{(0)}(x^u,x^s, x^0)=(x^u,G(x^u,x^s), x^0)$, where, for fixed $x^0$ and $x^u$, the map
$x^s\mapsto G(x^u,x^s)$ is a diffeomorphism of the stable leaf
$\{x^u\}\times \real^{d_s}\times \{x^0\}$, so that $\pi_2 \circ Q \circ \phi_m
\circ \Gamma^{(0)} (x^u,x^s,x^0)=x^s$, while
the map
$\Delta_m(x^u, x^s, x^0)(x^u, x^s, x^0+\delta_m(x^u,x^s))$ is  so that $
\pi_3\circ Q \circ \phi_m \circ \Delta_m^{-1}
\circ \Gamma^{(0)}  (x^u,x^s,x^0)\equiv x^0$
(note that $\pi_2 \circ Q \circ \phi_m
\circ \Gamma^{(0)} (x^u,x^s,x^0)=\pi_2 \circ Q \circ \phi_m\circ \Delta_m^{-1}
\circ \Gamma^{(0)} (x^u,x^s,x^0)$). In particular, $Q \circ
\phi_m\circ \Delta_m^{-1}\circ \Gamma^{(0)}(x^u,0,x^0)$ is of the form 
$(L^{(1)}(x^u),0,x^0)$
for some map $L^{(1)}$. Choosing
$$\Gamma^{(1)}(x^u,x^s,x^0)=((L^{(1)})^{-1}(x^u),x^s, x^0)$$ solves our last
problem: the map
$$\phi_m^{(1)}= Q \circ \phi_m \circ \Delta_m^{-1}
\circ \Gamma^{(0)} \circ \Gamma^{(1)}
$$ satisfies 
$\pi_2\circ \phi_m^{(1)}(x^u,x^s,x^0)=x^s$,
$\pi_3\circ \phi_m^{(1)}(x^u,x^s,x^0)=x^0$, 
and $\pi_1 \circ
\phi_m^{(1)}(x^u,0,x^0)=x^u$, as desired. 
(And it still has the property that
$\partial_{x^0}\pi_1\circ \phi_m^{(1)}(x^u,x^s,x^0)\equiv 0$.)
Then, the map
$\Psi_m=(\Gamma^{(0)}\circ \Gamma^{(1)})^{-1}$
sends
stable leaves to stable leaves and flow directions to flow directions, 
and $Q\circ \phi_m \phi_m^{(1)} \circ \Psi_m\circ \Delta_m$.

We shall now be more precise, justifying the existence of the
maps mentioned above, and estimating their domain of
definition, their range and their smoothness.

\emph{The maps $\Gamma^{(0)}$ and
$\Delta_m$.} For fixed $x^u$  the map $(x^u, x^s)\mapsto
G(x^u,x^s)$ should satisfy $\pi_2  \circ Q \circ
\phi_m(x^u,G(x^u,x^s),x^0)=x^s$, i.e., it should be the inverse to the map
\begin{equation}
\label{def_Lx}
L_{x^u} : x^s\mapsto \pi_2\circ Q \circ \phi_m(x^u,x^s, x^0)  x^s-\pi_2(PF_m(x^u,x^s))\, ,
\end{equation}
where we denote $\phi_m(x^u, x^s, x^0)=(F_m(x^u,x^s),x^s, 
\tilde F_m(x^s, x^0))$.  We claim that the
map  $L_{x^u}$ is invertible onto its image, and that there exists
$\epsilons^0>0$  such that
\begin{align}
\label{ellx_diffeo}
&|L_{x^u}(y^s)-L_{x^u}(x^s)|\ge 
\epsilons^0 |y^s- x^s|\, ,\\
\nonumber&\qquad \qquad\forall x^u\in B(m^u,C_0) \, , \quad \forall x^s,
y^s\in B(m^s,C_0)
\, .
\end{align}
Indeed, fix $x^u \in B(m^u, C_0)$, let $w=y^s-x^s$, write
$F(x^s)=F_m(x^u,x^s)$. We have
\begin{equation}
\label{jlqsmfdlkjqmdslf}
L_{x^u}( y^s)-L_{x^u}(x^s)    w -\pi_2\biggl (P \int_{\sigma =0}^1 \partial_{x^s} F(x^s+\sigma w) w\,  d \sigma \biggr )
\, .
\end{equation}
Define 
$$\CC^{ws}=\{v+(0,0, x^0)\mid v \in \CC^s, x^0 \in \real\}\, .$$
Each vector $(\partial_{x^s} F(x^s+\sigma w)w,w,x^0)$ belongs to $\tilde\CC^{ws}$.
Since this cone is  transversal to $\real^{d_u}\times
\{0\}\times \{0\}$ and defined
through a linear map
(see \eqref{conedeff} and the corresponding
footnote), the set $\tilde\CC^{ws} \cap (\real^{d_u} \times \{w\} \times
\{x^0\})$ is
convex, hence
\begin{equation}
\label{useconv}
v_1  \left(\int_{\sigma =0}^1 \partial_{x^s} F(x^s+\sigma w) w \, d \sigma, w,x^0\right)
\in \tilde \CC^{ws}\, .
\end{equation}
On the other hand, since the graph of $P$ is included in
$\tilde \CC^u$, 
$$v_2= (\int_{\sigma =0}^1 \partial_{x^s} F(x^s+\sigma w)
w \, d \sigma , P \int_{\sigma =0}^1 \partial_{x^s} F(x^s+\sigma w) w\, d \sigma)$$ 
belongs to
$\tilde\CC^u$. Let $\epsilons^0>0$ be such that
$B(v,\epsilons^0|v|) \cap \tilde\CC^u =\emptyset$ for any $v\in
\tilde\CC^{ws} -\{0\}$. Since $v_1 \in \tilde \CC^{ws}$ and $v_2 \in
\tilde \CC^u$, we get $|v_1-v_2|\ge \epsilons^0 |v_1|$. As
$v_1$ and $v_2$ have the same first
and third components if
$x^0:=\pi_3(P \int_{\sigma =0}^1 \partial_{x^s} F(x^s+\sigma w) w\, d \sigma )$, this gives
$|\pi_2(v_1)-\pi_2(v_2)|\ge \epsilons^0|v_1|$, i.e.,
\begin{equation*}
\left|w - \pi_2(P \int_{\sigma =0}^1 \partial_{x^s} F(x^s+\sigma w) w \, d \sigma)
\right|\ge \epsilons^0 |w|\, ,
\end{equation*}
which implies \eqref{ellx_diffeo} by \eqref{jlqsmfdlkjqmdslf}.

For each $x^s$  the map $(x^u, x^s)\mapsto
\delta_m(x^u,x^s)$ should satisfy 
$$
\pi_3  \circ Q \circ
\phi_m(x^u,G(x^u,x^s),x^0-\delta_m(x^u, G(x^u,x^s)))\equiv x^0\, ,
$$ 
i.e., 
\begin{equation}
\label{def_HH}
\tilde F_m(x^u,
G(x^u, x^s),x^0-\delta_m(x^u, G(x^u,x^s))) -\pi_3(PF_m(x^u, G(x^u,x^s)))
=x^0 \, .
\end{equation}
Letting $P_3=\pi_3 P$, and recalling that
$\partial_{x^0}\tilde F_m(x^s,x^0)\equiv 1$,
the condition reads, for $y^s=G(x^u, x^s)$,
\begin{equation}\label{defdelta}
\tilde f_m(x^u, y^s)
-\pi_3(PF_m(x^u, y^s))
=\delta_m(x^u, y^s)\, .
\end{equation}

The map $\Lambda^{(0)}:(x^u,x^s)\mapsto (x^u,L_{x^u}(x^s))$ is well defined
on $B( (m^u, m^s), C_0)$, its derivative is bounded by a constant $\Cs$,
and its second   component satisfies \eqref{ellx_diffeo}.
Then, \cite[Lemma
~A.1]{BG2} (with $x^u$ and $x^s$ exchanged) shows that
$\Lambda^{(0)} \in \DD(\Cs,B( (m^u, m^s), C_0))$ for some constant $\Cs$. In
particular, $\Lambda^{(0)}$ admits an inverse
which also belongs to $\DD(\Cs)$. 
By \cite[Lemma
~A.2]{BG2}, the range of $\Lambda^{(0)}$ 
contains the ball $B( \Lambda^{(0)}(m), C_0/\Cs)$. Moreover,
$\Lambda^{(0)}(m^u, m^s)=(m^u, \pi^2(Qm))$. By \eqref{piQm}, we have $|Qm - \Pi
m|\le \Cs$, hence the domain of definition of $(\Lambda^{(0)}))^{-1}(x^u, x^s)
=(x^u, G(x^u, x^s))$
contains $B( \Pi m, C_0/\Cs -\Cs)$. If $C_0$ is large enough,
this contains $B(\Pi m, C_0^{2/3})$.

Finally, one gets from   \eqref{defdelta}  and Lemma~\ref{lempropphi}
that
$\delta_m(x^u, x^s)$ is defined on  $B((m^u, 0), C_0^{2/3})$ with
$\|\delta_m\|_{C^1}\le \Cs$ (in particular,
$\Delta_m \in \DD(\Cs)$),  and we put
\begin{align*}
&  \Gamma^{(0)}(x^u, x^s, x^0)=(x^u, G(x^u, x^s), x^0)  ((\Lambda^{(0)}))^{-1}(x^u, x^s), x^0)\, ,\, \,\\
&  \Delta_m(x^u, x^s, x^0)=x^0+\delta_m(x^u, x^s)\, .
\end{align*}

\emph{The map $\Gamma^{(1)}$.} Consider $\phi_m^{(0)}Q\circ \phi_m \circ \Delta_m^{-1} \circ \Gamma^{(0)}$. It is a composition of maps
in $\DD(\Cs)$, hence it also belongs to $\DD(\Cs)$. Moreover,
its restriction to $\real^{d_u}\times \{0\}\times \real$ has the form
$(x^u,0,x^0)\mapsto (L^{(1)}(x^u),0,x^0)$. It follows that the map
$L^{(1)}$ (defined on a subset of $\real^{d_u}$) also satisfies
the inequalities defining $\DD(\Cs)$. In particular, this map is
invertible, and we may define
$\Gamma^{(1)}(x^u,x^s,x^0)=((L^{(1)})^{-1}(x^u),x^s, x^0)$. This map belongs to
$\DD(\Cs)$. By construction, $\phi_m^{(1)} = Q \circ
\phi_m  \circ \Delta_m^{-1} \circ\Gamma^{(0)} \circ \Gamma^{(1)}$ can be written as
$$(F_m^{(1)}(x^u,x^s),x^s,x^0)$$ 
with $F_m^{(1)}(x^u,0)=x^u$.

We have $\phi_m^{(0)}(Q m)=Q m$. Since $|\Pi m-Qm| \le \Cs$ by
\eqref{piQm}, and $\phi_m^{(0)}$ is Lipschitz, we obtain
$|\phi^{(0)}_m(\Pi m)-\Pi m|\le \Cs$, i.e.,
$|L^{(1)}(m^u)-m^u|\le \Cs$. Since $L^{(1)}\in \DD(\Cs)$,
\cite[Lemma~A.2]{BG2} shows that $L^{(1)}(B(m^u,C_0^{2/3}))$
contains the ball $B(m^u,C_0^{2/3}/\Cs-\Cs)$. Therefore, it
contains the ball $B(m^u,C_0^{1/2})$ if $C_0$ is large enough.
Hence, the domain of definition of the map $\Gamma^{(1)}$
contains $B(\Pi m, C_0^{1/2})$. This shows that $\phi_m^{(1)}$
is defined on $B(\Pi m, C_0^{1/2})$.

\emph{The map $\Psi_m$.} We can now define
$$
\Psi_m=(\Gamma^{(0)}\circ\Gamma^{(1)})^{-1}=(L^{(1)}(x^u),L_{x^u}(x^s),
x^0)\, ,
$$
so that $Q\circ
\phi_m=\phi_m^{(1)} \circ \Psi_m\circ \Delta_m$. We have seen that $\Psi_m\in
\DD(\Cs)$, hence $D\Psi_m$ and $D\Psi_m^{-1}$ are uniformly
bounded. To show that $\Psi_m \in D^1_{2}(\Cs)$, we
should check that $|D\Psi_m(x^u,x^s, x^0)-D\Psi_m(x^u, y^s, x^0)| \le
\Cs|x^s-y^s|$. This follows directly from the construction
and the inequality \eqref{smooths} for $DF_m$.
Finally, since $\Psi_m\in \DD(\Cs)$ and $\Delta_m\in \DD(\Cs)$,
\begin{equation*}
\Psi_m( \Delta_m \phi_m^{-1}(B(m,d))) \subset \Psi_m(B(m, \Cs))
\subset B( \Psi_m(m), \Cs)\, .
\end{equation*}
Since $Q m=\phi_m^{(1)}(\Psi_m(m))$ and $\Pi m=\phi_m^{(1)}(\Pi
m)$, we get $|\Psi_m(m) -\Pi m|\le \Cs |Qm-\Pi m|\le \Cs$ by
\eqref{piQm}. Therefore, $\Psi_m( \phi_m^{-1}(B(m,d))) \subset
B( \Pi m, \Cs)$, and this last set is included in $B(\Pi m,
C_0^{1/2}/2)$ if $C_0$ is large enough.

\emph{The regularity of $DF_m^{(1)}$.} To finish the proof, we
should prove that $DF_m^{(1)}$  satisfies the
bounds defining $\KK(\Cs)$ for some constant
$\Cs$ independent of $C_0$. Since $\phi_m^{(1)}=Q\circ \phi_m
\circ \Delta_m^{-1}
\circ \Gamma^{(0)}\circ \Gamma^{(1)}$, we have
\begin{align}
\label{eqdphim1}  D \phi_m^{(1)}&  ( DQ \circ \phi_m \circ \Delta_m^{-1}\circ \Gamma^{(0)}\circ \Gamma^{(1)})
\cdot (D\phi_m\circ \Delta_m^{-1}\circ \Gamma^{(0)}\circ \Gamma^{(1)})\\
\nonumber
&\qquad \qquad\qquad
\cdot (D \Delta_m^{-1}\circ \Gamma^{(0)}\circ \Gamma^{(1)})
\cdot( D\Gamma^{(0)} \circ \Gamma^{(1)})
\cdot D\Gamma^{(1)}\, .
\end{align}
Since $\KK$ is invariant under multiplication (\cite[Proposition~
A.4]{BG2}), and under composition by Lipschitz maps
sending stable leaves to stable leaves (\cite[Proposition~
A.5]{BG2}), it is sufficient to show that $D\phi_m$, $D\Delta_m^{-1}$,
$D\Gamma^{(0)}$, and $D\Gamma^{(1)}$ all satisfy the bounds
defining $\KK(\Cs)$
(note that this is where \eqref{smoothu}--\eqref{smoothsecond} will be used)
For $D\phi_m$, this follows from our assumptions,
and for $D\Delta_m^{-1}$ from our assumptions
and \eqref{defdelta}.

Since $\Gamma^{(0)}=((\Lambda^{(0)})^{-1}, H)$, we have
$D\Gamma^{(0)}=(D\Lambda^{(0)})^{-1}\circ \Gamma^{(0)}$. Since
$D\Lambda^{(0)}$ is expressed in terms of $DF_m$, it belongs to
$\KK$. As $\KK$ is invariant under inversion (\cite[Proposition
A.4]{BG2}) and composition, we obtain $D\Gamma^{(0)} \in
\KK(\Cs)$.

Since $D\phi_m^{(1)}(x^u,0,x^0)=\id$, it follows from
\eqref{eqdphim1} that, on the set $\{(x^u,0,x^0)\}$, $D\Gamma^{(1)}$
is the inverse of the restriction of a function in $\KK$, and in
particular $$D\Gamma^{(1)}(x^u,0, x^0)$$ is a $\beta$-H\"{o}lder continuous
function of $x^u$, by \cite[(A.7)]{BG2} and a Lipschitz function of
$x^0$ by construction. Since
$D\Gamma^{(1)}(x^u,x^s, x^0)$ only depends on $x^u$ and
$x^0$, it follows that
$D\Gamma^{(1)}$ belongs to $\KK$. 
This concludes the proof of
Lemma ~\ref{lemcomposeQ}.
\end{proof}

\smallskip

We return to the proof of Lemma~\ref{lemcompose}:

\emph{Second step: Glueing the foliations $\phi^{(1)}_{m}$
together.}

Just as in \cite{BG2}, a glueing step is necessary (see Step 3
of the proof of the Lasota-Yorke Lemma~\ref{LY0}: the localisation lemma
gives
$\sum_{m\in \integer^d} \norm{\eta_m \omega }{H_p^{r,s,q}}^p
\le \Cs  \norm{\omega}{H_p^{r,s,q}}^p$
but {\em not} 
$\sum_{m\in \integer^d} \norm{\eta_m \omega_m }{H_p^{r,s,q}}^p
\le \sup_m \Cs  \norm{\omega_m}{H_p^{r,s,q}}^p$). 

Let $\gamma(x^u, x^s)$ be a $C^\infty$ function equal to $1$ on the
ball $B(C_0^{1/2}/2)$, vanishing outside of $B(C_0^{1/2})$. Let
$\phi^{(1)}_{m}(x^u,x^s, x^0)=(F^{(1)}_m(x^u,x^s),x^s, x^0)$
be a foliation defined
by Lemma~ \ref{lemcomposeQ}, and put
\begin{align*}
\phi^{(2)}_{m}(x^u,x^s, x^0)&=(\gamma(x^u-m^u,x^s)(F^{(1)}_m(x^s, x^u)-x^u)+x^u, x^s,   x^0) \, .
\end{align*}
By construction, 
$$
\phi^{(2)}_{m}(x^u,x^s, x^0)=(F^{(2)}_m(x^u,x^s),x^s, x^0)\, ,
$$ 
with
$F^{(2)}_m(x^u,0)=x^u$. In addition
$
\phi^{(2)}_{m}$ defines a foliation on the ball of radius
$C_0^{1/2}$ around $\Pi m$, coinciding with $\phi^{(1)}_{m}$ on
$B(\Pi m,C_0^{1/2}/2)$, with $F^{(2)}_{m}$ equal to 
$x^u$ on the $(x^u, x^s)$-boundary of $B(\Pi m,C_0^{1/2})$. Moreover, $DF^{(2)}_m$ is expressed in
terms of $\gamma$, $D\gamma$, $F_m^{(1)}$ and $DF_m^{(1)}$. All
those functions belong to $\KK(\Cs)$ (the first three functions
are Lipschitz and bounded, hence in $\KK(\Cs)$, while we proved
in Lemma ~\ref{lemcomposeQ} that $DF_m^{(1)}\in \KK(\Cs)$).
Therefore, $DF^{(2)}_m \in \KK(\Cs)$ by a modification of \cite[Proposition
~A.4]{BG2}.
We proved in (a) that the balls $B(\Pi m, C_0^{1/2})$ for
$m\in\JJ'$ are disjoint, therefore all the foliations $\phi^{(2)}_{m}$ can be
glued together (with the trivial stable foliation outside of
$\bigcup_{m\in \JJ'} B(\Pi m,C_0^{1/2})$), to get a single
foliation parameterised by $\phi^{(2)}:\real^d \to \real^d$. We
emphasize that this new foliation is not necessarily contained
in the cone $Q( \tilde\CC^s)$, since the function $\gamma$
contributes to the derivative of $\phi^{(2)}$. Nevertheless, it
is uniformly transversal to the direction $\real^{d_u}\times\{0\} \times \real$,
and this will be sufficient for our purposes.  Write
$$\phi^{(2)}(x^u, x^s, x^0)(F^{(2)}(x^u,x^s),x^s, x^0)
$$
where $F^{(2)}$ coincides
everywhere with a function $F_m^{(2)}$ or with the function
$(x^u,x^s)\mapsto x^u$.
Since all the derivatives of those functions
belong to $\KK(\Cs)$, it follows that $DF^{(2)}\in \KK(\Cs)$ 
(for some other constant $\Cs$, worse than the previous one due
to the glueing). Since we will need to reuse this last constant,
let us denote it by $\Cs^{(0)}$.

\medskip

\emph{Third step: Pushing the  foliation $\phi^{(2)}$ with
$D^{-1}$.} This  very simple step is the heart of the argument and
this is where
\eqref{suffhyp} is needed: Define a new foliation by
\begin{align}
\label{heart}
 &F^{(3)}(x^u,x^s)=A^{-1}F^{(2)}(Ax^u, Bx^s) \, ,
\, \, \phi^{(3)}(x)=(F^{(3)}(x^u,x^s),x^s, x^0) \, ,
\end{align}
so that $D^{-1}\phi^{(2)}=\phi^{(3)}D^{-1}$. 
We have
$F^{(3)}(x^u,0)=x^u$. Moreover
\begin{align*}
& \partial_{x^u} F^{(3)}(x^u,x^s)=A^{-1}(\partial_{x^u} F^{(2)})(Ax^u, Bx^s) A\,,  \\&\partial_{x^s} F^{(3)}(x^u,x^s)=A^{-1}(\partial_{x^s} F^{(2)})(Ax^u, Bx^s) B \, .
\end{align*}
In particular, if $|A^{-1}|$ and $|B|$ are small enough (which
can be ensured by decreasing $\epsilon$ in \eqref{suffhyp}), we
can make $\partial_{x^s} F^{(3)}$  arbitrarily small. 
Since $|B|\le
1 \le |A|$, it also follows that (see \cite[(3.12)]{BG2})
\begin{equation}
\label{heart'}
\begin{split}
\raisetag{-30pt}
|D F^{(3)}(x^u,x^s)-D F^{(3)}(x^u, y^s)| &  \le
 |A^{-1}| |A| \Cs^{(0)} |B| |x^s-y^s| \, .
\end{split}
\end{equation}

In the same way as \eqref{heart'} (see \cite[(3.13)]{BG2}),
\begin{equation}
\label{heart''}
\begin{split}
|DF^{(3)}(x^u,x^s)-D&F^{(3)}(x^u, y^s)-DF^{(3)}(y^u,x^s)+DF^{(3)}(y^u,  y^s))|
\\&
\le |A^{-1}| |A| \Cs^{(0)} |A|^\beta |B|^{1-\beta} |x^u-y^u|^\beta
 |x^s-y^s|^{1-\beta}
\, .
\end{split}
\end{equation}
If the bunching constant $\epsilon$ in \eqref{suffhyp} is small
enough (depending on $C_1$), we can ensure that the two last
equations are bounded, respectively, by $|x^s-y^s|/(2C_1)$
and $|x^u-y^u|^\beta |x^s-y^s|^{1-\beta} / (4C_0^2C_1)$, i.e.,
the map $F^{(3)}$ satisfies the requirements \eqref{smooths} and
\eqref{smoothsecond} for admissible foliations, with better constants.

Taking $y^s=0$ in \eqref{heart''}, we obtain
\begin{align*}
&|DF^{(3)}(x^u,x^s) - DF^{(3)}(y^u,x^s)|\\
&\qquad \le |x^u-y^u|^\beta |x^s|^{1-\beta} / (4C_0^2C_1)+ 
|DF^{(3)}(x^u,0) - DF^{(3)}(y^u,0)|\, .
\end{align*}
Moreover, $\partial_{x^u} F^{(3)}(x^u,0)=\partial_{x^u}
F^{(3)}(y^u,0)=\id$, so that
(see the computation in the lines above \cite[(3.14)]{BG2}),
\begin{align*}
|DF^{(3)}(x^u,0) - DF^{(3)}(y^u,0)|
&
\le |A^{-1}| |B| \Cs^{(0)} |A|^\beta |x^u-y^u|^\beta\, .
\end{align*}
The quantity $|A^{-1}| |B| |A|^\beta$ is bounded by $\Cs
\lambda_u^{-1} \lambda_s \Lambda_u^\beta$. Choosing $\epsilon$
small enough in \eqref{suffhyp}, it can be made arbitrarily
small. For $|x^s|\le C_0^2$, this yields
\begin{equation}
|DF^{(3)}(x^u,x^s) - DF^{(3)}(y^u,x^s)| \le |x^u-y^u|^\beta/(2C_1)\, ,
\end{equation}
which is a small reinforcement of \eqref{smoothu}.

We see that for fixed  $\CC$, $\widetilde C$,  there is a constant
$C$ depending only on $C_0$ so that
the smallness condition on $\epsilon$ is of the
form
\begin{equation}\label{smalleps}
\epsilon \le \frac{C}{C_1}\, .
\end{equation}
\medskip

\emph{Fourth step: Pushing the foliation $\phi^{(3)}$ with
$(Q')^{-1}$.} Define maps 
$$
F^{(4)}(x^u,x^s)=F^{(3)}(x^u, x^s)+P'_u(x^s)\, ,\, \, \, 
\tilde F^{(4)}(x^s, x^0)= x^0+P'_0(x^s)
$$ and
let $\phi^{(4)}(x^u,x^s, x^0)=(F^{(4)}(x^u,x^s),x^s, \tilde F^{(4)}(x^s, x^0))$. The corresponding
foliation is the image of $\phi^{(3)}$ under $(Q')^{-1}$. Let
us fix a cone $\CC^s_1$ which sits compactly  between $\CC^s_0$
and $\CC^s$. Since the graph $\{(P_u'(x^s),x^s, P_0'(x^s)\}$ is contained in
$\CC^s_0$, the foliation $F^{(4)}$ is contained in $\CC^s_1$ if
$\partial_{x^s} F^{(3)}$ is everywhere small enough. Moreover, the
bounds of the previous step concerning $DF^{(3)}$ directly
translate into the following bounds for $DF^{(4)}$ for all
$x^u,y^u\in \real^{d_u}$ and all $x^s, y^s\in B(0,C_0^2)$:
\begin{align}
\label{controleF4}
&| DF^{(4)}(x^u,x^s) -DF^{(4)}(x^u,y^s) | \le |x^s-y^s| / (2C_1)\, ,
\\ \label{controleF4'}
&|DF^{(4)}(x^u,x^s)- DF^{(4)}(y^u, x^s)|\le  |x^u- y^u|^\beta /(4C_0^2C_1)\, ,
\\ \label{controleF4''}
\begin{split}
&| D F^{(4)}(x^u,x^s)- DF^{(4)} (x^u,y^s)-DF^{(4)}(y^u,x^s)+DF^{(4)}(y^u,y^s) |
\\& \hphantom{| DF^{(4)}(x^u,x^s) -DF^{(4)}(x^u,y^s) |}
\le |x^u-y^u|^\beta |x^s-y^s|^{1-\beta}/(2C_1)\, ,
 \end{split} \, .
\end{align}
In particular, since $\partial_x F^{(4)}(x^u,0)=\id$, the bound
\eqref{controleF4} implies that $\partial_{x^u} F^{(4)}$ is bounded
and has a bounded inverse on a ball of radius $C_1\ge 2C_0$.

In addition, linearity of $P'_0$ implies that
\begin{equation}
\label{controleF4tilde}
 | D\tilde F^{(4)}(x^s,x^0) -D\tilde F^{(4)}(y^s,y^0) |=0\, .
\end{equation}
\medskip

\emph{Last step: Pushing the foliation $\phi^{(4)}$ with
$\TT^{-1}\AAc$.} Let $\UU = \TT^{-1}\AAc$, and consider $\phi'$ the
foliation obtained by pushing  $\phi^{(4)}$ with
$\UU$. We claim that $\phi'$ belongs to $\FF(0,\CC^s, C_0,
C_1)$, and that we can write $\UU\circ \phi = \phi' \circ
\Psi'$ for some $\Psi'\in D_1^{2}(\Cs)$.

To prove this, we follow the arguments in the proof of Lemma
~\ref{lemcomposeQ} (with simplifications here since $\UU$ is
close to the identity, noting also that $\UU$ preserves the flow direction
so that $\UU(x^u, x^s, x^0)(U_u(x^u, x^s), U_s(x^u, x^s), U_0(x^u,x^s, x^0)$ with
$U_0(x^u, x^s, x^0)=x^0+u_0(x^u, x^s)$,
and in particular the property $\partial_{x^0}F'\equiv0$ will
be given for free). First, fix $x^u$, and consider the map
$L_{x^u} : x^s\mapsto \pi_2\circ \UU \circ \phi^{(4)}(x^u,x^s, x^0)$. Writing
$\UU=\id+\VV$ where $\norm{\VV}{C^{2}}\le \epsilon$, we
have $L_{x^u}(x^s)= x^s + \pi_2\circ \VV (F^{(4)}(x^u,x^s),x^s,\tilde F^{(4)}(x^s, x^0))$.
(Note that $L_{x^u}$ does not depend on $x^0$ because
$U_s$ doesn't.) Since
$F^{(4)}$ is bounded in $C^1$ on the ball $B(0,2C_1)$, it
follows that, if $\epsilon$ is small enough, then the
restriction of $L_{x^u}$ to the ball $B(0,2C_1)$ (in $\real^{d_s}$) is
arbitrarily close to the identity. Therefore, its inverse is
well defined, and we can set $\Gamma^{(0)}(x^u,x^s, x^0)=(x^u,
L_{x^u}^{-1}(x^s), x^0)$. By construction, the map $\UU\circ
\phi^{(4)}\circ \Gamma^{(0)}(x^u,0, x^0)$ has the form $(L^{(1)}(x^u),
0, x^0+ L^{(2)}_{x^u})$ for some function $L^{(1)}$, which is bounded in
$C^{1+\beta}$ and arbitrarily close to the identity in $C^{1}$
if $\epsilon$ is small, and
some function $L^{(2)}(x^u)$ which is bounded in
$C^{1+\beta}$ and arbitrarily close to zero in $C^{1}$,
if $\epsilon$ is small. Let
$\Gamma^{(1)}(x^u,x^s, x^0)=((L^{(1)})^{-1}(x^u), x^s, x^0-L^{(2)}(x^u))$, then the map
$\phi'= \UU \circ \phi^{(4)}\circ \Gamma^{(0)}\circ
\Gamma^{(1)}$ is defined on the set $\{(x^u,x^s, x^0) \st |x^s|\le C_1\}$
(which contains $B(0,C_0)$), and it takes the form
$\phi'(x^u,x^s,x^0)=(F'(x^u,x^s), x^s, x^0+\tilde f'(x^u, x^s))$ for some functions $F'$, $\tilde f'$, with
$F'(x^u,0)=x^u$ and $\tilde f'(x^u,0)=0$. 

Since $\phi'$ is obtained by composing $\phi^{(4)}$ with
diffeomorphisms arbitrarily close to the identity, it follows
from \eqref{controleF4}--\eqref{controleF4''} that $F'$
satisfies \eqref{smooths}--\eqref{smoothsecond} and
from \eqref{controleF4} (recall that
$\tilde F^4$ does not depend on $x^u$) that $\tilde f'$ satisfies
\eqref{smoothtilde}.  Indeed, the present
analogue of \eqref{eqdphim1} is
\begin{equation*}
D \phi'  ( D\UU\circ \phi^{(4)} \circ \Gamma^{(0)}\circ \Gamma^{(1)})
\cdot (D\phi^{(4)}\circ \Gamma^{(0)}\circ \Gamma^{(1)})
\cdot( D\Gamma^{(0)} \circ \Gamma^{(1)})
\cdot D\Gamma^{(1)}\, ,
\end{equation*}
where $\Gamma^{(0)}$ and $\Gamma^{(1)}$ here satisfy the same
properties as the maps with the same names in the proof of
Lemma ~\ref{lemcomposeQ}, and where $D\UU \circ \phi^{(4)}
\circ \Gamma^{(0)}\circ \Gamma^{(1)}$ is bounded and
Lipschitz and thus belongs to $\KK$. We may thus argue exactly
as in the last step of the proof of Lemma ~\ref{lemcomposeQ}
for $F'$, while the case of $\tilde f'$ is easier.

Moreover, since
$(\partial_{x^s} F^{(4)}(z)w,w, \partial_{x^s} \tilde F^{(4)} w)$ takes its values in the cone
$\CC^s_1$, it follows that $(\partial_{x^s} F'(z)w,w,\partial_{x^s} 
\tilde f'(z))$ lies in the
cone $\CC^s$ if $\UU$ is close enough to the identity. Hence,
the foliation defined by $\phi'$ is contained in $\CC^s$. This
shows that $\phi'$ belongs to $\FF(0,\CC^s, C_0, C_1)$.

Finally, the function $\Psi = (\Gamma^{(0)}\circ
\Gamma^{(1)})^{-1}$ belongs to $D_1^{2}(\Cs)$. This
concludes the proof of Lemma~\ref{lemcompose}.
\end{proof}


\section{Approximate unstable foliations}\label{sec:unstable}
Let us consider $\bar x\in U_{i,j,\ell,1}$ and $\varrho$ so small that $B(\bar x, \varrho)\subset U_{i,j,\ell,1}\cap B_{i,j}$.\footnote{By $B(\bar x, \varrho)$ we mean the ball of radius $\varrho$ centered at $\bar x$.} First we describe all the objects in $B(\bar x, \varrho)$ by using the chart $\kappa_{i,j,\ell}$. 

Let $W$ be a surface with curvature bounded by some fixed constant $\Cs$ such that $\bar x\in W$, $\partial(W\cup B(\bar x, \varrho))\subset  \partial B(\bar x, \varrho)$ and with tangent space, at each point,  given by the span of the flow direction and a vector, in the kernel of $\alpha$, contained in the stable cone. Recall that in the present coordinates the contact form has the expression $\alpha_0= dx^0-x^sdx^u$. Note that almost every point in $W$ has a well-defined unstable direction. We can thus assume without loss of generality that the unstable direction is well defined at $\bar x$. By Lemma \ref{lem:c2} we can then change coordinates so that $\bar x=0$, $W=\{(0,\xi,s)\}_{\xi,s\in\bR}$ and the unstable direction at $\bar x$ is given by $(1,0,0)$. From now on we will work in such coordinates without further mention.

Our goal is to define smooth approximate strong unstable foliations in a $c\rho<\varrho$ neighbourhood of $W$.  

More precisely, we look for $C^{1+Lip}$ foliations $\Gamma_m$, described by the triangular change of coordinates $\bG_m(\eta,\xi,s)=(\eta, G_m(\eta,\xi), H_m(\eta,\xi)+s)$, with domains $\Delta_m\subset \{\xi\in\bR^2\;;\; \|\xi\|\leq \varrho\}$ and constants $c,\ho>0\,, \sigma\in (0,1)$, $m_{0}\in\bN$ such that
\begin{enumerate}
\item\label{it:1} $G_m(0,\xi)=\xi, H_m(0,\xi)=0$\,;
\item\label{it:2} $\partial_\eta H_m=G_m$ (i.e., $\alpha_0(\partial_\eta\bG)=0$)\,;
\item\label{it:3} for all $m_{0}\geq m'\geq m$, $\Delta_{m'}\subset \Delta_m$\,;
\item\label{it:4} for all $m_{0}\geq m'\geq m$ and $\xi\in\Delta _{m'}$, $\|G_m(\xi,\cdot)- G_{m'}(\xi,\cdot)\|_{C^1}\leq c \sigma^{m}$, $\|H_m(\xi,\cdot)- H_{m'}(\xi,\cdot)\|_{C^1}\leq c \sigma^{m}$\,;
\item\label{it:5} $\|\partial_\xi\partial_\eta \bG_m\|_\infty+\|\partial_s\partial_\eta \bG_m\|_\infty\leq c \rho^{-1}$\,;
\item\label{it:6} $\|\partial_\xi\partial_\eta \bG_m\|_{C^\ho}+\|\partial_s\partial_\eta \bG_m\|_{C^\ho}\leq c \rho^{-1-\ho}$\,;
\item\label{it:8} if $\xi\in \Delta_m$, then  $\partial_\eta T_{-\vus n}\bG_m(\eta, \xi)$ is well defined and belongs to the unstable cone for all $n\leq m\leq m_0$ and $\|\eta\|\leq c\rho$\,;
\item\label{it:9} $m(\Delta_m^c\cap W)\leq c \rho$, for all $m\leq m_{0}$.
\end{enumerate}
\begin{remark} Note that the above properties are not all independent, we spelled them out in unnecessary details for the reader's convenience. In particular, (\ref{it:1}) is just a condition on the parametrisation used to describe the foliation and can be assumed without loss of generality; (\ref{it:3}, \ref{it:8}) imply (\ref{it:4}) due the the usual cone contraction of hyperbolic dynamics.
\end{remark}

\begin{lemma} \label{lem:unstable-fol} There exists $\ho_{0}>0$ such that, provided $\varrho>\rho>0$, $\ho\in [0,\ho_{0}]$,\footnote{In fact,  it should be possible to have the result for each $\ho\in[0,1)$, but  this is not necessary for our purposes.} for each $m_{0}\in \bN$, there exists at least one set of foliations $\Gamma_m$, $m\in \{0,\dots, m_{0}\}$, in a $\rho$-neighbourhood of $W$, which satisfies the properties (\ref{it:1}-\ref{it:9}).
\end{lemma}
\begin{proof}
First of all, notice that once we construct the foliation over $W_0=\{(0,\xi,0)\}_{\xi\in\bR}$ we can obtain the foliation on $W$ by simply flowing it with the dynamics and $G, H$ will have automatically the wanted $s$ dependence. Contrary to the notation in section \ref{carlangelo}, we will call $\cW_{m}$ the set of regular manifolds obtained by $W_0$ under backward iteration by $T_{\vus}$.

Second, for each $\ho\in (0,1)$, $K>0$, there exists $c>0$ and $\cu, \cu_{1}\in (0,1)$, $\cu_{1}<\cu$, close enough to one, with the following properties. Consider any foliation in a $\cu^m\rho$ neighbourhood of $\cW_m$ aligned with the unstable cone, with fibers in the kernel of the contact form and with $C^1$ and $C^{1+\ho}$ norm bounded by $Kc\cu^{-m}\rho^{-1}$ and $Kc\cu^{-m(1+\ho)}\rho^{-1-\ho}$ respectively (in the sense of conditions (\ref{it:5}, \ref{it:6}) above). Then the image under any $T_{k\vus}$ provides a foliation in the $\cu^{m-k}\rho$ neighbourhood of $\cW_{m-k}$ satisfying the same conditions of $\cW_m$ but with bounds $c\cu^{-m}\cu_{1}^k\rho^{-1}$ and $c\cu^{-m(1+\ho)}\cu_{1}^{k(1+\ho)}\rho^{-1-\ho}$, provided $\rho$ is chosen small enough and $c$ large enough. This follows by standard distortion estimates as in the construction of stable manifolds for a smooth Anosov map, see \cite{KH}. We will call such foliations {\em allowed.}

Our strategy will be as follows: For each given allowed foliation $\Gamma^*_m$ defined in a $\cu^m\rho$ neighbourhood of $\cW_m$ we will show how to construct foliations $\Gamma_n$ represented by coordinates $\bG_n$, $n\leq m$, satisfying conditions (\ref{it:1}-\ref{it:9}). The problem, of course, is what to do with the fibers that are cut by a singularity.

As a first, very rough, approximation of the unstable foliation, let us choose in each $V_{i,j,\ell,0}=\kappa_{i,j,\ell}(U_{i,j,\ell,0})$ the foliation given by the leaves $\{\eta, \xi,s+\xi\eta\}_{\eta\in\bR}$. We call $\Gamma$ the resulting set of foliations, note that the tangent space to the leaves belongs the unstable cone and to the kernel of the contact form.~\footnote{By, if necessary, restricting the chart, we can assume without loss of generality that the unstable cone is given by the condition $\|\xi\|+\|s\|\leq K_{0}\|\eta\|$ for some $K_{0}$ small.}

Next, we proceed by induction. Given $\Gamma^*_0$ we set $\Delta_0=W_0$ and chose $\Gamma^*_0$ itself as foliation. All the non vacuous conditions are then satisfied. Next, suppose we have defined a construction of the foliation for all $n<m$ and we are given an allowed foliation $\Gamma^*_m$ in the neighbourhood of $\cW_m$. Note that by the transversality condition on the singularities, there exists $c_*>1$ such that any fiber of $T_\vus\Gamma^*_m$ which has been cut by $\partial B_{i,j}$ and intersects $\cW_{m-1}$ must intersect $T_\vus(\partial_{c_*\cu^{m}\rho} \cW_m)$. We define ${\overline W}_m=\cup_{W\in\cW_m}W$ and $S_m$ to be the union of the elements of $\cW_m$ shorter than $2c_*\cu^m\rho$.

We are now ready to define an allowed foliation $\Gamma^*_{m-1}$. On the set ${\overline W}_{m-1}\setminus (T_\vus(\partial_{2c_*\cu^m\rho}\cW_m))$ it is simply given by $T_\vus\Gamma^*_m$. On $(T_\vus(\partial_{c_*\cu^m\rho}\cW_m))\cup T_\vus S_m$ we define it to be $\Gamma$.\footnote{Note that, by choosing $L_0$ in Lemma  \ref{lem:r-boundary}, we can always assume that each element of $\cW_m$ is contained in some $U_{i,j,\ell,1}$, thus one can use the corresponding element of $\Gamma$.} 
Inside small intervals at whose boundaries the foliation has
now been defined,  we must interpolate.
To do so precisely, it is best to write explicitly the objects of interest.

Let $(f(\xi),\xi, g(\xi))_{\xi\in [a',b']}$ be the graph of
an element of $\cW_{m-1}$ and $[a,b]$, $a'<a<b<b'$, the interval on which we want to define the interpolating foliation.\footnote{We can assume without loss of generality that the foliation has been already defined in $[a',a]\cup[b,b']$ and satisfies the wanted properties.} By construction there exist fixed constants $c_-, c_{+}>0$ such that $c_+\rho\cu^{m-1}\geq |b-a|\geq c_-\rho\cu^{m-1}$. Let $F(\xi)=(f(\xi),\xi, g(\xi))$ and $\gamma_a(\eta,\xi)=F(\xi)+(\eta, \sigma_a(\eta,\xi), \zeta_a(\eta,\xi))$, $\gamma_b(\eta,\xi)=F(\xi)+(\eta, \sigma_b(\eta,\xi), \zeta_b(\eta,\xi))$, $|\eta|\leq\cu^{m-1} \rho$, a parametrisation of the two foliations we must interpolate. Clearly, we can assume without loss of generality $\sigma_a(0,\xi)=\sigma_b(0,\xi)=\zeta_a(0,\xi)=\zeta_b(0,\xi)=0$. By construction the above curves are in the kernel of the contact form, that is $\partial_\eta \zeta_a(\eta,\xi)=\xi+\sigma_a(\eta,\xi)$, $\partial_\eta \zeta_b(\eta,\xi)=\xi+\sigma_b(\eta,\xi)$, and in the unstable cone, that is $|\partial_\eta\sigma_a|+|\partial_\eta\zeta_a|\leq\cu K_0$ and $|\partial_\eta\sigma_b|+|\partial_\eta\zeta_b|\leq \cu K_0$.\footnote{The $\cu$ comes from the fact that the foliation is either an image of a foliation  already in the unstable cone or is a fixed foliation well inside the cone.} In addition, since $\Gamma^{*}_{m}$ is allowed,
\[
\begin{split}
&|\partial_\xi\partial_\eta\sigma_a|_{C^0}\leq c\cu^{-m}\cu_{1} \rho^{-1}, \quad\quad |\partial_\xi\partial_\eta\sigma_a|_{C^\ho}\leq c[\cu^{-m}\cu_{1} \rho^{-1}]^{1+\ho}\\
&|\partial_\xi\partial_\eta\sigma_b|_{C^0}\leq c\cu^{-m}\cu_{1} \rho^{-1}, \quad\quad |\partial_\xi\partial_\eta\sigma_b|_{C^\ho}\leq c[\cu^{-m}\cu_{1} \rho^{-1}]^{1+\ho}\, .
\end{split}
\]

Next, fix once and for all $\bar\vf, \bar\psi\in C^2(\bR,[0,1])$,  such that $\bar \vf(0)=\partial_\xi\bar\vf(0)=\partial_\xi\bar\vf(1)=\bar\psi(0)=\bar\psi(1)=0$, $\bar\vf(1)=1$, $\int_0^1\bar\psi(\xi)=1-c_1$, $c_1$ to be chosen later small enough. Define then $\vf(\xi)=\bar\vf(\frac{\xi-a}{b-a})$ and $\psi(\xi)=\bar\psi(\frac{\xi-a}{b-a})$. Clearly, $\vf(a)=\partial_\xi\vf(a)=\partial_\xi\vf(b)=\psi(a)=\psi(b)=0$, $\vf(b)=1$, $\int_a^b\psi(\xi)=(1-c_{1})(b-a)$. Define 
\[
\begin{split}
&\theta_0(\eta,\xi)=\left(\partial_\xi\sigma_b(\eta,b) \frac{\xi-a}{b-a}+\partial_\xi\sigma_a(\eta,a) \frac{b-\xi}{b-a}\right)(1-\psi(\xi))\\
&\theta(\eta,\xi)=\theta_0(\eta,\xi)-\frac 1{(1-c_{1})(b-a)}\psi(\xi)\int_a^b\theta_0(\eta,z)dz\\
&\sigma(\eta,\xi)=\sigma_b(\eta,b)\vf(\xi)+\sigma_a(\eta,a)(1-\vf(\xi))+\int_a^\xi\theta(\eta,z)dz\, ,
\end{split}
\]
for all $\|\eta\|\leq \cu^{m-1} \rho$ and $\xi\in(a,b)$. Next, define $\zeta(\eta,\xi)=\xi\eta+\int_0^\eta \sigma(z,\xi)dz$. We can then define $\Gamma^*_{m-1}$ for $\xi\in[a',b']$ as the foliation with fibers
\[
\gamma(\eta,\xi)=F(\xi)+\begin{cases} \gamma_a(\eta,\xi)\quad &\text{for }\xi\in [a',a]\\
(\eta,\sigma(\eta,\xi),\zeta(\eta,\xi))\quad &\text{for }\xi\in (a, b)\\
\gamma_b(\eta,\xi)\quad &\text{for }\xi\in [b,b']\, .
\end{cases}
\]
The definition of $\zeta$ implies that the leaves of $\gamma$ are in the kernel of the contact form.
Moreover, note that $\sigma(\eta,a)=\sigma_{a}(\eta,a)$ and
\[
\sigma(\eta,b)=\sigma_{b}(\eta,b)+\int_{a}^{b}dz\theta_{0}(\eta,z)\left[1-\frac 1{(1-c_{1})(b-a)}\int_{a}^{b}dw\psi(w)\right]
=\sigma_{b}(\eta,b)\, .
\]
From this follows $\gamma\in C^{0}$. In addition, $\partial_{\eta}\sigma(\eta,a)=\partial_{\eta}\sigma_{a}(\eta,a)$ and, since
\[
\begin{split}
&\partial_{\eta}\theta_{0}(\eta,\xi)=\left(\partial_{\eta}\partial_\xi\sigma_b(\eta,b) \frac{\xi-a}{b-a}+\partial_{\eta}\partial_\xi\sigma_a(\eta,a) \frac{b-\xi}{b-a}\right)(1-\psi(\xi))\\
&\partial_{\eta}\theta(\eta,\xi)=\partial_{\eta}\theta_0(\eta,\xi)-\frac 1{(1-c_{1})(b-a)}\psi(\xi)\int_a^b\partial_{\eta}\theta_0(\eta,z)dz\\
&\partial_{\eta}\sigma(\eta,\xi)=\partial_{\eta}\sigma_b(\eta,b)\vf(\xi)+\partial_{\eta}\sigma_a(\eta,a)(1-\vf(\xi))+\int_a^\xi\partial_{\eta}\theta(\eta,z)dz\, .
\end{split}
\]
It follows that $\partial_{\eta}\gamma\in C^{0}$ and a similar computation shows $\partial_\xi\gamma\in C^0$. Since all the quantities are piecewise $C^2$, it follows that $\gamma\in C^{1+\text{Lip}}$. In addition,
\[
\begin{split}
\|\partial_{\eta}\sigma\|_{L^{\infty}}&\leq \cu K_{0}+c\cu^{-m+1}\rho^{-1}\left[\int_{a}^{b}(1-\psi)+\int_{a}^{b}\frac{\psi}{(1-c_{1})(b-a)}\int_{a}^{b}(1-\psi)\right]\\
&\leq \cu K_{0}+2c c_{+}c_{1} \, .
\end{split}
\]
Which, by choosing $c_{1}$ small enough, ensures that the fibers of $\gamma$ belong to the unstable cone.
Finally, we have,
\[
\partial_{\xi}\partial_{\eta}\sigma(\eta,\xi)=\left(\partial_{\eta}\sigma_b(\eta,b)-\partial_{\eta}\sigma_a(\eta,a)\right)\vf'(\xi)+\partial_{\eta}\theta(\eta,\xi)\, ,
\]
which implies $\partial_{\xi}\partial_{\eta}\sigma(\eta,a)=\partial_{\xi}\partial_{\eta}\sigma_{a}(\eta,a)$ and $\partial_{\xi}\partial_{\eta}\sigma(\eta,b)=\partial_{\xi}\partial_{\eta}\sigma_{b}(\eta,b)$ and also $\partial_{\eta}\partial_{\xi}\gamma\in C^{0}$. The last estimate is 
\begin{align*}
\|\partial_{\xi}\partial_{\eta}\sigma\|_{L^{\infty}}
&\leq 2K_{0}\|\vf'\|_{L^{\infty}}+\|\partial_{\eta}\theta(\eta,\xi)\|_{L^{\infty}}\\
&\leq 2K_{0}\|\bar \vf'\|_{L^{\infty}}c_-^{-1}\rho^{-1}\cu^{-m+1}+ c\cu^{-m}\cu_{1}\rho^{-1}\, ,
\end{align*}
which yields $\|\partial_{\xi}\partial_{\eta}\sigma\|_{L^{\infty}}\leq c\cu^{-m+1}\rho^{-1}$, provided $c$ is large enough.
Similar computations verify the $ C^{\ho}$ bounds, provided $\ho>0$ is chosen small enough.

In other words $\gamma$ is an allowed foliation.

By the inductive hypotheses we can then take $\Gamma^*_{m-1}=\gamma$ as the starting point to construct foliations $\tilde \bG_n^{m}$ and domains $\Delta_n$, for $n<m-1$, satisfying hypotheses (1-8). We then define the domain $\Delta_m=T_{(m-1)\vus}\left[T^{-(m-1)\vus}\Delta_{m-1}\setminus T_\vus(\partial_{2c_*\cu^m\rho}\cW_m)\right]$.

To conclude we define $\bG_{m}$, $m\leq m_{0}$, to be the foliations $\tilde \bG^{m_{0}}_n$ obtained by the above procedure when starting from the initial allowed foliation $\Gamma$. It is immediate to check that the above construction satisfies properties (\ref{it:1}-\ref{it:8}). Property (\ref{it:9}) follows by noticing that $\Delta_m^c\subset \cup_{n\leq m}T_{n\vus}(\partial_{2c_*\cu^n\rho\cW_n})$ and by applying Lemma \ref{lem:r-boundary} with $\delta=\varrho$, $r=2c_*\rho$ and $\vartheta=\cu$.  
\end{proof}

\begin{remark} \label{rem:lip} Note that, given two manifolds of size $\varrho$ uniformly $C^2$ and uniformly transversal to the unstable direction having a distance less than $\rho$, one can use the above lemma to construct a uniformly Lipschitz holonomy, approximating the unstable one, between the two manifolds.\footnote{Indeed $\partial_\eta\bG$ provides a Lipschitz vector field, hence one can consider the associated flow, the holonomy map is nothing else than the Poincar\'e map between the two manifolds. It follows then by standard computations that the  the time taken by the flow to go between the two surfaces is Lipschitz and bounded by a fixed constant times $\rho$. Finally,  (5) readily implies that the holonomy is Lipschitz as well.}
\end{remark}

\section{Cancellations estimates}\label{sec:hoihoi}

In this appendix we detail the basic, but technical, cancellations estimates in the Dolgopyat argument.
\begin{proof}[{\bfseries Proof of Lemma} {\normalfont \ref{lem:cancel}}]
To start with, let us obtain a formula for $\omega_\alpha$, this is the analogous of the formulae in \cite{Li, KB} adapted to the present context. By translating $W_\alpha$ along the flow direction we can assume, without loss of generality, that the leaf of $\Gamma_{i,j,r}^\up$ starting from $\bar x$ intersects $W_\alpha$. Let $\tilde W_\alpha=\Theta_{i,\alpha,\up} W_\alpha\subset W^i_*$ and consider the path $\gamma(\xi)$ running  from $\bar x$ along the leaf of $\Gamma_{i,j,r}^\up$ up to $W_\alpha$, then along $W_\alpha$ up to $\bar x+(F_\alpha(\xi),\xi,N_\alpha(\xi))$ and back to $W_*^i$ along the leaf of $\Gamma_{i,j,r}^\up$ again, then to the axis $(\eta,\xi, s)=(\bar x^u,\xi,\bar x^0)$ along the flow direction and finally back to $\bar x$ along the axis, (see Figure \ref{fig:omega} for a pictorial explanation). By construction,
\[
\begin{split}
\bar\omega_\alpha(\xi)=&\int_{\gamma(\xi)}\alpha=\int_{\Sigma_\alpha(\xi)}d\alpha+\int_{\Omega_\alpha(\xi)}d\alpha=\int_{\Sigma_\alpha(\xi)}d\alpha\\
=&\int_0^\xi dz\int_0^{F_\alpha(z)}d\eta\,\partial_\xi G_{i,j,\up}(\eta, h_\alpha(z))h'_\alpha(z)\\
=&\int_0^{h_\alpha(\xi)}\!\!\!\! dz\int_0^{F_\alpha\circ h^{-1}_\alpha(z)}\!\!\!\!d\eta\,\partial_\xi G_{i,j,\up}(\eta,z)\, ,
\end{split}
\]
where we have used the fact that all the curve $\gamma(\xi)$, apart for the piece in the flow direction, is in the kernel of $\alpha$ and  Stokes' Theorem. Moreover, $\Sigma_\alpha(\xi)$ is the surface traced by the fibers of $\Gamma^\up_{i,r}$ while moving along $W_\alpha$ up to $\xi$ and $\Omega_\alpha(\xi)$  the portion of $W^i_*$ between the $\xi$ axes and the curve $\Theta_{i,\alpha,\up}(W_\alpha)$ up to $h_\alpha(\xi)$.
\begin{figure}[ht]\ 
\begin{tikzpicture}
\draw (-2,0)--(0,0) node[right=1pt, yshift=-6pt]{$\bar x$}--(4,0)  node[below=13pt]{$W_*^i$}--(10,0) node[left=72pt, yshift=6pt] {$\scriptstyle\Omega_\alpha(\xi)$};
\draw(0,-2)--(0,2) node[left=2pt]{$\Gamma^\up_{i,r}(0,0)$}--(0, 6);
\draw(-2,-2)--(3,3) node[right=12pt] {$\Sigma_\alpha(\xi)$};
\draw[line width=1pt] (0,4).. controls (3.5, 4.8).. (6.6,4.4) node[above=2pt]{$W_\alpha$} .. controls (7.9, 4.4) .. (9.6, 4.7); 
\draw (9.6,4.7) node[left=55pt, below=40pt]{$\Gamma^\up_{i,r}(h_\alpha(\xi),\bar\omega_\alpha(\xi))$}..controls (9.8, 2.5).. (9.5,0.5);
\draw[->](9.2,3)--(9.7,3);
\draw[line width=1pt, dash pattern=on 2 pt  off 1 pt] (0,0)..controls (4.5, .4) ..(6, .4) node[above=28pt, xshift=-48pt]{$\Theta_{i,\alpha,\up}(W_\alpha)$} ..controls (6.3, .4) ..(9.5,.5);
\draw[->] (4.5,1.4)--(5.5,.45);
\draw[dash pattern=on 1 pt  off 1 pt](9,0) node[below=0pt,xshift=-6pt]{$\scriptscriptstyle h_\alpha(\xi)$}--(9.5, 0.5);
\draw[dash pattern=on 1 pt  off 1 pt](9.6,4.7)--(9.6,.4);
\draw[dash pattern=on 1 pt  off 1 pt](9.6,.4)--(9.2,0) node[below=0pt]{$\scriptscriptstyle \xi$};
\draw[dash pattern=on 1 pt  off 1 pt](9.5, 0.5)--(0.5,0.5) node[left=-2pt] {\mbox{\boldmath$\scriptstyle \bar\omega_\alpha(\xi)$}};
\pgfsetfillopacity{0.2}
\draw[fill=gray] (-3,-1)--(-1,1)--(-1,1)--(11,1)--(9,-1)--(-3,-1);
\pgfsetfillopacity{0.4}
\path[fill=gray](0,0)..controls (4.5, .4) ..(6, .4)..controls (6.3, .4) ..(9.5,.5)--(9,0)--(0,0);
\end{tikzpicture}
\caption{Definition of $\bar\omega_\alpha$.}
\label{fig:omega}
\end{figure}
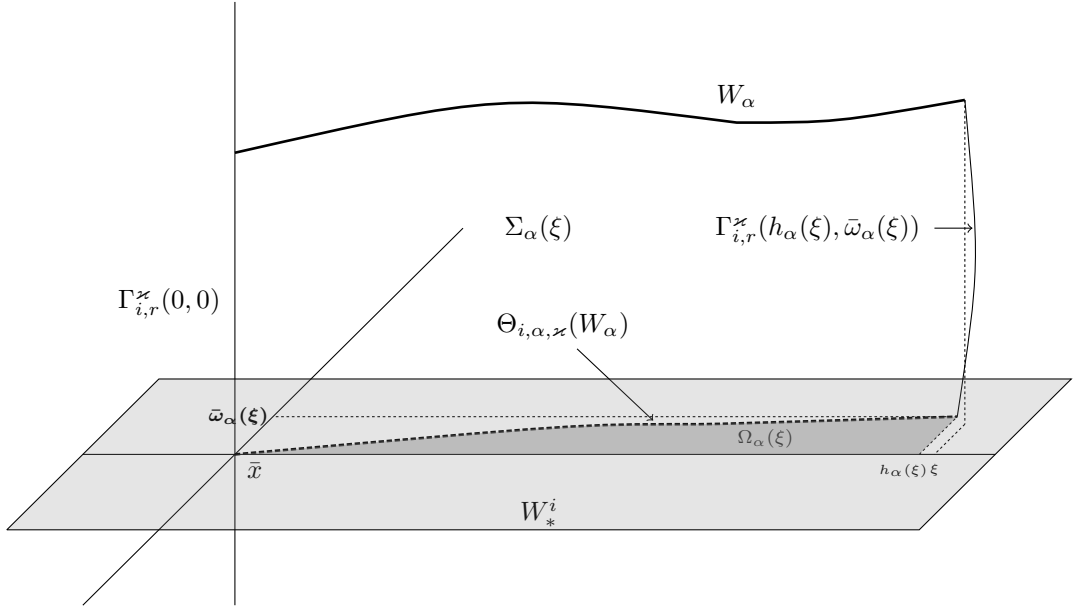

By Lemma \ref{lem:c2} we can assume without loss of generality that $\bar x=0$.
Using properties \eqref{it:1} and \eqref{it:5} of the approximate foliations in appendix \ref{sec:unstable}, it follows that 
\begin{equation}\label{eq:o-ab}
\begin{split}
\partial_\xi \omega_\beta(\xi)-\partial_\xi \omega_\alpha(\xi)&=\int_{F_\alpha\circ h^{-1}_\alpha(\xi)}^{F_\beta\circ h^{-1}_\beta(\xi)}\!d\eta\;\partial_\xi G_{i,j,\up}(\eta,\xi)\\
&=\int_{F_\alpha\circ h^{-1}_\alpha(\xi)}^{F_\beta\circ h^{-1}_\beta(\xi)}d\eta\left[ 1+\int_0^\eta dz\partial_\xi \partial_zG_{i,j,\up}(z,\xi)\right]\\
&=\left[F_\beta\circ h^{-1}_\beta(\xi)-F_\alpha\circ h^{-1}_\alpha(\xi)\right](1+\cO(r^{1-\varsigma}))\, .
\end{split}
\end{equation}
To continue note that, in analogy with \eqref{eq:h-a} we have 
\begin{equation}\label{eq:h-ab}
\begin{split}
\left|h^{-1}_\beta(\xi)-h^{-1}_\alpha(\xi)\right|&=\left|\int_{F_\alpha\circ h^{-1}_\alpha(\xi)}^{F_\beta\circ h^{-1}_\beta(\xi)}\!\!\!\!\!\!\!dz\int_0^\xi dw\; \partial_z\partial_wG_{i,\up}(z,w)\right|\\
&\leq \Cs r^{\theta-\varsigma}\left|F_\beta\circ h^{-1}_\alpha(\xi)-F_\alpha\circ h^{-1}_\beta(\xi)\right|\, .
\end{split}
\end{equation}
To conclude recall that we are working in coordinates in which $|F'_{\alpha}|\leq \Cs r^{1-\theta}$, cf. the proof of Lemma \ref{lem:smallbox}, hence
\[
\begin{split}
\left|F_\beta\circ h^{-1}_\beta-F_\alpha\circ h^{-1}_\alpha\right|&\geq\left| F_\beta\circ h^{-1}_\beta-F_\alpha\circ h^{-1}_\beta\right|-\Cs r^{1-\theta}|h^{-1}_\beta-h^{-1}_\alpha|\\
&\geq\left| F_\beta\circ h_{\beta}^{-1}-F_\alpha\circ h_{\beta}^{-1}\right|-\Cs r^{1-\varsigma}\left|F_\beta\circ h^{-1}_\beta-F_\alpha\circ h^{-1}_\alpha\right|
\end{split}
\]
which, together with \eqref{eq:o-ab}, proves \eqref{eq:below-o} provided we can show that the distance between the manifolds is comparable with $|F_\beta-F_\alpha|$ at any point. To prove such a fact recall that in \eqref{eq:cutcut}, and following lines, we have seen that 
\[
|F_\beta'(\xi)-F_\alpha'(\xi)|\leq \Cs c^{-1}r^{-\theta}|F_\beta(\xi)-F_\alpha(\xi)|\, ,
\]
provided that 
\[
\blambda^{2\ell}\geq \Cs r^\theta |F_\beta(\xi)-F_\alpha(\xi)|^{-1}\geq \Cs r^{\theta-\vartheta}\, ,
\]
which is ensured by our assumptions.
Thus setting $d_{\alpha,\beta}(\xi)=|F_\alpha(\xi)-F_\beta(\xi)|$ we have
\[
\frac d{d\xi} d_{\alpha,\beta}(\xi)\leq \Cs c^{-1}d_{\alpha,\beta}(\xi)r^{-\theta}\, ,
\] 
Hence, by Gronwall's lemma, $|d_{\alpha,\beta}|_\infty\leq \Cs d_{\alpha,\beta}(0)$ and 
\begin{equation}\label{eq:dab}
|d_{\alpha,\beta}(\xi)-d_{\alpha,\beta}(0)|\leq \frac 12d_{\alpha,\beta}(0)\, ,
\end{equation}
for all $|\xi|\leq r^\theta$, provided $c$ has been chosen large enough.
Next, remember that $H_{i,j,\up}'=G_{i,j,\up}$, thus  $|H_{i,j,\up}'|_\infty\leq \Cs r$ and
\[
|N_\alpha(0)-N_\beta(0)|= |H_{i,j,\up}(F_\alpha(0),0)-H_{i,j,\up}(F_\beta(0),0)|\leq\Cs r|F_\alpha(0)-F_\beta(0)|\, .
\]
In addition, since $N_\alpha'(\xi)=\xi F_\alpha(\xi)$, it follows that $|N_\alpha-N_\beta|_\infty\leq\Cs r|F_\alpha-F_\beta|_\infty$.
Which proves
\begin{equation}\label{eq:okone}
\partial_\xi \omega_\beta(\xi)-\partial_\xi \omega_\alpha(\xi)\geq \Cs d_{i,j}(W_{\alpha},W_{\beta})\, .
\end{equation}

To prove the second statement, let us introduce $\omega_{\alpha,\beta}(\xi)=\omega_\alpha(\xi)-\omega_\beta(\xi)$ and $A_{\alpha,\beta}=\frac{[(m-1)!]^2}{(\ell\vu)^{2m-2}}{\bf F}^*_{\ell,m,i,\alpha}\overline{{\bf F}^*_{\ell,m,i,\beta}}$.  Next we introduce a sequence $a_i$, $a_0=-cr^\theta$, such that $\partial_\xi\omega_{\alpha,\beta}(a_i)(a_{i+1}-a_i)=2\pi b^{-1}$ and let $M\in\bN$ be such that $a_M\leq c r^\theta$ and $a_{M+1}> cr^\theta$. Also, we establish the following notation: $\delta_i=a_{i+1}-a_i$. 
By Lemma \ref{lem:hpr} it follows that
\[
|\omega_{\alpha,\beta}(\xi)-\omega_{\alpha,\beta}(a_i)-\partial_\xi\omega_{\alpha,\beta}(a_i)(\xi-a_i)|\leq \Cs \delta_i^{1+\ho}\, .
\]
In addition, equations \eqref{eq:Xi}, \eqref{eq:defFs}, and Lemma \ref{lem:hpr} imply
\[
|A_{\alpha,\beta}(\xi,s)-A_{\alpha,\beta}(a_i,s)|\leq \Cs \{\delta_i^\ho r^{-2\theta}+\delta_ir^{-3\theta}\}\, .
\]
Then, remembering \ref{eq:dab} and using the first part of the lemma,
\[
\begin{split}
&\left|\int_{a_i}^{a_{i+1}} e^{-ib\omega_{\alpha,\beta}(\xi)} A_{\alpha,\beta}(\xi,s) d\xi\right|\\
&\quad=\left|
\int_{a_i}^{a_{i+1}}  e^{-ib[\partial_\xi\omega_{\alpha,\beta}(a_i)\xi+\cO(\delta_i^{1+\ho})]} [A_{\alpha,\beta}(a_i)+\cO((\delta_i^{\ho}      r^{-2\theta}+r^{-3\theta}\delta_i)] d\xi\right|\\
&\quad\leq\Cs\left(b\delta_i^{1+\ho}r^{-2\theta}+\delta_{i}^{\ho}r^{-2\theta}+r^{-3\theta}\delta_i\right)\delta_i\\
&\quad\leq\Cs\left(\frac{r^{-2\theta}}{d_{i,j}(W_\alpha,W_\beta)^{1+\ho} b^\ho}+\frac{r^{-3\theta}}{d_{i,j}(W_\alpha,W_\beta) b}\right)\delta_i \, .
\end{split}
\]
We may be left with the integral over the interval $[a_M,cr^\theta]$ which is trivially bounded by $\Cs r^{-2\theta}\delta_M\leq \Cs [r^{2\theta}bd_{i,j}(W_\alpha, W_\beta)]^{-1}$.
The statement follows since the manifolds we are considering have length at most $cr^\theta$, hence $\sum_{i=0}^{M-1}\delta_i \leq cr^\theta$.
\end{proof}

\begin{proof}[{\bfseries Proof of Lemma} {\normalfont \ref{lem:disco}}]
We start by introducing a function $\bar R:W^s_{\delta, \zeta}\to\bN$ such that $\bar R(\xi)$ is the first $t\in\bN$ at which $T_{-t\vu}\xi$ belongs to a regular component of $T_{-t\vu}W^s_{\delta, \zeta}$ of size larger than $L_{0}/2$. We define then $R(t)=\min\{\bar R(t),\ell\}$. Let $\cP=\{J_{i}\}$ be the coarser partition of $W^s_{\delta, \zeta}$ in intervals on which $R$ is constant. Note that, for each $W_{\beta,i}$, $T_{\ell\vu}W_{\beta,i}\subset J_{j}$ for some $J_{j}\in\cP$.

Let $\Sigma_{\ell,j}=\{(\beta,i)\;:\;W_{\beta,i}\in D_{\ell}(\widetilde O,\rho_{*}),\; T_{\ell\vu}W_{\beta,i}\subset J_{j}\}$. Then, by the usual distortion estimates, for each $(\beta,i)\in \Sigma_{\ell,j}$
\[
\begin{split}
Z_{\beta,i}&\leq \Cs\int_{W_{\beta,i}}\delta^{-1}J^{s}T_{\ell\vu}\leq \Cs\delta^{-1}\frac{|J_{j}|}{|\bar W_{j}|}\int_{W_{\beta,i}} J^{s}T_{(\ell-R_{j})\vu}\\
&\leq \Cs\delta^{-1}\frac{|J_{j}|}{|\bar W_{j}|}|T_{(\ell-R_{j})\vu}W_{\beta,i}|\,,
\end{split}
\]
where $R_{j}=R(J_{j})$ and $\bar W_{j}=T_{-R_{j}\vu}J_{j}$. Note that, by construction, either $|\bar W_{j}|\geq L_{0}/2$ or $R_j=\ell$.

Let us analyze first the case in which $R_j<\ell$.
We apply Lemma \ref{lem:unstable-fol} to $\bar W_{j}$, with $\varrho=|\bar W_j|\geq L_{0}/2$, in order to obtain a foliation $\Gamma$ transversal to $\bar W_{j}$ with leaves of size $\rho<\frac{L_0}{2c}$.\footnote{We work in coordinates in which $\bar W_j$ is flat, this can always be achieved by Lemma \ref{lem:c2}.}  Let $\Omega_{\beta,i}$ be the set of leaves that intersect $\Delta_{\ell-R_{j}}\cap T_{(\ell-R_{j})\vu}W_{\beta,i}$. By the construction of the covering $B^{\theta}_{cr}(x_i)$, the $\Omega_{\beta,i}$ have at most $\Cs$ overlaps. In addition, by the uniform transversality between stable and unstable direction, 
\[
\begin{split}
\sum_{(\beta,i)\in \Sigma_{\ell,j}}m(\Omega_{\beta,i})&\geq \Cs\!\!\sum_{(\beta,i)\in \Sigma_{\ell,j}}|\Delta_{\ell-R_{j}}\cap T_{(\ell-R_{j})\vu}W_{\beta,i}|\rho\\
&\geq \Cs\!\!\sum_{(\beta,i)\in \Sigma_{\ell,j}}\!\! |T_{(\ell-R_{j})\vu}W_{\beta,i}|\rho-\Cs\rho^{2}\, ,
\end{split}
\]
where we have used the estimate on the complement of $\Delta_{\ell-R_{j}}$ given by property \eqref{it:9} of the foliation.

Accordingly, for each $j$ such that $R_j<\ell$,
\begin{equation}\label{eq:common-1}
\sum_{(\beta,i)\in \Sigma_{\ell,j}}Z_{\beta,i}\leq \Cs\delta^{-1} |J_{j}|\left[\rho^{-1}\sum_{(\beta,i)\in \Sigma_{\ell,j}}m(T_{-(\ell-R_{j})\vu}\Omega_{\beta})+\rho\right]\,,
\end{equation}
where we have used the invariance of the volume associated to the contact form.
Remembering that the $T_{-(\ell-R_{j})\vu}\Omega_\beta$ have a fixed maximal number of overlaps and since they are all contained in a $\rho_{*}+r^\theta+\blambda^{-(\ell-R_j)\vu}\rho$ neighbourhood of $\widetilde O$ we have,\footnote{Remember that the $J_j$ are all disjoint , $\cup_j J_j =W^s_{\delta,\zeta}$ and $|W^s_{\delta,\zeta}|\leq \delta$.}
\begin{equation}\label{eq:common-2}
\begin{split}
\sum_j\sum_{(\beta,i)\in \Sigma_{\ell,j}}Z_{\beta,i}\leq &\sum_{\{j\;:\; R_{j}\leq \frac \ell 2\}} \Cs\delta^{-1} |J_{j}|\left[\rho^{-1}S(\rho_{*}+r^\theta)+S\blambda^{-\frac{\ell}2\vu}+\rho\right] \\
&+\sum_{\{j\;:\; R_{j}> \frac\ell 2\}}\sum_{(\beta,i)\in \Sigma_{\ell,j}}Z_{\beta,i}\, .
\end{split}
\end{equation}
By our assumption on complexity (Definition \ref{domin}), it follows that the number of pieces in $T_{-k\vu}W^{s}_{\delta}$ that have always been shorter than $L_{0}$ grows at most sub-exponentially with $k$. Remember that $\sigma\geq \blambda^{-\frac \vus2} $, see \eqref{eq:sigmachoice}. Then there exists $\ell_{0}\in\bN$ such that, the number of pieces that are never longer than $L_{0}/2$ in $k\geq \ell_{0}k_1$ time steps are bounded by $(\blambda^\vus\nu)^{\frac k{k_1}}$. Then, remembering again \eqref{eq:sigmachoice},
\[
\sum_{\{j\;:\; R_{j}>\frac  \ell 2\}}\sum_{(\beta,i)\in \Sigma_{\ell,j}}Z_{\beta,i}\leq \sum_{\{j\;:\; R_{j}> \frac \ell 2\}}\Cs\delta^{-1}|J_{j}|
\leq\Cs\sum_{k=\frac \ell 2}^\ell \delta^{-1}\nu^{k}
\leq \Cs\delta^{-1} \sigma^{\frac\ell{k_1}}\, .
\]

The result follows by choosing $\rho=S^{\frac 12}(\rho_{*}+r^\theta)^{\frac 12}$.
\end{proof}

\begin{proof}[{\bfseries Proof of Lemma} {\normalfont \ref{lem:discard}}]
We argue exactly like in the proof of Lemma \ref{lem:disco}, where $\tilde O$ is replaced by $W_{\alpha,i}$ and $\rho_{*}=r^{\vartheta}$, up to formula \eqref{eq:common-1}. At this point, we notice that $T_{-(\ell-R_{j})\vu}\Omega_\beta$ are all contained in a $r^{\vartheta}+\blambda^{-\vu\ell}\rho$ neighbourhood of $W_{\alpha,i}$. Then, arguing as in \eqref{eq:common-2}
\begin{equation}\label{eq:common-3}
\begin{split}
\sum_{(\beta,i)\in \Sigma_{\ell,j}}Z_\beta\leq &\sum_{\{j\;:\; R_{j}\leq \frac\ell 2\}} \Cs\delta^{-1} |J_{j}|\left[\rho^{-1}r^{\theta+\vartheta}+r^{\theta}\blambda^{-\ell\vu}+\rho\right] +\Cs\delta^{-1} \sigma^{\frac\ell{k_1}}\\
&\leq \Cs\left[\rho^{-1}r^{\theta+\vartheta}+\rho\right] +\Cs\delta^{-1} \sigma^{\frac\ell{k_1}}\, .
\end{split}
\end{equation}

The result follows by choosing $\rho=r^{\frac {\theta+\vartheta} 2}$.
\end{proof}

\bibliographystyle{alpha}

\end{document}